\documentclass[asterisque,english]{smfart}

\usepackage{subfiles}

\usepackage[english,french]{babel}
\usepackage{graphicx}
\usepackage{amsmath} 
\usepackage[all]{xy}
\usepackage{amsthm}
\usepackage{amsfonts}
\usepackage{amssymb}
\usepackage{mathrsfs}
\usepackage{amsbsy}
\usepackage{mathrsfs}
\usepackage{stmaryrd}
\usepackage{amsmath}
\usepackage{mathtools}
\usepackage[utf8]{inputenc}
\usepackage{calligra}
\usepackage{mathtools}
\usepackage{aurical}
\usepackage[T1]{fontenc}

\makeatletter
\DeclareFontFamily{OMX}{MnSymbolE}{}
\DeclareSymbolFont{MnLargeSymbols}{OMX}{MnSymbolE}{m}{n}
\SetSymbolFont{MnLargeSymbols}{bold}{OMX}{MnSymbolE}{b}{n}
\DeclareFontShape{OMX}{MnSymbolE}{m}{n}{
    <-6>  MnSymbolE5
   <6-7>  MnSymbolE6
   <7-8>  MnSymbolE7
   <8-9>  MnSymbolE8
   <9-10> MnSymbolE9
  <10-12> MnSymbolE10
  <12->   MnSymbolE12
}{}
\let\llangle\@undefined
\let\rrangle\@undefined
\DeclareMathDelimiter{\llangle}{\mathopen}%
                     {MnLargeSymbols}{'164}{MnLargeSymbols}{'164}
\DeclareMathDelimiter{\rrangle}{\mathclose}%
                     {MnLargeSymbols}{'171}{MnLargeSymbols}{'171}
\makeatother

\usepackage[usenames,dvipsnames]{color}
\usepackage{color}

\RequirePackage{color}
\definecolor{bwmagenta}{rgb}{0.0,0.45,0.6}

\definecolor{bwblue}{rgb}{0.4,0.1,0.2}

\usepackage[usenames,dvipsnames]{xcolor}

\usepackage[bookmarks,colorlinks=true,pagebackref,final,breaklinks=true,pdfsubject={LaTeX}]{hyperref}
\hypersetup{
        bookmarksnumbered=true,
        linkcolor=bwblue,
        citecolor=bwmagenta,
}

\usepackage{calligra}
\makeatletter
\def\@splitop#1#2\@nil{$\mathscr{#1}\!\!$\calligra#2\,\,}
\newcommand*\DeclareCursiveOperator[2]{%
  \newcommand#1{\mathop{\mbox{\@splitop#2\@nil}}\nolimits}}
\makeatother
\DeclareCursiveOperator{\TAY}{Hess}
\DeclareCursiveOperator{\HOM}{Hom}
\DeclareCursiveOperator{\Det}{D{}et}

\newtheorem{definition}{Definition}[]

\newtheorem{theo}{Theorem}[section]
\newtheorem{cor}[theo]{Corollary}

\newtheorem{lemma}[theo]{Lemma}
\newtheorem{remark}[theo]{Remark}

\newtheorem{proposition}[theo]{Proposition}

\newtheorem{assumption}[theo]{Assumption}
\newtheorem{theorem}{Theorem}[]

\newcommand{\quo}[1]{ \mathbf{Z}/p^{n}\mathbf{Z}  }

\usepackage[OT2,T1]{fontenc}
\DeclareSymbolFont{cyrletters}{OT2}{wncyr}{m}{n}
\DeclareMathSymbol{\sha}{\mathalpha}{cyrletters}{"58}

\newcommand{\lfre}[1]{\stackrel{#1}{\longrightarrow}}

\newcommand{\N}{\mathbf{N}}

\newcommand{\ges}{\geqslant}

\newcommand{\hlfre}[1]{\stackrel{#1}{\lhook\joinrel\relbar\joinrel\rightarrow}}

\author{F. Andreatta, M. Bertolini, M. Seveso, R. Venerucci}
\title{Syntomic formalism with coefficients}

\begin{document}
\maketitle

\begin{abstract} This paper provides the technical tools needed in ongoing work of the authors to compute $p$-adic \'etale Abel-Jacobi maps in order to obtain explicit reciprocity laws for $\mathrm{GSp}_4$. In particular, we define and study syntomic polynomial cohomology for filtered Frobenius log-isocrystals over proper and semistable schemes over the ring of integers of a local field, with smooth generic fiber, endowed with horizontal divisors. We introduce syntomic polynomial cohomology with support, we define Kunneth morphisms, trace maps and cup products,  Gysin maps with respect to divisors and we study some of their properties. We establish the relation with Hyodo-Kato cohomology of the special fiber and de Rham cohomology of the generic fiber. We also introduce overconvergent variants with and without support by restricting to open smooth formal subschemes. Most of all, in case that the filtered log-isocrystal is associated to a $p$-adic local system on the generic fiber, we establish comparison morphisms between \'etale and syntomic cohomology and compatibilities with Hochschild-Serre morphisms and between Gysin morphisms.

\end{abstract}

\tableofcontents

\section{The geometric setup}

\subsection{Notation}\label{sub:not}

Let $K$ be a finite extension of $\mathbb{Q}_p$ with $p$ odd, with ring of integers $\mathcal{O}_K$  and with  residue field $k$. We fix an algebraic closure $K\subset \overline{K}$ and we denote by $K_0\subset K$ the maximal absolutely unramified subfield. Let $\pi$ be a uniformizer of $\mathcal{O}_K$. Consider  the $\mathcal{O}_{K_0}$-algebra homomorphism $\gamma_\pi\colon \mathcal{O}_{K_0}\{Z\}\to \mathcal{O}_{K}$ sending $Z$ to $\pi$. Write
$P_\pi(Z)\in \mathcal{O}_{K_0}[Z]$ for the monic irreducible polynomial of $\pi$. It is a generator of the kernel of $\gamma_\pi$. We denote by $\mathcal{S}_k:=\mathrm{Spec}(k)$, $\mathcal{S}_0:=\mathrm{Spf}(\mathcal{O}_{K_0})$, 
$\mathcal{S}:=\mathrm{Spf}(\mathcal{O}_{K})$, $\overline{\mathcal{S}}:=\mathrm{Spf}(\mathcal{O}_{K}/p \mathcal{O}_K)$ and $\mathfrak{S}:=\mathrm{Spf}(\mathcal{O}_{K_0}[\![Z]\!])$ with log-structures $\mathcal{M}_k$, $\mathcal{M}_0$, $\mathcal{M}$, $\overline{\mathcal{M}}$ and $\mathfrak{M}$ on $\mathcal{S}_k$, $\mathcal{S}_0$, $\mathcal{S}$, $\overline{\mathcal{S}}$ and $\mathfrak{S}$ respectively associated to the prelog-structures $\N\longrightarrow  k$, $\N\longrightarrow \mathcal{O}_{K_0} $, $\N\longrightarrow \mathcal{O}_K $, $\N\longrightarrow \mathcal{O}_K/p \mathcal{O}_K $ and  $\N\longrightarrow \mathcal{O}_{K_0}[\![Z]\!]$ sending $0\neq n\in \N$ to $0$ (in $k$, resp.~in $\mathcal{O}_{K_0}$), $\pi^n$, $\pi^n$ (modulo $p$) and  $Z^n$ respectively. We have a strict closed immersion $j\colon \bigl(\mathcal{S},\mathcal{M}\bigr)\to (\mathfrak{S},\mathfrak{M})$. We write  $(\mathfrak{S}^{\rm DP},\mathfrak{M}^{\mathrm{DP}})$ for the $p$-adic formal log scheme defined by taking the $p$-adic completion of the DP envelope of $j$. We summarize the relation among these objects:

$$
\begin{array}{cccccccccccc}
(\mathcal{S}_k,\mathcal{M}_k) & \longrightarrow  &(\overline{\mathcal{S}},\overline{\mathcal{M}}) & \longrightarrow &   (\mathcal{S},\mathcal{M}) & \longrightarrow &  (\mathfrak{S}^{\rm DP},\mathfrak{M}^{\mathrm{DP}})\\
\downarrow & & \downarrow&& \downarrow & \swarrow \\
\mathcal{S}_k & = & \mathrm{Spec}(k) &\longrightarrow& \mathcal{S}_0 .
\end{array}
$$For simplicity we write  $\mathcal{S}_0^0$  for $(\mathcal{S}_0,\mathcal{M}_0)$.

\subsection{Setting}\label{sec:setting}

We recall the conventions and notations of \cite[\S 2]{LLZ}.  Let $X$ be a smooth rigid analytic variety over $\mathrm{Spa}(K,\mathcal{O}_K)$ and let $D\subset X$ be a strict normal crossing divisor with finitely many irreducible components $\{D_j\}_{j\in I}$ such that for every $J\subset I$ $$X_J:=\cap_{j\in J} D_j$$ is smooth. We let $\alpha_X\colon \mathcal{M}_X\to \mathcal{O}_{X}$ be the log-structure on $X$ defined by $D$. For every $J\subset I$ we let $X_J^\partial$ (resp.~$(X_J,\alpha_{X_J})$) be the analytic variety $X_J$ endowed with the log-structure obtained from $\mathcal{M}_X$ by pull-back via the closed immersion $\iota_J\colon X_J\to X$ (resp.~the log-structure defined by $X_J\cap \bigl(\cup_{i\in I\backslash J} D_i\bigr)$).  For every $a\in \mathbb{N}$ we let $X_{(a)}^\partial:=\amalg_{J\subset I, \vert J\vert=a} X_J^\partial$ and $\iota_{(a)}:=\amalg_{J\subset I, \vert J\vert=a} \iota_J$. We define similarly $(X_{(a)},\alpha_{X_{(a)}})$.

We denote by $X_{\mathrm{et}}$, $X_{\mathrm{ket}}$, $X_{\mathrm{pket}}$ the \'etale, Kummer-\'etale, pro-Kummer-\'etale sites associated to $X$ and $\alpha_X$. Let $\nu\colon X_{\overline{K},\mathrm{ket}} \to X_{\mathrm{ket}}$ be the profinite \'etale cover obtained from the inclusion $K\subset \overline{K}$. Let  $X_{\overline{K},\mathrm{pket}}$ be the localization of $X_{\mathrm{pket}}$ over $X_{\overline{K},\mathrm{pket}}$. Define $\nu_{\mathrm{geo}}\colon X_{\overline{K},\mathrm{pket}} \to X_{\overline{K},\mathrm{ket}}$ to be the corresponding morphism of sites. We also have the morphism of sites $\nu_{\mathrm{ar}}\colon X_{\overline{K},\mathrm{pket}}\to X_{\mathrm{ket}}$ given by $\nu_{\mathrm{ar}}=\nu\circ \nu_{\mathrm{geo}}$

In order to have a good crystalline theory we  strengthen the assumptions of \cite{LLZ} and we require the following:

\begin{assumption}\label{ass:formalcase} The rigid analytic spaces $X$ and $D_j$ for every $j\in I$ are the rigid analytic generic fibers of a formal scheme $\mathcal{X}$ and of Cartier divisors  $\mathcal{D}_j\subset \mathcal{X}$ relative to $\mathcal{O}_K$ such that, setting $\mathcal{D}:=\cup_{j\in I} \mathcal{D}_j$, then  $\bigl(\mathcal{X},\mathcal{D}\bigr)$ is a strict semistable pair, i.e., $\mathcal{X}$ is strict semistable over $\mathcal{O}_K$, $\mathcal{D}\subset \mathcal{X}$ is a strict normal crossing divisor and  for every $J\subset I$ we have that $\mathcal{X}_J:=\cap_{j\in J} \mathcal{D}_j$ is strict semistable over $\mathcal{O}_K$.  We let $\alpha_{\mathcal{X}}\colon \mathcal{M}_\mathcal{X}\to \mathcal{O}_{\mathcal{X}}$ be the fs log-structure on $\mathcal{X}$ defined by the base change $\mathcal{X}_k$ of $\mathcal{X}$ to $k$ and by $\mathcal{D}$. 

\end{assumption}

We have a formally log-smooth morphism of $p$-adic formal log-schemes $$f\colon (\mathcal{X},\alpha_\mathcal{X}) \longrightarrow (\mathcal{S},\mathcal{M}).$$Let $\overline{f}\colon (\overline{\mathcal{X}},\alpha_{\overline{\mathcal{X}}}) \longrightarrow (\overline{\mathcal{S}},\overline{\mathcal{M}}) $ be the mod $p$ reduction of $f$ and $f_k\colon (\mathcal{X}_k,\alpha_{\mathcal{X}_k}) \longrightarrow (\mathcal{S}_k,\mathcal{M}_k)$ be the reduction of $f$ modulo the maximal ideal of $\mathcal{O}_K$. For simplicity we might write $\alpha_k:=\alpha_{\mathcal{X}_k}$.

For every $J\subset I$ we let $\mathcal{X}_J^\partial$ (resp.~$(\mathcal{X}_J,\alpha_{\mathcal{X}_J})$) be the formal scheme $\mathcal{X}_J:=\cap_{j\in J} \mathcal{D}_j$ endowed with the log-structure $\mathcal{M}_{\mathcal{X}_J^\partial} $ obtained from $\mathcal{M}_{\mathcal{X}}$ by pull-back via the closed immersion $\iota_J\colon \mathcal{X}_J\to \mathcal{X}$ (resp.~the log-structure defined by $\mathcal{X}_{J,k}$ and the divisor $\mathcal{X}_J\cap \bigl(\cup_{i\in I\backslash J} \mathcal{D}_i\bigr)$). For every $a\in \mathbb{N}$ we let $\mathcal{X}_{(a)}^\partial:=\amalg_{J\subset I, \vert J\vert=a} \mathcal{X}_J^\partial$  and we define also the morphism $\iota_{(a)}:=\amalg_{J\subset I, \vert J\vert=a} \iota_J\colon \mathcal{X}_{(a)}^\partial\to \mathcal{X}$. We have an analogous notation  $(\mathcal{X}_{(a)},\alpha_{\mathcal{X}_{(a)}})$.

\section{Frobenius log-isocrystals}\label{sec:FrobIso}

We define a non-degenerate filtered Frobenius log-crystal relative to $(X,\alpha_X,\overline{\mathcal{X}},\alpha_{\overline{\mathcal{X}}})$ to be aquintuple  $ \bigl(\mathcal{E},\mathrm{Fil}^\bullet \mathcal{E}, \nabla,\mathfrak{E},\Phi_{\mathfrak{E}}\bigr)$ where:

\medskip
\begin{itemize}

\item[i.] $\mathcal{E}$ is a locally free $\mathcal{O}_X$-module, endowed with an integrable logarithmic connection $\nabla$ and a descending filtration $\mathrm{Fil}^\bullet \mathcal{E}$ by locally free $\mathcal{O}_X$-modules such that Griffiths' transversality holds and the graded pieces are locally free $\mathcal{O}_X$-modules as well.

% We require that the Hodge- de Rham spectral sequence degenerates and that the logarithmic connection has nilpotent residues. DA AGGIUNGERE nel caso proprio

\item[ii.]  $\bigl(\mathfrak{E},\Phi_{\mathfrak{E}}\bigr)$ defines a non-degenerate Frobenius log-crystal in locally free modules on $\bigl(\overline{\mathcal{X}},\alpha_{\overline{\mathcal{X}}}\bigr)$ relative to the PD thickening  $\mathrm{Spec}(k) \to \mathcal{S}_0$ (with trivial log-structure, considering $ (\overline{\mathcal{X}},\alpha_{\overline{\mathcal{X}}}) $ as a log-scheme over  $k=\mathcal{O}_{K_0}/p\mathcal{O}_{K_0}$).

\end{itemize}
\smallskip
The non-degeneracy means the following. Let $\varphi^\ast(\mathfrak{E})$ be the log-crystal defined by pull-back via the Frobenius morphism $\varphi$ on  $\bigl(\overline{\mathcal{X}},\alpha_{\overline{\mathcal{X}}}\bigr)$ relative to $\mathrm{Spec}(k)$. We require there exists a non-negative integer $h$ and a morphism $V_{\mathfrak{E}}\colon \mathfrak{E} \to \varphi^\ast(\mathfrak{E})$ such that $V_{\mathfrak{E}} \circ \Phi_{\mathfrak{E}}$ and $\Phi_{\mathfrak{E}}\circ V_{\mathfrak{E}} $ are multiplication by $p^h$.  

In the definition we also require the following.  Write $\mathfrak{E}_{\mathcal{O}_K}$ for the log-crystal on $\bigl(\overline{\mathcal{X}},\alpha_{\overline{\mathcal{X}}}\bigr)$ relative to $(\overline{\mathcal{S}},\overline{\mathcal{M}} )\to  (\mathcal{S},\mathcal{M})$ defined by  the pull-back of $\mathfrak{E}$  via  the map $ (\mathcal{S},\mathcal{M}) \to \mathcal{S}_0$. We assume:
\smallskip

\begin{itemize}
\item[iii.]    The evaluation of $\mathfrak{E}_{\mathcal{O}_K} $ at the $p$-adic DP thickening $\bigl(\overline{\mathcal{X}},\alpha_{\overline{\mathcal{X}}}\bigr) \subset \bigl( \mathcal{X},\alpha_{\mathcal{X}}\bigr)$ defines by passing to the generic fiber $(X,\alpha_X)$ the module with log-connection  $(\mathcal{E},\nabla)$.

\end{itemize}

Non-degenerate Frobenius log-crystals relative to $(X,\alpha_X,\overline{\mathcal{X}},\alpha_{\overline{\mathcal{X}}})$ form a category in which we have tensor products and an identity element $\mathbf{1}$  given by the triple $ (\mathcal{O}_X, \mathcal{O}_{\mathcal{X}_k}^{\rm crys}, \varphi )$ with $\mathcal{O}_X$ with trivial filtration and connection given by differentiation and $\mathcal{O}_{\mathcal{X}_k}^{\rm crys}$ the structure sheaf crystal with its Frobenius morphism.

\begin{remark}\label{rmk:dualisocrystal} If one forgets Frobenius and works in the resulting isogeny category, we get a category that we call the category of filtered log-isocrystals relative to $(X,\alpha_X,\overline{\mathcal{X}},\alpha_{\overline{\mathcal{X}}})$ in which we also have duals. It is given by considering the isogeny category of  quadruples $ \bigl(\mathcal{E},\mathrm{Fil}^\bullet \mathcal{E}, \nabla,\mathfrak{E}\bigr)$ satisfying (i) and (iii) above, where we only assume that $\mathfrak{E}$ is a log-crystal in locally free modules. By a non-degenerate Frobenius morphism we mean a morphism  $\Phi_{\mathfrak{E}}\colon \varphi^\ast(\mathfrak{E}) \to \mathfrak{E}$ in the isogeny category such that there exists an $r$ such that $\Phi_{\mathfrak{E}}':=p^r \Phi_{\mathfrak{E}}$ arises from a non-degenerate Frobenius log-crystal $ \bigl(\mathcal{E},\mathrm{Fil}^\bullet \mathcal{E}, \nabla,\mathfrak{E}, \Phi_{\mathfrak{E}}'\bigr)$, i.e., $\Phi_{\mathfrak{E}}'$ satisfies (ii). A non-degenerate filtered  Frobenius log-isocrystal relative to $(X,\alpha_X,\overline{\mathcal{X}},\alpha_{\overline{\mathcal{X}}})$ is a quintuple $ \bigl(\mathcal{E},\mathrm{Fil}^\bullet \mathcal{E}, \nabla,\mathfrak{E},\Phi_{\mathfrak{E}}\bigr)$ such that $ \bigl(\mathcal{E},\mathrm{Fil}^\bullet \mathcal{E}, \nabla,\mathfrak{E}\bigr)$ is a log-isocrystal and $\Phi_{\mathfrak{E}}$ is a non-degenerate Frobenius. In this category we have duals. The dual of $ \bigl(\mathcal{E},\mathrm{Fil}^\bullet \mathcal{E}, \nabla,\mathfrak{E},\Phi_{\mathfrak{E}}\bigr)$ is defined by the dual $\mathcal{E}^\vee$ with the dual filtration  and dual connection, by  the dual crystal $\mathfrak{E}^\vee$ and Frobenius $\Phi_{\mathfrak{E}^\vee}$ given as follows. If $V_{\mathfrak{E}}'$ is such that  $V_{\mathfrak{E}}' \circ \Phi_{\mathfrak{E}}'$ and $\Phi_{\mathfrak{E}}'\circ V_{\mathfrak{E}}' $ are multiplication by $p^h$ set $\Phi_{\mathfrak{E}^\vee}:=V_{\mathfrak{E}}^{',\vee}/p^{h-r}$.

We will usually state our results for non-degenerate Frobenius log-crystals. However, all the notions we are going to introduce and the results we are going to prove have their analogue for non-degenerate Frobenius log-isocrystals as well. On the other hand, when discussing dualities, we will need to consider the dual objects and, hence, we will be forced to work with non-degenerate Frobenius log-isocrystals.
\end{remark}

Using the Frobenius structure we can upgrade Frobenius crystals to convergent Frobenius isocrystals:

\begin{proposition}\label{prop:Econv}
The  Frobenius crystal $\bigl(\mathfrak{E},\Phi_{\mathfrak{E}}\bigr)$  defines a convergent Frobenius isocrystal  $\bigl(\mathfrak{E}^{\rm conv},\Phi_{\mathfrak{E}^{\rm conv}}\bigr)$ on the log-convergent site of $\bigl(\mathcal{X}_k,\alpha_{\mathcal{X}_k}\bigr)$ for the Zariski topology, relative to the  thickening  $\mathrm{Spec}(k) \to \mathcal{S}_0$ in the sense of \cite[Def.~2.1.5]{Shiho}.
\end{proposition}
\begin{proof} 
Consider a Zariski open $U \subset \overline{\mathcal{X}}$ and set $\alpha_U:= \alpha_{\overline{\mathcal{X}}}\vert_U$. Consider an exact closed immersion 

$$
\begin{array}{cccccccccc}
(U, \alpha_U) & \longrightarrow &   (P,L) \\
\downarrow&& \downarrow \\
\mathrm{Spec}(k) &\longrightarrow& \mathcal{S}_0 .\\
\end{array}
$$with $(P,L)\to \mathcal{S}_0$ formally log-smooth endowed with a lift of Frobenius $\varphi_P$. Let $\mathcal{O}_P^{\rm DP}$ be the structure sheaf on $U$ of the $p$-adic DP envelope of $U\subset P$. Then,  $\mathfrak{E}$ defines a locally free $\mathcal{O}_P^{\rm DP}$-module $\mathfrak{E}_P$ endowed with a logarithmic integrable connection that extends to an HPD stratification on $(U, \alpha_U)$ relative to $(P,L)$; see the discussion in \cite[\S 4.3]{ShihoI}. The Frobenius structure provides a morphism $\Phi_{\mathfrak{E}_P}\colon \varphi_P^\ast(\mathfrak{E}_P) \to \mathfrak{E}_P$, which is an isomorphism if we invert $p$. Arguing as in \cite[\S 2.4.1]{Berthelot}  one deduces that the non-degenerate Frobenius log-crystal $\bigl(\mathfrak{E}_P,\Phi_{\mathfrak{E}_P}\bigr)$ defines a module endowed with an integrable logarithmic connection,  extending to a stratification, and an horizontal Frobenius $(\mathfrak{E}_P^{\rm conv}, \Phi_{\mathfrak{E}_P^{\rm conv}})$ on the tubular neighborhood $]U[^{\rm log}_P:=](U, \alpha_U) [_{(P,L)}$. Recall from \cite[Def. 2.2.5]{Shiho} that $]U[^{\rm log}_P$ is  the rigid analytic fiber of the formal completion of $P$ along $U$. Since  $U_k\subset U$ is a nilpotent thickening, the tubular neighborhoods $]U_k[^{\rm log}_P$ and $]U[^{\rm log}_P$ coincide. Hence,  $(\mathfrak{E}_P^{\rm conv}, \Phi_{\mathfrak{E}_P^{\rm conv}})$ defines a convergent Frobenius isocrystal  $\bigl(\mathfrak{E}^{\rm conv},\Phi_{\mathfrak{E}^{\rm conv}}\bigr)\vert_{U}$   on the log-convergent site of $\bigl(U_k,\alpha_{\mathcal{X}_k}\vert_{U_k}\bigr)$ relative to the  thickening  $\mathrm{Spec}(k) \to \mathcal{S}_0$ by \cite[Prop. 2.2.7]{Shiho}. Since  convergent Frobenius isocrystals satisfy Zariski descent and for varying $U$'s the constructions above glue,  we can glue and get the sought for  convergent Frobenius isocrystal $\bigl(\mathfrak{E}^{\rm conv},\Phi_{\mathfrak{E}^{\rm conv}}\bigr)$.
\end{proof}

In \cite[Def. 2.2.12]{Shiho} the author defines complexes $\mathrm{R}\Gamma_{\mathrm{an}}\bigl((\mathcal{X}_k,\alpha_{\mathcal{X}_k})/\mathcal{S}_0,\mathfrak{E}^{\rm conv}\bigr)$ computing the analytic cohomology of $\mathfrak{E}^{\rm conv}$. We sometimes omit $\alpha_{\mathcal{X}_k}$ from the notation for simplicity. The Frobenius defines $\sigma$-linear isomorphisms $\Phi$. We have a natural map, compatible with Frobenii, in the derived category
$$\tau\colon \mathrm{R}\Gamma_{\mathrm{an}}\bigl((\mathcal{X}_k,\alpha_{\mathcal{X}_k})/\mathcal{S}_0,\mathfrak{E}^{\rm conv}\bigr) \longrightarrow \mathrm{R}\Gamma_{\mathrm{crys}}\bigl((\overline{\mathcal{X}},\alpha_{\overline{\mathcal{X}}})/\mathcal{S}_0,\mathfrak{E}\bigr) [p^{-1}].$$By our assumption (iii) above we also have a morphism $$\gamma_{\rm crys}\colon  \mathrm{R}\Gamma_{\mathrm{crys}}\bigl((\overline{\mathcal{X}},\alpha_{\overline{\mathcal{X}}})/\mathcal{S}_0,\mathfrak{E}\bigr) \longrightarrow \mathrm{R}\Gamma_{\mathrm{logdR}}\bigl(X, \mathcal{E}\bigr);$$see \S \ref{sssec:explcictcrystallne}.  Composing with $\gamma_{\rm crys}$ we get a natural morphism 
$$\gamma_{\rm an}=\gamma_{\rm crys}\circ \tau \colon \mathrm{R}\Gamma_{\mathrm{an}}\bigl((\mathcal{X}_k,\alpha_{\mathcal{X}_k})/K_0,\mathfrak{E}^{\rm conv}\bigr) \longrightarrow  \mathrm{R}\Gamma_{\mathrm{logdR}}\bigl(X, \mathcal{E}\bigr).$$Finally, we recall the following:

\begin{definition}\label{def:nilpotentconnection} Let $\mathcal{E}$ be a locally free $\mathcal{O}_X$-module, endowed with an integrable logarithmic connection $\nabla$. We say that $\nabla$ has nilptotent residues if for every $j$ the residue map $$\mathrm{Res}(\nabla)\vert_{D_j}\colon \mathcal{E}\vert_{D_j}\longrightarrow \mathcal{E}_{D_j}$$is nilpotent. We recall that it is the morphism of $\mathcal{O}_{D_J}$ modules defined by composing $\nabla\colon \mathcal{E}\to \mathcal{E}\otimes_{\mathcal{O}_X} \Omega^1_X$    with the base change  $\mathcal{E}\otimes_{\mathcal{O}_X}$ of the Poincar\'e residue map $\Omega^1_X \to \mathcal{O}_{D_j}$ along $D_J$ (see \cite[Ex. 1.2.13 \& Prop. 1.2.14]{OgusBook}).

\end{definition}

\section{The syntomic complex}

Consider a non-degenerate Frobenius log-crystal  $ \bigl(\mathcal{E},\mathrm{Fil}^\bullet \mathcal{E}, \nabla,\mathfrak{E},\Phi_{\mathfrak{E}}\bigr)$
relative to $(X,\alpha_X,\overline{\mathcal{X}},\alpha_{\overline{\mathcal{X}}})$.

\begin{definition}\label{def:syntomiccomplex} We define the (analytic) syntomic complex $\mathbf{Syn}_{\rm an}(\mathcal{E},\mathfrak{E},r)$ or simply $\mathbf{Syn}(\mathcal{E},\mathfrak{E},r)$ associated to $ \bigl(\mathcal{E},\mathrm{Fil}^\bullet \mathcal{E}, \nabla,\mathfrak{E},\Phi_{\mathfrak{E}}\bigr)$ to be the total complex associated to the morphism  of complexes $(1-\frac{\Phi}{p^r},\gamma_{\rm an})$:

$$\mathrm{R}\Gamma_{\mathrm{an}}\bigl((\mathcal{X}_k,\alpha_k)/\mathcal{S}_0,\mathfrak{E}^{\rm conv}\bigr) \longrightarrow \mathrm{R}\Gamma_{\mathrm{an}}\bigl((\mathcal{X}_k,\alpha_k)/\mathcal{S}_0,\mathfrak{E}^{\rm conv}\bigr) \oplus \frac{ \mathrm{R}\Gamma_{\mathrm{logdR}}\bigl(X, \mathcal{E}\bigr) }{\mathrm{Fil}^r \mathrm{R}\Gamma_{\mathrm{logdR}}\bigl(X, \mathcal{E}\bigr)}.$$

\end{definition}

Here $\mathrm{Fil}^r \mathrm{R}\Gamma_{\mathrm{logdR}}\bigl(X, \mathcal{E}\bigr)$ is defined upon passing to a Chech cover by the subcomplex $\mathrm{Fil}^{r-\bullet} \mathcal{E} \otimes_{\mathcal{O}_X}  \Omega_{X/K}^{\log,\bullet}$ of $\mathcal{E} \otimes_{\mathcal{O}_X}  \Omega_{X/K}^{\log,\bullet}$. We denote by $\mathrm{Gr}^r \bigl(\mathcal{E} \otimes_{\mathcal{O}_X}  \Omega_{X/K}^{\log,\bullet}\bigr)$ the quotient of  $\mathcal{E} \otimes_{\mathcal{O}_X}  \Omega_{X/K}^{\log,\bullet}$
by  $\mathrm{Fil}^r \mathrm{R}\Gamma_{\mathrm{logdR}}\bigl(X, \mathcal{E}\bigr)$.

\begin{proposition} \label{prop:spectralanal}
We have an exact sequence
$$ 0\longrightarrow F^1\mathrm{H}^i\bigl(\mathbf{Syn}_{\rm an}(\mathcal{E},\mathfrak{E},r)\bigr) \longrightarrow  \mathrm{H}^i\bigl(\mathbf{Syn}_{\rm an}(\mathcal{E},\mathfrak{E},r)\bigr) \longrightarrow  F^0 \mathrm{H}^i\bigl(\mathbf{Syn}_{\rm an}(\mathcal{E},\mathfrak{E},r)\bigr) \longrightarrow 0$$with $ F^1\mathrm{H}^i\bigl(\mathbf{Syn}_{\rm an}(\mathcal{E},\mathfrak{E},r)\bigr)$ isomorphic to the cokernel of the map $(1-\frac{\Phi}{p^r})  \oplus \gamma^{\rm an}$ $$\mathrm{H}^{i-1}_{\mathrm{an}}\bigl((\mathcal{X}_k,\alpha_k)/K_0,\mathfrak{E}^{\rm conv}\bigr) \longrightarrow \mathrm{H}^{i-1}_{\mathrm{an}}\bigl((\mathcal{X}_k,\alpha_k)/K_0,\mathfrak{E}^{\rm conv}\bigr) \oplus H^{i-1}\bigl(\mathrm{Gr}^r \bigl(\mathcal{E} \otimes_{\mathcal{O}_X}  \Omega_{X/K}^{\log,\bullet}\bigr)\bigr)  $$and $F^0 \mathrm{H}^i\bigl(\mathbf{Syn}_{\rm an}(\mathcal{E},\mathfrak{E},r)\bigr)$ isomorphic to the kernel of the map   $\bigl(1-\frac{\Phi}{p^r}\bigr) \oplus \gamma^{\rm an}$ $$\mathrm{H}^i_{\mathrm{an}}\bigl((\mathcal{X}_k,\alpha_k)/K_0,\mathfrak{E}^{\rm conv}\bigr) \longrightarrow \mathrm{H}^i_{\mathrm{an}}\bigl((\mathcal{X}_k,\alpha_k)/K_0,\mathfrak{E}^{\rm conv}\bigr) \oplus H^{i}\bigl(\mathrm{Gr}^r \bigl(\mathcal{E} \otimes_{\mathcal{O}_X}  \Omega_{X/K}^{\log,\bullet}\bigr)\bigr).$$
When $X$ is proper over $K$ and the Hodge-de Rham spectral sequence for $\mathcal{E}$ degenerates, then we can replace $H^{j}\bigl(\mathrm{Gr}^r \bigl(\mathcal{E} \otimes_{\mathcal{O}_X}  \Omega_{X/K}^{\log,\bullet}\bigr)\bigr)$ with $\mathrm{Gr}^r \mathrm{H}^{j}_{\mathrm{logdR}}(X,\mathcal{E})$ for $j \in \{i-1,i\}$ in the above displayed equations.

\end{proposition}
\begin{proof} The result is a formal consequence of the definition, taking into account that, under the hypotheses above, the map  $$\mathrm{H}^i_{\mathrm{dR}}(X,\mathcal{E})/ \mathrm{Fil}^r \mathrm{H}^i_{\mathrm{dR}}(X,\mathcal{E}) \to \mathrm{H}^i_{\mathrm{dR}}(X,\mathrm{Gr}^r \mathcal{E}) $$ is an isomorphism for every $r$.
\end{proof}

\begin{remark}\label{rmk:crysanal}  We also have a crystalline syntomic complex $\mathbf{Syn}_{\rm crys}(\mathcal{E},\mathfrak{E},r)$ defined replacing $\mathrm{R}\Gamma_{\mathrm{an}}\bigl((\mathcal{X}_k,\alpha_k)/K_0,\mathfrak{E}^{\rm conv}\bigr)$ with  $\mathrm{R}\Gamma_{\mathrm{crys}}\bigl((\overline{\mathcal{X}},\alpha_{\overline{\mathcal{X}}})/\mathcal{S}_0,\mathfrak{E}\bigr) [p^{-1}]$ and $\gamma_{\rm an}$ with $\gamma_{\rm crys}$. Its cohomology $\mathrm{H}^i\bigl(\mathbf{Syn}_{\rm crys}(\mathcal{E},\mathfrak{E},r)\bigr)$ admits a description via an exact sequence as in Proposition \ref{prop:spectralanal} above.  The two complexes are related by a morphism $$\tau_{\rm crys}\colon\mathbf{Syn}_{\rm an}(\mathcal{E},\mathfrak{E},r)\longrightarrow \mathbf{Syn}_{\rm crys}(\mathcal{E},\mathfrak{E},r) $$which induces a morphism on cohomology $$\mathrm{H}^i\bigl(\mathbf{Syn}_{\rm an}(\mathcal{E},\mathfrak{E},r)\bigr) \to \mathrm{H}^i\bigl(\mathbf{Syn}_{\rm crys}(\mathcal{E},\mathfrak{E},r)\bigr). $$
Note also that $\tau_{\rm crys}$ is an isomorphism in cohomology and, hence, the resulting syntomic cohomologies are the same when $K=K_0$ by \cite[Thm.~3.1.1]{Shiho}, while in general the analytic cohomology computes the convergent cohomology by \cite[Thm.~2.3.9]{Shiho}.

\end{remark}

\section{Explicit complexes}

\subsection{Chech covers}\label{sec:Chech}

We fix $\mathcal{X}_\beta\subset \mathcal{X}$, with $\beta\in B$ varying in a  finite totally ordered set $B$, an open affine cover of $\mathcal{X}$ such that 

\begin{itemize}

\item[i.] $\mathcal{X}_\beta=\mathrm{Spf}(R_\beta)$ is formally \'etale over  $\mathrm{Spf}(R_\beta^0)$ with $$R_\beta^0:=\mathcal{O}_{K}\{X_1,\ldots,X_a,Y_1,\ldots,Y_b\}/(X_1\cdots X_a-\pi) $$with chart $$\mathbb{N}^{a+b+1}  \to R_\beta^0,\quad (n,\alpha_1,\ldots,\alpha_a,\beta_1,\ldots,\beta_b)\mapsto \pi^n X_1^{\alpha_1}\cdots X_a^{\alpha_a}Y_1^{\beta_1}\cdots Y_b^{\beta_b};$$

\item[ii.] the divisor $\mathcal{D}_\beta=\mathcal{X}_\beta\cap \mathcal{D}$ is defined  by the equation $Y_1\cdots Y_b=0$

\end{itemize}

We write $X_\beta\subset X$ for the induced open. We take the unique formally \'etale lift $\mathfrak{X}_\beta=\mathrm{Spf}(\mathfrak{R}_\beta)$ over  $\mathrm{Spf}(\mathfrak{R}_\beta^0)$ with $$\mathfrak{R}_\beta^0:=\mathcal{O}_{K_0}[\![Z]\!]\{X_1,\ldots,X_a,Y_1,\ldots,Y_b\}/(X_1\cdots X_a-Z),$$endowed with the $(p,Z)$-adic topology, with chart $$\mathbb{N}^{a+b+1}  \to \mathfrak{R}_\beta^0,\quad (n,\alpha_1,\ldots,\alpha_a,\beta_1,\ldots,\beta_b)\mapsto Z^n X_1^{\alpha_1}\cdots X_a^{\alpha_a}Y_1^{\beta_1}\cdots Y_b^{\beta_b}.$$We denote by $\widetilde{\alpha}_\beta\colon \mathfrak{M}_{\mathfrak{X}_\beta}\to \mathcal{O}_{\mathfrak{X}_\beta}$ the induced  log-structure on $\mathfrak{X}_\beta$. We have the following cartesian diagram of formally log-smooth morphisms
$$
\begin{array}{cccccccccc}
(\mathcal{X}_\beta,\mathcal{M}_{\mathcal{X}_\beta}) &\longrightarrow & (\mathfrak{X}_\beta,\mathfrak{M}_{\mathfrak{X}_\beta})\\
f_\beta\downarrow&& \downarrow\mathfrak{f}_\beta\\
(\mathcal{S},\mathcal{M}) &\longrightarrow& (\mathfrak{S},\mathfrak{M}),
\end{array}
$$sending $X_i\mapsto X_i$, $Y_j\mapsto Y_j$ and $Z\mapsto \pi$. The divisor $Y_1\cdots Y_b=0$ defines a Cartier divisor $\mathfrak{D}_\beta$ on $\mathfrak{X}_\beta$ whose base change to $\mathcal{S}$ is $\mathcal{D}_\beta$. For each $j\in I$ we get  Cartier divisors $\mathfrak{D}_{\beta,j}\subset \mathfrak{X}_\beta$ relative to $\mathcal{O}_{K_0}$   such that
\begin{itemize} 
\item[i.] $\mathfrak{D}_\beta:=\cup_{j\in I} \mathfrak{D}_{\beta,j}\subset \mathfrak{X}_\beta$ is a strict normal crossing divisor relative to $\mathcal{O}_{K_0}$ 
\item[ii.]  for each $i\in J$ we have $\mathfrak{D}_{\beta,i}\cap \mathcal{X}_\beta=\mathcal{D}_i \cap \mathcal{X}_\beta$  
\item[iii.]  for every $J\subset I$ we have that $\mathfrak{X}_{\beta,J}:=\cap_{j\in J} \mathfrak{D}_{\beta,j}$ is smooth over $\mathcal{O}_{K_0}$.  
\end{itemize}

We further take for every $\beta\in B$ a lift  $\varphi_\beta$ of the characteristic $p$ Frobenius  to $\mathfrak{X}_\beta$. For example, we might take the lift of Frobenius on $\mathfrak{R}_\beta^0$ defined as follows: it  is  Frobenius on $\mathcal{O}_{K_0}$, raises the variables  $Z,X_1,\ldots,X_a,Y_1,\ldots,Y_b$ to the $p$-th power and is multiplication by $p$ on  $\mathbb{N}^{a+b+1}$. In particular, it sends each $\mathfrak{D}_{\beta,i}$ to itself. 

\

Given  $\underline{\beta}:=(\beta_1,\ldots,\beta_n)\in B^n$ such that $\beta_1<\ldots < \beta_n$, we write $\mathcal{X}_{\underline{\beta}}:=\cap_{i=1}^n \mathcal{X}_{\beta_i}\subset \mathcal{X}$, $X_{\underline{\beta}}=\cap_{i=1}^n X_{\beta_i} \subset X$ and we let $\mathfrak{X}_{\underline{\beta}}\to \prod_{i=1,\mathfrak{S}}^n \mathfrak{X}_{\beta_i}$ be the $(p,Z)$-adic formal scheme defined as the $(p,Z)$-adic completion of the log-envelope of the product $ \prod_{i=1,\mathfrak{S}}^n \mathfrak{X}_{\beta_i} $; see \cite[Lemma 2.16]{AI}. It comes endowed with a log-structure $\widetilde{\alpha}_{\underline{\beta}}\colon \mathcal{M}_{\mathfrak{X}_{\underline{\beta}}}\to \mathcal{O}_{\mathfrak{X}_{\underline{\beta}}}$. Explicitly, given $$\mathfrak{R}_{\beta_1}^0:=\mathcal{O}_{K_0}[\![Z]\!]\{X_1,\ldots,X_a,Y_1,\ldots,Y_b\}/(X_1\cdots X_a-Z),$$
 then  $\mathfrak{X}_{\underline{\beta}}$ is formally \'etale over

$$\mathfrak{R}_{\underline{\beta}}^0:=\mathcal{O}_{K_0}[\![Z]\!]\{X_1,\ldots,X_a,Y_1,\ldots,Y_b, v_{i,j}^{\pm 1},u_{i,k}^{\pm 1}\}_{i=2,\ldots,n, j=1,\ldots,a, k=1,\ldots,b}/(X_1\cdots X_a-Z).$$Notice that for every $i=2,\ldots,n$ the structural morphism $\mathfrak{R}_{\beta_i}\to \mathfrak{R}_{\underline{\beta}}$ is provided by the morphism
$$\mathfrak{R}_{\beta_i}^0:=\mathcal{O}_{K_0}[\![Z]\!]\{X_{i,1},\ldots,X_{i,a},Y_{i,1},\ldots,Y_{i,b}\}/(X_{i,1}\cdots X_{i,a}-Z) \longrightarrow \mathfrak{R}_{\underline{\beta}}^0,$$ given by $X_{i,j}\mapsto v_{i,j} X_i$ and $Y_{i,k}\mapsto u_{i,k} Y_k$. The  closed immersion $\mathcal{X}_{\underline{\beta}} \subset \prod_{i=1,\mathcal{S}}^n \mathcal{X}_{\beta_i} \subset \prod_{i=1,\mathfrak{S}}^n \mathfrak{X}_{\beta_i}$ lifts to a closed immersion of $\iota_{\underline{\beta}}\colon \mathcal{X}_{\underline{\beta}} \to  \mathfrak{X}_{\underline{\beta}}$. Furthermore, the product of the Fobenius lifts $\prod_{i=1}^n \varphi_{\beta_i}$ on  $\prod_{i=1,\mathfrak{S}}^n \mathfrak{X}_{\beta_i}$ defines a lift  $\varphi_{\underline{\beta}}$ of Frobenius  to $\mathfrak{X}_{\underline{\beta}}$ by functoriality of the log-envelope. This also defines a morphism $\varphi_{\underline{\beta}}$ on the tubular neighborhood $ \big]\mathcal{X}_{\underline{\beta},k}\big[_{\mathfrak{X}_{\underline{\beta}}}^{\log}  $; see \cite[Def.~2.2.5]{Shiho}. Notice that $ \big]\mathcal{X}_{\underline{\beta},k}\big[_{\mathfrak{X}_{\underline{\beta}}}^{\log} =\big]\overline{\mathcal{X}}_{\underline{\beta}}\big[_{\mathfrak{X}_{\underline{\beta}}}^{\log}$ as $\mathcal{X}_{\underline{\beta},k}\subset \overline{\mathcal{X}}_{\underline{\beta}}$ is a closed immersion defined by a nilpotent ideal.

Finally,  for every  $j\in I$  we have a divisor $\mathfrak{D}_{\underline{\beta},j} \subset \mathfrak{X}_{\underline{\beta}}$ defined in the coordinates above by the ideal generated by $Y_j$. Given a subset  $J\subset I$ we then define the closed formal subschemes $\mathfrak{X}_{J,\underline{\beta}}=\cap_{j\in J} \mathfrak{D}_{\underline{\beta},j} \subset \mathfrak{X}_{\underline{\beta}}$. In the coordinates above, it is defined by the ideal sheaf $\mathfrak{I}_{J,\underline{\beta}}=(Y_j\vert j\in J) \mathcal{O}_{\mathfrak{X}_{\underline{\beta}}}$.

\begin{remark}
\label{rmk:Stein} Consider the open unit disk $C$ defined as the rigid analytic space  associated to the formal scheme $\mathcal{O}_{K_0}\left[ \left[ Z\right] \right] $. It is quasi--Stein as it can be written as $C=\cup_{h\in \mathbb{N}} C_h$  with $C_h\subset C$ the admissible open affinoid defined by $\vert Z^h/p\vert\leq 1$.  Similarly, for every $\underline{\beta}\in B^n$ and every open or closed affine subscheme $Y \subset \mathcal{X}_{\underline{\beta},k}$, the tube $ \big]Y\big[_{\mathfrak{X}_{\underline{\beta}}}^{\log}  $ naturally maps to $C$ and is also quasi--Stein. In fact, it can be written as $\cup_h \big]Y\big[_{\mathfrak{X}_{\underline{\beta}},h}^{\log}  $ with $\big]Y\big[_{\mathfrak{X}_{\underline{\beta}},h}^{\log}  $ an admissible affinoid open over $C_h$. If $Y \subset \mathcal{X}_{\underline{\beta},k}$ is open such affinoid is the inverse image of $C_h$. It it is closed and defined by the ideal $(\overline{g}_1,\ldots,\overline{g}_t)$ and we choose lifts $g_1,\ldots,g_t\in \mathfrak{R}_\beta$, then $\big]Y\big[_{\mathfrak{X}_{\underline{\beta}}}^{\log}  $ is defined by the rigid analytic space associated to the $(p,Z,g_1,\ldots,g_t)$-adic completion of $\mathfrak{R}_\beta$. One then defines $\big]Y\big[_{\mathfrak{X}_{\underline{\beta}},h}^{\log}  $  by  the conditions $\vert Z^h/p\vert\leq 1$, $\vert g_1^h/p\vert\leq 1$, $\ldots$, $\vert g_t^h/p\vert\leq 1$. 

For $Y \subset \mathcal{X}_{\underline{\beta},k}$ closed,  we can define a   fundamental system of open admissible covers  of the complement of $ \big]Y\big[_{\mathfrak{X}_{\underline{\beta}}}^{\log}  $ in the rigid analytic fiber of $\mathfrak{X}_{\underline{\beta}}$ as the rigid analytic spaces associated to the $(p,Z)$-adic completion of the ring $\mathfrak{R}_{\underline{\beta}}\bigl[Z/g_i^s\bigr]$ for $i=1,\ldots,t$ and $s\in \mathbb{N}$. Reasoning as before, each of them is quasi-Stein.
 
In each of these cases, the property of being quasi-Stein implies by Kiehl's Theorem B that every coherent module has trivial  higher cohomology groups.

\end{remark}

Consider a non-degenerate Frobenius log-crystal $ \bigl(\mathcal{E},\mathrm{Fil}^\bullet \mathcal{E}, \nabla,\mathfrak{E},\varphi_{\mathfrak{E}}\bigr)$ relative to $(X,\alpha_X,\overline{\mathcal{X}},\alpha_{\overline{\mathcal{X}}})$.

\subsection{De Rham complexes}\label{sec:dR}

We write  $\mathrm{dR}(\mathcal{E})_{\underline{\beta}}$ for the  complex obtained by taking sections of the logarithmic de Rham complex  $\mathcal{E}\otimes_{\mathcal{O}_{X_{\underline{\beta}}}}  \Omega_{X_{\underline{\beta}}}^{\log,\bullet}$. For every integer $r$ we let $\mathrm{Fil}^r\mathrm{dR}(\mathcal{E})_{\underline{\beta}}$ be the complex given by taking  the global sections of the de Rham complex $\mathrm{Fil}^{r-\bullet} \mathcal{E}_{\underline {\beta}} \otimes_{\mathcal{O}_{X_{\underline{\beta}}}}  \Omega_{X_{\underline{\beta}}}^{\log,\bullet}$. For $n\in \mathbb{N}$ let $$\mathrm{dR}(\mathcal{E})_n:=\oplus_{\underline{\beta}\in B^n, \beta_1<\ldots <\beta_n} \mathrm{dR}(\mathcal{E})_{\underline{\beta}}.$$ Then $ \mathrm{R}\Gamma_{\mathrm{logdR}}\bigl(X, \mathcal{E}\bigr)$  is represented by the total complex associated to 
$$\mathrm{dR}(\mathcal{E})_1\longrightarrow \mathrm{dR}(\mathcal{E})_2 \longrightarrow \mathrm{dR}(\mathcal{E})_3 \longrightarrow \cdots. $$ Define similarly $\mathrm{Fil}^r \mathrm{R}\Gamma_{\mathrm{logdR}}\bigl(X, \mathcal{E}\bigr)$ using the $\mathrm{Fil}^r \mathrm{dR}(\mathcal{E})_{\underline{\beta}}$'s. 

Using the coordinates of \S \ref{sec:Chech} we get an explicit description of the residue $$\mathrm{Res}(\nabla_{\underline{\beta},j})\colon \mathcal{E}_{\underline{\beta}}/ Y_j \mathcal{E}_{\underline{\beta}}\longrightarrow \mathcal{E}_{\underline{\beta}}/Y_j \mathcal{E}_{\underline{\beta}}$$of $\nabla_{\underline{\beta}}\colon \mathcal{E}_{\underline{\beta}} \to  \mathcal{E}_{\underline{\beta}}\otimes_{\mathcal{O}_{X_{\underline{\beta}}}} \Omega_{X_{\underline{\beta}}}^{\log}$ along the component of the divisor $D$ defined by $Y_j$, as recalled in Definition \ref{def:nilpotentconnection}, as the reduction modulo $Y_j$ of the composite $(1 \otimes \delta_j)  \circ  \nabla_{\underline{\beta}} $ with $\delta_j:=Y_j \frac{\partial}{\partial Y_j} $.

\subsection{Crystalline complexes}\label{sssec:explcictcrystallne}

For every $\underline{\beta}$ we let $\mathcal{O}_{\mathfrak{X}_{\underline{\beta}}}^{\rm DP}$ be the $p$-adic completion of the sheaf of divided powers with respect to the ideal defining the closed immersion $\mathcal{X}_{\underline{\beta}} \subset  \mathfrak{X}_{\underline{\beta}}$. 
The construction of Frobenius lifts on $\mathfrak{X}_{\underline{\beta}}$, explained in \S \ref{sec:Chech}, defines  a lift of Frobenius  $\varphi_{\underline{\beta}}$ on this divided powers envelope.

Define $\bigl(\mathfrak{E}_{\underline {\beta}}, \nabla_{\underline{\beta}}, \Phi_{\underline{\beta}}\bigr)$ as the $\mathcal{O}_{\mathfrak{X}_{\underline{\beta}}}^{\rm DP}$-module with log-connection $\nabla_{\underline{\beta}}$ relative to $\mathcal{O}_{K_0}$ and Frobenius $\Phi_{\underline{\beta}} \colon \varphi_{\underline{\beta}}^\ast\bigl(\mathfrak{E}_{\underline {\beta}}\bigr)\to \mathfrak{E}_{\underline {\beta}} $ obtained by evaluating the Frobenius log-crystal at $\bigl(\mathfrak{X}_{\underline{\beta}}, \alpha_{\mathfrak{X}_{\underline{\beta}}}\bigr)$.  Define $\mathrm{dR}(\mathfrak{E})_{\underline{\beta}}$ to be the complex obtained by taking global sections of the complex  $\mathfrak{E}_{\underline {\beta}} \otimes_{\mathcal{O}_{\mathfrak{X}_{\underline{\beta}}}^{\rm DP}}  \Omega_{\mathfrak{X}_{\underline{\beta}}}^{\log,{\rm DP},\bullet}[p^{-1}]=\mathfrak{E}_{\underline {\beta}} \otimes_{\mathcal{O}_{\mathfrak{X}_{\underline{\beta}}}}  \Omega_{\mathfrak{X}_{\underline{\beta}}}^{\log,\bullet}[p^{-1}]$.  Proceeding as above we represent the complex $\mathrm{R}\Gamma_{\mathrm{crys}}\bigl((\overline{\mathcal{X}},\alpha_{\overline{\mathcal{X}}})/\mathcal{S}_0,\mathfrak{E}\bigr) [p^{-1}] $, which comes equipped with a Frobenius morphism (of complexes).

For every $\underline{\beta}\in B^n$ such that $\beta_1<\ldots < \beta_n$ the evaluation of $\mathfrak{E}_{\mathcal{O}_K}$ at the open $\mathcal{X}_{\underline{\beta}}$ is the pull-back of   $(\mathfrak{E}_{\underline {\beta}}, \nabla_{\underline{\beta}})$ via the closed immersion $\mathcal{X}_{\underline{\beta}}\subset  \mathfrak{X}_{\underline{\beta}}$. In particular, the restriction of $(\mathcal{E},\nabla)$ to $X_{\underline{\beta}}$ coincides with the module with connection  defined on the analytic generic fiber  by the pull-back of  $(\mathfrak{E}_{\underline {\beta}}, \nabla_{\underline{\beta}})$  to $\mathcal{X}_{\underline{\beta}}$.  This provides a surjective map $\mathrm{dR}(\mathfrak{E})_{\underline{\beta}}\to \mathrm{dR}(\mathcal{E})_{\underline{\beta}}$ for every $\underline{\beta}$ and, hence, a surjective morphism $$\gamma_{\rm crys}\colon \mathrm{R}\Gamma_{\mathrm{crys}}\bigl((\overline{\mathcal{X}},\alpha_{\overline{\mathcal{X}}})/\mathcal{S}_0,\mathfrak{E}\bigr) [p^{-1}] \longrightarrow \mathrm{R}\Gamma_{\mathrm{logdR}}\bigl(X, \mathcal{E}\bigr).$$

\subsection{Convergent complexes}\label{sec:analsyn}

For every $\underline{\beta}$ evaluating $\mathfrak{E}^{\rm conv}$ on the tubular neighborhood  $]\mathcal{X}_{\underline{\beta},k}   [_{\mathfrak{X}_{\underline{\beta}}}^{\log}$ we get the unique module with logarithmic connection and Frobenius $(\mathfrak{E}^{\rm conv}_{\underline{\beta}},\nabla_{\underline{\beta}}^{\rm conv}, \Phi_{\underline{\beta}}^{\rm conv})$ extending $\bigl(\mathfrak{E}_{\underline {\beta}}, \nabla_{\underline{\beta}}, \Phi_{\underline{\beta}}\bigr)$.

Define $\mathrm{dR}(\mathfrak{E}^{\rm conv})_{\underline{\beta}}$ to be the complex $\mathfrak{E}^{\rm conv}_{\underline{\beta}} \otimes_{\mathcal{O}_{]\mathcal{X}_{\underline{\beta},k}   [_{\mathfrak{X}_{\underline{\beta}}}^{\log}} }   \Omega^{\bullet,\log}_{]\mathcal{X}_{\underline{\beta},k}   [_{\mathfrak{X}_{\underline{\beta}}}^{\log}}$, which is also $\mathfrak{E}^{\rm conv}_{\underline{\beta}}  \otimes_{\mathcal{O}_{\mathfrak{X}_{\underline{\beta}}}}  \Omega_{\mathfrak{X}_{\underline{\beta}}}^{\log,\bullet}[p^{-1}] $,  obtained by taking global sections of the logarithmic de Rham  complex of $(\mathfrak{E}^{\rm conv}_{\underline{\beta}},\nabla_{\underline{\beta}}^{\rm conv})$. Thanks to Remark \ref{rmk:Stein} it computes the cohomology of the latter complex on the rigid analytic space $]\mathcal{X}_{\underline{\beta},k}   [_{\mathfrak{X}_{\underline{\beta}}}^{\log} $.  It maps to $\mathrm{dR}(\mathfrak{E})_{\underline{\beta}}$ and it is equipped with a Frobenius morphism (of complexes). 

Proceeding as in \S \ref{sec:dR} we get a representative of $\mathrm{R}\Gamma_{\mathrm{an}}\bigl(\mathcal{X}_k/K_0,\mathfrak{E}^{\rm conv}\bigr)$ and, by Proposition \ref{prop:Econv},  a Frobenius equivariant morphism $$\tau \colon \mathrm{R}\Gamma_{\mathrm{an}}\bigl((\mathcal{X}_k,\alpha_{\mathcal{X}_k})/\mathcal{S}_0,\mathfrak{E}^{\rm conv}\bigr) \longrightarrow  \mathrm{R}\Gamma_{\mathrm{crys}}\bigl((\overline{\mathcal{X}},\alpha_{\overline{\mathcal{X}}})/\mathcal{S}_0,\mathfrak{E}\bigr) [p^{-1}] .$$

\subsection{Log-rigid complexes}\label{sec:rigsyn}

Let $U\subset Y \subset \mathcal{X}_k$ be subschemes with $U\subset \mathcal{X}_k$ open and contained in the smooth locus $\mathcal{X}_k^{\rm sm}$ of $\mathcal{X}_k$ and $Y\subset \mathcal{X}_k$ closed and reduced. Let $\alpha_{U}$, resp.~$\alpha_Y$ be the log-structure on $U$, resp.~$Y$ induced from the log-structure $\alpha_k:=\alpha_{\mathcal{X}_k}$.

Consider $\underline{\beta}=(\beta_1,\ldots,\beta_n) \in B^n$ such that $\beta_1<\ldots < \beta_n$. Define $U_{\underline{\beta}}:= U\cap \mathcal{X}_{\underline{\beta},k}$ and $Y_{\underline{\beta}}:= Y\cap \mathcal{X}_{\underline{\beta},k}$. Restricting $(\mathfrak{E}^{\rm conv}_{\underline{\beta}},\nabla_{\underline{\beta}}^{\rm conv}, \Phi_{\underline{\beta}}^{\rm conv})$ to strict neighborhoods  of the tubular neighborhood of $U_{\underline{\beta}}$ in the tube  of $Y_{\underline{\beta}}$ in the rigid analytic fiber of $\mathfrak{X}_{\underline{\beta}}$ we get a Frobenius overconvergent log-isocrystal $\mathfrak{E}_{\underline{\beta}}^\dagger$ relative to $U_{\underline{\beta}}\subset Y_{\underline{\beta}}$ and the  rigid analytic fiber of $\mathfrak{X}_{\underline{\beta}}$.  See \cite[Rmk. 2.2.19]{Shiho}. Proceeding as in the previous sections using open covers by quasi-Stein spaces as in Remark \ref{rmk:Stein}, we get a complex $\mathrm{R}\Gamma_{\mathrm{rig}}\bigl((U\subset Y,\alpha_k)/K_0,\mathfrak{E}^{\rm conv}\bigr)$ given by taking the total complex of the logarithmic de Rham  complexes
$$\mathrm{dR}(\mathfrak{E}^\dagger)_{\underline{\beta}}:=\mathfrak{E}^\dagger_{\underline {\beta}} \otimes_{\mathcal{O}_{\mathfrak{X}_{\underline{\beta}}}}  \Omega_{\mathfrak{X}_{\underline{\beta}}}^{\log,\bullet}[p^{-1}].$$

Let $\mathcal{E}^\dagger$ be the overconvergent sheaf with connection defined by restricting $\mathcal {E}$ to strict neighbourhoods of the tube $]U[_\mathcal{X}$ in the tube $]Y[_\mathcal{X}$ in $X$. For every $\underline{\beta}$ as above the pull-back of  $\mathfrak{E}_{\underline{\beta}}^\dagger$ to  the tube $]Y_{\underline{\beta}}[_{\mathcal{X}_{\underline{\beta}}}$ in $X_{\underline{\beta}}$ coincides with the overconvergent module with connection $\mathcal{E}^\dagger_{\underline{\beta}}$ given by restricting $\mathcal{E}^\dagger$ to  $]Y_{\underline{\beta}}[_{\mathcal{X}_{\underline{\beta}}}$. These identifications provide a surjective map $\mathrm{dR}(\mathfrak{E}^\dagger)_{\underline{\beta}}\to \mathrm{dR}(\mathcal{E}^\dagger)_{\underline{\beta}}$, given by sending $Z\to \pi$.  We then get a morphism $$\gamma_{\rm rig}\colon \mathrm{R}\Gamma_{\mathrm{rig}}\bigl((U\subset Y,\alpha_k)/K_0,\mathfrak{E}^{\rm conv}\bigr) \longrightarrow \mathrm{R}\Gamma_{\mathrm{rig}}\bigl((U\subset Y,\alpha_k)/K,\mathcal{E}\bigr) .$$As $U$ is contained in the smooth locus of $\mathcal{X}_k$, one can use the embedding $U\subset \mathcal{X}$ to represent $\mathrm{R}\Gamma_{\mathrm{rig}}\bigl((U\subset Y,\alpha_k)/K,\mathcal{E}\bigr)$ as the de Rham complex of $\mathcal{E}$ over strict  neighbourhoods of $]U[_{\mathcal{X}}$ in $]Y[_\mathcal{X}$. We endow $\mathrm{R}\Gamma_{\mathrm{rig}}\bigl((U\subset Y,\alpha_k)/K,\mathcal{E}\bigr)$  with the filtration  $\mathrm{Fil}^r\mathrm{R}\Gamma_{\mathrm{rig}}\bigl((U\subset Y,\alpha_k)/K,\mathcal{E}\bigr)$ defined by restriction of $\mathrm{Fil}^{r-\bullet} \mathcal{E} \otimes_{\mathcal{O}_X}  \Omega_X^\bullet$.

Write $\mathbf{Syn}_{{\rm rig},U\subset Y, \alpha_k}\bigl(\mathcal{E},\mathfrak{E},r\bigr)$ for the total complex associated to the morphism of complexes $ (1-\frac{\Phi}{p^r})\oplus \gamma_{\rm rig} $ from $\mathrm{R}\Gamma_{\mathrm{rig}}\bigl((U\subset Y,\alpha_k)/K_0,\mathfrak{E}^{\rm conv}\bigr) $ to $$\mathrm{R}\Gamma_{\mathrm{rig}}\bigl((U\subset Y,\alpha_k)/K_0,\mathfrak{E}^{\rm conv}\bigr) \oplus \frac{ \mathrm{R}\Gamma_{\mathrm{rig}}\bigl((U\subset Y,\alpha_k)/K,\mathcal{E}\bigr)}{\mathrm{Fil}^r  \mathrm{R}\Gamma_{\mathrm{rig}}\bigl((U\subset Y,\alpha_k)/K,\mathcal{E}\bigr)}.$$By construction we have a morphism of complexes $$\tau_{{\rm rig},U\subset Y, \alpha_k}\colon  \mathbf{Syn}_{\rm an}(\mathcal{E},\mathfrak{E},r) \longrightarrow  \mathbf{Syn}_{{\rm rig},U\subset Y, \alpha_k}\bigl(\mathcal{E},\mathfrak{E},r\bigr)$$which  gives a morphism on cohomology $$\tau_{{\rm rig},U\subset Y, \alpha_k}^i\colon \mathrm{H}^i\bigl(\mathbf{Syn}_{\rm an}(\mathcal{E},\mathfrak{E},r)\bigr) \longrightarrow \mathrm{H}^i\bigl(\mathbf{Syn}_{{\rm rig},U\subset Y, \alpha_k}(\mathcal{E},\mathfrak{E},r)\bigr). $$
As the notation suggests, under the assumption that  $U\subset \mathcal{X}_k^{\rm sm}$, the  cohomology groups we obtain are independent of the choices made in \S \ref{sec:Chech}. This is not clear in general. See \cite[Conj. 2.2.20]{Shiho}.

\begin{lemma}\label{lemma:indchoices} The  complex $\mathrm{R}\Gamma_{\mathrm{rig}}\bigl((U\subset Y,\alpha_k)/K_0,\mathfrak{E}^{\rm conv}\bigr) $ is independent, up to homotopy, of the choices in \S \ref{sec:Chech}. \end{lemma}
\begin{proof}  
Consider the natural morphisms of log-formal schemes $$\mathfrak{X}_{\underline{\beta}} \stackrel{a}{\longrightarrow} \mathfrak{X}_{\underline{\beta}}' \stackrel{b}{\longrightarrow} \mathfrak{X}_{\beta_1 } \times_{\mathfrak{S}} \cdots \times_{\mathfrak{S}}  \mathfrak{X}_{\beta_n } ,$$where $b$ is defined by exactifying the log-structure defined on each $\mathfrak{X}_{\beta_i}$ by the special fiber $\mathcal{X}_{\beta_i,k}$   and $a$ is defined by exactifying the log-structure given on each $\mathfrak{X}_{\beta_i}$  by the divisor $\mathfrak{D}_{\beta_i}$. It follows from \cite[Claim p. 114]{Shiho} and the assumption that $U\subset \mathcal{X}_k^{\rm sm}$ that 
the morphism $b$ induces an isomorphism of log-rigid spaces $b_V\colon V_{\underline{\beta}}^{\log} \to V_{\underline{\beta}}$ on suitable strict neighborhoods  $V_{\underline{\beta}}^{\log}$ and $V_{\underline{\beta}}$ of the tubular neighborhood of $U_{\underline{\beta}}$ in the tube  of $\mathcal{X}_{\underline{\beta},k}$ in the rigid analytic fiber of $\mathfrak{X}_{\underline{\beta}}'$, resp.~of $\prod_{i=1,\mathfrak{S}}^n \mathfrak{X}_{\beta_i }$. Here, possibly enlarging the index set $B^n$, we can replace $ V_{\underline{\beta}}$ with an open admissible cover by quasi-Stein spaces, see Remark \ref{rmk:Stein}, living over the rigid analytic fiber of $\mathfrak{S}$, with the  log-structure pull-back of $\mathfrak{M}$.

Since $\mathfrak{X}_{\beta_i} \to \mathfrak{S}$ is a log-smooth lift of $\mathcal{X}_{\beta_i,k}$ for each $\beta_i$, considering also the log-structure defined by the divisors $\mathfrak{D}_{\beta_i}$ and $\mathcal{D}_{\beta_i,k}$ respectively, it follows from formal log-smoothness that we get isomorphic $V_{\underline{\beta}}$'s, considering also the log-structure induced by the $\mathfrak{D}_{\beta_i}$'s,  if we take different choices of $\mathfrak{X}_{\beta_i}$ in \S \ref{sec:Chech}. This implies that the inverse image of $V_{\underline{\beta}}^{\log}$ via the map of log-rigid spaces induced by $a$ is independent of the choices made, up to isomorphism as log-rigid spaces. The Lemma follows.

\end{proof}

Finally, we remark that  $\mathbf{Syn}_{{\rm rig},U\subset Y, \alpha_k}\bigl(\mathcal{E},\mathfrak{E},r\bigr)$ is functorially contravariant in the pair $(U,Y)$ so that if $U'\subset U$ is an open subscheme and $Y'\subset Y$ is a closed reduced subscheme such that $U'\subset Y'$ we have a morphism of complexes: 
$$\mathbf{Syn}_{{\rm rig},U\subset Y, \alpha_k}\bigl(\mathcal{E},\mathfrak{E},r\bigr)\longrightarrow \mathbf{Syn}_{{\rm rig},U'\subset Y', \alpha_k}\bigl(\mathcal{E},\mathfrak{E},r\bigr).$$

\section{Relation with Hyodo-Kato cohomology I}\label{sec:HK}

Consider a non-degenerate  Frobenius log-crystal $ \bigl(\mathcal{E},\mathrm{Fil}^\bullet \mathcal{E}, \nabla,\mathfrak{E},\Phi_{\mathfrak{E}}\bigr)$ relative to $(X,\alpha_X,\overline{\mathcal{X}},\alpha_{\overline{\mathcal{X}}})$.  Let $\bigl(\mathfrak{E}_{\mathrm{rel}},\Phi_{\mathfrak{E}_{\mathrm{rel}}}\bigr)$ be the Frobenius log-crystal on $\bigl(\overline{\mathcal{X}},\alpha_{\overline{\mathcal{X}}}\bigr)$ relative to  $(\overline{\mathcal{S}},\overline{\mathcal{M}}) \to (\mathfrak{S}^{\mathrm DP}, \mathfrak{M}^{\mathrm DP})$ given as pull-back of $\bigl(\mathfrak{E},\Phi_{\mathfrak{E}}\bigr)$ via the natural projections $(\overline{\mathcal{S}},\overline{\mathcal{M}})\to \mathrm{Spec}(k)$ and  $(\mathfrak{S}^{\mathrm DP}, \mathfrak{M}^{\mathrm DP})\to \mathcal{S}_0$.

\begin{definition}\label{def:HK}

Let  $\bigl(\mathfrak{E}_0,\Phi_{\mathfrak{E}_0}\bigr)$ be the Frobenius log-crystal on $(\mathcal{X}_k,\alpha_{\mathcal{X}_k})$ relative to $ (\mathcal{S}_0,\mathcal{M}_0)$ given by the pull-back of $\bigl(\mathfrak{E}_{\mathrm{rel}},\Phi_{\mathfrak{E}_{\mathrm{rel}}}\bigr)$ via the map  $(\mathcal{S}_0,\mathcal{M}_0) \to (\mathfrak{S}^{\mathrm DP}, \mathfrak{M}^{\mathrm DP})$ defined by sending $Z\mapsto 0$. We write $$\mathrm{R}\Gamma_{\mathrm{HK}}\bigl((\mathcal{X}_k,\alpha_{\mathcal{X}_k})/ \mathcal{S}_0^0,\mathfrak{E}_0\bigr):=\mathrm{R}\Gamma_{\mathrm{crys}}\bigl((\mathcal{X}_k,\alpha_{\mathcal{X}_k})/(\mathcal{S}_0,\mathcal{M}_0),\mathfrak{E}_0\bigr)[p^{-1}] $$and $\Phi$ for the induced Frobenius morphism.

\end{definition}

Consider a complex $ \mathrm{R}\Gamma_{\mathrm{an}}\bigl((\mathcal{X}_k,\alpha_{\mathcal{X}_k})/(\mathfrak{S},\mathfrak{M}),\mathfrak{E}^{\rm conv}\bigr)$ computing the relative analytic cohomology with  Frobenius $\Phi$ induced by $\Phi_{\mathfrak{E}_{\mathrm{rel}}}$ proceeding as in \S \ref{sec:analsyn}.  Its base change to the $(\mathcal{S}_0,\mathcal{M}_0)$-valued point of $(\mathfrak{S},\mathfrak{M})$ defined by the ideal $\mathcal{I}$ generated by $Z$ represents $\mathrm{R}\Gamma_{\mathrm{HK}}\bigl((\mathcal{X}_k,\alpha_{\mathcal{X}_k})/ \mathcal{S}_0^0,\mathfrak{E}_0\bigr)$. The map $$Z \frac{\partial}{\partial Z}\colon \mathrm{R}\Gamma_{\mathrm{an}}\bigl((\mathcal{X}_k,\alpha_{\mathcal{X}_k})/(\mathfrak{S},\mathfrak{M}),\mathfrak{E}^{\rm conv}\bigr) \to \mathrm{R}\Gamma_{\mathrm{an}}\bigl((\mathcal{X}_k,\alpha_{\mathcal{X}_k})/(\mathfrak{S},\mathfrak{M}),\mathfrak{E}^{\rm conv}\bigr)$$of complexes defines the {\em monodromy operator} as a morphism of complexes $$N\colon \mathrm{R}\Gamma_{\mathrm{HK}}\bigl((\mathcal{X}_k,\alpha_{\mathcal{X}_k})/ \mathcal{S}_0^0,\mathfrak{E}_0\bigr) \longrightarrow \mathrm{R}\Gamma_{\mathrm{HK}}\bigl((\mathcal{X}_k,\alpha_{\mathcal{X}_k})/  \mathcal{S}_0^0,\mathfrak{E}_0\bigr)$$satisfying the relation $ N \circ \Phi=p \cdot \Phi \circ N$

\begin{definition}\label{def:abs} We denote by $ \mathrm{R}\Gamma_{\mathrm{abs}}\bigl((\mathcal{X}_k,\alpha_{\mathcal{X}_k})/\mathcal{S}_0,\mathfrak{E}^{\rm conv}\bigr) $ the cone of the morphism $Z \frac{\partial}{\partial Z}$ above, by $\mathrm{R}\Gamma_{\mathrm{abs}}\bigl((\mathcal{X}_k,\alpha_{\mathcal{X}_k})/\mathcal{S}_0,\mathfrak{E}_0\bigr) $ the cone of  $N$ and by $$\delta\colon \mathrm{R}\Gamma_{\mathrm{abs}}\bigl((\mathcal{X}_k,\alpha_{\mathcal{X}_k})/\mathcal{S}_0,\mathfrak{E}^{\rm conv}\bigr) \longrightarrow \mathrm{R}\Gamma_{\mathrm{abs}}\bigl((\mathcal{X}_k,\alpha_{\mathcal{X}_k})/\mathcal{S}_0,\mathfrak{E}_0\bigr) $$the projection map.  The Frobenius on the source and $p$ times Frobenius on the target of the cones of $Z \frac{\partial}{\partial Z}$ and of $N$  define a Frobenius morphism on  $ \mathrm{R}\Gamma_{\mathrm{abs}}\bigl((\mathcal{X}_k,\alpha_{\mathcal{X}_k})/\mathcal{S}_0,\mathfrak{E}^{\rm conv}\bigr) $  and on   $\mathrm{R}\Gamma_{\mathrm{abs}}\bigl((\mathcal{X}_k,\alpha_{\mathcal{X}_k})/\mathcal{S}_0,\mathfrak{E}_0\bigr) $ so that $\delta$ is Frobenius equivariant.

\end{definition}

\begin{proposition}\label{prop:anabs} Assume that $\mathcal{X}$ is proper over $\mathcal{O}_K$. Then, \smallskip

 (1) We have an isomorphism, equivariant with respect to Frobenius, $$\mathrm{R}\Gamma_{\mathrm{an}}\bigl((\mathcal{X}_k,\alpha_{\mathcal{X}_k})/\mathcal{S}_0,\mathfrak{E}^{\rm conv}\bigr)[1]\stackrel{\sim}{ \longrightarrow} \mathrm{R}\Gamma_{\mathrm{abs}}\bigl((\mathcal{X}_k,\alpha_{\mathcal{X}_k})/ \mathcal{S}_0,\mathfrak{E}^{\rm conv}\bigr).$$

(2) Each $\mathrm{H}^i_{\mathrm{abs}}\bigl((\mathcal{X}_k,\alpha_{\mathcal{X}_k})/ \mathcal{S}_0,\mathfrak{E}_0\bigr)$ is a finite dimensional $K_0$-vector space.

\smallskip

(3) For every polynomial $P(T)\in K_0[T]$ of the form $P(T)=1+ T Q(T)$ the map on cohomology $$\mathrm{H}^{i}_{\mathrm{an}}\bigl((\mathcal{X}_k,\alpha_{\mathcal{X}_k})/\mathcal{S}_0,\mathfrak{E}^{\rm conv}\bigr)\longrightarrow \mathrm{H}^{i-1}_{\mathrm{abs}}\bigl((\mathcal{X}_k,\alpha_{\mathcal{X}_k})/\mathcal{S}_0,\mathfrak{E}_0\bigr),$$induced by  $\delta$ and (1) defines an isomorphism on kernel and cokernel of $P(\Phi)$.

\smallskip

(4) For every $i$, omitting $\alpha_{\mathcal{X}_k}$ from the notation for simplicity, we have a Frobenius equivariant, exact sequence $$0\to \frac{\mathrm{H}^{i-1}_{\mathrm{HK}}\bigl(\mathcal{X}_k/ \mathcal{S}_0^0,\mathfrak{E}_0\bigr)(-1)}{N\bigl(\mathrm{H}^{i-1}_{\mathrm{HK}}\bigl(\mathcal{X}_k/ \mathcal{S}_0^0,\mathfrak{E}_0\bigr)(-1)\bigr)} \to \mathrm{H}^{i-1}_{\mathrm{abs}}\bigl(\mathcal{X}_k/\mathcal{S}_0,\mathfrak{E}_0\bigr) \to \mathrm{H}^i_{\mathrm{HK}}\bigl(\mathcal{X}_k/ \mathcal{S}_0^0,\mathfrak{E}_0\bigr)^{N=0} \to 0.$$
Here, the twist by $-1$ means that the Frobenius acts as $p\phi$ if $\phi$ denotes the Frobenius acting on the Hyodo-Kato cohomology.

\end{proposition}
\begin{proof}
(1) Recall that we have constructed the formal schemes $\mathfrak{X}_{\underline{\beta}}$ so that they factor through formally log smooth morphisms $\mathfrak{f}_{\underline{\beta}}\colon \mathfrak{X}_{\underline{\beta}} \to (\mathfrak{S},\mathfrak{M})$.  The complex $\mathfrak{E}^{\rm conv}_{\underline{\beta}}\otimes_{\mathcal{O}_{\mathfrak{X}_{\underline{\beta}}}} \Omega^{\log,\bullet}_{\mathfrak{X}_{\underline{\beta}}}$ computes analytic cohomology relative to $\mathcal{S}_0$ with trivial log-structure (cfr. condition ii. in the definition of non-degenerate Frobenius log-crystal). We may then also consider  the complex $\mathfrak{E}^{\rm conv}_{\underline{\beta}}\otimes_{\mathcal{O}_{\mathfrak{X}_{\underline{\beta}}}} \Omega^{\log,\bullet}_{\mathfrak{X}_{\underline{\beta}}/(\mathfrak{S},\mathfrak{M})}$ computing the analytic cohomology relative to  $(\mathfrak{S},\mathfrak{M})$.   Taking the total complex for varying $\underline{\beta}$'s, as we did in \S \ref{sec:analsyn}, we get a complex defining $ \mathrm{R}\Gamma_{\mathrm{an}}\bigl((\mathcal{X}_k,\alpha_{\mathcal{X}_k})/(\mathfrak{S},\mathfrak{M}),\mathfrak{E}^{\rm conv}\bigr)$.
Claim (1) follows then  from  the exact sequences below, compatible for varying $\underline{\beta}$'s and with Frobenii, noticing that the boundary map coincides with $Z \frac{\partial}{\partial Z}$:
$$0\to \mathfrak{E}^{\rm conv}_{\underline{\beta}}\otimes_{\mathcal{O}_{\mathfrak{X}_{\underline{\beta}}}}  \Omega_{\mathfrak{X}_{\underline{\beta}}/(\mathfrak{S},\mathfrak{M})}^{\log,\bullet}\wedge \mathcal{O}_{\mathfrak{X}_{\underline{\beta}}}\frac{dZ}{Z}\longrightarrow \mathfrak{E}^{\rm conv}_{\underline{\beta}}\otimes_{\mathcal{O}_{\mathfrak{X}_{\underline{\beta}}}}  \Omega^{\log,\bullet}_{\mathfrak{X}_{\underline{\beta}}}[1]  \longrightarrow $$ $$\longrightarrow \mathfrak{E}^{\rm conv}_{\underline{\beta}}\otimes_{\mathcal{O}_{\mathfrak{X}_{\underline{\beta}}}}  \Omega^{\log,\bullet}_{\mathfrak{X}_{\underline{\beta}}/(\mathfrak{S},\mathfrak{M})}[1]\to 0.$$Note that the induced isomorphism in cohomology is compatible with the Frobenii defined above, since Frobenius on $dZ/Z$ is multiplication by $p$.

(2) Since each $\mathrm{H}^{j}_{\mathrm{HK}}\bigl((\mathcal{X}_k,\alpha_{\mathcal{X}_k})/ \mathcal{S}_0^0,\mathfrak{E}_0\bigr)$ is a finite dimensional $K_0$-vector space due to the properness assumption (see \cite[\S 7]{BerthelotOgusBook}), also each $\mathrm{H}^j_{\mathrm{abs}}\bigl((\mathcal{X}_k,\alpha_{\mathcal{X}_k})/\mathcal{S}_0,\mathfrak{E}_0\bigr)$ is finite dimensional as $K_0$-vector space and, hence, as $\mathbb{Q}_p$-vector space.

(3) We first reduce the claim to the following:

\smallskip
CLAIM:  $P(\Phi)$ is an isomorphism on each  $\mathcal{I} \mathfrak{E}^{\rm conv}_{\underline{\beta}}\otimes_{\mathcal{O}_{\mathfrak{X}_{\underline{\beta}}}} \Omega^{\log,\bullet}_{\mathfrak{X}_{\underline{\beta}}} $.
\smallskip

\noindent 
By (1) we can replace the total complex $\mathcal{I} \mathfrak{E}^{\rm conv}_{\underline{\beta}}\otimes_{\mathcal{O}_{\mathfrak{X}_{\underline{\beta}}}} \Omega^{\log,\bullet}_{\mathfrak{X}_{\underline{\beta}}}[1]$ with the complex
$\mathcal{I}  \mathrm{R}\Gamma_{\mathrm{abs}}\bigl((\mathcal{X}_k,\alpha_{\mathcal{X}_k})/\mathcal{S}_0,\mathfrak{E}^{\rm conv}\bigr)$. Also,  the Claim implies that $P(\Phi)$ is an isomorphism on the cohomology groups of  $\mathcal{I}  \mathrm{R}\Gamma_{\mathrm{abs}}\bigl((\mathcal{X}_k,\alpha_{\mathcal{X}_k})/\mathcal{S}_0,\mathfrak{E}^{\rm conv}\bigr)$. Considering the long exact sequence of cohomology defined by the short exact sequence of complexes $$0 \longrightarrow \mathcal{I}  \mathrm{R}\Gamma_{\mathrm{abs}}\bigl((\mathcal{X}_k,\alpha_{\mathcal{X}_k})/\mathcal{S}_0,\mathfrak{E}^{\rm conv}\bigr) \longrightarrow  \mathrm{R}\Gamma_{\mathrm{abs}}\bigl((\mathcal{X}_k,\alpha_{\mathcal{X}_k})/\mathcal{S}_0,\mathfrak{E}^{\rm conv}\bigr) \stackrel{\delta}{\longrightarrow} $$ $$\longrightarrow \mathrm{R}\Gamma_{\mathrm{abs}}\bigl((\mathcal{X}_k,\alpha_{\mathcal{X}_k})/\mathcal{S}_0,\mathfrak{E}_0\bigr) \longrightarrow 0$$and using the finite dimensionality of $\mathrm{H}^j_{\mathrm{abs}}\bigl((\mathcal{X}_k,\alpha_{\mathcal{X}_k})/\mathcal{S}_0,\mathfrak{E}_0\bigr)$,  we deduce claim (3) from Lemma \ref{lemma:Marco}. 

 We are left to prove the Claim. It suffices to prove the Claim for each term $Z \cdot \mathfrak{E}^{\rm conv}_{\underline{\beta}} \otimes_{\mathcal{O}_{\mathfrak{X}_{\underline{\beta}}}} \Omega^{\log,n}_{\mathfrak{X}_{\underline{\beta}}} $. This follows from Lemma \ref{lemma:PHIM0} (for $r=1$ and $Z_1=Z$) if we restrict to the open of the strict neighborhood  $]\mathcal{X}_{\underline{\beta},k}[^{\log}_{\mathfrak{X}_{\underline{\beta}}}$ defined by the condition $\vert P_\pi(Z)\vert \leq \vert p \vert$ or equivalently $\vert Z^e\vert \leq \vert p \vert$ with $e$ the degree of $P_\pi(Z)$. 
The assumptions of loc.~cit.~are satisfied as  on this open $\mathfrak{E}^{\rm conv}$ admits an integral model $\mathfrak{E}_{\underline{\beta}}$ which is an $\mathcal{O}_{e,1}$-module, with an integrally defined  Frobenius $\Phi_{\mathfrak{E}_{\underline{\beta}}}$ compatible with the one on  $\mathcal{O}_{e,1}$, as explained in the proof of Proposition \ref{prop:Econv}. Since $\mathfrak{E}^{\rm conv}_{\underline{\beta}}$ is defined by Frobenius-pull back of $\mathfrak{E}_{\underline{\beta}}$ it follows that for every affinoid open $W$ in $]\mathcal{X}_{\underline{\beta},k}[^{\log}_{\mathfrak{X}_{\underline{\beta}}}$, stable under Frobenius,  there exists an $m$  such that  $P(\Phi)$  is bijective on $Z^m \cdot \mathfrak{E}^{\rm conv}_{\underline{\beta}} \otimes_{\mathcal{O}_{\mathfrak{X}_{\underline{\beta}}}} \Omega^{\log,n}_{\mathfrak{X}_{\underline{\beta}}/\mathcal{S}_0} $ over $W$. Since $\phi(Z)=Z^p$ this implies that  $P(\Phi)$  is bijective also for $m=1$; see the argument in the second part of the proof of Lemma \ref{lemma:PHIM0}. 

(4) This is formal, taking into account that the Frobenius of Definition \ref{def:abs} makes the short exact sequence of the cone that defines the absolute cohomology groups  ${\rm H}^i_{\rm abs}$ Frobenius equivariant.
\end{proof}

For every $n$, $r\in \mathbb{N}_{\geq 1}$, let $\mathcal{O}_{n,r}:=\mathcal{O}_{K_0}\left[ \underline{Z},\underline{T}\right] /\left( \underline{Z}^{n}-p
\underline{T}\right) $ where $\underline{T}=\left( T_{1},\ldots,T_{r}\right) $, 
$\underline{Z}=\left( Z_{1},\ldots,Z_{r}\right) $ and similarly for the
relations. Let $\varphi =\varphi _{n,r}:\mathcal{O}_{n,r}\rightarrow 
\mathcal{O}_{n,r}$ be the ring homomorphism which is Frobenius on $\mathcal{O
}_{K_0}$, sends $Z_{i}$ to $Z_{i}^{p}$: then necessarily $T_{i}=
Z_{i}^{n}/p$ (the equality after inverting $p$) is sent to $
Z_{i}^{pn}/p=p^{p-1}T_{i}^{p}$. Let $\mathcal{O}_{n,r}\rightarrow \mathcal{
O}_{K_0}$ be the homomorphism of $\mathcal{O}_{K_0}$-algebras sending $
Z_{i}$ and $T_{i}$ to $0$. It is Frobenius equivariant. We denote by $
\mathfrak{m}=\left( \underline{Z},\underline{T}\right) $ its kernel.

\begin{lemma}\label{lemma:PHIM0}
Let $M_{0}$ be an $\mathcal{O}_{n,r}$-module and which is $p$-adically
complete and separated and is endowed with a Frobenius linear map $\Phi
_{0}:M_{0}\rightarrow M_{0}$, i.~e., $\Phi _{0}\left( \alpha m\right) =\varphi
\left( \alpha \right) \Phi _{0}\left( m\right) $ for every $\alpha \in 
\mathcal{O}_{n,r}$ and $m\in M_{0}$. Let $M:=M_{0}\left[ 1/p\right] $ and
let $\Phi $ be the map induced by $\Phi _{0}$. Then, for every polynomial $
P\left( X\right) =1+Q\left( X\right) $ with $Q\left( X\right) \in XK_0
\left[ X\right] $, we have that $P\left( \Phi \right) $ defines an
automorphism of the submodule $\mathfrak{m}M=\left( \underline{Z}\right)
M=\left( \underline{T}\right) M$ of $M$.
\end{lemma}

\begin{proof}
Assume that $p^{h}Q\left( X\right) \in \mathcal{O}_{K_0}\left[ X\right] $.
For every $i=1,\ldots,r$ we have $\varphi \left( T_{i}\right) =pT_{i}g_{i}$ for some $g_{i}\in \mathcal{O}
_{n,r}$. Hence, for $a_1,\ldots,a_r\in \mathbb{N}$ such that $a_1+\ldots + a_r=m$, we have $\varphi \left( T_{1}^{a_1}\cdots T_r^{a_r}\right)
=p^{m}T_{1}^{a_1}\cdots T_r^{a_r}g_{1}^{a_1}\cdots g_r^{a_r}$ and, taking $m$ such that $m\geq h+1$, we see
that $Q\left( \Phi \right) \left( \left( \underline{T}\right)
^{m}M_{0}\right) \subset p\left( \underline{T}\right) ^{m}M_{0}$. Let us
regard $End_{\mathbb{Z}_{p}}\left( \left( \underline{T}\right)
^{m}M_{0}\right) $ as an $\mathcal{O}_{n,r}$-module via the composition with
the multiplication by elements of $\mathcal{O}_{n,r}$ on the codomain: then the inclusion $Q\left( \Phi \right) \left( \left( 
\underline{T}\right) ^{m}M_{0}\right) \subset p\left( \underline{T}\right)
^{m}M_{0}$ implies that $Q\left( \Phi \right) \in \mathrm{End}_{\mathbb{Q}_{p}}\left(
\left( \underline{T}\right) ^{m}M_{0}\right) $ is topologically nilpotent.
Because $M_{0}$ is $p$-adically separated and complete, $\left( \underline{T}%
\right) ^{m}M_{0}$ is also $p$-adically separated and complete and $End_{\mathbb{Q}
_{p}}\left( \left( \underline{T}\right) ^{m}M_{0}\right) $ is $p$
-adically separated and complete as well. In particular, a geometric series argument shows that
$P\left( \Phi \right) =1+Q\left( \Phi \right)$ defines an automorphism of
$\left(\underline{T}\right) ^{m}M_{0}$ and, hence, of $\left(\underline{T}\right) ^{m}M$.

We are left to show that $P\left( \Phi \right) $ is an isomorphism on $
\mathfrak{m}M/ \left( \underline{T}\right) ^{m}M$. Since $\mathfrak{m}
^{2mnr}\subset \left( \underline{T}\right) ^{m}\mathcal{O}
_{n,r}$, $\Phi \left( \mathfrak{m}^{s}M\right) \subset \mathfrak{m}^{s+1}M$
(because $\mathfrak{m}M=\left( \underline{Z}\right) M$ and $\varphi \left(
Z_{i}\right) =Z_{i}^{p}$) and $Q\left( X\right) \in XK_0\left[ X\right] $,
we deduce that $Q\left( \Phi \right) ^{2mnr}\equiv 0$ on $
\mathfrak{m}M/\left( \underline{T}\right) ^{m}M$. Hence $P\left(
\Phi \right) $ is an isomorphism on $\mathfrak{m}M/\left( \underline{T
}\right) ^{m}M$ with inverse $\sum\nolimits_{i=0}^{2mnr}\left( -1\right) ^{i}Q\left( \Phi \right) ^{i}$, as wanted.
\end{proof}

\begin{lemma}\label{lemma:Marco}
Suppose that we are given a morphism between distinguished triangle in the derived category of
complexes of modules over some ring
\begin{equation*}
\begin{array}{ccccccc}
X & \longrightarrow  & Y & \overset{\pi }{\longrightarrow } & Z & 
\longrightarrow  & X\left[ 1\right]  \\ 
x\downarrow \text{ } &  & y\downarrow \text{ } &  & z\downarrow \text{ } & 
& x\left[ 1\right] \downarrow \text{ \ \ } \\ 
X & \longrightarrow  & Y^{\prime } & \overset{\pi }{\longrightarrow } & Z & 
\longrightarrow  & X\left[ 1\right] 
\end{array}
\end{equation*}
satisfying the following properties. The morphism $H^{j+1}\left( x\right) $
is a monomorphisms, the morphism $H^{j}\left( x\right) $ is an isomorphism
and $H^{j}\left( Z\right) $ and $H^{j-1}\left( Z\right) $ are artinian. Then
the morphism $\pi $\ induces isomorphisms%
\begin{equation*}
\mathrm{ker}\left( H^{j}\left( y\right) \right) \overset{\sim }{%
\longrightarrow }\mathrm{ker}\left( H^{j}\left( z\right) \right) \text{ and }%
\mathrm{coker}\left( H^{j}\left( y\right) \right) \overset{\sim }{%
\longrightarrow }\mathrm{coker}\left( H^{j}\left( z\right) \right) \text{.}
\end{equation*}
\end{lemma}
\begin{proof} Left to the reader.
\end{proof}

\section{Syntomic complexes with support conditions}
\label{sec:compactsupport}

The divisor $\mathcal{D}_k$ defines a crystal of invertible ideals $\mathcal{O}_{\mathcal{X}_k}^{\rm crys}(-\mathcal{D}_k) \subset \mathcal{O}_{\mathcal{X}_k}^{\rm crys} $; see \cite[Def. 4.3 \& Rmk, 4.5]{EY}. Given a Chech cover as in \S \ref{sec:Chech}, for every $\beta\in B$  its value on  $\mathfrak{X}_\beta$ is the ideal sheaf in $\mathcal{O}_{\mathfrak{X}_\beta}$  defined by the Cartier divisor $\mathfrak{D}_{\beta}\subset \mathfrak{X}_\beta$ of loc.~cit.~Let $\mathbf{1}(-D)\subset \mathbf{1}$ be the Frobenius subcrystal defined by $\mathcal{O}_{X}(-D)\subset \mathcal{O}_{X}$ and $\mathcal{O}_{\mathcal{X}_k}^{\rm crys}(-\mathcal{D}_k) \subset \mathcal{O}_{\mathcal{X}_k}^{\rm crys} $. Given  a non-degenerate Frobenius log-crystal  $ \bigl(\mathcal{E},\mathrm{Fil}^\bullet \mathcal{E}, \nabla,\mathfrak{E},\Phi_{\mathfrak{E}}\bigr)$ relative to $(X,\alpha_X,\overline{\mathcal{X}},\alpha_{\overline{\mathcal{X}}})$  we write $ \bigl(\mathcal{E}(-D),\mathrm{Fil}^\bullet \mathcal{E}(-D), \nabla,\mathfrak{E}(-\mathfrak{D}),\Phi_{\mathfrak{E}}(-\mathfrak{D})\bigr)$ for the tensor product with $\mathbf{1}(-D)$.

\begin{definition}\label{def:syntomiccompactcomplex} Define the analytic syntomic complex $\mathbf{Syn}_{\rm an}^{c-\infty} (\mathcal{E},\mathfrak{E},r)$, or simply $\mathbf{Syn}^{c-\infty}(\mathcal{E},\mathfrak{E},r)$, with support along $(\mathcal{X}_k\backslash \mathcal{D}_k, X\backslash D)$,   associated to $ \bigl(\mathcal{E},\mathrm{Fil}^\bullet \mathcal{E}, \nabla,\mathfrak{E},\Phi_{\mathfrak{E}}\bigr)$ to be the total complex associated to the morphism  of complexes $(1-\frac{\Phi}{p^r},\gamma_{\rm an})$  from the complex $\mathrm{R}\Gamma_{\mathrm{an}}\bigl((\mathcal{X}_k,\alpha_k)/\mathcal{S}_0,\mathfrak{E}^{\rm conv}(-\mathfrak{D})\bigr)$ to the complex $$ \mathrm{R}\Gamma_{\mathrm{an}}\bigl((\mathcal{X}_k,\alpha_k)/\mathcal{S}_0,\mathfrak{E}^{\rm conv}(-\mathfrak{D})\bigr) \oplus \frac{ \mathrm{R}\Gamma_{\mathrm{logdR}}\bigl(X, \mathcal{E}(-D)\bigr) }{\mathrm{Fil}^r \mathrm{R}\Gamma_{\mathrm{logdR}}\bigl(X, \mathcal{E}(-D)\bigr)}.$$

\end{definition}
Similarly we have the crystalline analytic syntomic complex $\mathbf{Syn}_{\rm crys}^{c-\infty}(\mathcal{E},\mathfrak{E},r)$ with support along $D$ with a morphism 
$\mathbf{Syn}_{\rm an}^{c-\infty}(\mathcal{E},\mathfrak{E},r)\to \mathbf{Syn}_{\rm crys}^{c-\infty}(\mathcal{E},\mathfrak{E},r)$; see Remark \ref{rmk:crysanal}. One has also an analogue of Proposition \ref{prop:spectralanal}.

\smallskip

%Proceeding as  in \S \ref{sec:rigsyn} consider $U\subset \mathcal{X}_k^{\mathrm{sm}}$ an open subscheme such that $\overline{U}'\backslash U'\cap \mathcal{D}_k=\emptyset$. We define the partially compactly supported rigid syntomic cohomology $\mathbf{Syn}_{\mathrm{rig} ,U}^{c-\infty }\left( \mathcal{E},\mathfrak{E},r\right) :=\mathbf{Syn}_{\mathrm{rig},U}\left( \mathcal{E}\left(-\mathcal{D}\right) ,\mathfrak{E}\left(-\mathfrak{D}\right),r\right) $. 

\subsection{Alternative descriptions}

In order to compare our definition of these syntomic complexes with rigid cohomology and with \'etale cohomology we need the following alternative definition.

We start with a definition close to \cite{Ber}. Let $\mathbf{Syn}^\partial_{\mathrm{an}}(\mathcal{E},\mathfrak{E},r)$ or simply $\mathbf{Syn}^\partial(\mathcal{E},\mathfrak{E},r)$  be defined as the total complex associated to the morphism  of complexes $(1-\frac{\Phi}{p^r},\gamma_{\rm an})$ $$\mathrm{R}\Gamma_{\mathrm{an}}\bigl(\mathcal{X}_{(\bullet),k}^\partial/\mathcal{S}_0,\mathfrak{E}_{(\bullet)}^{\rm conv}\bigr)\to \mathrm{R}\Gamma_{\mathrm{an}}\bigl(\mathcal{X}_{(\bullet),k}^\partial/\mathcal{S}_0,\mathfrak{E}_{(\bullet)}^{\rm conv}\bigr) \oplus \frac{ \mathrm{R}\Gamma_{\mathrm{logdR}}\bigl(X_{(\bullet)}, \widehat{\mathcal{E}}^\partial_{(\bullet)}\bigr) }{\mathrm{Fil}^r \mathrm{R}\Gamma_{\mathrm{logdR}}\bigl(X_{(\bullet)},\widehat{\mathcal{E}}^\partial_{(\bullet)}\bigr)}.$$See \S \ref{sec:setting} for the notation $X_{(\bullet)}$, $\mathcal{X}_{(\bullet),k}^\partial$ etc. We denote $\widehat{\mathcal{E}}^\partial_{(a)}$ (resp.~$\mathfrak{E}_{(a)}^{\rm conv}$)  the pull-back of $\mathcal{E}_{(a)}$ to the rigid analytic space associated to formal completion of $\mathcal{X}$ along each $\mathcal{X}_{J}$ with $J\subset I$ and $\vert J\vert=a$ (resp.~of $\mathfrak{E}^{\rm conv}$ to $\mathcal{X}_{(a),k}^\partial\to \mathcal{X}_k$).

We have the following explicit description. Recall from \S \ref{sec:Chech} that for every $\underline{\beta}\in B^n$ with $\beta_1<\ldots < \beta_n$ and every $J\subset I$ we have defined closed formal subschemes $\mathfrak{X}_{J,\underline{\beta}} \subset \mathfrak{X}_{\underline{\beta}}$ defined by an ideal sheaf  $\mathfrak{I}_{J,\underline{\beta}}$. Let $\widehat{\mathfrak{X}}^J_{\underline{\beta}}$ be the formal completion along this closed immersion.  Then for each $a\in \mathbb{N }$ the complex $\mathrm{R}\Gamma_{\mathrm{an}}\bigl(\mathcal{X}_{(a),k}^\partial/\mathcal{S}_0,\mathfrak{E}_{(a)}^{\rm conv}\bigr)$ is represented by the product over all $J\subset I$ of cardinality $\vert J \vert=a$ of the global sections of the logarithmic de Rham complex of the pull-back of $(\mathfrak{E}_{\underline{\beta}}^{\rm conv},\nabla_{\underline{\beta}})$ to the tube $]\mathcal{X}_{J,\underline{\beta},k}[_{\widehat{\mathfrak{X}}^J_{\underline{\beta}}} \to ]\mathcal{X}_{\underline{\beta},k}[_{\mathfrak{X}_{\underline{\beta}}}^{\log} $.

For each $\underline{\beta}$ let $\mathfrak{X}_{J,\underline{\beta}}^\partial $ be the formal scheme $\mathfrak{X}_{J,\underline{\beta}} $ endowed with the log-structure pulled-back from the one on $\mathfrak{X}_{\underline{\beta}}$ and similarly for  $X_J^\partial \subset X$ which is $X_J$ with the log-structure defined by pull-back from $X$. 

We let $\overline{\mathbf{Syn}}_{\mathrm{an}}(\mathcal{E},\mathfrak{E},r)$, or simply $\overline{\mathbf{Syn}}(\mathcal{E},\mathfrak{E},r)$,  be defined as the total complex associated, for varying $\underline{\beta}$ and $J$  to the morphism  of complexes $(1-\frac{\Phi}{p^r},\gamma_{\rm an})$ $$\mathrm{R}\Gamma_{\mathrm{logdR}}\bigl(]\mathcal{X}_{J,\underline{\beta}, k}[_{\mathfrak{X}_{J,\underline{\beta}}^\partial},\mathfrak{E}_{\underline{\beta}}^{\rm conv}\bigr)\to\mathrm{R}\Gamma_{\mathrm{logdR}}\bigl(]\mathcal{X}_{J,\underline{\beta}, k}[_{\mathfrak{X}_{J,\underline{\beta}}^\partial},\mathfrak{E}_{\underline{\beta}}^{\rm conv}\bigr)  \oplus \frac{ \mathrm{R}\Gamma_{\mathrm{logdR}}\bigl(X_{J,\underline{\beta}}^\partial, \mathcal{E}_{J,\underline{\beta}}\bigr) }{\mathrm{Fil}^r \mathrm{R}\Gamma_{\mathrm{logdR}}\bigl(X_{J,\underline{\beta}}^\partial, \mathcal{E}_{J,\underline{\beta}}\bigr)}.$$Here $\mathcal{E}_{J,\underline{\beta}}$ is the pull back of $\mathcal{E}$ to the rigid analytic fiber of $ \mathcal{X}_{J,\underline{\beta}}^\partial$.

\begin{proposition}\label{prop:samecohogroups} Assume that $(\mathcal{E},\nabla)$ has nilpotent residues (see Definition \ref{def:nilpotentconnection}). Then, for every $i$ the morphisms  $\mathbf{Syn}^{c-\infty}(\mathcal{E},\mathfrak{E},r)
\to\overline{\mathbf{Syn}}(\mathcal{E},\mathfrak{E},r)$ and $\mathbf{Syn}^\partial(\mathcal{E},\mathfrak{E},r)\to\overline{\mathbf{Syn}}(\mathcal{E},\mathfrak{E},r)$ are quasi-isomorphisms and, hence, induce isomorphisms on cohomology.

\end{proposition}
\begin{proof} The first statement follows as the natural map $\mathbf{Syn}^{c-\infty}(\mathcal{E},\mathfrak{E},r)\to  \overline{\mathbf{Syn}}(\mathcal{E},\mathfrak{E},r)$ is a quasi-isomorphism.  See the proof of \cite[Lemma 3.6]{EY} for the computation.

We prove the second claim. It suffices to show that, for every $J\subset I$ and every $\underline{\beta}\in B^n$ such that $\beta_1<\ldots < \beta_n$, the following two assertions hold: 

(1) the map of de Rham complexes  $\mathrm{R}\Gamma_{\mathrm{logdR}}\bigl(X_{J,\underline{\beta}}, \widehat{\mathcal{E}}^\partial_{J,\underline{\beta}}\bigr) \longrightarrow \mathrm{R}\Gamma_{\mathrm{logdR}}\bigl(X_{J,\underline{\beta}}^\partial, \mathcal{E}_{J,\underline{\beta}}\bigr) $ is a quasi-isomorphism and similarly for the $ \mathrm{Fil}^r \mathrm{R}\Gamma_{\mathrm{logdR}}$;

(2) $1-\frac{\Phi}{p^r}$ defines an isomorphism on the kernel of surjective morphism of the representing complexes $$\mathrm{R}\Gamma_{\mathrm{logdR}}\bigl(]\mathcal{X}_{J,\underline{\beta},k}[_{\widehat{\mathfrak{X}}^J_{\underline{\beta}}},\mathfrak{E}_{\underline{\beta}}^{\rm conv}\bigr) \to  \mathrm{R}\Gamma_{\mathrm{logdR}}\bigl(]\mathcal{X}_{J,\underline{\beta}, k}[_{\mathfrak{X}_{J,\underline{\beta}}^\partial},\mathfrak{E}_{\underline{\beta}}^{\rm conv}\bigr).$$

We prove (1). Let $I_{J,\underline{\beta}}$  be the kernel of the map $\mathcal{O}_{X_{\underline{\beta}}} \to \mathcal{O}_{X_{{J,\underline{\beta}}}}$. Then $I_{J,\underline{\beta}} \widehat{\mathcal{E}}^\partial_{J,\underline{\beta}}$ is the kernel of the map $\widehat{\mathcal{E}}^\partial_{J,\underline{\beta}}\to \mathcal{E}_{J,\underline{\beta}}$. It suffices to prove that  the de Rham complex $I_{J,\underline{\beta}} \widehat{\mathcal{E}}^\partial_{J,\underline{\beta}}\otimes_{\mathcal{O}_{X_{\underline{\beta}}}} \Omega_{X_{\underline{\beta}}}^{\log,\bullet}$ is exact. This will be proved in \S\ref{section:incarnation}.

%applying Lemma \ref{lemma:Kos L1} to the de Rham complex of formal scheme given by the completion of $\mathcal{X}_{\underline{\beta}}$ along the closed formal subscheme  $\mathcal{X}_{J,\underline{\beta}}$. Indeed, thanks to Proposition \ref{prop:Kosexplicit} this de Rham complex is identified with the Koszul compelx of loc.~cit. Moreover,  hypothesis of Lemma \ref{lemma:Kos L1}(2) holds true as can be checked after inverting $p$, namely passing to the rigid analytic fiber where it follows from the explicit description of the residues of \S \ref{sec:dR} and the assumption that the connection $\nabla$ on $\mathcal{E}$ has nilpotent residues along the $D_s$'s. 

To prove (2), let  $\mathfrak{I}_{J,\underline{\beta}}$  be the kernel of the map $\mathcal{O}_{\mathfrak{X}_{\underline{\beta}}} \to \mathcal{O}_{\mathfrak{X}_{{J,\underline{\beta}}}}$.  It  suffices to show that $1-\frac{\Phi_{\underline{\beta}}}{p^r}$
is invertible on the restriction of $\mathfrak{I}_{J,\underline{\beta}}\mathfrak{E}_{\underline{\beta}}^{\rm conv}$ to $]\mathcal{X}_{J,\underline{\beta},k}[_{\widehat{{\mathfrak{X}}_{\underline{\beta}}^J}}$.  One checks on local coordinates \ref{sec:Chech} that  $\mathfrak{I}_{J,\underline{\beta}}$ is stable under Frobenius. The claim then follows as in the proof of the Claim in part (3) of Proposition \ref{prop:anabs}, replacing $Z_1,\ldots,Z_r$ with generators of $\mathfrak{I}_{J,\underline{\beta}}$ and using Lemma \ref{lemma:PHIM0} for $n=1$ and $r=\vert J \vert$.
\end{proof}

\subsection{Relation with Hyodo-Kato cohomology with support}\label{sec:HKcompact}

One defines complexes  $\mathrm{R}\Gamma_{\mathrm{abs}}\bigl((\mathcal{X}_k,\alpha_{\mathcal{X}_k})/\mathcal{S}_0,\mathfrak{E}^{\rm conv}(-\mathfrak{D})\bigr)$,  $\mathrm{R}\Gamma_{\mathrm{abs}}\bigl((\mathcal{X}_k,\alpha_{\mathcal{X}_k})/\mathcal{S}_0,\mathfrak{E}_0(-\mathfrak{D}_0)\bigr)$ and $\mathrm{R}\Gamma_{\mathrm{HK}}\bigl(\mathcal{X}_k/ \mathcal{S}_0^0,\mathfrak{E}_0(-\mathfrak{D}_0)\bigr)$ proceeding as in \S \ref{sec:HK}. We have the following:

\begin{proposition}\label{prop:anabscompact} Assume that $\mathcal{X}$ is proper over $\mathcal{O}_K$. Then, \smallskip

(1) We have an isomorphism, equivariant with respect to Frobenius, $$\mathrm{R}\Gamma_{\mathrm{an}}\bigl((\mathcal{X}_k,\alpha_{\mathcal{X}_k})/\mathcal{S}_0,\mathfrak{E}^{\rm conv}(-\mathfrak{D})\bigr)[1]\stackrel{\sim}{ \longrightarrow} \mathrm{R}\Gamma_{\mathrm{abs}}\bigl((\mathcal{X}_k,\alpha_{\mathcal{X}_k})/\mathcal{S}_0,\mathfrak{E}^{\rm conv}(-\mathfrak{D})\bigr).$$

(2) Each $\mathrm{H}^i_{\mathrm{abs}}\bigl((\mathcal{X}_k,\alpha_{\mathcal{X}_k})/\mathcal{S}_0,\mathfrak{E}_0(-\mathfrak{D}_0)\bigr)$ is a finite dimensional $K_0$-vector space.

\smallskip

(3) For every polynomial $P(T)\in K_0[T]$ of the form $P(T)=1+ T Q(T)$ the map on cohomology $$\mathrm{H}^i_{\mathrm{an}}\bigl((\mathcal{X}_k,\alpha_{\mathcal{X}_k})/\mathcal{S}_0,\mathfrak{E}^{\rm conv}(-\mathfrak{D})\bigr)\longrightarrow \mathrm{H}^{i-1}_{\mathrm{abs}}\bigl((\mathcal{X}_k,\alpha_{\mathcal{X}_k})/\mathcal{S}_0,\mathfrak{E}_0(-\mathfrak{D}_0)\bigr),$$induced by  $\delta$ and (1) defines an isomorphism on kernel and cokernel of $P(\Phi)$.

\smallskip

(4) For every $i$, omitting $\alpha_{\mathcal{X}_k}$ from the notation for simplicity, we have a Frobenius equivariant, exact sequence $$0\longrightarrow \frac{\mathrm{H}^{i-1}_{\mathrm{HK}}\bigl(\mathcal{X}_k/ \mathcal{S}_0^0,\mathfrak{E}_0(-\mathfrak{D}_0)\bigr)(-1)}{N\bigl(\mathrm{H}^{i-1}_{\mathrm{HK}}\bigl(\mathcal{X}_k/ \mathcal{S}_0^0,\mathfrak{E}_0(-\mathfrak{D}_0)\bigr)(-1)\bigr)} \longrightarrow \mathrm{H}^{i-1}_{\mathrm{abs}}\bigl(\mathcal{X}_k/\mathcal{S}_0,\mathfrak{E}_0(-\mathfrak{D}_0)\bigr) \longrightarrow $$ $$\longrightarrow \mathrm{H}^i_{\mathrm{HK}}\bigl(\mathcal{X}_k/ \mathcal{S}_0^0,\mathfrak{E}_0(-\mathfrak{D}_0)\bigr)^{N=0} \longrightarrow 0.$$
Recall that the twist by $-1$ means that the Frobenius acts as $p\phi$ if $\phi$ denotes the Frobenius acting on the untwisted cohomology.

\end{proposition}
\begin{proof}  It follows from Proposition \ref{prop:anabs}.

\end{proof}

\subsection{Log-rigid cohomology with support conditions}\label{sec:rigcmpct}

Consider $U\subset Y \subset \mathcal{X}_k$ subschemes with $U\subset \mathcal{X}_k^{\rm sm}$ open in the smooth locus of $\mathcal{X}_k$ and $Y\subset \mathcal{X}_k$ closed and reduced, as in \S \ref{sec:rigsyn}.  

First of all we introduce a rigid syntomic cohomology  of $U-\mathcal{D}
_{k}\subset Y$ with compact support along $\mathcal{D}_{k}$
by setting $\mathbf{Syn}_{\mathrm{rig},U\subset Y,\alpha
_{k}}^{c-\infty }\left( \mathcal{E},\mathfrak{E},r\right) :=\mathbf{
Syn}_{\mathrm{rig},U\subset Y,\alpha _{k}}\left( \mathcal{E}\left( -\mathcal{D
}\right) ,\mathfrak{E}\left( -\mathfrak{D}\right) ,r\right) $.

Next we define a rigid syntomic cohomology of $U\subset Y$
with compact support as follows. Let $Z:=Y \backslash U$ and $U^o=U\backslash \mathcal{D}_k$. With the notation of \S \ref{sec:Chech} and \S \ref{sec:rigsyn}, for every $\underline{\beta}\in B^n$ with $\beta_1<\ldots < \beta_n$  we let $$j_{\underline{\beta}} \colon ]Z_{\underline{\beta}} [_{\mathfrak{X}_{\underline{\beta}}}^{\log}\longrightarrow ]Y_{\underline{\beta}} [_{\mathfrak{X}_{\underline{\beta}}}^{\log}, \quad \iota_{\underline{\beta}}\colon ]Z_{\underline{\beta}} [_{\mathcal{X}}\longrightarrow ]Y_{\underline{\beta}} [_{\mathcal{X}} $$be the inclusions of the logarithmic tubes of $Z_{\underline{\beta}}:=Z \cap \mathcal{X}_{\underline{\beta},k}$  in the rigid analytic fiber of $\mathfrak{X}_{\underline{\beta}}$ in the tubes of $Y_{\underline{\beta}}$ in  the rigid analytic fiber  of $\mathfrak{X}_{\underline{\beta}}$, resp.~of $\mathcal{X}$. We get  natural morphisms  $$\mathfrak{E}^{\rm conv}(-\mathfrak{D})_{\underline{\beta}} \to j_{\underline{\beta},\ast} \circ j_{\underline{\beta}}^\ast\bigl(\mathfrak{E}^{\rm conv}(-\mathfrak{D})_{\underline{\beta}}\bigr), \quad \mathcal{E}(-\mathcal{D})_{\underline{\beta}} \to \iota_{\underline{\beta},\ast} \circ \iota_{\underline{\beta}}^\ast\bigl(\mathcal{E}(-\mathcal{D})_{\underline{\beta}}\bigr)$$compatible with logarithmic connections.  In particular, we get a morphism of the sections of the logarithmic de Rham complexes $$ \mathrm{dR}\bigl(\mathfrak{E}^{\rm conv}(-\mathfrak{D})_{\underline{\beta}}\bigr)\to \mathrm{dR}\bigl( j_{\underline{\beta},\ast} \circ j_{\underline{\beta}}^\ast\bigl(\mathfrak{E}^{\rm conv}(-\mathfrak{D})_{\underline{\beta}}\bigr)\bigr)$$and $$ \mathrm{dR}\bigl(\mathcal{E}(-\mathcal{D})_{\underline{\beta}}\bigr)\to \mathrm{dR}\bigl( \iota_{\underline{\beta},\ast} \circ \iota_{\underline{\beta}}^\ast\bigl(\mathcal{E}(-\mathcal{D})_{\underline{\beta}}\bigr)\bigr);$$see \S \ref{sec:rigsyn}. 
Define
$$\mathrm{dR}_{(U\subset Y,\alpha_k)/K_0}\bigl(\mathfrak{E}^{\rm conv}(-\mathfrak{D})\bigr)_{\underline{\beta}}, \quad \mathrm{dR}_{(U\subset Y,\alpha_k)/K} (\mathcal{E}(-\mathcal{D}))_{\underline{\beta}} $$to be the total complexes. Using Remark \ref{rmk:Stein}, one can prove that they compute the cohomology of the total complexes of the corresponding morphisms of logarithmic de Rham complexes of sheaves.  Following \cite[\S 3 \& \S 4]{Ber} denote by  $$ \mathrm{R\Gamma}_{\rm rig,c}\bigl( (U^o\subset Y,\alpha_k)/K_0,\mathfrak{E}^{\rm conv}\bigr), \quad  \mathrm{R\Gamma}_{\rm rig,c}\bigl( (U^o\subset Y,\alpha_k)/K ,\mathcal{E}\bigr)$$the total complexes defined varying the $\underline{\beta}$'s as in \S \ref{sec:dR} and let $\gamma_{{\rm rig},c}$ be the induced map obtained as in \S \ref{sec:rigsyn}.

\begin{definition}\label{def:syntomiccompactonU}  As in \S \ref{sec:rigsyn} define the {\em rigid syntomic complex} $\mathbf{Syn}_{{\rm rig},U^o\subset Y, \alpha_k}^c(\mathcal{E},\mathfrak{E},r)$  with support on $U^o=U\backslash \mathcal{D}_k$  to be the total complex associated to the morphism  of complexes $(1-\frac{\Phi}{p^r},\gamma_{{\rm rig},c})$  from the complex $\mathrm{R\Gamma}_{\rm rig,c}\bigl( (U^o\subset Y,\alpha_k)/K_0,\mathfrak{E}^{\rm conv}\bigr)$ to the complex $$ \mathrm{R\Gamma}_{\rm rig,c}\bigl( (U^o\subset Y,\alpha_k)/K_0,\mathfrak{E}^{\rm conv}\bigr) \oplus \frac{\mathrm{R\Gamma}_{\rm rig,c}\bigl( (U^o\subset Y,\alpha_k)/K ,\mathcal{E}\bigr)}{\mathrm{Fil}^r \mathrm{R\Gamma}_{\rm rig,c}\bigl( (U^o\subset Y,\alpha_k)/K ,\mathcal{E}\bigr)\bigr)}.$$
Consider the complex $\mathbf{Syn}^{c-\infty}(\mathcal{E},\mathfrak{E},r)$  introduced in Definition \ref{def:syntomiccompactcomplex}. By construction, we get a morphism of complexes $$\tau_{{\rm rig},c}\colon  \mathbf{Syn}_{{\rm rig},U^o\subset Y, \alpha_k}^c(\mathcal{E},\mathfrak{E},r) \longrightarrow \mathbf{Syn}^{c-\infty}(\mathcal{E},\mathfrak{E},r).$$
\end{definition}

\begin{lemma}\label{lemma:indchoicescompact}  The complexes $\mathrm{R\Gamma}_{\rm rig,c}\bigl( (U^o\subset Y,\alpha_k)/K_0,\mathfrak{E}^{\rm conv}\bigr)$ and $\mathrm{R\Gamma}_{\rm rig,c}\bigl( (U^o\subset Y,\alpha_k)/K ,\mathcal{E}\bigr)$, and hence also the complex $\mathbf{Syn}_{{\rm rig},U^o\subset Y, \alpha_k}^c(\mathcal{E},\mathfrak{E},r)$, are independent, up to homotopy, of the choices made in \S \ref{sec:Chech}.
\end{lemma}
\begin{proof}
By Remark \ref{rmk:Stein} we can replace $]Y_{\underline{\beta}} [_{\mathfrak{X}_{\underline{\beta}}}^{\log}$ and $]Y_{\underline{\beta}} [_{\mathcal{X}}$  in the above with admissible opens $V_{\underline{\beta},i}^{\log}$ and $V_{\underline{\beta},i}$ in strict neighbourhoods  of  $]U\cap Y_{\underline{\beta}}[_{\mathfrak{X}_{\underline{\beta}}}^{\log}$ and of $]U\cap Y_{\underline{\beta}} [_{\mathcal{X}}$ in $]Y_{\underline{\beta}} [_{\mathfrak{X}_{\underline{\beta}}}^{\log}$ and $]Y_{\underline{\beta}} [_{\mathcal{X}}$ respectively, in such a way that $V_{\underline{\beta},i}\subset V_{\underline{\beta},i}^{\log}$. Replace 
$]Z_{\underline{\beta}} [_{\mathfrak{X}_{\underline{\beta}}}^{\log}$ with its intersection with $V_{\underline{\beta},i}^{\log}$ and replace $ ]Z_{\underline{\beta}} [_{\mathcal{X}}$ with its intersection with $V_{\underline{\beta},i}$. Making compatible choices for varying $\underline{\beta}$'s and varying $i$'s we get complexes $ \mathrm{R\Gamma}_{\rm rig,c}\bigl( (U^o\subset Y,\alpha_k)/K_0,\mathfrak{E}^{\rm conv}\bigr)$ and $  \mathrm{R\Gamma}_{\rm rig,c}\bigl( (U^o\subset Y,\alpha_k)/K ,\mathcal{E}\bigr)$ analogous to those defined above.  Since, for fixed  $\underline{\beta}$ and varying $i$'s, the spaces $]Z_{\underline{\beta}} [_{\mathfrak{X}_{\underline{\beta}}}^{\log}$ and the open subspaces $V_{\underline{\beta},i}^{\log}$ give admissible covers of  $]Y_{\underline{\beta}} [_{\mathfrak{X}_{\underline{\beta}}}^{\log}$ and $ ]Z_{\underline{\beta}} [_{\mathcal{X}}$ and the open subspaces $V_{\underline{\beta},i}$ give  an admissible cover of 
$]Y_{\underline{\beta}} [_{\mathcal{X}}$, we deduce that these complexes are quasi-isomorphic to those used in Definition \ref{def:syntomiccompactonU}. The claim then follows  arguing as in the proof of Lemma \ref{lemma:indchoices}.
\end{proof}

It follows from the proof of Lemma \ref{lemma:indchoicescompact} that the complexes appearing in Definition \ref{def:syntomiccompactonU} can be replaced by analogous complexes defined on strict neighbourhoods  of  $]U\cap Y_{\underline{\beta}}[_{\mathfrak{X}_{\underline{\beta}}}^{\log}$ and of $]U\cap Y_{\underline{\beta}} [_{\mathcal{X}}$ in $]Y_{\underline{\beta}} [_{\mathfrak{X}_{\underline{\beta}}}^{\log}$ and $]Y_{\underline{\beta}} [_{\mathcal{X}}$ respectively. Considering the sheaves of overconvergent sections $\mathfrak{E}^\dagger$ and $\mathcal{E}^\dagger$  of \S\ref{sec:rigsyn} and the corresponding syntomic complex $ \mathbf{Syn}_{{\rm rig},U^o\subset Y, \alpha_k}^c(\mathcal{E}^\dagger,\mathfrak{E}^\dagger,r)$, we then obtain the following corollary that justifies the notation in Definition \ref{def:syntomiccompactonU}:

\begin{cor}\label{cor:BerCoh1} The natural morphisms $$\mathrm{R\Gamma}_{\rm rig,c}\bigl( (U^o\subset Y,\alpha_k)/K_0,\mathfrak{E}^{\rm conv}\bigr)\longrightarrow \mathrm{R\Gamma}_{\rm rig,c}\bigl( (U^o\subset Y,\alpha_k)/K_0,\mathfrak{E}^\dagger\bigr) $$ and $$\mathrm{R\Gamma}_{\rm rig,c}\bigl( (U^o\subset Y,\alpha_k)/K ,\mathcal{E}\bigr)\longrightarrow \mathrm{R\Gamma}_{\rm rig,c}\bigl( (U^o\subset Y,\alpha_k)/K ,\mathcal{E}^\dagger\bigr)$$are quasi-isomorphisms. In particular,  we have a natural quasi-isomorphism of complexes 
$$\mathbf{Syn}_{{\rm rig},U^o\subset Y, \alpha_k}^c(\mathcal{E},\mathfrak{E},r)\longrightarrow \mathbf{Syn}_{{\rm rig},U^o\subset Y, \alpha_k}^c(\mathcal{E}^\dagger,\mathfrak{E}^\dagger,r).$$
\end{cor}

\subsubsection{The case of $\mathcal{X}_k^{\rm sm}$}\label{sec:caseXsm}
We now specialize to the case $U=\mathcal{X}_k^{\rm sm}$, $U^o=\mathcal{X}_k^{\rm sm,o}=\mathcal{X}_k^{\rm sm}\backslash\mathcal{D}_k$ and $Y=\mathcal{X}_k$. 
Let $W_1,\ldots, W_s\subset \mathcal{X}_k$ be the irreducible components and let $W_i^o=W_i\cap U^o$.   For each $i$ and each $\underline{\beta}$, replacing the tubes of $\mathcal{X}_{\underline{\beta},k}$ and of $Z_k$ in the rigid analytic fiber of $\mathfrak{X}_{\underline{\beta}}$, resp.~of $\mathcal{X}$ with the tubes of $W_i\cap \mathcal{X}_{\underline{\beta},k}$ and of $Z_k\cap W_i\cap \mathcal{X}_{\underline{\beta},k}$  in the construction above we get complexes $\mathrm{R\Gamma}_{\rm rig,c}\bigl((W_i^o\subset W_i,\alpha_k)/K_0,\mathfrak{E}^{\rm conv}\bigr) $, $\mathrm{R\Gamma}_{\rm rig,c}\bigl((W_i^o\subset W_i,\alpha_k)/K ,\mathcal{E}\bigr)$  and  $\mathbf{Syn}_{{\rm rig},W_i^o\subset W_i,\alpha_k}^c(\mathcal{E},\mathfrak{E},r)$. We also have natural morphisms  from $$\mathrm{R\Gamma}_{\rm rig,c}\bigl( (\mathcal{X}_k^{\rm sm,o}\subset \mathcal{X}_k,\alpha_k)/K_0,\mathfrak{E}^{\rm conv}\bigr)\longrightarrow \oplus_{i=1}^s \mathrm{R\Gamma}_{\rm rig,c}\bigl( (W_i^o\subset W_i,\alpha_k)/K_0,\mathfrak{E}^{\rm conv}\bigr) $$and $$\mathrm{R\Gamma}_{\rm rig,c}\bigl((\mathcal{X}_k^{\rm sm,o}\subset \mathcal{X}_k,\alpha_k)/K ,\mathcal{E}\bigr) \longrightarrow\oplus_{i=1}^s\mathrm{R\Gamma}_{\rm rig,c}\bigl((W_i^o\subset W_i,\alpha_k)/K ,\mathcal{E}\bigr)$$and, hence,
$$\xi\colon \mathbf{Syn}_{{\rm rig},\mathcal{X}_k^{\rm sm,o}\subset \mathcal{X}_k, \alpha_k}^c(\mathcal{E},\mathfrak{E},r) \longrightarrow \oplus_{i=1}^s \mathbf{Syn}_{{\rm rig},W_i^o\subset W_i,\alpha_k}^c(\mathcal{E},\mathfrak{E},r) .$$

\begin{lemma}\label{lemma:xiiso} The morphisms above, and in particular $\xi$, are  quasi-isomorphisms.
\end{lemma}
\begin{proof} Write $U$ for $\mathcal{X}_k^{\rm sm}$ for simplicity. It suffices to show that  the first two morphisms are quasi-isomorphisms and that the morphism $\mathrm{Fil}^s\mathrm{R\Gamma}_{\rm rig,c}\bigl( U^o\subset \mathcal{X}_k/K ,\mathcal{E}\bigr) \to \oplus_{i=1}^s \mathrm{Fil}^r \mathrm{R\Gamma}_{\rm rig,c}\bigl( W_i^o \subset W_i/K ,\mathcal{E}\bigr)$, induced on $\mathrm{Fil}^r$, is also a quasi-isomorphism. This reduces  to prove that for every $\underline{\beta}$ the morphisms $$ \mathrm{dR}_{(U\subset \mathcal{X}_k,\alpha_k)/K_0} \bigl(\mathfrak{E}^{\rm conv}(-\mathfrak{D})\bigr)_{\underline{\beta}}\longrightarrow \oplus_{i=1}^r
\mathrm{dR}_{(U\cap W_i \subset W_i,\alpha_k)/K_0} \bigl(\mathfrak{E}^{\rm conv}(-\mathfrak{D})\bigr)_{\underline{\beta}},$$
$$\mathrm{dR}_{(U\subset Y,\alpha_k)/K} (\mathcal{E}(-\mathcal{D}))_{\underline{\beta}}\longrightarrow \oplus_{i=1}^s \mathrm{dR}_{(U\cap W_i \subset W_i,\alpha_k/K} (\mathcal{E}(-\mathcal{D}))_{\underline{\beta}}$$and similarly with $\mathrm{Fil}^r$ are quasi-isomorphisms. This follows if we show that for every locally free $\mathcal{O}_{\mathfrak{X}_{\underline{\beta}}}$-module $F$, resp.~$\mathcal{O}_{X_{\underline{\beta}}}$-module $F$ the cone of $F\to  j_{\underline{\beta},\ast} \circ j_{\underline{\beta}}^\ast(F)$, resp.~$F\to   \iota_{\underline{\beta},\ast} \circ \iota_{\underline{\beta}}^\ast(F)$ is quasi isomorphic to the cone provided by the sum of their restrictions to the tubes of $W_1\cap \mathcal{X}_{\underline{\beta},k}, \ldots, W_s\cap \mathcal{X}_{\underline{\beta},k}$  in the rigid analytic fiber of $\mathfrak{X}_{\underline{\beta}}$, resp.~of $X$.  This follows as the morphism $\amalg_i \bigl(W_i\cap \mathcal{X}_{\underline{\beta},k}\bigr) \to \mathcal{X}_{\underline{\beta},k}$ defines on logarithmic tubes in the rigid analytic fiber of $\mathfrak{X}_{\underline{\beta}}$, resp.~of $\mathcal{X}$ a cover by analytic open subspaces and $\amalg_{i<j} (W_i\cap W_j)\subset Z_k$.  
\end{proof}

Let  $U\subset W_\ell \subset \mathcal{X}_k$ for $\ell=1,\ldots,s$ and $U\subset \mathcal{X}_k^{\rm sm}$ an open non-empty subscheme. Set $U^o=U\backslash \mathcal{D}_k$. It follows from Definition \ref{def:syntomiccompactonU} that we have a morphism of complexes $$\mathbf{Syn}_{{\rm rig},U^o\subset W_\ell, \alpha_k}^c(\mathcal{E},\mathfrak{E},r) {\longrightarrow} \mathbf{Syn}_{{\rm rig},W_\ell^o\subset W_\ell, \alpha_k}^c(\mathcal{E},\mathfrak{E},r).$$ Using Lemma \ref{lemma:xiiso} and the morphism $\tau_{{\rm rig},c}$ from Definition \ref{def:syntomiccompactonU}, we get, for every $i\in \mathbb{N}$  a morphism
\begin{equation}\label{eq:tauRigc}
\tau_{\mathrm{rig},U,\ell,c}^i\colon\mathrm{H}^{i}\left( \mathbf{Syn}_{\mathrm{rig}
,U^{o}\subset W_\ell, \alpha_k}^{c}\left( \mathcal{E},\mathfrak{E},r\right)
\right) \longrightarrow \mathrm{H}^{i}\left( \mathbf{Syn}^{c-\infty}\left( \mathcal{
E},\mathfrak{E},r\right) \right) .  
\end{equation}
Recall from \S \ref{sec:rigsyn} that we also have a morphism 
\begin{equation} \label{eq:tauRig}
\tau_{{\rm rig},U,\ell}^i\colon \mathrm{H}^i\bigl(\mathbf{Syn}_{\rm an}(\mathcal{E},\mathfrak{E},r)\bigr) \longrightarrow \mathrm{H}^i\bigl(\mathbf{Syn}_{{\rm rig},U\subset W_\ell, \alpha_k}(\mathcal{E},\mathfrak{E},r)\bigr),
\end{equation}
as well as the analogue in which $\mathbf{Syn}$ is replaced by $\mathbf{Syn}%
^{c\text{-}\infty }$ obtained applying the construction to the twisted
sheaves.

\subsection{Coefficients with good reduction}\label{sec:Egoodred}

Consider the log-structure  $\alpha _{\mathcal{X}}^{\rm hor}$ on $\mathcal{X}_k$ arising from the horizontal divisor $\mathcal{D}_k$. We keep the notation of \S \ref{sec:rigsyn} and \S \ref{sec:rigcmpct} so that $U\subset Y\subset \mathcal{X}_k^{\rm sm}$ with $U \subset \mathcal{X}_k^{\rm sm}$ open  and $Y \subset \mathcal{X}_k$ closed and reduced. The good reduction alluded to in the title refers to the following.\smallskip

{\em Assume}  that $Y$ is smooth over $k$ and that the non-degenerate filtered Frobenius log-crystal  $ \bigl(\mathcal{E},\mathrm{Fil}^\bullet \mathcal{E}, \nabla,\mathfrak{E},\Phi_{\mathfrak{E}}\bigr)$ defines a Frobenius overconvergent log-isocrystal $\mathfrak{E}^\dagger$ as in \S \ref{sec:rigsyn} with respect to the log-structure $\alpha^{\rm hor}_k$.\smallskip

We define complexes $\mathrm{R}\Gamma_{\mathrm{rig}}\bigl((U\subset Y,\alpha_k^{\rm hor})/K_0,\mathfrak{E}^{\rm conv}\bigr)$ and $\mathbf{Syn}_{{\rm rig},U\subset Y, \alpha_k^{\rm hor}}\bigl(\mathcal{E},\mathfrak{E},r\bigr)$ proceeding as in \S \ref{sec:rigsyn} and $\mathbf{Syn}_{\mathrm{rig},U\subset Y,\alpha _{k}^{hor}}^{c\text{-}\infty
}\left( \mathcal{E},\mathfrak{E},r\right) :=\mathbf{Syn}_{\mathrm{rig}%
,U\subset Y,\alpha _{k}^{hor}}\left( \mathcal{E}\left( \mathcal{D}\right) ,%
\mathfrak{E}\left( \mathfrak{D}\right) ,r\right) $, $\mathrm{R}\Gamma_{\mathrm{rig},c}\bigl((U^o\subset Y,\alpha_k^{\rm hor})/K_0,\mathfrak{E}^{\rm conv}\bigr)$ and
$\mathbf{Syn}_{{\rm rig},U^o\subset Y, \alpha_k}\bigl(\mathcal{E},\mathfrak{E},r\bigr)$ as in \S \ref{sec:rigcmpct}. In terms of the complexes of loc.~cit.~we take the logarithmic complexes only with respect to the variables $Y_1,\ldots,Y_b$ and {\em not} with respect to the variables $X_1,\ldots,X_a$.
In particular, we have natural morphisms of complexes $\Xi$ from $$\mathrm{R}\Gamma_{\mathrm{rig}}\bigl((U\subset Y,\alpha_k^{\rm hor})/K_0,\mathfrak{E}^{\rm conv}\bigr)\otimes_{K_0} \mathrm{R}\Gamma_{\mathrm{rig}}\bigl((\mathcal{S}_k,\mathcal{M}_k)/K_0\bigr)$$to $\mathrm{R}\Gamma_{\mathrm{rig}}\bigl((U\subset Y,\alpha_k)/K_0,\mathfrak{E}^{\rm conv}\bigr)$ and $\Xi_c$ from
$$\mathrm{R}\Gamma_{\mathrm{rig},c}\bigl((U^o\subset Y,\alpha_k^{\rm hor})/K_0,\mathfrak{E}^{\rm conv}\bigr)\otimes_{K_0} \mathrm{R}\Gamma_{\mathrm{rig}}\bigl((\mathcal{S}_k,\mathcal{M}_k)/K_0\bigr)$$ to $\mathrm{R}\Gamma_{\mathrm{rig},c}\bigl((U^o \subset Y,\alpha_k)/K_0,\mathfrak{E}^{\rm conv}\bigr)$; recall that $\mathcal{S}_k=\mathrm{Spec}(k)$ with the log-structure $\mathcal{M}_k$ given by $\mathbb {N }\to k$ sending $0\neq n\mapsto 0$. 

\begin{lemma} The morphisms of complexes $\Xi$ and $\Xi_c$ are quasi-isomorphisms.
\end{lemma}
\begin{proof} It follows from Lemma \ref{lemma:indchoices} and  Lemma \ref{lemma:indchoicescompact}
and the Assumption that one can compute $\mathrm{R}\Gamma_{\mathrm{rig}}\bigl((U\subset Y,\alpha_k)/K_0,\mathfrak{E}^{\rm conv}\bigr)$ and $\mathrm{R}\Gamma_{\mathrm{rig},c}\bigl((U^o \subset Y,\alpha_k)/K_0,\mathfrak{E}^{\rm conv}\bigr)$ as the tensor product over $K_0$ of (1) de Rham complexes of $\mathfrak{E}^\dagger$, with or without support conditions, evaluated on the rigid analytic fiber of local lifts of $Y$, with log-structure given by the divisor $\mathcal{D}_k\cap Y\subset Y$, as a log smooth scheme over $\mathcal{O}_{K_0}$ and (2)  the de Rham complex of the open unit disk associated to the formal scheme  $\mathcal{O}_{K_0}[\![Z]\!]$ (with log-structure defined by the divisor $Z=0$).  As (1) represents $\mathrm{R}\Gamma_{\mathrm{rig}}\bigl((U\subset Y,\alpha_k^{\rm hor})/K_0,\mathfrak{E}^{\rm conv}\bigr)$, respectively $\mathrm{R}\Gamma_{\mathrm{rig},c}\bigl((U^o\subset Y,\alpha_k^{\rm hor})/K_0,\mathfrak{E}^{\rm conv}\bigr)$ and (2) represents  $\mathrm{R}\Gamma_{\mathrm{rig}}\bigl((\mathcal{S}_k,\mathcal{M}_k)/K_0\bigr)$, the conclusion follows.
\end{proof}

Notice that  $K_0 \oplus K_0(-1)[-1]\cong \mathrm{R}\Gamma_{\mathrm{rig}}\bigl((\mathcal{S}_k,\mathcal{M}_k)/K_0\bigr) $, where $[-1]$ stands for a shift as a complex while $(-1)$ stands for the fact that Frobenius is multiplied by $p$. In terms of the logarithmic rigid cohomology of the rigid space $\mathfrak{S}^{\rm rig}$ associated to $\mathcal{O}_{K_0}[\![Z]\!]$, the term $K_0$ is identified with $\mathrm{H}^0$ of the logarithmic de Rham complex of $\mathfrak{S}^{\rm rig}$,  while $K_0(-1)[-1]$ is identified with the element $K_0 \cdot dZ/Z\subset \Omega^{\rm log,1}_{\mathfrak{S}^{\rm rig}}$ which maps isomorphically onto $\mathrm{H}^1$ of the logarithmic de Rham complex of $\mathfrak{S}^{\rm rig}$.
We then get morphisms of complexes
$$\nu_{{\rm rig}, U\subset Y,\ast}\colon \mathbf{Syn}_{{\rm rig},U\subset Y, \alpha_k^{\rm hor}}\bigl(\mathcal{E},\mathfrak{E},r\bigr) \longrightarrow \mathbf{Syn}_{{\rm rig},U\subset Y, \alpha_k}\bigl(\mathcal{E},\mathfrak{E},r\bigr)$$with a retract that we denote $\nu_{{\rm rig}, U\subset Y}^\ast$ that identifies the latter as the direct sum of the former and the total complex associated to the endomorphism $1-\frac{\Phi}{p^r}$ of the complex $ \mathrm{R}\Gamma_{\mathrm{rig}}\bigl((U\subset Y,\alpha_k^{\rm hor})/K_0,\mathfrak{E}^{\rm conv}\bigr)(-1)[-1]$.
Similarly for the complexes with support we get morphisms
$$\nu_{{\rm rig}, U\subset Y,c,\ast}\colon \mathbf{Syn}^c_{{\rm rig},U\subset Y, \alpha_k^{\rm hor}}\bigl(\mathcal{E},\mathfrak{E},r\bigr)\longrightarrow \mathbf{Syn}^c_{{\rm rig},U\subset Y, \alpha_k}\bigl(\mathcal{E},\mathfrak{E},r\bigr)$$with a retract denoted $\nu_{{\rm rig}, U\subset Y,c}^\ast$.
In case $Y=W_\ell$ is one of the irreducible components of $\mathcal{X}_k$ and it is smooth over $k$, we simply write $\nu_{{\rm rig},U,\ell}$ and $\nu_{{\rm rig},U,\ell,c}$ for the morphisms above. We then get for every $i\in \mathbb{N}$ morphisms in cohomology
\begin{equation} \label{eq:rhoRig}
\nu_{{\rm rig},U,\ell,\ast}^i\colon \mathrm{H}^i\bigl(\mathbf{Syn}_{{\rm rig},U\subset W_\ell, \alpha_k^{\rm hor}}(\mathcal{E},\mathfrak{E},r)\bigr) \longrightarrow\mathrm{H}^i\bigl(\mathbf{Syn}_{{\rm rig},U\subset W_\ell, \alpha_k}(\mathcal{E},\mathfrak{E},r)\bigr)
\end{equation}with a left inverse $ \nu_{{\rm rig},U,\ell}^{\ast,i},$
as well as the analogue in which $\mathbf{Syn}$ is replaced by $\mathbf{Syn}%
^{c\text{-}\infty }$ obtained applying the construction to the twisted
sheaves, and, similarly, a morphism
\begin{equation}\label{eq:rhoRigc}
\nu_{\mathrm{rig},U,\ell,c,\ast}^i\colon \mathrm{H}^{i}\left( \mathbf{Syn}_{\mathrm{rig}
,U^{o}\subset W_\ell, \alpha_k^{\rm hor}}^{c}\left( \mathcal{E},\mathfrak{E},r\right)
\right) \longrightarrow \mathrm{H}^{i}\left( \mathbf{Syn}_{\mathrm{rig}
,U^{o}\subset W_\ell, \alpha_k}^{c}\left( \mathcal{E},\mathfrak{E},r\right)
\right)
\end{equation}
with a left inverse $\nu_{\mathrm{rig},U,\ell,c}^{\ast,i}$.

\subsection{An acyclicity result} 
Let $n$ be a positive integer. For every  $p\in \mathbb{N}$,
write $I_{p}\subset \left\{ 1,\ldots,n\right\} ^{p}$ for the set of
multi-indexes $\mathbf{i}=\left( i_{1},\ldots,i_{p}\right) $ such that $
i_{1}<\ldots<i_{p}$. For
every $\mathbf{i}\in I_{p}$, we denote by $\left\{ \mathbf{i}\right\} $ the
set $\left\{ i_{1},\ldots,i_{p}\right\} $. If $k\in
\left\{ 1,\ldots,n\right\} $, we set $k\wedge \mathbf{i}=\phi $ in case $k\in
\left\{ \mathbf{i}\right\} $ and, otherwise, we write $k\wedge \mathbf{i}$
for the unique element of $I_{p+1}$ such that $\left\{ k\wedge \mathbf{i}
\right\} =\left\{ k\right\} \cup \left\{ \mathbf{i}\right\} $. 

\begin{definition}
Let $M$ be an $R$-module and let $\varphi _{1},\ldots,\varphi _{n}\in \mathrm{End}_{R}\left( M\right) $  be commuting $R$-linear maps. For each $p\in \N$ set $$K^{p}\left( \varphi _{1},\ldots,\varphi _{n}\right) :=M\otimes
_{R}\wedge _{R}^{p}\left( R^{n}\right) =\bigoplus\nolimits_{\mathbf{i}\in
I_{p}}M.$$If $\mathbf{i}\in I_{p}$, we write $e_{\mathbf{i}
}:=e_{i_{1}}\wedge \ldots\wedge e_{i_{p}}\in \wedge _{R}^{p}\left( R^{n}\right) $ for  the canonical basis elements. If $m\in M$, we write $m_{\mathbf{i}}\in \bigoplus\nolimits_{\mathbf{i}\in
I_{p}}M$ for the element associated to $m\otimes _{R}e_{\mathbf{i}
}$.  

Let $d^{p}\colon K^{p}\left( \varphi _{1},\ldots,\varphi _{n}\right)
\rightarrow K^{p+1}\left( \varphi _{1},\ldots,\varphi _{n}\right) $ to be the $R$-linear map defined by 
\begin{equation}
d^{p}\left( m\otimes _{R}e_{\mathbf{i}}\right)
=\sum\nolimits_{k=1}^{n}\varphi _k\left( m\right) \otimes _{R}e_k\wedge
e_{\mathbf{i}}=\sum\nolimits_{k=1}^{n}\left( -1\right) ^{\varepsilon _{\mathbf{i}}\left( k\right)}\varphi
_k\left( m\right) \otimes _{R}e_{k\wedge \mathbf{i}},
\label{Kos F1}
\end{equation}where  $
\varepsilon _{\mathbf{i}}\left( k\right) :=\#\left\{ i\in \left\{ \mathbf{i}
\right\} \vert i<k\right\} $. 

Define  $K^\ast\left( \varphi
_{1},\ldots,\varphi _{n}\right)=\oplus_{p=0,\ldots,n} K^{p}\left( \varphi _{1},\ldots,\varphi _{n}\right) $ and let $d^\ast$ be the degree $1$ map $d^\ast=\sum_p d^p\colon K^\ast\left( \varphi
_{1},\ldots,\varphi _{n}\right) \to K^\ast\left( \varphi
_{1},\ldots,\varphi _{n}\right)$.

\end{definition}

We have the following lemma whose proof we leave to the reader:

\begin{lemma}\label{lemma:Koszul}  The following hold:\smallskip

(i) Given $\mathbf{j}\in I_{p+1}$, then $d^{p}\left( m_{\mathbf{i}
}\right) _{\mathbf{j}}=0$ unless  $\left\{ \mathbf{i}\right\} \subset \left\{ \mathbf{j}\right\} $. In this case, if  $\mathbf{j}=\left( j_{1},\ldots,j_{p+1}\right) $, then we have $\mathbf{i}=
\left( j_{1},\ldots,\widehat{j_k},\ldots,j_{p+1}\right) $ for a unique $j_k\in
\left\{ \mathbf{j}\right\} $ and $d^{p}\left( m_{\mathbf{i}
}\right) _{\mathbf{j}}=\left( -1\right) ^{k+1}\varphi _{j^{k}}\left(
m\right) _{\mathbf{j}}.$

\smallskip

(ii) We have $d^\ast\circ d^\ast=0$

\smallskip

(iii) Let $\iota \colon R^{n-1} \to R^n$ be the map induced by the inclusion of $R^{n-1}$ into the first $n-1$-components. Then, the canonical isomorphism $\wedge
_{R}^{p}\left( R^{n-1}\right) \oplus \wedge _{R}^{p-1}\left( R^{n-1}\right) 
\overset{\sim }{\rightarrow }\wedge _{R}^{p}\left( R^{n}\right) $ sending $
\left( x,y\right) $ to $\iota ^{p}\left( x\right) +e_{n}\wedge \iota
^{p-1}\left( y\right) $ identifies
\begin{equation*}
K^\ast\left( \varphi _{1},\ldots,\varphi _{n}\right) =\mathrm{Cone}\left(
-\varphi _{n}:K^\ast\left( \varphi _{1},\ldots,\varphi _{n-1}\right)
\rightarrow K^\ast\left( \varphi _{1},\ldots,\varphi _{n-1}\right) \right) %
\left[ -1\right].
\end{equation*}

\smallskip

(iv) Suppose that
\begin{equation*}
0\longrightarrow M^{\prime }\longrightarrow M\longrightarrow M^{\prime
\prime }\longrightarrow 0
\end{equation*}
is an exact sequence of $R$-modules and that  $\varphi
_{i}\left( M^{\prime }\right) \subset M^{\prime }$ for every $i=1,\ldots,n$.
Write $\left\{ \varphi _{1}^{\prime },\ldots,\varphi _{n}^{\prime }\right\}
\subset \mathrm{End}_{R}\left( M^{\prime }\right) $ and $\left\{ \varphi _{1}^{\prime
\prime },\ldots,\varphi _{n}^{\prime \prime }\right\} \subset \mathrm{End}_{R}\left(
M^{\prime \prime }\right) $ for the induced morphisms. Then we have 
an exact sequence of complexes of $R$-modules%
\begin{equation*}
0\longrightarrow K^\ast\left( \varphi _{1}^{\prime },\ldots,\varphi
_{n}^{\prime }\right) \longrightarrow K^\ast\left( \varphi
_{1},\ldots,\varphi _{n}\right) \longrightarrow K^\ast\left( \varphi
_{1}^{\prime \prime },\ldots,\varphi _{n}^{\prime \prime }\right)
\longrightarrow 0.
\end{equation*}

\smallskip

(v) Suppose that $A$ is an $R$-algebra and $M$ is an $A$-module. Then $K^{p}\left( \varphi
_{1},\ldots,\varphi _{n}\right) \cong M\otimes _{A}\wedge _{A}^{p}\left(
A^{n}\right) $, compatibly with differentials where on the RHS we simply replace $\otimes _{R}$ with $\otimes _{A}$ in (\ref{Kos F1}).

(vi) With the notation as in (v), assume that $I\subset A$ is an ideal such that $\varphi
_{i}\left( IM\right) \subset IM$ for every $i=1,\ldots,n$. Then, we get a
commuting elements $\varphi _{1\mid IM},\ldots,\varphi _{n\mid
IM}\in \mathrm{End}_{R}\left( IM\right) $ and a natural identification $$IK^{\ast
}\left( \varphi _{1},\ldots,\varphi _{n}\right) =K^{p}\left( \varphi _{1\mid IM},\ldots,\varphi _{n\mid IM}\right) \subset K^\ast\left(
\varphi _{1},\ldots,\varphi _{n}\right) .$$We also have an identification $K^\ast\left( \widehat{\varphi }_{1},\ldots,\widehat{\varphi }
_{n}\right) =\widehat{K}^\ast\left( \varphi _{1},\ldots,\varphi _{n}\right) 
$ of the $I$-adic completions.

\end{lemma}

We now prove the following key technical result. Let $A$ be a noetherian $R$-algebra, $I:=\left(
x_{1},\ldots,x_{n_{o}}\right) \subset A$ an ideal and
$M$ is a finitely generated $A$-module.   For every non-empty subset $S\subset \left\{
1,\ldots,n_{o}\right\} $, set $I_{S}\subset A$ to be the ideal $I_S:=\left( x_{i}:i\in S\right) $.

\begin{lemma}\label{lemma:Kos L1} 
Let $d_{1},\ldots d_{n}\in \mathrm{End}_{R}(
A) $  be endomorphisms such that for $j=1,\ldots,n_{o}$ we have  $d_{i}\left( x_{j}\right) =0$ if $i\neq j$ and $d_{j}\left( x_{j}\right) =x_{j}$. Let $ \nabla _{1},\ldots,\nabla _{n}\in \mathrm{End}_{R}(M) $ be commuting homomorphisms such that, for every $i=1,\ldots,n$,we have
\begin{equation}
\nabla _{i}\left( am\right) =d_{i}\left( a\right) m+a\nabla _{i}\left(
m\right) \text{ for every }a\in A\text{ and }m\in M  \label{Kos L1 F1 Clm}
\end{equation}

Assume that the residue $\mathrm{Res}(\nabla _{i})\in \mathrm{End}_R(M/x_i M)$ is
nilpotent for every $i=1,\ldots,n_{o}$.  Then  $I_{S}K^\ast\left( 
\widehat{\nabla }_{1},\ldots,\widehat{\nabla }_{n}\right) $ is an acyclic
subcomplex of $K^\ast\left( \widehat{\nabla }_{1},\ldots,\widehat{
\nabla }_{n}\right) $.  
\end{lemma}

\begin{proof} We remark that due our assumption $\nabla_i(x_i M)\subset x_i M$ so that $\mathrm{Res}(\nabla _{i})$ is well defined. Using (\ref{Kos L1 F1 Clm}) and the assumption that $ d_{i}\left( x_{j}\right) =0$ or $x_{j}$ if $i=j$ for $j\in \left\{ 1,\ldots,n_{o}\right\} $, we deduce from Property (v) that $I_{S}K^\ast\left( \widehat{\nabla }_{1},\ldots,\widehat{\nabla }
_{n}\right) \subset K^\ast\left( \widehat{\nabla }_{1},\ldots,\widehat{
\nabla }_{n}\right) $ is a subcomplex.

Notice that  $K^\ast\left( \widehat{\nabla }_{1},\ldots,\widehat{\nabla }
_{n}\right) =K^\ast\left( \nabla _{1},\ldots,\nabla _{n}\right) $ when $M$
is $I$-adically complete and separated by Prperty (vi). Since we have the equality of $I$-adic completions $\widehat{\frac{M}{I_{\left\{
i\right\} }M}}\simeq \frac{\widehat{M}}{I_{\left\{ i\right\} } \widehat{M}}$
the nilpotency of $\mathrm{Res}\left( \nabla _{i}\right) $
implies the nilpotency of $\mathrm{Res}\left( \widehat{\nabla }_{i}\right) $. 
In order to prove the claim we can then assume that $M$ is already $I$
-adically complete and separated. Setting $r:=\left\vert S\right\vert $, we
will prove our claim by a double induction.

\emph{Step 1: }$\left( r,n_{o},n\right) =\left( 1,1,1\right) $. The generic
element $z$ of $x_{1}K^{0}\left( \nabla _{1}\right) $ is of the form $
z=x_{1}z_{1}$ where $z_{1}\in K^{0}\left( \nabla _{1}\right) $\ and, setting 
$N:=\nabla _{1}-{\rm Id}$, we have $N\left( z\right) =x_{1}\nabla _{1}\left( z_{1}\right) $. By
induction, we find that $N^{t}\left( z\right) =x_{1}\nabla _{1}^{t}\left(
z_{1}\right) $ for every $t\in \mathbb{N}_{\geq 1}$. Our assumption that $
\mathrm{Res}\left( \nabla _{1}\right) $ is nilpotent implies that there
exists some $t\in \mathbb{N}_{\geq 1}$ such that $\nabla _{1}^{t}\left(
y\right) \in x_{1}M$ for every $y\in M$. Hence, $N^{t}\left(
z\right) \in x_{1}^{2}M$ for every $z\in x_{1}K^{0}\left( \nabla _{1}\right) 
$ or, equivalently, considering $End_{R}\left( x_{1}M\right) $ as a left $A$
-module via the composition with the multiplication by elements of $A$ on
the codomain, then $N^{t}\in
x_{1}End_{R}\left( x_{1}M\right) $, implying that $N$ is topologically
nilpotent in the $I$-adic topology. By a geometric series argument, we deduce that
\begin{equation*}
\nabla _{1}=1+ N_1\in A^{\times }\left( 1+ x_1 \mathrm{End}_{R}\left(
x_{1}M\right) \right) \subset \mathrm{Aut}_{R}\left( x_{1}M\right) \text{.}
\end{equation*}
is an isomorphism, proving the claim in this case.

\emph{Step 2: }$\left( r,n_{o},n\right) =\left( 1,n_{o},n\right) $. We may
assume without loss of generality that $S=\left\{ 1\right\} $ and we prove
our assertion by induction on $n$, the case $n=1$ being discussed in Step 1. 
Suppose that $n>1$. According to $\left( v\right) $
above, $x_{1}K^\ast\left( \nabla _{1},\ldots,\nabla _{m}\right) =K^{\ast
}\left( \nabla _{1\mid x_{1}M},\ldots,\nabla _{m\mid x_{1}M}\right) $ for $m\in
\left\{ n-1,n\right\} $ and, hence, thanks to Property (iii) we know
that $x_{1}K^\ast\left( \nabla _{1\mid
x_{1}M},\ldots,\nabla _{n\mid x_{1}M}\right) $ is the cone of
$$ -\nabla
_{n\mid x_{1}M}:x_{1}K^\ast\left( \nabla _{1\mid x_{1}M},\ldots,\nabla
_{n-1\mid x_{1}M}\right) \rightarrow x_{1}K^\ast\left( \nabla _{1\mid
x_{1}M},\ldots,\nabla _{n-1\mid x_{1}M}\right)$$twisted by $-1$. Using the inductive step one deduces that  $
x_{1}K^\ast\left( \nabla _{1\mid x_{1}M},\ldots,\nabla _{n-1\mid
x_{1}M}\right) $ is acyclic and we conclude the claimed acyclicity of $
x_{1}K^\ast\left( \nabla _{1\mid x_{1}M},\ldots,\nabla _{n\mid
x_{1}M}\right) $ using the identification with the cone.

\emph{Step 3: arbitrary }$\left( r,n_{o},n\right) $\emph{\ with }$1\leq
r\leq n_{o}\leq n$. We may assume without loss of generality that $S=\left\{
1,\ldots,r\right\} $ and we prove our assertion by induction on $r$, the case $
r=1$ being discussed in Step 2. Suppose that $r>1$ and consider
the exact sequence
\begin{equation*}
0\longrightarrow I_{\left\{ 1,\ldots,r-1\right\} }M\longrightarrow I_{\left\{
1,\ldots,r\right\} }M\longrightarrow \frac{I_{\left\{ 1,\ldots,r\right\} }M}{
I_{\left\{ 1,\ldots,r-1\right\} }M}\longrightarrow 0\text{.}
\end{equation*}
Set $M^{\prime \prime }:=\frac{M}{I_{\left\{ 1,\ldots,r-1\right\} }M}$. Then $
x_{r}M^{\prime \prime }=\frac{I_{\left\{ 1,\ldots,r\right\} }M}{I_{\left\{
1,\ldots,r-1\right\} }M}$. Write $\nabla
_{i}^{\prime \prime }$ for the operator induced by $\nabla_i$ on $x_{r}M^{\prime \prime }$.
Thanks to Property (iv)  we have an exact sequence of complexes of $R$-modules expressing $K^\ast\left( \nabla _{1\mid I_{\left\{
1,\ldots,r\right\} }M},\ldots,\nabla _{n\mid I_{\left\{ 1,\ldots,r\right\} }M}\right)$ as an extension of 
$K^\ast\left( \nabla _{1}^{\prime \prime },\ldots,\nabla
_{n}^{\prime \prime }\right)$ by $K^\ast\left( \nabla _{1\mid I_{\left\{
1,\ldots,r-1\right\} }M},\ldots,\nabla _{n\mid I_{\left\{ 1,\ldots,r-1\right\}
}M}\right) $.  By induction, we know that the latter is acyclic. Hence, it suffices to show that  that $K^\ast\left( \nabla _{1}^{\prime \prime },\ldots,\nabla
_{n}^{\prime \prime }\right)$ is acyclic. This follows from the case $\left(
r,n_{o},n\right) =\left( 1,n_{o},n\right) $ discussed in Step 2.
\end{proof}

\subsubsection{End of proof of Proposition \ref{prop:samecohogroups}}\label{section:incarnation} 
Going back to the proof of Proposition \ref{prop:samecohogroups} let $A$ be 
the sheaf of global sections of $\mathcal{X}_{
\underline{\beta }}$.  Let $\Omega_A $ be the $A$-module of continuous logarithmic  K\"{a}hler differentials of $\mathcal{X}_{\underline{\beta }}$.  Since we are in the setting of \S \ref{sec:dR}, we can take $
\left\{ x_{1},\ldots,x_{n_{o}}\right\} $ to be the set $\left\{
Y_1,\ldots, Y_b\right\} $ and $\left\{ x_{n_o},\ldots,x_n\right\}$ which corresponds to the variables  $\left\{X_1,\ldots,X_{a-1}, v_{i,j},u_{i,k}\vert j=1,\ldots,a, k=1,\ldots,b\right\}$. Their $d\log$ define a basis of $\Omega_A$ and, hence,  derivations $d_i\colon A \to A$ such that $d_i(x_j)=0$ for $i\neq j$ and $d_i(x_i)=x_i$ if $i=1,\ldots, n_o$. Similarly let $M$ be the global section of $\mathcal{E}$ over $\mathcal{X}_{
\underline{\beta }}$. Since $\mathcal{E}$ arises for a log-isocrystal, it comes endowed with an integrable connection $\nabla$. Using the given basis of $\Omega_A$ we get commuting operators $\nabla_1,\ldots,\nabla_n$ that satisfy Condition (\ref{Kos L1 F1 Clm}) of Lemma  \ref{lemma:Kos L1}. Also the nilpotency assumption on the residues in Lemma  \ref{lemma:Kos L1} holds since $\mathcal{E}$ is locally free and the nilpotency holds on the generic fiber by assumption in Proposition \ref{prop:samecohogroups}. Finally, one checks that via the given trivialization of $\Omega_A\cong A^n$, the logarithmic de Rham complex $\mathrm{dR}(M)$ of $M$ coincides with the complex obtained from $K^\ast\left( \nabla _{1},\ldots,\nabla _{n}\right)$ and the operator $d^\ast$.  Lemma \ref{lemma:Kos L1} then implies that $I\widehat{\mathrm{dR}(M)}$ is acyclic. If we invert $p$ this complex is the complex $I_{J,\underline{\beta }
}R\Gamma _{\mathrm{logdR}}\left( X_{J,\underline{\beta }},\widehat{\mathcal{E
}}_{J,\underline{\beta }}^{\partial }\right) $ so that Claim (1) in the proof of Proposition \ref{prop:samecohogroups} follows from Lemma  \ref{lemma:Kos L1}.

\section{Period sheaves}\label{sec:periodsheaves}

Following \cite[\S 2.2]{DLLZb} we have period sheaves $\widehat{\mathcal{O}}_X$ and $\mathbb{A}_{\mathrm{inf}}$ on $X_{\mathrm{pket}}$ with a morphism  $\vartheta\colon \mathbb{A}_{\mathrm{inf}} \to \widehat{\mathcal{O}}_X^+$ with locally principal kernel. From these we get filtered period sheaves $\mathbb{B}_{\mathrm{crys}}$ and $\mathbb{B}_{\mathrm{dR}}$. For every $J$ we have the sheaves $\widehat{\mathcal{O}}_{X_J}^{\partial,+}$, $\mathbb{A}_{\mathrm{inf},J}^\partial$ with the morphism $\vartheta_J^\partial\colon \mathbb{A}_{\mathrm{inf},J}^\partial \to \widehat{\mathcal{O}}_{X_J}^{\partial,+}$ 
and the filtered period sheaves $\mathbb{B}_{\mathrm{crys},J}^\partial$ and $\mathbb{B}_{\mathrm{dR},J}^\partial$ on $X_{\mathrm{pket}}$  obtained as push-forward via $\iota_J\colon X_J^\partial\to X$ of the corresponding period sheaves on $X^\partial_{J,\mathrm{pket}}$; see \cite[Def. 2.5.2]{LLZ} for the de Rham version. For the crystalline version one defines $\mathbb{A}_{\mathrm{crys},J}^\partial$ as in \cite[Def. 2.1]{TT} as the $p$-adic completion of the DP 
envelope of the ideal $\mathrm{ker}(\vartheta_J^\partial)$, which is locally principal by \cite[Cor. 2.5.3]{LLZ}. Then $\mathbb{B}_{\mathrm{crys},J}^\partial= \mathbb{A}_{\mathrm{crys},J}^\partial[t^{-1}]$, with $t$ Fontaine's element. 
The sheaves $\widehat{\mathcal{O}}_X$, $\mathbb{B}_{\mathrm{crys}}$, $\mathbb{B}_{\mathrm{dR}}$ etc. correspond to the case $J=\emptyset$. Set $\mathbb{B}_{\mathrm{crys},(a)}^\partial:=\oplus_{J\subset I, \vert J\vert=a} \mathbb{B}_{\mathrm{crys},J}^\partial$ and $\mathbb{B}_{\mathrm{dR},(a)}^\partial=\oplus_{J\subset I, \vert J\vert=a} \mathbb{B}_{\mathrm{dR},J}^\partial$. Define the  period sheaves with compact support
$$\mathbb{B}_{\mathrm{crys}}^{c-\infty}:=\mathbb{B}_{\mathrm{crys},(\bullet)}^\partial, \qquad 
\mathbb{B}_{\mathrm{dR}}^{c-\infty}:=\mathbb{B}_{\mathrm{dR},(\bullet)}^\partial $$with filtration $\mathrm{Fil}^r \mathbb{B}_{\mathrm{crys}}^{c-\infty}:=\mathrm{Fil}^r\mathbb{B}_{\mathrm{crys},(\bullet)}^\partial$ and $\mathrm{Fil}^r \mathbb{B}_{\mathrm{dR}}^{c-\infty}:=\mathrm{Fil}^r\mathbb{B}_{\mathrm{dR},(\bullet)}^\partial$.  The crystalline period sheaves are endowed with Frobenius endomorphisms and we have injections, strictly compatible with filtrations,  $$\mathbb{B}_{\mathrm{crys}}\subset \mathbb{B}_{\mathrm{dR}}, \qquad \mathbb{B}_{\mathrm{crys}}^{c-\infty}\subset \mathbb{B}_{\mathrm{dR}}^{c-\infty}.$$
\begin{remark} In \cite[Def.~3.3.2]{LLZ} one finds the definition of $\mathbb{B}_{\mathrm{dR}}^{c-\infty}$ as the kernel of $\mathbb{B}_{\mathrm{dR},(0)}^\partial \to \mathbb{B}_{\mathrm{dR},(1)}^\partial$. It is proven in 
\cite[Lemma 3.3.7]{LLZ} that the natural map from this to $\mathbb{B}_{\mathrm{dR},(\bullet)}^\partial$ is an isomorphism. 
\end{remark}

In  \cite{DLLZb} and \cite{LLZ} relative versions of these sheaves are introduced, $\mathcal{O}\mathbb{B}_{\mathrm{dR}}$ and $\mathcal{O}\mathbb{B}_{\mathrm{dR}}^{c-\infty}$. Furthermore by \cite[Cor. 2.4.2]{DLLZb} and  \cite[Cor. 2.5.13]{LLZ} we have quasi-siomorphisms of filtered complexes $$ \mathbb{B}_{\mathrm{dR}} \longrightarrow\mathcal{O}\mathbb{B}_{\mathrm{dR}}\otimes_{\mathcal{O}_X} \Omega^{\log,\bullet}_X,\qquad \mathbb{B}_{\mathrm{dR}}^{c-\infty} \longrightarrow\mathcal{O}\mathbb{B}_{\mathrm{dR}}^{c-\infty}\otimes_{\mathcal{O}_X} \Omega^{\log,\bullet}_X.$$

Under Assumption \ref{ass:formalcase} and fixing $\beta\in B$ one has also relative versions ${\mathcal{O}}_{\mathfrak{X}_\beta}\mathbb{B}_{\mathrm{crys}}$ and ${\mathcal{O}}_{\mathfrak{X}_\beta}\mathbb{B}_{\mathrm{crys}}^{c-\infty}$ on $X_{\beta,\mathrm{pket}}$ as follows. Let $w\colon X_\beta \to \mathcal{X}_{\beta,k}$ be the specialization map. It induces a mophism on \'etale sites.  Define ${\mathcal{O}}_{\mathfrak{X}_\beta}\mathbb{A}_{\mathrm{crys}}$ as the $p$-adic completion of the log DP envelope of $$\vartheta_\beta\colon w^{-1}\bigl(\mathcal{O}_{\mathfrak{X}_\beta}\bigr)\otimes_{\mathcal{O}_{K_0}} \mathbb{A}_{\mathrm{inf}} \to \widehat{\mathcal{O}}_{X_\beta}^+$$and similarly for ${\mathcal{O}}_{\mathfrak{X}_\beta}\mathbb{A}_{\mathrm{crys},J}^\partial $ as the pushforward via the open immersion $X_\beta\subset X$ of the $p$-adic completion of the log DP envelope of $$\vartheta_\beta^\partial \colon w^{-1}\bigl(\mathcal{O}_{\mathfrak{X}_\beta}\bigr)\otimes_{\mathcal{O}_{K_0}} \mathbb{A}_{\mathrm{inf},J}^\partial \to \widehat{\mathcal{O}}_{X_{J,\beta}}^{\partial,+};$$see \cite[Lemma 2.16]{AI}. They are endowed with Frobenius morphisms induced by the given  Frobenius on $\mathfrak{X}_\beta$ and Frobenius on $ \mathbb{A}_{\mathrm{inf}}$. Define ${\mathcal{O}}_{\mathfrak{X}_\beta}\mathbb{B}_{\mathrm{crys}}$ and ${\mathcal{O}}_{\mathfrak{X}_\beta}\mathbb{B}_{\mathrm{crys},J}^{c-\infty}$ as ${\mathcal{O}}_{\mathfrak{X}_\beta}\mathbb{A}_{\mathrm{crys}}[t^{-1}]$ and 
${\mathcal{O}}_{\mathfrak{X}_\beta}\mathbb{A}_{\mathrm{crys},J}^\partial [t^{-1}]$ respetively. We consider the de Rham complexes $${\mathcal{O}}_{\mathfrak{X}_\beta}\mathbb{B}_{\mathrm{crys}}\otimes_{\mathcal{O}_{\mathfrak{X}_\beta}} \Omega_{\mathfrak{X}_\beta}^{\log,\bullet},\qquad{\mathcal{O}}_{\mathfrak{X}_\beta}\mathbb{B}_{\mathrm{crys}}^{c-\infty}\otimes_{\mathcal{O}_{\mathfrak{X}_\beta}} \Omega_{\mathfrak{X}_\beta}^{\log,\bullet}$$
as in \S\ref{sec:Chech}.

\begin{remark}\label{rmk:Blog}
Consider the $0$-dimensional case, i.e., the point. Denote by $A_{\rm log}$  the period ring one gets  taking the $p$-adic completion of the log DP envelope of the map $$\vartheta \colon \mathcal{O}_{K_0}[\![Z]\!] \otimes_{\mathcal{O}_{K_0}} A_{\mathrm{inf}}(\mathbb{C}_p) \to \mathcal{O}_{\mathbb{C}_p},$$sending $Z\mapsto \pi$, in the sense of \cite[Lemma 2.16]{AI}. Let $B_{\rm log}$ be $A_{\rm log}[t^{-1}]$.  As explained in \cite[\S 2.1.1]{AI}  this coincides with the ring $\widehat{B}_{\mathrm{st}}$ introduced by Breuil in  \cite{Breuil}. As explained in loc.~cit., it contains Fontaine's classical period ring $B_{\rm st}$ as the subring of  $\widehat{B}_{\mathrm{st}}$ on which the monodromy operator $N$ is nilpotent. One can use it to characterize semistable $p$-adic representations of the absolute Galois group of $K$.

\end{remark}

\begin{theorem}\label{thm:OBcrys}
 The natural morphism of sheaves  on $X_{\beta,\mathrm{pket}}$ $$ \mathbb{B}_{\mathrm{crys}}\vert_{X_\beta} \longrightarrow{\mathcal{O}}_{\mathfrak{X}_\beta}\mathbb{B}_{\mathrm{crys}}\otimes_{\mathcal{O}_{\mathfrak{X}_\beta}} \Omega_{\mathfrak{X}_\beta}^{\log,\bullet} ,\qquad \mathbb{B}_{\mathrm{crys}}^{c-\infty} \longrightarrow{\mathcal{O}}_{\mathfrak{X}_\beta}\mathbb{B}_{\mathrm{crys}}^{c-\infty}\otimes_{\mathcal{O}_{\mathfrak{X}_\beta}} \Omega_{\mathfrak{X}_\beta}^{\log,\bullet}$$are quasi isomorphisms, compatible with Frobenius.

\end{theorem}
\begin{proof}  The proof is local.  We may assume that $\mathcal{X}_\beta=\mathrm{Spf}(R_\beta)\subset  \mathfrak{X}_\beta=\mathrm{Spf}(\mathfrak{R}_\beta)$ admits a chart as in \S \ref{sec:Chech}.  As explained in \cite[\S 3.1.2]{AI}  such a choice of chart defines an integral perfectoid extension $\widehat{\overline{R}}^0_\beta$ of $R^0_\beta$ and hence $\widehat{\overline{R}}_\beta$ of $R_\beta$ and a perfectoid open $\mathcal{U}_\beta$ of $X_{\mathrm{pket}}$.  Fix compatible $p$-power roots $\underline{\pi}$, $\underline{X}_1$, $\cdots$, $\underline{X}_a$, $\underline{Y}_1$, $\cdots$, $\underline{Y}_b$ in $\widehat{\overline{R}}^\flat_\beta$ of $\pi$, $X_1$, $\cdots$, $X_a$, $Y_1$, $\cdots$, $Y_b$ in $\widehat{\overline{R}}_\beta$ such that $\underline{X}_1\cdots \underline{X}_a=\underline{\pi}$.
We denote by $[\_]$ their Teichm\"uller lifts in $\mathbf{W}\bigl(\widehat{\overline{R}}^\flat_\beta)$. 
Arguing as in \cite[\S 3.4]{AI}, especially Remark 3.26, one gets that $\mathbb{A}_{\mathrm{crys}}(\mathcal{U}_\beta)\cong A_{\mathrm{crys}}\bigl(\widehat{\overline{R}}_\beta\bigr)$ and $$\mathcal{O}_{\mathfrak{X}_\beta}\mathbb{A}_{\mathrm{crys}}(\mathcal{U}_\beta)\cong A_{\mathrm{crys}}\bigl(\widehat{\overline{R}}_\beta\bigr)\langle v_1-1, \ldots, v_a-1,w_1-1,\ldots,w_b-1\rangle$$with  $v_1:=\frac{[\underline{X}_1]}{X_1}$, $\cdots$, $v_a:=\frac{[\underline{X}_a]}{X_a}$, $w_1:=\frac{[\underline{Y}_1]}{Y_1}$, $\cdots$, $w_b:=\frac{[\underline{Y}_b]}{Y_b}$ with derivations   $X_1 \frac{\partial}{\partial X_1}$, $\ldots$, $X_a \frac{\partial}{\partial X_a}$, $Y_1 \frac{\partial}{\partial Y_1}$, $\ldots$, $Y_b \frac{\partial}{\partial Y_b}$ that are zero on $A_{\mathrm{crys}}\bigl(\widehat{\overline{R}}_\beta\bigr)$. Notice that $ \frac{d X_1}{X_1}$, $\ldots$, $ \frac{d X_a}{X_a}$, $ \frac{d Y_1}{Y_1}$, $\ldots$, $ \frac{d Y_b}{Y_b}$  form a basis of generators of the continuous differentials $\Omega^{1,\log}_{\mathfrak{R}_\beta/\mathcal{O}_{K_0}}$. One readily computes that the de Rham complex is exact except in degree $0$ where the kernel fo the derivative on $\mathcal{O}_{\mathfrak{X}_\beta}\mathbb{A}_{\mathrm{crys}}(\mathcal{U}_\beta)$ is $A_{\mathrm{crys}}\bigl(\widehat{\overline{R}}_\beta\bigr)$. The first claim then follows. 

%$u=\frac{[\underline{\pi}]}{Z}$,$Z \frac{\partial}{\partial Z}$,

For the second claim, assume that  $\mathfrak{X}_{\beta,J}\subset \mathfrak{X}_\beta$ is defined by $Y_1=0,\ldots,Y_j=0$. Then, $$\mathbb{A}_{\mathrm{crys},J}^\partial(\mathcal{U}_\beta)\cong A_{\mathrm{crys}}\bigl(\widehat{\overline{R}}_\beta\bigr)/([\underline{Y}_1^s],\ldots, [\underline{Y}_j^s])_{s\in\mathbb{Q}_{>0}},$$using \cite[Cor.~2.5.4(1)]{LLZ} and \cite[Lemma 2.5.8]{LLZ} and $$ \mathcal{O}_{\mathfrak{X}_\beta} \mathbb{A}_{\mathrm{crys}}^\partial(\mathcal{U}_\beta)\cong \mathbb{A}_{\mathrm{crys},J}^\partial(\mathcal{U}_\beta) \langle v_1-1, \ldots, v_a-1,w_1-1,\ldots,w_b-1\rangle,$$arguing as in \cite[\S 3.4]{AI}, Remak 3.26, for the second. From this also the second claim follows.
\end{proof}

Consider the morphisms of sites $\nu \colon X_{\beta,\overline{K}} \to X_\beta$,  $\nu_{\mathrm{geo}}\colon X_{\beta,\overline{K},\mathrm{pket}}\to X_{\beta,\overline{K}}$ and $\nu_{\mathrm{ar}}=\nu\circ \nu_{\mathrm{geo}}\colon X_{\beta,\overline{K},\mathrm{pket}}\to X_\beta$. Consider a  free $\mathcal{O}_{\mathfrak{X}_\beta}^{\mathrm{DP}}$-module  with Frobenius  $\bigl(\mathfrak{E}_\beta,\Phi_\beta\bigr)$ on  $\mathfrak{X}_\beta$. Consider $ \mathfrak{E}_\beta\otimes_{\mathcal{O}_{\mathfrak{X}_\beta}^{\mathrm{DP}}}\mathcal{O}_{\mathfrak{X}_\beta}\mathbb{B}_{\mathrm{crys}}$ as a sheaf on $X_{\beta,\overline{K},\mathrm{pket}}$.  Then we have the following:

\begin{proposition}\label{prop:phi-1inv} We have $\mathrm{R}^i (w\circ \nu_{\mathrm{geo}})_\ast\left( \mathfrak{E}_\beta\otimes_{\mathcal{O}_{\mathfrak{X}_\beta}^{\mathrm{DP}}}\mathcal{O}_{\mathfrak{X}_\beta}\mathbb{B}_{\mathrm{crys}}\right) =0 $ for $ i\geq 1$, $$\mathfrak{E}_\beta\widehat{\otimes}_{\mathcal{O}_{\mathfrak{S}^{\rm PD}}} B_{\mathrm{log}}\stackrel{\sim}{\longrightarrow}\bigl(w\circ \nu_{\mathrm{geo}}\bigr)_\ast\left( \mathfrak{E}_\beta\otimes_{\mathcal{O}_{\mathfrak{X}_\beta}^{\mathrm{DP}}}\mathcal{O}_{\mathfrak{X}_\beta}\mathbb{B}_{\mathrm{crys}}\right) $$ (see Remark \ref{rmk:Blog} for the definition of the classical period ring $B_{\rm log}$).  The map $$\mathfrak{E}_\beta[p^{-1}] \longrightarrow  (w\circ \nu_{\mathrm{ar}})_\ast\left( \mathfrak{E}_\beta\otimes_{\mathcal{O}_{\mathfrak{X}_\beta}^{\mathrm{DP}}}\mathcal{O}_{\mathfrak{X}_\beta}\mathbb{B}_{\mathrm{crys}}\right)$$is injective with cokernel annihilated by $\Phi_\beta^3$. In particular,  it induces an isomorphism on the kernels of $1 -\frac{\Phi_\beta}{p^r}$.

\end{proposition} 
\begin{proof} We may replace $\mathfrak{E}_\beta$ with $\mathcal{O}_{\mathfrak{X}_\beta}$ as $\mathfrak{E}$ is locally free on the logarithmic crystalline site of $\mathcal{X}_k$. One then argues as in the proof of \cite[Lemma 5.3]{TT} reducing the claims to  caims in group cohomology, namely Theorem 6.12 of loc.~cit.~Its analogue in our context is \cite[Thm.~3.39]{AI} for the computation of  $\mathrm{R}^i (w\circ \nu_{\mathrm{geo}})_\ast \left(\mathcal{O}_{\mathfrak{X}_\beta}\mathbb{B}_{\mathrm{crys}}\right)$ for $i\geq 0$  and \cite[Prop. 3.40]{AI} for the computation of $  (w\circ \nu_{\mathrm{ar}})_\ast\left(\mathcal{O}_{\mathfrak{X}_\beta}\mathbb{B}_{\mathrm{crys}}\right)$.

For the last claim notice that for any local section $x$ lying in the kernel of $1 -\frac{\Phi_\beta}{p^r}$ we have that $\Phi^3_\beta(x)=p^{3r} x$. As the displayed map  is injective with cokernel annihilated by $\Phi_\beta^3$, the claim follows. The last claim is obvious.
\end{proof}

\section{Crystalline associated sheaves}\label{sec:comparison}

We let $\mathcal{L}$ be a $\mathbb{Z}_p$-local system on $X_{\mathrm{ket}}$ with unipotent geometric monodromy along $D$ (\cite[Def. 6.3.1 \& Def. 6.3.7]{DLLZa}). We write $\widehat{\mathcal{L}}$ for the associated sheaf of $\widehat{\mathbb{Z}}_p$-modules in $X_{\mathrm{pket}}$ as in \cite[Def. 6.3.2]{DLLZa}. Let $\mathcal{E}$ be a locally free $\mathcal{O}_X$-module, endowed with an integrable connection $\nabla$ and a descending filtration $\mathrm{Fil}^\bullet \mathcal{E}$ by locally free $\mathcal{O}_X$-modules such that Griffiths' transversality holds and the graded pieces are locally free $\mathcal{O}_X$-modules as well. 

\begin{definition} We say that $\mathcal{L}$ and $\bigl(\mathcal{E},\mathrm{Fil}^\bullet \mathcal{E},\nabla\bigr)$ are de Rham associated if there is an isomorphism of $\mathcal{O}\mathbb{B}_{\mathrm{dR}}$-modules in $X_{\mathrm{pket}}$, compatibly with filtrations and connections: $$\xi_{\mathrm{dR}}\colon  \widehat{\mathcal{L}}\otimes_{\widehat{\mathbb{Z}}_p} \mathcal{O}\mathbb{B}_{\mathrm{dR}}\cong \mathcal{E}\otimes_{\mathcal{O}_X}  \mathcal{O}\mathbb{B}_{\mathrm{dR}}.$$
\end{definition}

\begin{remark}\label{rmk:unipotentnilpotent} As $\mathcal{L}$ is assumed to have unipotent geometric monodromy along $D$, it follows from \cite[Thm.~3.2.12]{DLLZb} that $(\mathcal{E},\nabla)$ has nilpotent residues along $D$ in the sense of Definition \ref{def:nilpotentconnection}.
\end{remark}

Analogously, assume that \ref{ass:formalcase} holds and fix coverings as in \S \ref{sec:Chech}.  Consider a non-degenerate Frobenius log-crystal $ \bigl(\mathcal{E},\mathrm{Fil}^\bullet \mathcal{E}, \nabla,\mathfrak{E},\varphi_{\mathfrak{E}}\bigr)$ relative to $(X,\alpha_X,\overline{\mathcal{X}},\alpha_{\overline{\mathcal{X}}})$.

\begin{definition}\label{def:crysass}  We say that $\mathcal{L}$ and $\bigl(\mathcal{E},\mathrm{Fil}^\bullet \mathcal{E}, \nabla,\mathfrak{E},\varphi_{\mathfrak{E}}\bigr) $ are  crystalline  associated  if the following hold:

\begin{itemize} 

\item[i.]  $\mathcal{L}$ and $\bigl(\mathcal{E},\mathrm{Fil}^\bullet \mathcal{E},\nabla\bigr)$ are de Rham associated;

\item[ii.] for every $\beta\in B$ there is an isomorphism of $\mathcal{O}_{\mathfrak{X}_\beta}\mathbb{B}_{\mathrm{crys}}$-modules in $X_{\beta,\mathrm{pket}}$  compatible with connections and Frobenius: $$\xi_{\beta,\mathrm{crys}}\colon \widehat{\mathcal{L}}\vert_{X_{\beta,\mathrm{pket}}}\otimes_{\widehat{\mathbb{Z}}_p} \mathcal{O}_{\mathfrak{X}_\beta} \mathbb{B}_{\mathrm{crys}}\cong \mathfrak{E}_\beta\otimes_{\mathcal{O}_{\mathfrak{X}_\beta}}  \mathcal{O}_{\mathfrak{X}_\beta}\mathbb{B}_{\mathrm{crys}};$$

\item[iii.] for every $\beta\in B$ the isomorphisms $\xi_{\mathrm{dR}}$ and $\xi_{\beta,\mathrm{crys}}$ are compatible via the natural morphism $\mathcal{O}_{\mathfrak{X}_\beta}\mathbb{B}_{\mathrm{crys}}\to \mathcal{O}\mathbb{B}_{\mathrm{dR}}\vert_{X_{\beta,\mathrm{pket}}}$ and $\mathfrak{E}_\beta\to \mathcal{E}\vert_{X_\beta}$. 

\end{itemize}
\end{definition}

We remark that the Definition above can be given also for the category of a non-degenerate Frobenius log-isocrystals of  Remark \ref{rmk:dualisocrystal}.

\begin{lemma}\label{lemma:uniquenabla}
There is a unique connection and a unique Frobenius structure on $\mathfrak{E}_\beta$ compatible with the isomorphism $\xi_{\beta,\mathrm{crys}}$ so that $\widehat{\mathcal{L}}$ is horizontal for the connection and Frobenius acts as the identity on $\widehat{\mathcal{L}}$.
\end{lemma} 
\begin{proof}
 Proposition \ref{prop:phi-1inv} implies that, Zariski locally on $\mathcal{X}_\beta$,  the isomorphism $\xi_{\beta,\mathrm{crys}}$ defines an injection $$\mathfrak{E}_\beta[p^{-1}]\subset   (w\circ \nu_{\mathrm{ar}})_\ast\left( \widehat{\mathcal{L}}\vert_{X_{\beta,\mathrm{pket}}}\otimes_{\widehat{\mathbb{Z}}_p} \mathcal{O}_{\mathfrak{X}_\beta} \mathbb{B}_{\mathrm{crys}}\right)$$compatible with connection and Frobenius, with cokernel annihilated by $\Phi_{\beta}^3$. In particular, the connection and Frobenius on $\mathfrak{E}_\beta$ are uniquely determined by those on $ \widehat{\mathcal{L}}\vert_{X_{\beta,\mathrm{pket}}}\otimes_{\widehat{\mathbb{Z}}_p} \mathcal{O}_{\mathfrak{X}_\beta} \mathbb{B}_{\mathrm{crys}}$. 
\end{proof}

\begin{proposition}\label{prop:FunctAss}
(1) Being de Rham or crystalline associated preserves tensor products; i.e., if $\mathcal{L}_1$ and $\bigl(\mathcal{E}_1,\mathrm{Fil}^\bullet \mathcal{E}_1, \nabla_1,\mathfrak{E}_1,\varphi_{\mathfrak{E}_1}\bigr) $ and $\mathcal{L}_2$ and $\bigl(\mathcal{E}_2,\mathrm{Fil}^\bullet \mathcal{E}_2, \nabla_2,\mathfrak{E}_2,\varphi_{\mathfrak{E}_2}\bigr) $  are crystalline associated then $\mathcal{L}_1\otimes_{\mathbb{Z}_p} \mathcal{L}_2$ and $\bigl(\mathcal{E}_1\otimes_{\mathcal{O}_X} \mathcal{E}_2, \mathrm{Fil}^\bullet (\mathcal{E}_1\otimes_{\mathcal{O}_X} \mathcal{E}_2), \nabla_1 \otimes 1 + 1 \otimes \nabla_2,\mathfrak{E}_1\otimes_{\mathcal{O}_{\mathfrak{X}}} \mathfrak{E}_2,\varphi_{\mathfrak{E}_1}\otimes \varphi_{\mathfrak{E}_1}\bigr) $ are crystalline associated. 

Similarly, being de Rham or crystalline associated preserves duals, namely $\mathcal{L}_1^\vee$ and the dual of $\bigl(\mathcal{E}_1,\mathrm{Fil}^\bullet \mathcal{E}_1, \nabla_1,\mathfrak{E}_1,\varphi_{\mathfrak{E}_1}\bigr) $, as defined in Remark \ref{rmk:dualisocrystal},  are crystalline associated. 

\medskip 
(2) Let $f\colon Y\to X$ be a morphism of smooth rigid analytic variety over $\mathrm{Spa}(K,\mathcal{O}_K)$ and let  $D\subset X$ and $E\subset Y$ be strict normal crossing divisors such that $f^{-1}(D)\subset E$. Assume that Assumption (\ref{ass:formalcase}) holds also for $Y$ and $E$. In particular, $E\subset Y$ are the generic fibers of formal schemes $\mathcal{E}\subset \mathcal{Y}$. Suppose that $f$ extends to a morphism of formal schemes $f\colon \mathcal{X}\to \mathcal{Y}$ such that $f^{-1}(\mathcal{D})\subset \mathcal{E}$ and $f$ is log-smooth with respect to the log-structures defined by the special fiber and the divisors $\mathcal{D}$ and $\mathcal{E}$. Choose compatible local charts $f_\beta\colon \mathfrak{Y}_\beta  \to \mathfrak{X}_\beta$ for $\beta\in B$, which exists by the log-smoothness of $f$. If  
$\mathcal{L}$ and $\bigl(\mathcal{E},\mathrm{Fil}^\bullet \mathcal{E}, \nabla,\mathfrak{E},\varphi_{\mathfrak{E}}\bigr) $ are crystalline associated on $X$, then the pull-back of $\mathcal{L}$ to $Y_{\mathrm{ket}}$ and the pull-back of $\bigl(\mathcal{E},\mathrm{Fil}^\bullet \mathcal{E}, \nabla,\mathfrak{E},\varphi_{\mathfrak{E}}\bigr) $ to $Y$ (and to the $\mathfrak{Y}_\beta$'s)  are crystalline associated.

\end{proposition}
\begin{proof} (1) Due to our assumption that $\mathcal{L}_1$ and $\mathcal{L}_2$ have unipotent geometric monodromy along $D$, the statement on being de Rham associated follows from  \cite[Thm. 3.2.12]{DLLZb}. The fact that they are also crystalline associated follows easily from the definition.

(2) The pull-backs are de Rham associated due to \cite[Thm. 3.2.7(4)]{DLLZb}. The statement then follows from Lemma \ref{lemma:uniquenabla} as $f$ defines a morphism of sheaves $f^{-1}\left(\mathcal{O}_{\mathfrak{X}_\beta} \mathbb{B}_{\mathrm{crys}}\right)  \to \mathcal{O}_{\mathfrak{Y}_\beta} \mathbb{B}_{\mathrm{crys}}$ on $Y_{\mathrm{pket}}$ compatible with connections. 
\end{proof}

\begin{remark}\label{rmk:traceasociated} We take the assumptions of Proposition \ref{prop:FunctAss}(ii). Assume that the map $f\colon \mathcal{X}\to \mathcal{Y}$ of log-schemes is finite and Kummer \'etale. Then, the induced map of generic fibers $f\colon X\to Y$ induces a map of toposes $\tilde{f}=(\tilde{f}_\ast,\tilde{f}^\ast)$ between the sheaves on  $X_{\mathrm{pket}}$ and the sheaves on $Y_{\mathrm{pket}}$. It follows from \cite[Prop. 4.5.2]{DLLZb}  that $\tilde{f}_\ast=\tilde{f}_!\colon \mathrm{Sh}_{\rm Ab}\bigl(Y_{\mathrm{pket}}\bigr) \rightarrow \mathrm{Sh}_{\rm Ab}\bigl(X_{\mathrm{pket}}\bigr)$ on sheaves of abelian groups. In particular, they are both exact and there is a natural transformation $\tilde{f}_\ast \circ \tilde{f}^\ast\to \mathrm{Id}$, called the trace of $\tilde{f}$, compatible with the traces for sheaves arising from  Kummer \'etale sheaves on $X_K$.  We conclude that, for  $\mathcal{L}$ and $\bigl(\mathcal{E},\mathrm{Fil}^\bullet \mathcal{E}, \nabla,\mathfrak{E},\varphi_{\mathfrak{E}}\bigr) $  crystalline  associated  as in Definition \ref{def:crysass}, the trace of the maps $\xi_{\beta,\mathrm{crys}}$ and $\xi_{\mathrm{dR}}$ are compatible  with the trace of $\mathcal{L}$ on the Kummer \'etale site, the trace of $\mathfrak{E}_{\beta}$ on  $\mathfrak{X}_{\beta,k}$ and the trace of $\mathcal{E}$ on  $X$. 
\end{remark}

As an example of crystalline associated sheaves, we consider the case that $\mathcal{X}$ is the $p$-adic completion of a smooth toroidal compactification of a Siegel modular variety with prime to $p$ level structure and prime to $p$ polarization, $\mathcal{D}\subset \mathcal{X}$ is the boundary and
\smallskip

1)  $\mathcal{L}$ is the relative $p$-adic \'etale cohomology of the universal abelian scheme $A$ over $U:=X\backslash D$. It defines a a $\mathbb{Z}_p$-local system on $X_{\mathrm{ket}}$ of rank twice the dimension of $A$ which has unipotent geometric  monodromy  along $D$; see the discussion of \cite[\S 5.2]{DLLZb}.
\smallskip

2) $\bigl(\mathcal{E},\mathrm{Fil}^\bullet \mathcal{E}, \nabla,\mathfrak{E},\varphi_{\mathfrak{E}}\bigr) $  is the  Frobenius log-crystal relative to $(X,\alpha_X,\overline{\mathcal{X}},\alpha_{\overline{\mathcal{X}}})$ extending the relative de Rham cohomology of $A\to U$ and the relative crystalline cohomology of the mod $p$ special fiber of $\mathcal{A} \to \mathcal{U}=\mathcal{X}\backslash \mathcal{D}$ relative to $\mathcal{U}$,  constructed in \cite[\S 3.2]{MP}.
\smallskip

\begin{proposition}\label{prop:crysass} The sheaves $\mathcal{L}$ and  $\bigl(\mathcal{E},\mathrm{Fil}^\bullet \mathcal{E}, \nabla,\mathfrak{E},\varphi_{\mathfrak{E}}\bigr) $ are crystalline associated.
\end{proposition}
\begin{proof}  By \cite[Thm.~VI.1.1]{FC} we can extend $A\to U$ to a proper and log smooth morphism $\overline{A}\to X$.  Then $ \bigl(\mathcal{E},\mathrm{Fil}^\bullet \mathcal{E}, \nabla\bigr)$ is identified with the first logarithmic de Rham cohomology sheaf of $\overline{A}$ relative to $X$ as they both coincide with the canonical extension of the first logarithmic de Rham cohomology sheaf of $A$ relative to $U$. The fact that $\mathcal{L}$ and $ \bigl(\mathcal{E},\mathrm{Fil}^\bullet \mathcal{E}, \nabla\bigr)$  are de Rham associated follows then from \cite[Theorem 3.2.7(5)]{DLLZb}.
Thanks to  \cite[Theorem 5.5]{TT} they are also crystalline associated when we restrict to the rigid analytic 
fiber $\mathcal{U}_K$ associated to $\mathcal{U}$.  In this case, using the compatibility with the comparison isomorphism for $p$-divisible groups and Dieudonn\'e theory as in \cite[\S 3.5]{SW},
the isomorphism $\xi_{\beta,\mathrm{crys}}$  over the pro-Kummer \'etale site of $\mathcal{U}_K \cap X_{\beta} $  is induced by a map 
$$\xi_{\beta,\mathrm{crys}}'\colon \widehat{\mathcal{L}}\vert_{X_{\beta,\mathrm{pket}}}\otimes_{\widehat{\mathbb{Z}}_p} \mathcal{O}_{\mathfrak{X}_\beta} \mathbb{A}_{\mathrm{crys}}\longrightarrow  \mathfrak{E}_\beta\otimes_{\mathcal{O}_{\mathfrak{X}_\beta}}  \mathcal{O}_{\mathfrak{X}_\beta}\mathbb{A}_{\mathrm{crys}}.$$We claim that $\xi_{\beta,\mathrm{crys}}'$ is defined over the pro-Kummer \'etale site of $X_{\beta} $. Before proving the claim, we show how this implies the Proposition. Using the analogous map for the dual abelian scheme $A^\vee$, dualizing and identifying its $p$-divisible group with the dual $\mathcal{L}^\vee(1)$ and its relative Dieudonn\'e module with the dual of $\mathfrak{E}_\beta$,  we get a map  $$\xi_{\beta,\mathrm{crys}}''\colon  \mathfrak{E}_\beta\otimes_{\mathcal{O}_{\mathfrak{X}_\beta}}  \mathcal{O}_{\mathfrak{X}_\beta}\mathbb{A}_{\mathrm{crys}}(-1)\longrightarrow \widehat{\mathcal{L}}\vert_{X_{\beta,\mathrm{pket}}}\otimes_{\widehat{\mathbb{Z}}_p} \mathcal{O}_{\mathfrak{X}_\beta} \mathbb{A}_{\mathrm{crys}}.$$The composite $\xi_{\beta,\mathrm{crys}}' \circ \xi_{\beta,\mathrm{crys}}''$ is multiplication by Fontaine's $t$. %Elaborate, referenza ScholzeWeinstein  ?????
This implies that $\xi_{\beta,\mathrm{crys}}'$ defines an isomorphism after passing to $\mathcal{O}_{\mathfrak{X}_\beta}\mathbb{B}_{\mathrm{crys}}$.

We are left to prove the claim. Assuming that $\mathfrak{E}_\beta$ is a free $\mathcal{O}_{\mathfrak{X}_\beta}$-module, the map $\xi_{\beta,\mathrm{crys}}'$ is realized by a $2g \times 2g$-matrix with coefficients in  $ \mathcal{O}_{\mathfrak{X}_\beta}\mathbb{A}_{\mathrm{crys}}$ over the pro-Kummer \'etale site of $\mathcal{U}_K \cap X_{\beta} $. Using the notation in the proof of Theorem \ref{thm:OBcrys}, if  $\mathcal{X}_{\beta} =\mathrm{Spf}(R_\beta)$ then $\mathcal{U}\cap \mathcal{X}_{\beta}=\mathrm{Spf}(R_\beta\{f^{-1}\}) $ where   $R_\beta\{f^{-1}\}$ is the $p$-adic completion of $R_\beta[f^{-1}]$. Notice that $f\in R_\beta$ is a local coordinate of the boundary of $\mathcal{D}$ in $\mathcal{X}_\beta$ and, as such, defines a regular element of $R_\beta/\pi R_\beta$. We claim that $f$ defines a regular element of $\overline{R}_\beta/\pi \overline{R}_\beta$ as well. One reduces to proving that $f$ is not a zero divisor in $S/\pi S$ for every finite and normal extension of $R_\beta$ contained in $\overline{R}_\beta$. Since $S$ is normal, then the map $S/\pi S \to \prod_{\mathcal{P}} S_{\mathcal{P}}/\pi S_{\mathcal{P}}$, to the localization of $S/\pi S$ at all height one prime ideals $\mathcal{P}$ of $S$ containg $\pi$, is injective. Since for any such $\mathcal{P}$ the intersection $Q:=\mathcal{P} \cap R_\beta$ is equal to $\pi R_\beta$ and $f$ is invertible in $R_{\beta,Q}$, then $f$ is invertible in $S_\mathcal{P}$ as well  and we conclude.

Then evaluating $\xi_{\beta,\mathrm{crys}}'$ at the affinoid perfectoid defined by $R_\beta\{f^{-1}\}\subset \overline{R}_\beta\{f^{-1}\}$ we get a $2g \times 2g$ matrix $M$ with values in $ \mathcal{O}_{\mathfrak{X}_\beta}\mathbb{A}_{\mathrm{crys}}\bigl(\overline{R}_\beta\{f^{-1}\}\bigr)$.  The claim follows if we show that $M$ takes values in  $ \mathcal{O}_{\mathfrak{X}_\beta}\mathbb{A}_{\mathrm{crys}}\bigl(\widehat{\overline{R}}_\beta\bigr)$ as $\xi_{\beta,\mathrm{crys}}'$ is then defined over the pro-Kummer \'etale site of $X_{\beta} $. As $ \mathcal{O}_{\mathfrak{X}_\beta}\mathbb{A}_{\mathrm{crys}}\bigl(\overline{R}_\beta\{f^{-1}\}\bigr)$ and $ \mathcal{O}_{\mathfrak{X}_\beta}\mathbb{A}_{\mathrm{crys}}\bigl(\widehat{\overline{R}}_\beta\bigr)$ are $p$-adically complete and separated, it suffices to prove that the reduction $M_n$ of $M$ modulo $p^n$ takes values in $ \mathcal{O}_{\mathfrak{X}_\beta}\mathbb{A}_{\mathrm{crys}}\bigl(\widehat{\overline{R}}_\beta\bigr)/(p^n)$ for every $n$. It follows from the discussion in \cite[Prop. A.2.2 \& \S A.2.6]{MP} that $\xi_{\beta,\mathrm{crys}}'$ extends to the completion of $R_\beta$ at every maximal ideal; in loc.~cit.~the author works over a local field but one checks that the proof of Prop.~A.2.2 uses the constructions of \S A.1 that hold over a complete local normal domain.

Since $p$ is a regular element in $ \mathcal{O}_{\mathfrak{X}_\beta}\mathbb{A}_{\mathrm{crys}}\bigl(\widehat{\overline{R}}_\beta\bigr)$ by \cite[Cor.~3.27]{AI}, arguing for each entry of $M$ and proceeding by induction on $n$, one reduces to prove the following: let $s$  be an element of $ \mathcal{O}_{\mathfrak{X}_\beta}\mathbb{A}_{\mathrm{crys}}\bigl(\overline{R}_\beta\{f^{-1}\}\bigr)/(p)$ that belongs to $ \mathcal{O}_{\mathfrak{X}_\beta}\mathbb{A}_{\mathrm{crys}}\bigl(\widehat{\overline{R}}_\beta\otimes_{R_\beta} \widehat{R}_{\beta,x}\bigr)/(p)$ for the completion $\widehat{R}_{\beta,x}$ of the localization $R_{\beta,x}$ at each maximal ideal of $R_\beta$, then $s\in  \mathcal{O}_{\mathfrak{X}_\beta}\mathbb{A}_{\mathrm{crys}}\bigl(\widehat{\overline{R}}_\beta\bigr)/(p)$. Using the explicit description of  $ \mathcal{O}_{\mathfrak{X}_\beta}\mathbb{A}_{\mathrm{crys}}\bigl(\widehat{\overline{R}}_\beta\bigr)/(p)$ provided in \cite[Lemma 3.25]{AI}, at the end of page 231, as a free $\overline{R}_\beta/\pi \overline{R}_\beta$-module for some pseudouniformizer $\pi \in \mathcal{O}_{\mathbb{C}_p}$ (and similarly for $ \mathcal{O}_{\mathfrak{X}_\beta}\mathbb{A}_{\mathrm{crys}}\bigl(\widehat{\overline{R}}_\beta[f^{-1}]\bigr)/(p)$ and $\mathcal{O}_{\mathfrak{X}_\beta}\mathbb{A}_{\mathrm{crys}}\bigl(\widehat{\overline{R}}_\beta\otimes_{R_\beta} \widehat{R}_{\beta,x}\bigr)/(p)$ as free $\overline{R}_\beta/\pi \overline{R}_\beta[f^{-1}]$-module and respectively as $(\overline{R}_\beta/\pi\overline{R}_\beta )\otimes_{R_\beta} \widehat{R}_{\beta,x}$-module with the same basis), one reduces to the following claim.  Let $s\in \overline{R}_\beta/\pi \overline{R}_\beta[f^{-1}]$ be such that it lies in $(\overline{R}_\beta/\pi \overline{R}_\beta)\otimes_{R_\beta} \widehat{R}_{\beta,x}$ for every maximal ideal $x$ of $R_\beta$, then $s\in \overline{R}_\beta/\pi \overline{R}_\beta$. This is clear since the map $\overline{R}_\beta/\pi \overline{R}_\beta\to\overline{R}_\beta/\pi \overline{R}_\beta[f^{-1}]$ is injective by the regularity result proven above and using that $R_{\beta,x}\to \widehat{R}_{\beta,x}$ is faithfully flat.
\end{proof}

\section{\'Etale syntomic complexes}\label{sec:synL}

Let $\mathcal{L}$ be a $\mathbb{Z}_p$-local system on $X_{\mathrm{ket}}$ and take an integer $r\in \mathbb{Z}$. Define $\mathbf{Syn}(\mathcal{L},r)$ to be the object in the bounded derived category of $\mathbb{Q}_p$-vector spaces  given by the total complex associated to the morphism of complexes  
$(1-\frac{\varphi}{p^r}) \oplus \iota $: $$ \mathrm{R}\Gamma\left(X_{\mathrm{pket}},\widehat{\mathcal{L}}\otimes_{\widehat{\mathbb{Z}}_p}  \mathbb{B}_{\mathrm{crys}}\right) \to   \mathrm{R}\Gamma\left(X_{\mathrm{pket}},\widehat{\mathcal{L}}\otimes_{\widehat{\mathbb{Z}}_p}  \mathbb{B}_{\mathrm{crys}}\right)  \oplus  \mathrm{R}\Gamma\left(X_{\mathrm{pket}},  \widehat{\mathcal{L}}\otimes_{\widehat{\mathbb{Z}}_p} \mathrm{Gr}^r \mathbb{B}_{\mathrm{dR}}\right)$$ where  $\mathrm{Gr}^r \mathbb{B}_{\mathrm{dR}}$ is the quotient of $\mathbb{B}_{\mathrm{dR}}$ by $\mathrm{Fil}^r \mathbb{B}_{\mathrm{dR}}$. The map $\iota$ is induced by the natural map $\mathbb{B}_{\mathrm{crys}}\to \mathbb{B}_{\mathrm{dR}}$.

Analogously define $\mathbf{Syn}^{c-\infty}(\mathcal{L},r)$ to be the object in the bounded derived category of sheaves on $X$ given by the total complex associated to the morphism of complexes $\iota^{c-\infty} \circ (1-\frac{\varphi}{p^r}) \oplus \iota^{c-\infty} $: $$\mathrm{R}\Gamma\left(X_{\mathrm{pket}},\widehat{\mathcal{L}}\otimes_{\widehat{\mathbb{Z}}_p}  \mathbb{B}_{\mathrm{crys}}^{c-\infty}\right) \to   \mathrm{R}\Gamma\left(X_{\mathrm{pket}},\widehat{\mathcal{L}}\otimes_{\widehat{\mathbb{Z}}_p}  \mathbb{B}_{\mathrm{crys}}^{c-\infty}\right)  \oplus  \mathrm{R}\Gamma\left(X_{\mathrm{pket}},  \widehat{\mathcal{L}}\otimes_{\widehat{\mathbb{Z}}_p} \mathrm{Gr}^r \mathbb{B}_{\mathrm{dR}}^{c-\infty}\right) .$$

\begin{definition} Let $\mathrm{H}^i\bigl(\mathbf{Syn}(\mathcal{L},r)\bigr)$, resp.~$\mathrm{H}^i\bigl(\mathbf{Syn}^{c-\infty}(\mathcal{L},r)\bigr)$ to be the $i$-th homology groups of the complexes $\mathbf{Syn}(\mathcal{L},r)$ and $\mathbf{Syn}^{c-\infty}(\mathcal{L},r)$ respectively. 
\end{definition}

Assume that   $\mathcal{L}$ and  $\bigl(\mathcal{E},\mathrm{Fil}^\bullet \mathcal{E}, \nabla,\mathfrak{E},\varphi_{\mathfrak{E}}\bigr) $ are crystalline associated as in Definition~\ref{def:crysass}. Then

\begin{proposition}\label{prop:comparecomplexes} We have morphisms of complexes $\mathbf{Syn}(\mathcal{E},\mathfrak{E},r)\to \mathbf{Syn}(\mathcal{L},r)$ and  $\mathbf{Syn}^\partial(\mathcal{E},\mathfrak{E},r)\to \mathbf{Syn}^{c-\infty}(\mathcal{L},r)$. 

\end{proposition}
\begin{proof} It follows from Theorem \ref{thm:OBcrys} and Definition \ref{def:crysass} and by \cite[Cor. 2.4.2]{DLLZb} that have a quasi isomorphism $$\widehat{\mathcal{L}}\otimes_{\widehat{\mathbb{Z}}_p} \mathbb{B}_{\mathrm{dR}} \longrightarrow\mathcal{E}\otimes_{\mathcal{O}} \mathcal{O}\mathbb{B}_{\mathrm{dR}}\otimes_{\mathcal{O}_X} \Omega^{\log,\bullet}_X$$of filtered complexes and for every $\beta\in B$ a quasi isomorphism $$\widehat{\mathcal{L}}\vert_{X_{\beta,\mathrm{pket}}}\otimes_{\widehat{\mathbb{Z}}_p}  \mathbb{B}_{\mathrm{crys}}\vert_{X_{\beta,\mathrm{pket}}} \longrightarrow  \mathfrak{E}_\beta \otimes_{\mathcal{O}_{\mathfrak{X}_\beta}} \mathcal{O}_{\mathfrak{X}_\beta}\mathbb{B}_{\mathrm{crys}}\otimes_{\mathcal{O}_{\mathfrak{X}_\beta}} \Omega^ {\log,\bullet}_{\mathfrak{X}_\beta}$$compatible with Frobenius. The same holds with support conditions, see \cite[Cor. 2.5.13]{LLZ} for the de Rham case and Theorem \ref{thm:OBcrys} in the crystalline case, noticing that the isomorphisms $\xi_{\mathrm{dR}}$ and $\xi_{\beta,\mathrm{crys}}$ of  Definition \ref{def:crysass} induce analogous and compatible isomorphisms $\xi_{\mathrm{dR}}^{c-\infty}$ and $\xi_{\beta,\mathrm{crys}}^{c-\infty}$ replacing $\mathcal{O}\mathbb{B}_{\mathrm{dR}}$ and $\mathcal{O}_{\mathfrak{X}_\beta}\mathbb{B}_{\mathrm{crys}}$ with  $\mathcal{O}\mathbb{B}_{\mathrm{dR}}^{c-\infty}$ and $\mathcal{O}_{\mathfrak{X}_\beta}\mathbb{B}_{\mathrm{crys}}^{c-\infty}$. It follows that we have morphisms of complexes  $$\eta\colon \mathrm{R}\Gamma\left(X,\mathrm{Gr}^r \bigl( \mathcal{E}\otimes_{\mathcal{O}_X} \Omega^{\log,\bullet}_X\bigr) \right)\longrightarrow  \mathrm{R}\Gamma\left(X_{\mathrm{pket}},\widehat{\mathcal{L}}\otimes_{\widehat{\mathbb{Z}}_p} \mathrm{Gr}^r  \mathbb{B}_{\mathrm{dR}}\right)   $$ and for every $\beta\in B$ a morphism of complexes $$\psi_\beta\colon \mathrm{R}\Gamma\left(\mathcal{X}_{\beta,k}, \mathfrak{E}_\beta\otimes_{\mathcal{O}_{\mathfrak{X}_\beta}} \Omega^ {\log,\bullet}_{\mathfrak{X}_\beta}\right)\longrightarrow   \mathrm{R}\Gamma\left(X_{\beta,\mathrm{pket}},\widehat{\mathcal{L}}\otimes_{\widehat{\mathbb{Z}}_p}  \mathbb{B}_{\mathrm{crys}}\right)   $$compatible with Frobenius and also with $\eta$ using the compatibility of Definition \ref{def:crysass}(iii)   (and similarly considering support conditions). For every $\underline{\beta}=(\beta_1,\ldots,\beta_n) \in B^n$ such that $\beta_1<\ldots < \beta_n$ and every $i$ the  morphism of complexes $$\zeta_i\colon \mathrm{R}\Gamma\left(\mathcal{X}_{\underline{\beta},k},\bigl( \mathfrak{E}_{\beta_i}\otimes_{\mathcal{O}_{\mathfrak{X}_{\beta_i}}} \Omega^ {\log,\bullet}_{\mathfrak{X}_{\beta_i}}\bigr)\vert_{\mathcal{X}_{\underline{\beta},k}} \right) \longrightarrow \mathrm{R}\Gamma\left(\mathcal{X}_{\underline{\beta},k}, \mathfrak{E}_{\underline{\beta}}\otimes_{\mathcal{O}_{\mathfrak{X}_{\underline{\beta}}}} \Omega^ {\log,\bullet}_{\mathfrak{X}_{\underline{\beta}}} \right)$$is a quasi-isomorphism compatible with Frobenius  and we deduce a map of complexes 

 $$\psi_{\underline{\beta},i}\colon \mathrm{R}\Gamma\left(\mathcal{X}_{\underline{\beta},k}, \mathfrak{E}_{\underline{\beta}}\otimes_{\mathcal{O}_{\mathfrak{X}_{\underline{\beta}}}} \Omega^ {\log,\bullet}_{\mathfrak{X}_{\underline{\beta}}} \right)\longrightarrow   \mathrm{R}\Gamma\left(X_{\underline{\beta},\mathrm{pket}},\widehat{\mathcal{L}}\otimes_{\widehat{\mathbb{Z}}_p}  \mathbb{B}_{\mathrm{crys}}\right)   $$compatible with Frobenius. This is independent of $i$: for every $i$ and $j$ the affine open formal subschemes of $\mathfrak{X}_{\beta_i}$ and of $\mathfrak{X}_{\beta_j}$ corresponding to $\mathcal{X}_{\underline{\beta},k}$, viewed as an open of $\mathcal{X}_{\beta_i,k}$ and $\mathcal{X}_{\beta_j,k}$ respectively, are isomorphic by log-smoothness over $(\mathcal{S}_k,\mathcal{M}_k)$ (see \S\ref{sub:not} for the notation) and this isomophism identifies the restrictions to $X_{\underline{\beta},\mathrm{pket}}$ of the domains of $\psi_{\beta_i}$ and $\psi_{\beta_j}$ compatibly with the restriction of the codomains. This  isomorphism is realized in the derived category via $\zeta_j \circ \zeta_i^{-1}$ as $\mathfrak{E}$ is a crystal. 

Taking the total complex for varying $\underline{\beta}$'s we get a map of complexes, compatible with Frobenius and with $\eta$

$$\mathrm{R}\Gamma_{\mathrm{crys}}\bigl((\overline{\mathcal{X}},\alpha_{\overline{\mathcal{X}}})/\mathcal{S}_0,\mathfrak{E}\bigr) [p^{-1}] \longrightarrow \mathrm{R}\Gamma\left(X_{\mathrm{pket}},\widehat{\mathcal{L}}\otimes_{\widehat{\mathbb{Z}}_p}  \mathbb{B}_{\mathrm{crys}}\right).$$This provides a morphism $\mathbf{Syn}_{\rm crys}(\mathcal{E},\mathfrak{E},r)\to \mathbf{Syn}(\mathcal{L},r)$. Composing with the map $\tau_{\rm crys}\colon \mathbf{Syn}(\mathcal{E},\mathfrak{E},r) \to \mathbf{Syn}_{\rm crys}(\mathcal{E},\mathfrak{E},r)$ of Remark
\ref{rmk:crysanal} we get the first morphism. A similar argument holds for the case of compact support.
% Independence on the choice of $\beta_i$'s? 
\end{proof}

\section{Comparison with \'etale cohomology}\label{sec:compareetale}

Let $U:=X\backslash D$ and let $j\colon U\to X$ be the inclusion. Assume that $X$ is proper (except in Proposition \ref{prop:comp0} below).
We let $\mathcal{L}$ be a $\mathbb{Z}_p$-local system on $X_{\mathrm{ket}}$. Given an integer $r\in \mathbb{Z}$ we let $\mathcal{L}(r)$ be the Tate twist, given by twisting by the $r$-th power of the cyclotomic character. As usual we write $\widehat{\mathcal{L}}$, resp.~$\widehat{\mathcal{L}}(r)$ for the associated sheaves in $X_{\mathrm{pket}}$. Recall the morphsims of sites $\nu\colon X_{\overline{K}} \to X$, $\nu_{\mathrm{geo},X}\colon X_{\overline{K},\mathrm{pket}} \to X_{\overline{K}}$ and $\nu_{\mathrm{ar}}=\nu \circ \nu_{\mathrm{geo}} \colon X_{\overline{K},\mathrm{pket}}\to X$. We start with the following:

\begin{proposition}\label{prop:etcoh}
We have $\mathrm{H}^i\bigl(X_{\mathrm{pket}}, \widehat{\mathcal{L}}(r)\bigr)\cong \mathrm{H}^i\bigl(X_{\mathrm{ket}}, \mathcal{L}(r)\bigr)\cong  \mathrm{H}^i_{\mathrm{et}}\bigl(U, \mathcal{L}(r)\bigr)$.  Similarly we have $\mathrm{H}^i_{\mathrm{et},c}\bigl(U, \mathcal{L}(r)\bigr) \cong \mathrm{H}^i\bigl(X_{\mathrm{pket}}, j_!\bigl(\widehat{\mathcal{L}}(r)\bigr)\bigr)$ where, following \cite[Def. 2.3.2]{LLZ}, we define $\mathrm{H}^i_{\mathrm{et},c}\bigl(U, \mathcal{L}(r)\bigr):=\mathrm{H}^i\bigl(X_{\mathrm{ket}}, j_!
\bigl(\mathcal{L}(r)\bigr)\bigr)$. 

The cohomology groups $\mathrm{H}^i_{\mathrm{et}}\bigl(U_{\overline{K}},  \mathcal{L}(r)\bigr)$ and $\mathrm{H}^i_{\mathrm{et},c}\bigl(U_{\overline{K}},  \mathcal{L}(r)\bigr)$ are of finite rank and vanish for $i>2 \mathrm{dim}X$.  Finally, we have  $G_K=\mathrm{Gal}(\overline{K}/K)$-equivariant isomorphisms  

$$\mathrm{H}^i_{\mathrm{et}}\bigl(U_{\overline{K}},  \mathcal{L}(r)\bigr)\otimes  B_{\mathrm{crys}}\cong \mathrm{H}^i\bigl(X_{\overline{K},\mathrm{pket}},  j_\ast(\mathcal{L}(r))\bigr)\otimes B_{\mathrm{crys}}  \cong \mathrm{H}^i\bigl(X_{\overline{K},\mathrm{pket}},\widehat{\mathcal{L}}\otimes_{\widehat{\mathbb{Z}}_p} \mathbb{B}_{\mathrm{crys}}\bigr),$$ $$\mathrm{H}^i_{\mathrm{et}}\bigl(U_{\overline{K}},  \mathcal{L}(r)\bigr)\otimes  B_{\mathrm{dR}}\cong \mathrm{H}^i\bigl(X_{\overline{K},\mathrm{pket}},  \widehat{\mathcal{L}}(r)\bigr)\otimes  B_{\mathrm{dR}} \cong  \mathrm{H}^i\bigl(X_{\overline{K},\mathrm{pket}},\widehat{\mathcal{L}}\otimes_{\widehat{\mathbb{Z}}_p} \mathbb{B}_{\mathrm{dR}}\bigr)$$and similarly $$\mathrm{H}^i_{\mathrm{et}}\bigl(U_{\overline{K}},  \mathcal{L}(r)\bigr)\otimes \mathrm{Gr}^r B_{\mathrm{dR}}\cong \mathrm{H}^i\bigl(X_{\overline{K},\mathrm{pket}},  \widehat{\mathcal{L}}(r)\bigr)\otimes  \mathrm{Gr}^r B_{\mathrm{dR}} \cong  \mathrm{H}^i\bigl(X_{\overline{K},\mathrm{pket}},\widehat{\mathcal{L}}\otimes_{\widehat{\mathbb{Z}}_p} \mathrm{Gr}^r \mathbb{B}_{\mathrm{dR}}\bigr).$$

\end{proposition}
\begin{proof} We start with the first claim. First  we have isomorphisms if we replace $\mathcal{L}$ with $\mathcal{L}/p^n \mathcal{L}$ for the isomorphisms $\mathrm{H}^i\bigl(X_{\mathrm{pket}}, \widehat{\mathcal{L}}(r)/p^n  \widehat{\mathcal{L}}(r) \bigr)\cong \mathrm{H}^i\bigl(X_{\mathrm{ket}}, \mathcal{L}(r)/p^n \mathcal{L}(r)\bigr) $ by \cite[Prop. 5.1.7]{DLLZa}   and  $ \mathrm{H}^i\bigl(X_{\mathrm{ket}}, \mathcal{L}(r)/p^n \mathcal{L}(r)\bigr) \cong  \mathrm{H}^i_{\mathrm{et}}\bigl(U, \mathcal{L}(r)/p^n \mathcal{L}(r)\bigr)$ by \cite[Cor. 4.6.7]{DLLZa}. They are also finite modules;  thanks to \cite[Thm. 6.2.1]{DLLZa} the groups $ \mathrm{H}^i\bigl(X_{\overline{K},\mathrm{ket}}, \mathcal{L}(r)/p^n \mathcal{L}(r)\bigr)$ are finite and vanish for $i> 2 \mathrm{dim}X$. Using the  Hochschild-Serre spectral sequence this implies that also   $\mathrm{H}^i\bigl(X_{\mathrm{ket}}, \mathcal{L}(r)/p^n \mathcal{L}(r)\bigr)$ is finite for every $i$. In particular, the inverse system for varying $n$ is Miitag-Leffler and the result follows.

For the second claim one argues similarly, using the finiteness and the vanishing for $i> 2 \mathrm{dim}X$ of the group $\mathrm{H}^i_c\bigl(X_{\overline{K},\mathrm{ket}}, j_!\bigl(\mathcal{L}(r)/p^n \mathcal{L}(r)\bigr)\bigr)$ which is the content of \cite[Thm. 2.3.5]{LLZ}.
The isomorphism $\mathrm{H}^i_{\mathrm{et}}\bigl(U_{\overline{K}},  \mathcal{L}(r)\bigr)\cong \mathrm{H}^i\bigl(X_{\overline{K},\mathrm{pket}},  j_\ast(\mathcal{L}(r))\bigr)$ and last statements for the de Rham period sheaves follow from \cite[Lemma 3.6.1]{DLLZb} using \cite[Thm 6.2.1 \& Cor. 6.3.4]{DLLZa}. For the crystalline statement it suffices to show that the natural map $\mathrm{H}^i\bigl(X_{\overline{K},\mathrm{pket}},  j_\ast(\mathcal{L}(r))\bigr)\otimes A_{\mathrm{crys}}  \to\mathrm{H}^i\bigl(X_{\overline{K},\mathrm{pket}},\widehat{\mathcal{L}}\otimes_{\widehat{\mathbb{Z}}_p} \mathbb{A}_{\mathrm{crys}}\bigr) $ is an isomorphism. Since $p$ is a regular element in $\mathbb{A}_{\mathrm{crys}}$ by \cite[Cor.~3.27]{AI}, we are reduced to prove that the natural map $\mathrm{H}^i\bigl(X_{\overline{K},\mathrm{pket}},  j_\ast(\mathcal{L}(r))\bigr)\otimes A_{\mathrm{crys}}/p A_{\mathrm{crys}} \to\mathrm{H}^i\bigl(X_{\overline{K},\mathrm{pket}},\widehat{\mathcal{L}}\otimes_{\widehat{\mathbb{Z}}_p} \mathbb{A}_{\mathrm{crys}}/p \mathbb{A}_{\mathrm{crys}}\bigr) $ is an isomorphism.  As  $\mathbb{A}_{\mathrm{crys}}/p \mathbb{A}_{\mathrm{crys}}$ is a free module over $\mathcal{O}_{X_{\mathrm{pket}}}^+/p \mathcal{O}_{X_{\mathrm{pket}}}^+$ and $A_{\mathrm{crys}}/p A_{\mathrm{crys}}$ is a free $\mathcal{O}_{\overline{K}}/p \mathcal{O}_{\overline{K}} $ with the same basis by \cite[Cor.~3.24]{AI}, one is reduced to the primitive comparison theorem \cite[Thm 6.2.1]{DLLZa}.

\end{proof}

\begin{lemma}\label{lemma:fundexactseq}   The following sequences on $X_{\mathrm{pket}}$ are exact:

$$0 \longrightarrow\widehat{\mathcal{L}}(r)[p^{-1}] \longrightarrow  \widehat{\mathcal{L}}\otimes_{\widehat{\mathbb{Z}}_p} \mathbb{B}_{\mathrm{crys}}  \stackrel{( \iota \circ (1-\frac{\varphi}{p^r}),\iota)}{\longrightarrow} \widehat{\mathcal{L}}\otimes_{\widehat{\mathbb{Z}}_p} \mathbb{B}_{\mathrm{crys}} \oplus \widehat{\mathcal{L}}\otimes_{\widehat{\mathbb{Z}}_p} \mathrm{Gr}^r \mathbb{B}_{\mathrm{dR}} \longrightarrow 0$$and, for each $J\subset I$, 

$$0 \to\iota_{J,\ast}\bigl(\widehat{\mathcal{L}}(r)\vert_{X_J}\bigr)[p^{-1}] \longrightarrow  \widehat{\mathcal{L}}\otimes_{\widehat{\mathbb{Z}}_p} \mathbb{B}_{\mathrm{crys},J}^\partial \stackrel{(\iota \circ (1-\frac{\varphi}{p^r}),\iota)}{\longrightarrow} \widehat{\mathcal{L}}\otimes_{\widehat{\mathbb{Z}}_p} \mathbb{B}_{\mathrm{crys},J}^\partial \oplus \widehat{\mathcal{L}}\otimes_{\widehat{\mathbb{Z}}_p} \mathrm{Gr}^r\mathbb{B}_{\mathrm{dR},J}^\partial \to 0.$$In particular, this gives an exact sequence $$0 \longrightarrow j_!\bigl(\widehat{\mathcal{L}}(r)\vert_U\bigr)[p^{-1}] \longrightarrow  \widehat{\mathcal{L}}\otimes_{\widehat{\mathbb{Z}}_p} \mathbb{B}_{\mathrm{crys}} ^{c-\infty} \stackrel{\iota\circ (1-\frac{\varphi}{p^r}),\iota)}{\longrightarrow} \widehat{\mathcal{L}}\otimes_{\widehat{\mathbb{Z}}_p} \mathbb{B}_{\mathrm{crys}}^{c-\infty}\oplus  \widehat{\mathcal{L}}\otimes_{\widehat{\mathbb{Z}}_p} \mathrm{Gr}^r\mathbb{B}_{\mathrm{dR}}^{c-\infty}  \longrightarrow 0.$$

\end{lemma}
\begin{proof} The first two  sequences are obtained from the analogous sequence with $\mathcal{L}=\mathbb{Z}_p$. Hence, we can assume that $\mathcal{L}=\mathbb{Z}_p$. Since $\mathrm{Fil}^r \mathbb{B}_{\mathrm{crys}}=t^r  \mathrm{Fil}^0 \mathbb{B}_{\mathrm{crys}} $ and $\varphi(t^r)=p^r t^r$, we may also assume that $r=0$. Passing to a basis of  $X_{\mathrm{pket}}$ by log perfectoid affinoids, one reduces the proof to  \cite[\S 6.2.2]{Brinon}.

The last statement follows from the fact that $$0 \to j_!\bigl(\mathcal{L}/p^n \mathcal{L}(r))\vert_U\bigr) \longrightarrow  \iota_{(\bullet),\ast}\bigl((\mathcal{L}/p^n \mathcal{L}(r))\vert_{X_{(\bullet)}} \bigr) $$ is exact for every $n\in \mathbb{N}$ by  \cite[Lemma 2.1.5]{LLZ}.
\end{proof}

Using Proposition \ref{prop:comparecomplexes} we obtain morphisms of complexes $$\mathbf{Syn}(\mathcal{E},\mathfrak{E},r)\to \mathbf{Syn}(\mathcal{L},r), \quad \mathbf{Syn}^\partial(\mathcal{E},\mathfrak{E},r)\to \mathbf{Syn}^{c-\infty}(\mathcal{L},r).$$
So we obtain from Proposition \ref{prop:etcoh} and from Lemma \ref{lemma:fundexactseq}:

\begin{theorem}\label{thm:comparesynt} We have  morphisms
$$\rho_{\mathrm{Syn}}^i(r)\colon \mathrm{H}^i\bigl(\mathbf{Syn}(\mathcal{E},\mathfrak{E},r)\bigr) \longrightarrow \mathrm{H}^i\bigl(\mathbf{Syn}(\mathcal{L},r)\bigr)\cong  \mathrm{H}^i_{\mathrm{et}}\bigl(U, \mathcal{L}(r)\bigr)[p^{-1}]$$and, using the identification $\mathrm{H}^i\bigl(\mathbf{Syn}^{c-\infty}(\mathcal{E},\mathfrak{E},r)\bigr)\cong \mathrm{H}^i\bigl(\mathbf{Syn}^\partial(\mathcal{E},\mathfrak{E},r)\bigr)$   of Proposition \ref{prop:samecohogroups},
$$\rho_{\mathrm{Syn},c}^i(r)\colon  \mathrm{H}^i\bigl(\mathbf{Syn}^{c-\infty}(\mathcal{E},\mathfrak{E},r)\bigr) \longrightarrow \mathrm{H}^i\bigl(\mathbf{Syn}^{c-\infty}(\mathcal{L},r)\bigr)\cong  \mathrm{H}^i_{\mathrm{et},c}\bigl(U, \mathcal{L}(r)\bigr)\bigr)[p^{-1}] .$$
\end{theorem}

\begin{proposition}\label{prop:comp0}For every integer $r$ the map  $\rho_{\mathrm{Syn}}^0(r)$ is an isomorphism. (Here we do not need to assume that $X$ is proper) 

\end{proposition}
\begin{proof}  The claim follows from  Proposition \ref{prop:phi-1inv}.
\end{proof}

We have the following functoriality results. Consider a morphism $f\colon Y\to X$ of smooth and proper rigid analytic variety over $K$ which is the generic fiber of a  log smooth morphism of formal schemes $f\colon \mathcal{X}\to \mathcal{Y}$ satisfying the Assumptions of  Proposition \ref{prop:FunctAss}(ii). By loc.~cit.~the pull backs $\mathcal{L}_Y$ of $\mathcal{L}$ and  $\mathbb{E}_Y=\bigl(\mathcal{E}_Y,\mathrm{Fil}^\bullet \mathcal{E}_Y, \nabla_Y,\mathfrak{E}_{\mathcal{Y}_k},\varphi_{\mathfrak{E},{\mathcal{Y}_k}}\bigr)$ of $\mathcal{L}$ and $\mathbb{E}:=\bigl(\mathcal{E},\mathrm{Fil}^\bullet \mathcal{E}, \nabla,\mathfrak{E},\varphi_{\mathfrak{E}}\bigr) $ are also crystalline associated. Then:

\begin{proposition} The  morphisms $\rho_{\mathrm{Syn}}^i(r)$ and $\rho_{\mathrm{Syn},c}^i(r)$ are functorial with respect to the morphisms in syntomic and \'etale cohomology given by pull back via $f$.

If $f\colon \mathcal{X}\to \mathcal{Y}$ is further assumed to be finite Kummer \'etale, we have also trace maps from the syntomic and \'etale cohomology of $\mathbb{E}_Y$ and $\mathcal{L}_Y$  to the syntomic and \'etale cohomology of $\mathbb{E}$ and $\mathcal{L}$ respectively. Then, the  morphisms $\rho_{\mathrm{Syn}}^i(r)$ and $\rho_{\mathrm{Syn},c}^i(r)$ commute with respect to these trace maps. 

\end{proposition}
\begin{proof}
The first claim is clear. The second follows from Remark \ref{rmk:traceasociated}.
\end{proof}

Assume we have two finite and Kummer \'etale morphisms $f_1$, $f_2\colon Y\to X$ which are the generic fiber of a  log smooth morphism of formal schemes $f_1$ and $f_2\colon \mathcal{X}\to \mathcal{Y}$ satisfying the Assumptions of  Proposition \ref{prop:FunctAss}(ii). Assume that the latter is also finite and Kummer \'etale. Suppose we have isomorphisms $\alpha$ and $\beta$  between the pull-backs of $\mathcal{L}$ and $\mathbb{E}$ via $f_1$ and $f_2$. Using $f_1^\ast$ and the trace map for $f_{2,\ast} \circ f_2^\ast$  we get an induced endomorphism  $(f_{2,\ast} \circ f_1^\ast)_{\rm syn}$ on $\mathrm{H}^i\bigl(\mathbf{Syn}(\mathcal{E},\mathfrak{E},r)\bigr)$ and $(f_{2,\ast} \circ f_1^\ast)_{\rm et}$ on $\mathrm{H}^i_{\mathrm{et}}\bigl(U, \mathcal{L}(r)\bigr)[p^{-1}]$ and similarly for the cohomology groups with compact support. Then,

\begin{cor}\label{cor:Heckeop} Assume that $\alpha$ and $\beta$ are compatible with the morphisms $\xi_{\mathrm{dR}}$ and $\xi_{\beta,\mathrm{crys}}$ of Definition \ref{def:crysass} on $Y_{\mathrm{pket}}$. Then  $(f_{2,\ast} \circ f_1^\ast)_{\rm et}\circ  \rho_{\mathrm{Syn}}^i(r)= \rho_{\mathrm{Syn}}^i(r)  \circ  (f_{2,\ast} \circ f_1^\ast)_{\rm syn}$ and similarly for $\rho_{\mathrm{Syn},c}^i(r)$.
\end{cor}

\section{Relation with Hyodo-Kato cohomology II}\label{sec:HKII}

\begin{definition}
\label{CupProd Def3}We say that $\mathcal{X}$ is of relative dimension $d$ when it is the finite disjoint union of formal schemes with geometrically connected generic and special fibers all of the same dimension 
$d$. If there is just one component, we say that $\mathcal{X}$ is geometrically connected of relative dimension $d$.
\end{definition}

Let $U:=X\backslash D$ and let $j\colon U\to X$ be the inclusion. Assume that $\mathcal{X}$ is proper of relative dimension $d$.
We let $\mathcal{L}$ be a $\mathbb{Z}_p$-local system on $X_{\mathrm{ket}}$ and we assume that it is crystalline associated to a non-degenerate Frobenius log-crystal $ \bigl(\mathcal{E},\mathrm{Fil}^\bullet \mathcal{E}, \nabla,\mathfrak{E},\varphi_{\mathfrak{E}}\bigr)$ relative to $(X,\alpha_X,\overline{\mathcal{X}},\alpha_{\overline{\mathcal{X}}})$; see Definition \ref{def:crysass}. The aim of this section is to prove the following

\begin{theorem}\label{thm:DstHK}  Assume that $\mathrm{H}^i_{\mathrm{et}}\bigl(U_{\overline{K}}, \mathcal{L}(r)\bigr)[p^{-1}]$ is a semistable representation of the absolute Galois group $G_K=\mathrm{Gal}(\overline{K}/K)$.
Then, we have an isomorphism of $K_0$-vector spaces, compatible with Frobenius and monodromy operator, $$\mathrm{H}^{i}_{\mathrm{HK}}\bigl((\mathcal{X}_k,\alpha_{\mathcal{X}_k})/ \mathcal{S}_0^0,\mathfrak{E}_0\bigr)(r) \stackrel{\sim}{\longrightarrow} D_{\rm st}\bigl(\mathrm{H}^i_{\mathrm{et}}\bigl(U_{\overline{K}}, \mathcal{L}(r)\bigr) \bigr).$$Similarly, if we assume that $\mathrm{H}^i_{\mathrm{et},c}\bigl(U_{\overline{K}}, \mathcal{L}(r)\bigr)[p^{-1}]$ is a semistable representation then  we have an isomorphism of $K_0$-vector spaces, compatible with Frobenius and monodromy operator
$$ \mathrm{H}^{i}_{\mathrm{HK}}\bigl((\mathcal{X}_k,\alpha_{\mathcal{X}_k})/ \mathcal{S}_0^0,\mathfrak{E}_0(-\mathfrak{D}_0)\bigr)(r) \stackrel{\sim}{\longrightarrow} D_{\rm st}\bigl(\mathrm{H}^i_{\mathrm{et},c}\bigl(U_{\overline{K}}, \mathcal{L}(r)\bigr) \bigr).$$
\end{theorem}

We have also the following compatibilities:

\begin{cor} \label{cor:DstHKDdR} Assume $\mathrm{H}^i_{\mathrm{et}}\bigl(U_{\overline{K}}, \mathcal{L}(r)\bigr)[p^{-1}]$ is a semistable representation. Then following diagram is commutative:
$$
\xymatrix{ \mathrm{H}_{
\mathrm{an}}^{i}\left( \left( \mathcal{X}_k,\alpha _{\mathcal{X}
_k}\right) /\left( \mathfrak{S},\mathfrak{M}\right) ,\mathfrak{E}^{\mathrm{
conv}}\right)  \left( r\right) \ar[d] \ar[r] & \mathrm{H}_{\mathrm{logdR
}}^{i}\left( X,\mathcal{E}\left( r\right) \right) \ar[r]^-{\cong} & D_{\mathrm{dR}
}\left( \mathrm{H}_{\mathrm{et}}^{i}\left( U_{\overline{K}},\mathcal{
L}\left( r\right) \right) \right) \\ \mathrm{H}_{\mathrm{HK}}^{i}\left( \left( \mathcal{X}_k,\alpha _{\mathcal{X
}_k}\right) /\mathcal{S}_{0}^{0},\mathfrak{E}_{0}\right) \left( r\right) \ar[r] & D_{\mathrm{st}}\left( \mathrm{H}_{\mathrm{et}}^{i}\left(
U_{\overline{K}},\mathcal{L}\left( r\right) \right) \right) \text{.} \ar@/_{5pt}/[ur] & }
$$where the morphism on the bottom line is the one in Theorem \ref{thm:DstHK},   the most left vertical morphism is defined in Proposition \ref{prop:anabs} by   specializing to $Z=0$, the first top morphism is the morphism $\gamma^{\rm{an}}$ given by specializing $Z\mapsto \pi$, the isomorphism $$D_{\mathrm{dR}
}\left( \mathrm{H}_{\mathrm{et}}^{i}\left( U_{\overline{K}},\mathcal{
L}\left( r\right) \right) \right) \simeq \mathrm{H}_{\mathrm{logdR}
}^{i}\left( X,\mathcal{E}\left( r\right) \right) $$is the one in  \cite[Thm. 3.2.7(3)]{DLLZb} and finally we have Fontaine's morphism 
$$D_{\mathrm{st}}\left( \mathrm{H}_{\mathrm{et}}^{i}\left(
U_{\overline{K}},\mathcal{L}\left( r\right) \right) \right)\longrightarrow D_{\mathrm{dR}}\left( \mathrm{H}_{\mathrm{et}}^{i}\left(
U_{\overline{K}},\mathcal{L}\left( r\right) \right) \right).$$An analogue compatibility holds if we assume that $\mathrm{H}^i_{\mathrm{et},c}\bigl(U_{\overline{K}}, \mathcal{L}(r)\bigr)[p^{-1}]$ is a semistable representation replacing $\mathfrak{E}^{\rm conv}$ with $\mathfrak{E}^{\rm conv}(-\mathfrak{D})$, $\mathfrak{E}_0$ with $\mathfrak{E}_0(-\mathfrak{D}_0)$ and $\mathcal{E}$ with $\mathcal{E}(-\mathcal{D})$. 
\end{cor}

We prove the Theorem and its Corollary in the case where we do not have  support conditions. The case with support proceeds similarly with one important difference, namely the analogue of Corollary \ref{cor:isoHietcryslog} provides a morphism from $\mathrm{H}^{i}_{\mathrm{HK}}\bigl((\mathcal{X}_k,\alpha_{\mathcal{X}_k})/ \mathcal{S}_0^0,\mathfrak{E}_0(-\mathfrak{D}_0)\bigr)(r) $ to $ D_{\rm st}\bigl(\mathrm{H}^i_{\mathrm{et},c}\bigl(U_{\overline{K}}, \mathcal{L}(r)\bigr) \bigr)$ that we can not prove it is an isomorphism directly as the analogue of  Proposition \ref{prop:phi-1inv} is not available for the period sheaves with support. The proof that it is an isomorphism uses Poincar\'e duality; see \S\ref{sec:conclusionproffofTHM}. For this reason, beside the properness, we need to suppose that $\mathcal{X}$ is of relative dimension $d$ and we can (and we will) assume that $\mathcal{X}$ is geometrically connected of relative dimension $d$.

\smallskip

Denote by ${\mathcal{O}}_{\mathfrak{S}}\mathbb{A}_{\mathrm{crys}}$ the $p$-adic completion of the log DP envelope of $$\vartheta\colon  \mathcal{O}_{\mathfrak{S}}\otimes_{\mathcal{O}_{K_0}} \mathbb{A}_{\mathrm{inf}} \to \widehat{\mathcal{O}}_{X}^+.$$Define $\mathcal{O}_{\mathfrak{S}}\mathbb{B}_{\mathrm{crys}}$  as ${\mathcal{O}}_{\mathfrak{S}}\mathbb{A}_{\mathrm{crys}}[t^{-1}]$. They are sheaves on $X_{\mathrm{pket}}$ endowed with a Frobenius.

\begin{proposition}\label{prop:OBlog} For every $\beta\in B$ the natural morphism of sheaves on $X_{\beta,\mathrm{pket}}$ $$\mathcal{O}_{\mathfrak{S}}\mathbb{B}_{\mathrm{crys}}\vert_{X_{\beta, \mathrm{pket}}} \longrightarrow{\mathcal{O}}_{\mathfrak{X}_\beta}\mathbb{B}_{\mathrm{crys}}\otimes_{\mathcal{O}_{\mathfrak{X}_\beta}} \Omega_{(\mathfrak{X}_\beta,\alpha_{\mathfrak{X}_\beta})/(\mathfrak{S},\mathfrak{M})}^{\log,\bullet} $$is  quasi isomorphisms, compatible with Frobenius. 
\end{proposition}
\begin{proof} The proof is the same as the proof of Theorem  \ref{thm:OBcrys}. Using the notation of loc.~cit.~one  proves that

$$\mathcal{O}_{\mathfrak{S}}\mathbb{A}_{\mathrm{crys}}\bigl(\widehat{\overline{R}}_\beta\bigr) \cong \mathbb{A}_{\mathrm{crys}}\bigl(\widehat{\overline{R}}_\beta\bigr) )\langle u-1 \rangle$$with $u=\frac{[\overline{\pi}]}{Z}$ for $\overline{\pi}=(\pi, \pi^{\frac{1}{p}},\cdots )\in A_{\rm inf}$ and that 

$$\mathcal{O}_{\mathfrak{X}_\beta}\mathbb{A}_{\mathrm{crys}}(\mathcal{U}_\beta)\cong {\mathcal{O}}_{\mathfrak{S}}\mathbb{A}_{\mathrm{crys}}\bigl(\widehat{\overline{R}}_\beta\bigr)\langle v_2-1, \ldots, v_a-1,w_1-1,\ldots,w_b-1\rangle.$$See \cite[Lemma 3.25]{AI}. Using that $\frac{d X_2}{X_2}$, $\ldots$, $ \frac{d X_a}{X_a}$, $\frac{d Y_1}{Y_1}$, $\ldots$, $\frac{d Y_b}{Y_b}$  form a basis of generators of the logarithmic differentials $\Omega^{\log}_{\widetilde{R}_\beta/\mathfrak{S}}$ one computes that the de Rham complex is exact except in degree $0$ where the kernel of the derivative on $\mathcal{O}_{\mathfrak{X}_\beta}\mathbb{A}_{\mathrm{crys}}(\mathcal{U}_\beta)$ is ${\mathcal{O}}_{\mathfrak{S}}\mathbb{A}_{\mathrm{crys}}\bigl(\widehat{\overline{R}}_\beta\bigr)$.
%Warning about preScholze notation in AI?
\end{proof}

\begin{cor}\label{cor:isoHietcryslog} We have a Frobenius equivariant isomorphism of $B_{\rm log}$-modules $$\mathrm{H}^i_{\mathrm{crys}}\left((\mathcal{X}_k,\alpha_{\mathcal{X}_{\beta,k}})/(\mathfrak{S}^{\rm DP},\mathfrak{M}),\mathfrak{E} \widehat{\otimes}_{\mathcal{O}_{\mathfrak{S}^{\rm DP}}} B_{\rm log}\right)(r)\longrightarrow \mathrm{H}^i_{\mathrm{et}}\bigl(U_{\overline{K}},  \mathcal{L}(r)\bigr)\otimes  B_{\rm log}.$$

\end{cor}
\begin{proof} We assume $r=0$. It follows from Proposition \ref{prop:OBlog} and Proposition \ref{prop:phi-1inv} that for every $\beta\in B$ we have an isomorphism $\psi_{\beta}$: $$\mathrm{R}\Gamma_{\mathrm{crys}}\left((\mathcal{X}_{\beta,k},\alpha_{\mathcal{X}_{\beta,k}})/(\mathfrak{S}^{\rm DP},\mathfrak{M}),\mathfrak{E}_\beta \widehat{\otimes}_{\mathcal{O}_{\mathfrak{S}^{\rm DP}}} B_{\rm log}\right) \longrightarrow \mathrm{R}\Gamma\left(X_{\beta,\overline{K},\mathrm{pket}},\widehat{\mathcal{L}}\otimes_{\widehat{\mathbb{Z}}_p} \mathcal{O}_{\mathfrak{S}} \mathbb{B}_{\mathrm{crys}}\right),$$compatible with Frobenius. Here $B_{\rm log}$ is as in Remark \ref{rmk:Blog}. For every $\underline{\beta}=(\beta_1,\ldots,\beta_n) \in B^n$ such that $\beta_1<\ldots < \beta_n$, this provides an isomorphism  
$\psi_{\underline{\beta}}$: $$\mathrm{R}\Gamma_{\mathrm{crys}}\left((\mathcal{X}_{\underline{\beta},k},\alpha_{\mathcal{X}_{\underline{\beta},k}})/(\mathfrak{S}^{\rm DP},\mathfrak{M}),\mathfrak{E}_{\underline{\beta}} \widehat{\otimes}_{\mathcal{O}_{\mathfrak{S}^{\rm DP}}} B_{\rm log}\right)\longrightarrow \mathrm{R}\Gamma\left(X_{\underline{\beta},\overline{K},\mathrm{pket}},\widehat{\mathcal{L}}\otimes_{\widehat{\mathbb{Z}}_p} \mathcal{O}_{\mathfrak{S}} \mathbb{B}_{\mathrm{crys}}\right)   ,$$compatible with Frobenius,  restricting the domain of $\psi_{\beta_i}$ to  $\mathcal{X}_{\underline{\beta},k}$ and the codomain of $\psi_{\beta_i}$ to $X_{\underline{\beta},\overline{K},\mathrm{pket}}$ and using that the former is isomorphic to  $$ \mathrm{R}\Gamma_{\mathrm{crys}}\left((\mathcal{X}_{\underline{\beta},k},\alpha_{\mathcal{X}_{\underline{\beta},k}})/(\mathfrak{S}^{\rm DP},\mathfrak{M}),\mathfrak{E}_{\underline{\beta}} \widehat{\otimes}_{\mathcal{O}_{\mathfrak{S}^{\rm DP}}} B_{\rm log}\right) . $$Arguing as in the proof of Proposition \ref{prop:comparecomplexes} one proves that this is independent of the choice of $\beta_i$ so that the $\psi_{\underline{\beta}}$'s are compatible for varying $\underline{\beta}$'s, inducing an isomorphism 
$$ \mathrm{R}\Gamma_{\mathrm{crys}}\left((\mathcal{X}_k,\alpha_{\mathcal{X}_{\beta,k}})/(\mathfrak{S}^{\rm DP},\mathfrak{M}),\mathfrak{E} \widehat{\otimes}_{\mathcal{O}_{\mathfrak{S}^{\rm DP}}} B_{\rm log}\right)  \longrightarrow \mathrm{R}\Gamma\left(X_{\overline{K},\mathrm{pket}},\widehat{\mathcal{L}}\otimes_{\widehat{\mathbb{Z}}_p} \mathcal{O}_{\mathfrak{S}} \mathbb{B}_{\mathrm{crys}}\right) ,$$compatible with Frobenius. Arguing as in the proof of Proposition \ref{prop:etcoh} one has $$\mathrm{H}^i_{\mathrm{et}}\bigl(U_{\overline{K}},  \mathcal{L}\bigr)\otimes  B_{\rm log}
\cong  \mathrm{H}^i\left(X_{\overline{K},\mathrm{pket}},\widehat{\mathcal{L}}\right)\otimes B_{\rm log}\cong \mathrm{H}^i\left(X_{\overline{K},\mathrm{pket}},\widehat{\mathcal{L}}\otimes_{\widehat{\mathbb{Z}}_p} \mathcal{O}_{\mathfrak{S}} \mathbb{B}_{\mathrm{crys}}\right) .$$The claim follows.
\end{proof} 

Our assumption that $ \mathrm{H}^i_{\mathrm{et}}\bigl(U_{\overline{K}},  \mathcal{L}(r)\bigr)$ is a semistable representation and \cite[Thm. 3.3]{Breuil} imply that we have a Frobenius equivariant isomorphism of $B_{\rm log}$-modules:
\begin{equation}\label{eq:breuilst} D_{\rm st}\bigl(\mathrm{H}^i_{\mathrm{et}}\bigl(U_{\overline{K}}, \mathcal{L}(r)\bigr) \bigr) \otimes_{K_0} B_{\rm log} \cong \mathrm{H}^i_{\mathrm{et}}\bigl(U_{\overline{K}},  \mathcal{L}(r)\bigr)\otimes_{\mathbb{Z}_p}  B_{\rm log}.
\end{equation}We deduce that we have  a Frobenius equivariant isomorphism of $B_{\rm log}$-modules:
$$\tau\colon \mathrm{H}^i_{\mathrm{crys}}\left((\mathcal{X}_k,\alpha_{\mathcal{X}_{\beta,k}})/(\mathfrak{S}^{\rm DP},\mathfrak{M}),\mathfrak{E} \widehat{\otimes}_{\mathcal{O}_{\mathfrak{S}^{\rm DP}}} B_{\rm log}\right)(r)\longrightarrow D_{\rm st}\bigl(\mathrm{H}^i_{\mathrm{et}}\bigl(U_{\overline{K}}, \mathcal{L}(r)\bigr) \bigr) \otimes_{K_0} B_{\rm log} .$$

Let $P(T)$ be a polynomial in $K_0[T]$ of the form $P(T)=1+TQ(T)$.  We write $\Phi$ for the Frobenius linear homomorphism defined either on the LHS or the RHS of $\tau$.  The Theorem \ref{thm:DstHK} follows from the following Lemma, taking $P(T)$ such that $P(\Phi)$ annihilates  $D_{\rm st}\bigl(\mathrm{H}^i_{\mathrm{et}}\bigl(U_{\overline{K}}, \mathcal{L}(r)\bigr) \bigr)$ and  $ \mathrm{H}^i_{\mathrm{HK}}\left((\mathcal{X}_k,\alpha_{\mathcal{X}_k})/ \mathcal{S}_0^0,\mathfrak{E}_0 \right)(r)$. We remark that such a $P(T)$ exists as they are both finite dimensional $K_0$-vector spaces on which $\Phi$ is an isomorphism and  $K_0$ is a finite extension of $\mathbb{Q}_p$; take $m$ such that $\varphi^m$ is the identity on $K_0$ and let $M(T)$ be the  product of the characteristic polynomials of $\Phi^m$ on this two $K_0$-vector spaces; then take $P(T)$ to be $M(T^m)$ divided by the constant term of $M(T)$.

\begin{lemma}\label{lemma:tauimples} The map $\tau$ defines an isomorphism on the  subspaces of the $G _K$-invariants killed by $P(\Phi)$ and domain and codomain of $\tau$ are identified with the spaces $ \mathrm{H}^i_{\mathrm{HK}}\left((\mathcal{X}_k,\alpha_{\mathcal{X}_k})/ \mathcal{S}_0^0,\mathfrak{E}_0 \right)(r)^{P(\Phi)=0}$ and $D_{\rm st}\bigl(\mathrm{H}^i_{\mathrm{et}}\bigl(U_{\overline{K}}, \mathcal{L}(r)\bigr) \bigr)^{P(\Phi)=0}$  respectively.

\end{lemma}

The rest of this section is devoted to the proof of the Lemma. This proves the Theorem. Also  the Corollary follows as we now explain.

\subsection{Proof of Corollary \ref{cor:DstHKDdR}}\label{sec:pfCorDst} First of all note that Fontaine's ring $B_{\rm st}$ is a subring of $B_{\rm log}$, see Remark \ref{rmk:Blog}, and the composite with the natural map $B_{\rm log}\to B_{\rm dR}$ is Fontaine's map of rings $B_{\rm st}\to B_{\rm dR}$. Hence, the isomorphism (\ref{eq:breuilst}) is compatible with Fontaine's isomorphism given by replacing $B_{\rm log}$ with $B_{\rm st}$ and $ B_{\rm dR}$. In fact, $D_?\bigl(\mathrm{H}^i_{\mathrm{et}}\bigl(U_{\overline{K}}, \mathcal{L}(r)\bigr) \bigr) $ is defined by taking the $G_K$-invariants of  $\mathrm{H}^i_{\mathrm{et}}\bigl(U_{\overline{K}},  \mathcal{L}(r)\bigr)\otimes_{\mathbb{Z}_p}  B_?$ for $?={\rm st}$ or ${\rm dR}$. The de Rham comparison in \cite[Thm. 3.2.7(3)]{DLLZb} follows arguing as in Proposition \ref{prop:OBlog} and Corollary \ref{cor:isoHietcryslog} using the period ring $\mathcal{O} \mathbb{B}_{\rm dR}$ instead of $\mathcal{O}_{\mathfrak{S}}\mathbb{B}_{\mathrm{crys}}$ and the logarithmic de Rham cohomology instead of the log-crystalline one. The conclusion of the Corollary follows as by assumption (iii) in Definition \ref{def:crysass} the two comparisons, crystalline and de Rham,  are compatible via the natural map of sheaves $\mathcal{O}_{\mathfrak{S}}\mathbb{B}_{\mathrm{crys}}\to \mathcal{O} \mathbb{B}_{\rm dR}$.

\subsection{The ring $B_{\rm max}$}
Let $A_{\rm max}$ be the  $p$-adic completion of the $(\mathcal{O}_{K_0}[\![Z]\!] \otimes_{\mathcal{O}_{K_0}} A_{\mathrm{inf}}(\mathbb{C}_p))$-subalgebra of the log envelope $(\mathcal{O}_{K_0}[\![Z]\!] \otimes_{\mathcal{O}_{K_0}} A_{\mathrm{inf}}(\mathbb{C}_p))^{\rm log}[p^{-1}]$ generated by $\frac{\mathrm{ker}(\theta_{\rm log})}{p}$ with 
$$\vartheta_{\rm log} \colon \mathcal{O}_{K_0}[\![Z]\!] \otimes_{\mathcal{O}_{K_0}} A_{\mathrm{inf}}(\mathbb{C}_p)^{\rm log} \to \mathcal{O}_{\mathbb{C}_p}$$the extension of the map $\vartheta$ in Remark \ref{rmk:Blog}. We set $B_{\rm max}=A_{\rm max}[t^{-1}]$. We have $$A_{\rm max}\cong A_{\mathrm{inf}}(\mathbb{C}_p) 
\left\{\frac{P_\pi([\overline{\pi}])}{p},\frac{u-1}{p}\right\}, \quad u=\frac{[\overline{\pi}]}{Z}$$where $P_\pi(Z)$ is  the minimal polynomial of $\pi$; see \cite[Lemma 3.23]{AI}.  In particular, it contains $A_ {\rm log}$ as a subring and $B_{\rm max}$ contains $B_{\rm log}$ as a subring.

Since $P_\pi([\overline{\pi}])-[\overline{\pi}]^e=p w$ with $e$ the degree of $P_\pi(Z)$ and with $w\in A_{\mathrm{inf}}(\mathbb{C}_p)$ and since $Z=[\overline{\pi}] u$, with $u$ an invertible element in $A_{\rm max}$, we have $A_{\rm max}\cong A_{\mathrm{inf}}(\mathbb{C}_p)\{\frac{Z^e}{p},\frac{u-1}{p}\}$ so that $A_{\rm max}$  is an algebra over $\mathcal{O}_{\rm max}:=\mathcal{O}_{K_0}\{Z,\frac{Z^e}{p}\}$. Let $\mathfrak{m}\subset \mathcal{O}_{\rm max}$ be the ideal $(Z, Z^e/p)$. It is the kernel of the projection $\mathcal{O}_{\rm max}\to \mathcal{O}_{K_0}$. We let $\mathcal{S}_{\rm max}^0$ denote the associated $p$-adic formal scheme endowed with the log strcture defined by the divisor $Z$.  We deduce that for every $m\in \mathbb{N}$ the $\mathcal{O}_{K_0}/p^m \mathcal{O}_{K_0}$- algebra $$A_{\rm max}/(p^m,\mathfrak{m}) A_{\rm max}\cong  A_{\mathrm{inf}}(\mathbb{C}_p)/(p^m,[\overline{\pi}]) \left[\frac{u-1}{p}\right]$$is flat. Notice that Frobenius on $\mathcal{O}_{K_0}[\![Z]\!] \otimes_{\mathcal{O}_{K_0}} A_{\mathrm{inf}}(\mathbb{C}_p)$ defines a Frobenius $\varphi$ on  $A_{\rm max}$ and $B_{\rm max}$, which is Frobenius on  $A_{\mathrm{inf}}(\mathbb{C}_p)$ and sends $Z\mapsto Z^p$ and $u\mapsto u^p$. Since $A_{\rm log}\cong A_{\mathrm{cris}}(\mathbb{C}_p)\{\langle u-1\rangle \}\cong A_{\mathrm{inf}}(\mathbb{C}_p)\{\langle P_\pi([\overline{\pi}]), u-1\rangle \}$ (see \cite[\S 2.1.1]{AI}), we have $\varphi(A_{\rm max})\subset A_{\rm log}\subset A_{\rm max}$ and $\varphi(B_{\rm max})\subset B_{\rm log}\subset B_{\rm max}$.

\subsection{Passing to $\mathfrak{E}^{\rm max}$}
For every $\underline{\beta}\in B^n$ such that $\beta_1<\ldots < \beta_n$ let $\mathfrak{E}_{\underline{\beta}}^{\rm max}$ be the restriction of $\mathfrak{E}_{\underline{\beta}}^{\rm conv}$ (see \S\ref{sec:analsyn} for the notation) to the tube  $]\mathcal{X}_{\underline{\beta},k}[_{\mathfrak{X}_{\underline{\beta},\frac{1}{p}}}$ of radius $1/p$. 
The maps $\mathfrak{E}_{\underline{\beta}}\to\mathfrak{E}_{\underline{\beta}}^{\rm max}$  and $B_{\rm log}\to B_{\rm max}$ provide the following commutative diagram

$$\begin{matrix}
\mathrm{H}^i_{\mathrm{crys}}\left((\mathcal{X}_k,\alpha_{\mathcal{X}_{\beta,k}})/(\mathfrak{S}^{\rm DP},\mathfrak{M}),\mathfrak{E} \widehat{\otimes}_{\mathcal{O}_{\mathfrak{S}^{\rm DP}}} B_{\rm log}\right)(r) & \stackrel{\tau}{\longrightarrow} & D_{\rm st}\bigl(\mathrm{H}^i_{\mathrm{et}}\bigl(U_{\overline{K}}, \mathcal{L}(r)\bigr) \bigr) \otimes_{K_0} B_{\rm log} \cr \big\downarrow & & \big\downarrow \cr \mathrm{H}^i_{\mathrm{an}}\left((\mathcal{X}_k,\alpha_{\mathcal{X}_{\beta,k}})/\mathcal{S}_{\rm max}^0,\mathfrak{E}^{\rm max} \widehat{\otimes}_{\mathcal{O}_{\rm max}} B_{\rm max}\right)(r)
& \stackrel{\tau_{\rm max}}{\longrightarrow} &  D_{\rm st}\bigl(\mathrm{H}^i_{\mathrm{et}}\bigl(U_{\overline{K}}, \mathcal{L}(r)\bigr) \bigr) \otimes_{K_0} B_{\rm max} ,\cr
\end{matrix}$$where the bottom left cohomology group is defined as in \S \ref{sec:analsyn}. As remarked above $B_{\rm log}\to B_{\rm max}$ is injective and has cokernel annihilated by $\varphi$. A similar explicit computation shows that the map $\mathfrak{E}_{\underline{\beta}}\to\mathfrak{E}_{\underline{\beta}}^{\rm max}$ is injective and has cokernel annihilated by  the restriction of  $\Phi_{\underline{\beta}}^{\rm conv}$. Since $P(T)$ is congruent to the identity modulo $T$, one deduces that $P(\varphi)$ acts as an isomorphism on $D_{\rm st}\bigl(\mathrm{H}^i_{\mathrm{et}}\bigl(U_{\overline{K}}, \mathcal{L}(r)\bigr) \bigr) \otimes_{K_0} \frac{B_{\rm log}}{B_{\rm max}}$. Similarly, the complexes computing the domain and the codomain of the left vertical arrow $\pi$ are one contained in the other, they are endowed with (compatible lifts of) the Frobenii $\Phi$ (cfr. \S \ref{sssec:explcictcrystallne} and \S \ref{sec:analsyn}) and the cokernel $C$ is killed by $\Phi$. It follows that $\Phi$ also kills $H^{i-1}(C)$ and $H^{i}(C)$ as well as the image $X$ of $H^{i-1}(C)$ in the source of $\pi$ and the image $W$ of the target of $\pi$ in $H^{i}(C)$, from which it follows that $P(\Phi)$ acts as an isomorphism on $X$ and $W$. It then follows from Remark \ref{remark:Marco} below that the vertical arrows in the diagram above are isomorphisms if we take the kernels of $P(\Phi)$.  

\begin{remark}\label{remark:Marco}
Suppose that we are given a commutative diagram of modules over some ring
\begin{equation*}
\begin{array}{ccccccccc}
0\longrightarrow  & X & \longrightarrow  & Y & \overset{\pi }{
\longrightarrow } & Z & \longrightarrow  & W & \longrightarrow 0 \\ 
& x\downarrow \text{ } &  & y\downarrow \text{ } &  & z\downarrow \text{ } & 
& w\downarrow \text{ } &  \\ 
0\longrightarrow  & X^{\prime } & \longrightarrow  & Y^{\prime } & \overset{
\pi }{\longrightarrow } & Z^{\prime } & \longrightarrow  & W^{\prime } & 
\longrightarrow 0
\end{array}
\end{equation*}
such that $x$ and $w$ are isomorphisms. Then the morphism $\pi $\ induces
isomorphisms
\begin{equation*}
\mathrm{ker}\left( y\right) \overset{\sim }{\longrightarrow }\mathrm{ker}
\left( z\right) \text{ and }\mathrm{coker}\left( y\right) \overset{\sim }{
\longrightarrow }\mathrm{coker}\left( z\right) .
\end{equation*}
\end{remark}

Since $\mathcal{O}_{\rm max}[p^{-1}]\subset B_{\rm max}^{G_K}$ has cokernel annihilated by $\varphi^3$ by \cite[Prop. 5.1.1]{Breuil}, we deduce (once again by Remark \ref{remark:Marco}) that 
$$\bigl(D_{\rm st}\bigl(\mathrm{H}^i_{\mathrm{et}}\bigl(U_{\overline{K}}, \mathcal{L}(r)\bigr) \bigr)\otimes_{K_0} \mathcal{O}_{\rm max}[p^{-1}]\bigr)^{P(\Phi)=0}  \cong\left( D_{\rm st}\bigl(\mathrm{H}^i_{\mathrm{et}}\bigl(U_{\overline{K}}, \mathcal{L}(r)\bigr) \bigr) \otimes_{K_0} B_{\rm max} \right)^{G_K,P(\Phi)=0}.  $$Using that $\mathcal{O}_{\rm max}[p^{-1}]/ \mathfrak{m}\mathcal{O}_{\rm max}[p^{-1}]\cong K_0$, we conclude from Lemma \ref{lemma:PHIM0} (with $M= D_{\rm st}\bigl(\mathrm{H}^i_{\mathrm{et}}\bigl(U_{\overline{K}}, \mathcal{L}(r)\bigr) \bigr)\otimes_{K_0} \mathcal{O}_{\rm max}[p^{-1}]$ and $\mathcal{O}_{\rm max}= \mathcal{O}_{e,1}$) that $$\bigl(D_{\rm st}\bigl(\mathrm{H}^i_{\mathrm{et}}\bigl(U_{\overline{K}}, \mathcal{L}(r)\bigr) \bigr)\otimes_{K_0} \mathcal{O}_{\rm max}[p^{-1}]\bigr)^{P(\Phi)=0}  \cong D_{\rm st}\bigl(\mathrm{H}^i_{\mathrm{et}}\bigl(U_{\overline{K}}, \mathcal{L}(r)\bigr) \bigr)^{P(\Phi)=0}  .$$We now show that we have an isomorphism
$$ \mathrm{H}^i_{\mathrm{HK}}\bigl((\mathcal{X}_k,\alpha_{\mathcal{X}_k})/\mathcal{S}_0^0, \mathfrak{E}_0\bigr)(r)^ {P(\Phi)=0}\cong  \mathrm{H}^i_{\mathrm{an}}\left((\mathcal{X}_k,\alpha_{\mathcal{X}_{\beta,k}})/\mathcal{S}_{\rm max}^0,\mathfrak{E}^{\rm max} \widehat{\otimes}_{\mathcal{O}_{\rm max}} B_{\rm max}\right)(r)^{G_K,P(\Phi)=0}.$$This and the fact that $\tau$ is an isomorphism imply Lemma \ref{lemma:tauimples}.

\

We start with  $\mathrm{H}^i_{\mathrm{cris}}\bigl((\mathcal{X}_k,\alpha_{\mathcal{X}_k})/\mathcal{S}_0^0, \mathfrak{E}_0\bigr)(r)$. Note that for every $n$  the  $\mathcal{O}_{K_0}/p^n \mathcal{O}_{K_0}$-module $\mathrm{H}^i_{\mathrm{cris}}\bigl((\mathcal{X}_k,\alpha_{\mathcal{X}_k})/\mathcal{S}_0^0, \mathfrak{E}_0/p^n \mathfrak{E}_0\bigr)(r)$ is finite so that for varying $n$ the corresponding inverse system is Mittag-Leffler.  Since $A_{\rm max}$ is $p$-torsion free, we deduce that $$\mathrm{H}^i_{\mathrm{cris}}\bigl((\mathcal{X}_k,\alpha_{\mathcal{X}_k})/\mathcal{S}_0^0,  \mathfrak{E}_0\bigr) (r) \otimes_{\mathcal{O}_{K_0}} A_{\rm max}\longrightarrow
\mathrm{H}^i_{\mathrm{cris}}\bigl((\mathcal{X}_k,\alpha_{\mathcal{X}_k})/\mathcal{S}_0^0, \mathfrak{E}_0 \widehat{\otimes}_{\mathcal{O}_{K_0}}  A_{\rm max}\bigr)(r)$$is an isomorphism and,  thus, also replacing $ A_{\rm max}$ with $B_{\rm max} $ we get an isomorphism. Using the flatness of $A_{\rm max}/(p^m,\mathfrak{m}) A_{\rm max}$ as $\mathcal{O}_{K_0}/p^m \mathcal{O}_{K_0}$-module for every $m$, the same result holds replacing $ A_{\rm max}$ with $A_{\rm max} /\mathfrak{m} A_{\rm max}$ and $ B_{\rm max}$ with $B_{\rm max} /\mathfrak{m} B_{\rm max}$. We conclude that the maps
$$\mu\colon \mathrm{H}^i_{\mathrm{HK}}\bigl((\mathcal{X}_k,\alpha_{\mathcal{X}_k})/ \mathcal{S}_0^0, \mathfrak{E}_0\bigr)(r)\otimes_{K_0}  B_{\rm max}\longrightarrow \mathrm{H}^i_{\mathrm{cris}}\bigl((\mathcal{X}_k,\alpha_{\mathcal{X}_k})/\mathcal{S}_0^0,\mathfrak{E}_0 \widehat{\otimes}_{\mathcal{O}_{K_0}}  B_{\rm max}\bigr)(r)$$and
 $$\overline{\mu}\colon \mathrm{H}^i_{\mathrm{HK}}\bigl((\mathcal{X}_k,\alpha_{\mathcal{X}_k})/ \mathcal{S}_0^0, \mathfrak{E}_0\bigr)(r)\otimes_{K_0}  \frac{B_{\rm max}}{\mathfrak{m} B_{\rm max}}\longrightarrow \mathrm{H}^i_{\mathrm{cris}}\bigl((\mathcal{X}_k,\alpha_{\mathcal{X}_k})/\mathcal{S}_0^0,\mathfrak{E}_0 \widehat{\otimes}_{\mathcal{O}_{K_0}}  \frac{B_{\rm max}}{\mathfrak{m} B_{\rm max}}\bigr)(r)$$are isomorphisms.

\subsection{CLAIM} We claim that the induced map:
$$\nu\colon \mathrm{H}^i_{\mathrm{HK}}\bigl((\mathcal{X}_k,\alpha_{\mathcal{X}_k})/ \mathcal{S}_0^0, \mathfrak{E}_0\bigr)(r) \otimes_{K_0} B_{\rm max} \longrightarrow \mathrm{H}^i_{\mathrm{cris}}\bigl((\mathcal{X}_k,\alpha_{\mathcal{X}_k})/\mathcal{O}_{K_0}^0,\mathfrak{E}_0 \widehat{\otimes}_{\mathcal{O}_{K_0}}  \frac{B_{\rm max}}{\mathfrak{m} B_{\rm max}}\bigr)(r)$$defines an isomorphism between
$\mathrm{H}^i_{\mathrm{HK}}\bigl((\mathcal{X}_k,\alpha_{\mathcal{X}_k})/ \mathcal{S}_0^0, \mathfrak{E}_0\bigr)(r)^{P(\Phi)=0}$ and  $$\left(\mathrm{H}^i_{\mathrm{cris}}\bigl((\mathcal{X}_k,\alpha_{\mathcal{X}_k})/\mathcal{O}_{K_0}^0,\mathfrak{E}_0 \widehat{\otimes}_{\mathcal{O}_{K_0}}  \frac{B_{\rm max}}{\mathfrak{m} B_{\rm max}}\bigr)\right)(r)^{G_K,P(\Phi)=0}.$$

{\em Proof:}    Consider the $\mathcal{O}_{\rm max}$-modules  $\mathrm{H}^i_{\mathrm{cris}}\bigl((\mathcal{X}_k,\alpha_{\mathcal{X}_k})/\mathcal{S}_0^0,  \mathfrak{E}_0\bigr) (r) \otimes_{\mathcal{O}_{K_0}} A_{\rm max} t^{-n}$, for $n\in\N$. Applying Lemma \ref{lemma:PHIM0} we deduce that $P(\Phi)$ is an isomorphism on these modules and taking the direct limit over $n\in\N$ that $P(\Phi)$ is an isomorphism on  $\mathrm{H}^i_{\mathrm{HK}}\bigl((\mathcal{X}_k,\alpha_{\mathcal{X}_k})/ \mathcal{S}_0^0, \mathfrak{E}_0\bigr)(r) \otimes_{K_0} (\mathfrak{m} B_{\rm max} )$.
Hence, using that $\overline{\mu}$ are isomorphisms, we get that $\nu$ induces an isomorphism after taking the kernels of $P(\Phi)$.  One argues as before for $D_{\rm st}\bigl(\mathrm{H}^i_{\mathrm{et}}\bigl(U_{\overline{K}}, \mathcal{L}(r)\bigr) \bigr)$  that $$\bigl(\mathrm{H}^i_{\mathrm{HK}}\bigl((\mathcal{X}_k,\alpha_{\mathcal{X}_k})/ \mathcal{S}_0^0, \mathfrak{E}_0\bigr)(r) \otimes_{K_0} \mathcal{O}_{\rm max}[p^{-1}]\bigr)^{P(\Phi)=0}$$ is isomorphic to $$\bigl(\mathrm{H}^i_{\mathrm{HK}}\bigl((\mathcal{X}_k,\alpha_{\mathcal{X}_k})/ \mathcal{S}_0^0, \mathfrak{E}_0\bigr)(r) \otimes_{K_0} B_{\rm max}\bigr)^{G_K,P(\Phi)=0} $$and $$\bigl(\mathrm{H}^i_{\mathrm{HK}}\bigl((\mathcal{X}_k,\alpha_{\mathcal{X}_k})/ \mathcal{S}_0^0, \mathfrak{E}_0\bigr)(r) \otimes_{K_0} \mathcal{O}_{\rm max}\bigr)^{P(\Phi)=0}\cong \mathrm{H}^i_{\mathrm{HK}}\bigl((\mathcal{X}_k,\alpha_{\mathcal{X}_k})/ \mathcal{S}_0^0, \mathfrak{E}_0\bigr)(r) ^{P(\Phi)=0}.$$The conclusion follows.

\subsection{Conclusion}
By Lemma \ref{lemma:PHIM0}  the map $P(\Phi)$ is an isomorphism on each $\mathfrak{m} \mathfrak{E}^{\rm max}_{\underline{\beta}} $ for every $\underline{\beta}\in B^n$. Hence, $P(\Phi)$ is an isomorphism on the cohomology groups  $\mathrm{H}^i_{\rm an}\left((\mathcal{X}_k,\alpha_{\mathcal{X}_k})/\mathcal{S}_{\rm max}^0,\mathfrak{m} \mathfrak{E}^{\rm max}\right)(r)$, defined as in  \S \ref{sec:analsyn}. Since $\mathrm{H}^i_{\rm cris}\left((\mathcal{X}_k,\alpha_{\mathcal{X}_k})/ \mathcal{S}_0^0,\mathfrak{E}_0 \right)(r)$ is a finite dimensional $K_0$-vector space, it follows from Lemma \ref{lemma:Marco} that the map
$$\rho\colon \mathrm{H}^i_{\rm an}\left((\mathcal{X}_k,\alpha_{\mathcal{X}_k})/\mathcal{S}_{\rm max}^0,\mathfrak{E}^{\rm max}\right)(r)^{P(\Phi)=0}\longrightarrow \mathrm{H}^i_{\mathrm{HK}}\left((\mathcal{X}_k,\alpha_{\mathcal{X}_k})/ \mathcal{S}_0^0,\mathfrak{E}_0 \right)(r)^{P(\Phi)=0} $$is an isomorphism.
 Consider the natural map $\zeta_i$ from
$$\mathrm{H}^i_{\rm an}\left((\mathcal{X}_k,\alpha_{\mathcal{X}_k})/\mathcal{S}_{\rm max}^0,\mathfrak{E}^{\rm max} \widehat{\otimes}_{\mathcal{O}_{\rm max}}B_{\rm max}\right)(r)$$ to $$\mathrm{H}^i_{\rm an}\left((\mathcal{X}_k,\alpha_{\mathcal{X}_k})/\mathcal{S}_{\rm max}^0,\mathfrak{E}^{\rm max} \widehat{\otimes}_{\mathcal{O}_{K_0}}\frac{B_{\rm max}}{\mathfrak{m} B_{\rm max}}\right)(r) ,$$which coincides with  $\mathrm{H}^i_{\rm cris}\left((\mathcal{X}_k,\alpha_{\mathcal{X}_k})/\mathcal{S}_0^0,\mathfrak{E}_0 \widehat{\otimes}_{\mathcal{O}_{K_0}}\frac{B_{\rm max}}{\mathfrak{m} B_{\rm max}}\right)(r)$.  As remarked before the Lemma, there exists a polynomial $R(\Phi)$ annihilating $\mathrm{H}^i_{\mathrm{HK}}\bigl((\mathcal{X}_k,\alpha_{\mathcal{X}_k})/ \mathcal{S}_0^0, \mathfrak{E}_0\bigr)(r)$. The isomorphism $\rho$ for $R(\Phi)$ and the CLAIM imply that $\zeta_i$ is surjective on the subspaces of the $G_K$-invariants killed by $R(\Phi)$.  In particular, the image of the map $\zeta_i$ contains $ \mathrm{H}^i_{\mathrm{HK}}\left((\mathcal{X}_k,\alpha_{\mathcal{X}_k})/ \mathcal{S}_0^0,\mathfrak{E}_0 \right)(r)$. Hence, $\zeta_i$  is surjective as it is $B_{\rm max}$-linear and $\overline{\mu}$ is an isomorphism.

Since $P(\Phi)$ is an isomorphism on $\mathfrak{m} \mathfrak{E}^{\rm max}_{\underline{\beta}}\widehat{\otimes}_{\mathcal{O}_{\rm max}}B_{\rm max} $ for every $\underline{\beta}\in B^n$, we deduce that $P(\Phi)$ is an isomorphism on $\mathrm{H}^i_{\rm an}\left((\mathcal{X}_k,\alpha_{\mathcal{X}_k})/\mathcal{S}_{\rm max}^0,\mathfrak{m} \mathfrak{E}^{\rm max} \widehat{\otimes}_{\mathcal{O}_{\rm max}}B_{\rm max}\right)(r)$ and, hence, on the kernel of each $\zeta_i$. It follows from Remark \ref{remark:Marco} that  $\zeta_i$ defines an isomorphism on the subspaces annihilated by $P(\Phi)$ and, hence, also on their $G_K$-invariants. This concludes the proof of the Lemma.

\subsection{End of proof of Theorem \ref{thm:DstHK}}\label{sec:conclusionproffofTHM} In order to conclude the proof of the Theorem we  are left to show that if $\mathrm{H}^i_{\mathrm{et},c}\bigl(U_{\overline{K}}, \mathcal{L}(r)\bigr)[p^{-1}]$ is a semistable representation then  
$$\upsilon_c\colon \mathrm{H}^{i}_{\mathrm{HK}}\bigl((\mathcal{X}_k,\alpha_{\mathcal{X}_k})/ \mathcal{S}_0^0,\mathfrak{E}_0(-\mathfrak{D}_0)\bigr)(r) \longrightarrow D_{\rm st}\bigl(\mathrm{H}^i_{\mathrm{et},c}\bigl(U_{\overline{K}}, \mathcal{L}(r)\bigr) \bigr)$$is an isomorphism.   First notice that $\mathcal{L}^\vee$ is crystalline associated to the dual of $ \bigl(\mathcal{E},\mathrm{Fil}^\bullet \mathcal{E}, \nabla,\mathfrak{E},\varphi_{\mathfrak{E}}\bigr)$ by Proposition \ref{prop:FunctAss}. We then consider the diagrams of Corollary \ref{cor:DstHKDdR} for the cohomologies $\mathrm{H}^{i}$ with compact support of  $\mathcal{L}(r)$  and $ \bigl(\mathcal{E},\mathrm{Fil}^\bullet \mathcal{E}, \nabla,\mathfrak{E},\varphi_{\mathfrak{E}}\bigr)$
and for $\mathrm{H}^{2d-i}$  of $\mathcal{L}^\vee(-r)$ and the dual of $ \bigl(\mathcal{E},\mathrm{Fil}^\bullet \mathcal{E}, \nabla,\mathfrak{E},\varphi_{\mathfrak{E}}\bigr)$.  We have perfect Poincar\'e dualities in the \'etale  and log de Rham settings and they are compatible after tensoring with $B_{\rm dR}$ by  \cite[Thm.~4.4.1(4)]{LLZ}. This implies, in particular, that $\mathrm{H}^{2d-i}_{\mathrm{et}}\bigl(U_{\overline{K}}, \mathcal{L}^\vee(-r)\bigr)[p^{-1}]$ is also a semistable representation as it is dual to  $\mathrm{H}^i_{\mathrm{et},c}\bigl(U_{\overline{K}}, \mathcal{L}(r)\bigr)$. Similarly, we have perfect Poincar\'e duality  in the Hyodo-Kato setting by \cite[Prop.~4.6]{EY} and it is compatible with log de Rham Poincar\'e duality  thanks to Lemma \ref{lemma:traceHKlogdRcomp} proved below. This implies that $\upsilon$ and the isomorphism $$\upsilon\colon \mathrm{H}^{2d-i}_{\mathrm{HK}}\bigl((\mathcal{X}_k,\alpha_{\mathcal{X}_k})/ \mathcal{S}_0^0,\mathfrak{E}_0^\vee)(-r) \stackrel{\sim}{\longrightarrow} D_{\rm st}\bigl(\mathrm{H}^{2d-i}_{\mathrm{et}}\bigl(U_{\overline{K}}, \mathcal{L}^\vee(-r)\bigr) \bigr)$$of Theorem \ref{thm:DstHK} are compatible with Hyodo-Kato Poincar\'e duality and \'etale Poincar\'e duality  after tensoring  $\otimes_{K_0} B_{\rm dR}$. Since these are perfect dualities, $\upsilon_c$ is an isomorphism after tensoring  with $B_{\rm dR}$ and, hence, it is an isomorphism.

\section{Hochschild-Serre}\label{sec:HS}

We keep the notations of section \ref{sec:compareetale}. In particular, we assume that $X$ is proper.  Consider Fontaine's classical fundamental exact sequence $$0 \to \mathbb{Q}(r) \longrightarrow  B_{\mathrm{crys}} \stackrel{(1-\frac{\varphi}{p^r},\iota)}{\longrightarrow} B_{\mathrm{crys}} \oplus \frac{B_{\mathrm{dR}}}{ \mathrm{Fil}^rB_{\mathrm{dR}}}  \to 0.$$Tensoring with $\mathrm{H}^i_{\mathrm{et}}\bigl(U_{\overline{K}},  \mathcal{L}(r)\bigr)$ and taking $G_K$-cohomology, we obtain the exponential map 
$$\exp_r\colon \frac{D_{\mathrm{dR}}\bigl(\mathrm{H}^i_{\mathrm{et}}\bigl(U_{\overline{K}},  \mathcal{L}(r)\bigr)\bigr)}{\mathrm{Fil}^r D_{\mathrm{dR}}\bigl(\mathrm{H}^i_{\mathrm{et}}\bigl(U_{\overline{K}},  \mathcal{L}(r)\bigr)\bigr)} \longrightarrow \mathrm{H}^1(G_K,\mathrm{H}^i_{\mathrm{et}}\bigl(U_{\overline{K}},  \mathcal{L}(r)\bigr)\bigr)[p^{-1}].$$Recall that given a $\mathbb{Q}_p$-representation $V$ of $G_K$ one defines  $ \mathrm{D}_{\mathrm{dR}}(V):=\bigl(V\otimes B_{\mathrm{dR}}\bigr)^{G_K}$. We also have the Hochschild-Serre spectral sequence

$$\mathrm{H}^j\bigl(G_K,\mathrm{H}^i_{\mathrm{et}}\bigl(U_{\overline{K}},  \mathcal{L}(r)\bigr) \bigr) \Rightarrow 
 \mathrm{H}^{i+j}_{\mathrm{et}}\bigl(U,\mathcal{L}(r)\bigr).
$$Defining  $F^1\mathrm{H}^i_{\mathrm{et}}\bigl(U,  \mathcal{L}(r)\bigr)$ as the kernel of the natural map $\mathrm{H}^i_{\mathrm{et}}\bigl(U,  \mathcal{L}(r)\bigr)\to \mathrm{H}^i_{\mathrm{et}}\bigl(U_{\overline{K}},  \mathcal{L}(r)\bigr)$ we get a morphism $$\mathrm{HS }^i\colon F^1 \mathrm{H}^{i}_{\mathrm{et}}\bigl(U,  \mathcal{L}(r)\bigr)\to \mathrm{H}^1(G_K, \mathrm{H}^{i-1}_{\mathrm{et}}\bigl(U_{\overline{K}},  \mathcal{L}(r)\bigr)\bigr).$$Analogously, we have the exponential map
$$\exp_{r,c}\colon \frac{D_{\mathrm{dR}}\bigl(\mathrm{H}^i_{\mathrm{et},c}\bigl(U_{\overline{K}}, \mathcal{L}(r)\bigr)\bigr)}{\mathrm{Fil}^r  D_{\mathrm{dR}}\bigl(\mathrm{H}^i_{\mathrm{et},c}\bigl(U_{\overline{K}},  \mathcal{L}(r)\bigr)\bigr) }
 \longrightarrow \mathrm{H}^1(G_K,\mathrm{H}^i_{\mathrm{et},c}\bigl(U_{\overline{K}}, \mathcal{L}(r)\bigr)\bigr)[p^{-1}],$$the Hochschild-Serre spectral sequence
$$\mathrm{H}^j\bigl(G_K,\mathrm{H}^i_{\mathrm{et},c}\bigl(U_{\overline{K}},  \mathcal{L}(r)\bigr) \bigr) \Rightarrow 
\mathrm{H}^i_{\mathrm{et},c}\bigl(U_{\mathrm{pket}}, \mathcal{L}(r)\bigr)\bigr),
$$inducing a filtration   $F^1 \mathrm{H}^{i}_{\mathrm{et},c}\bigl(U, \mathcal{L}(r)\bigr)\subset \mathrm{H}^{i}_{\mathrm{et},c}\bigl(U, \mathcal{L}(r)\bigr)$ and a map $$\mathrm{HS }^i_c\colon F^1\mathrm{H}^{i}_{\mathrm{et},c}\bigl(U, \mathcal{L}(r)\bigr)\to \mathrm{H}^1(G_K, \mathrm{H}^{i-1}_{\mathrm{et},c}\bigl(U_{\overline{K}},  \mathcal{L}(r)\bigr)\bigr).$$

It follows from Proposition \ref{prop:spectralanal} and its analogue with support that we have canonical morphisms
$$\frac{\mathrm{H}_{\mathrm{logdR}}^{i-1}\left( X,\mathcal{E}\left( r\right)
\right) }{\mathrm{Fir}^{r}\mathrm{H}_{\mathrm{logdR}}^{i-1}\left( X,
\mathcal{E}\left( r\right) \right) }\longrightarrow F^{1}\mathrm{H}
^{i}\left( \mathbf{Syn}\left( \mathcal{E},\mathfrak{E},r\right)
\right) $$and 
$$\frac{\mathrm{H}_{\mathrm{logdR}}^{i-1}\left( X,\mathcal{E}(-D)\left( r\right)
\right) }{\mathrm{Fir}^{r}\mathrm{H}_{\mathrm{logdR}}^{i-1}\left( X,
\mathcal{E}(-D)\left( r\right) \right) }\longrightarrow F^{1}\mathrm{H}
^{i}\left( \mathbf{Syn}^{c-\infty }\left( \mathcal{E},\mathfrak{E},r\right)
\right) $$

\begin{proposition}\label{prop:HS}
The morphisms $\rho _{\mathrm{Syn}}^{i}\left( r\right) $
(resp.~$\rho _{\mathrm{Syn},c}^{i}\left( r\right) $) of Theorem \ref
{thm:comparesynt} map  $F^{1}\mathrm{H}^{i}\left( \mathbf{Syn}
\left( \mathcal{E},\mathfrak{E},r\right) \right) $ (respectively $F^{1}\mathrm{H}
^{i}\left( \mathbf{Syn}^{c-\infty }\left( \mathcal{E},\mathfrak{E},r\right)
\right) $) to $F^{1}\mathrm{H}_{\mathrm{et}}^i\left( U,\mathcal{L}
\left( r\right) \right) \left[ p^{-1}\right] $ (resp.~$F^{1}\mathrm{H}_{
\mathrm{et},c}^i\left( U,\mathcal{L}\left( r\right) \right) \left[
p^{-1}\right] $) and they make the following diagram commutative
$$\xymatrix{ F^{1}\mathrm{H}^{i}\left( \mathbf{Syn}\left( \mathcal{E},\mathfrak{E}
,r\right) \right) \ar[d]_-{\rho _{\mathrm{Syn}}^{i}\left( r\right)} & \frac{\mathrm{H}_{\mathrm{logdR}
}^{i-1}\left( X,\mathcal{E}\left( r\right) \right) }{\mathrm{Fir}^{r}\mathrm{
H}_{\mathrm{logdR}}^{i-1}\left( X,\mathcal{E}\left( r\right) \right) } \ar[l] \ar[r]^-{\cong} & \frac{D_{\mathrm{dR}}\left( \mathrm{H}_{\mathrm{et}
}^{i-1}\left( U_{\overline{K}},\mathcal{L}(r)\right) \right) }{\mathrm{Fil}
^{r}D_{\mathrm{dR}}\left( \mathrm{H}_{\mathrm{et}}^{i-1}\left( U_{
\overline{K}},\mathcal{L}(r)\right) \right) } \ar@/^{5pt}/[dl]^-{\mathrm{exp}_{r}} \\ F^{1}\mathrm{H}_{\mathrm{et}}^{i}\left( U,\mathcal{L}\left(
r\right) \right) \left[  p^{-1}\right] \ar[r]^-{\mathrm{HS}^{i}} & \mathrm{H}^{1}\left( G_{K},\mathrm{H}_{\mathrm{et
}}^{i-1}\left( U_{\overline{K}},\mathcal{L}(r)\right) \right) \left[ p^{-1}
\right] & }
$$(resp.~for the compactly supported versions). Here, we identify $D_{\mathrm{
dR}}\left( \mathrm{H}_{\mathrm{et}}^{i-1}\left( U_{\overline{K}},
\mathcal{L}(r)\right) \right) \simeq \mathrm{H}_{\mathrm{logdR}}^{i-1}\left( X,
\mathcal{E}\left( r\right) \right) $ (resp.~$D_{\mathrm{dR}}\left( \mathrm{H}
_{\mathrm{et},c}^{i-1}\left( U_{\overline{K}},\mathcal{L}(r)\right)
\right) \simeq \mathrm{H}_{\mathrm{logdR},c}^{i-1}\left( X,\mathcal{E}\left(
r\right) \right) $) via the filtered isomorphism of \cite[Thm. 3.2.7(3)]
{DLLZb} (resp.~\cite[Thm. 3.1.10]{LLZ}).
\end{proposition}
\begin{proof}  
We prove the statement for $\rho_{\mathrm{Syn}}^i(r) $. The one for $\rho_{\mathrm{Syn},c}^i(r) $ is proven similarly. By definition we have exact an exact sequence $$0 \longrightarrow F^1 \mathrm{H}^i\bigl(\mathbf{Syn}(\mathcal{L},r)\bigr)\longrightarrow \mathrm{H}^i\bigl(\mathbf{Syn}(\mathcal{L},r)\bigr) \longrightarrow F^0 \mathrm{H}^i\bigl(\mathbf{Syn}(\mathcal{L},r)\bigr) \longrightarrow 0$$with   $F^1 \mathrm{H}^i\bigl(\mathbf{Syn}(\mathcal{L},r)\bigr)$ the cokernel of the map $(1-\frac{\varphi}{p^r}) \oplus   \iota$ $$ \mathrm{H}^{i-1}_{\mathrm{pket}}\bigl( \widehat{\mathcal{L}}(r)\otimes_{\widehat{\mathbb{Z}}_p} \mathbb{B}_{\mathrm{crys}}  \bigr) \to \mathrm{H}^{i-1}_{\mathrm{pket}}\bigl(\widehat{\mathcal{L}}(r)\otimes_{\widehat{\mathbb{Z}}_p} \mathbb{B}_{\mathrm{crys}} \bigr) \oplus \mathrm{H}^{i-1}_{\mathrm{pket}}\bigl( \widehat{\mathcal{L}}(r)\otimes_{\widehat{\mathbb{Z}}_p}\mathrm{Gr}^r \mathbb{B}_{\mathrm{dR}} \bigr).$$Here $\mathrm{H}^{i-1}_{\mathrm{pket}}$ stands for $\mathrm{H}^{i-1}\bigl(X_{\mathrm{pket}},\_\bigr) $.
It maps to the cokernel of  the morphism $(1-\frac{\varphi}{p^r}) \oplus   \iota$:

\begin{equation}\label{eq:gpkt}
\mathrm{H}^{i-1}_{\mathrm{gpkt}}\bigl( \widehat{\mathcal{L}}(r)\otimes_{\widehat{\mathbb{Z}}_p} \mathbb{B}_{\mathrm{crys}}  \bigr) \longrightarrow \mathrm{H}^{i-1}_{\mathrm{gpkt}}\bigl(\widehat{\mathcal{L}}(r)\otimes_{\widehat{\mathbb{Z}}_p} \mathbb{B}_{\mathrm{crys}} \bigr) \oplus \mathrm{H}^{i-1}_{\mathrm{gpkt}}\bigl( \widehat{\mathcal{L}}(r)\otimes_{\widehat{\mathbb{Z}}_p}\mathrm{Gr}^r \mathbb{B}_{\mathrm{dR}} \bigr),
\end{equation} 
where $\mathrm{H}^{i-1}_{ \mathrm{gpkt}}$ stands for $\mathrm{H}^{i-1}\bigl(X_{\overline{K},\mathrm{pket}},\_\bigr) $, which is $0$.  In particular, via  the isomorphism  $\mathrm{H}^i\bigl(\mathbf{Syn}(\mathcal{L},r)\bigr)\cong\mathrm{H}_{\mathrm{et}}^{i}\left( U,\mathcal{L}\left(
r\right) \right)[p^{-1}] $ the subspace $F^1\mathrm{H}^i\bigl(\mathbf{Syn}(\mathcal{L},r)\bigr)$ maps to the kernel $F^1 \mathrm{H}_{\mathrm{et}}^{i}\left( U,\mathcal{L}\left(
r\right) \right)[p^{-1}] $ of the map  $$\mathrm{H}_{\mathrm{et}}^{i}\left( U,\mathcal{L}\left(
r\right) \right) [p^{-1}]\longrightarrow \mathrm{H}_{\mathrm{et}}^{i}\left( U_{\overline{K}},\mathcal{L}\left(
r\right) \right)[p^{-1}]  ( \cong \mathrm{H}^i_{\mathrm{gpkt}}\bigl( \widehat{\mathcal{L}}(r)\bigr)[p^{-1}]).$$By construction  $\rho_{\mathrm{Syn}}^i(r) $ maps $F^1 \mathrm{H}^i\bigl(\mathbf{Syn}(\mathcal{E},\mathfrak{E},r)\bigr)$ to $F^1 \mathrm{H}^i\bigl(\mathbf{Syn}(\mathcal{L},r)\bigr)$ and the first assertion of the Proposition  is proven.

In fact,  $F^1 \mathrm{H}^i\bigl(\mathbf{Syn}(\mathcal{L},r)\bigr)$ maps to the cokernel of  the morphism $(1-\frac{\varphi}{p^r}) \oplus   \iota$:
 $$\mathrm{H}^{i-1}_{\mathrm{gpkt}}\bigl( \widehat{\mathcal{L}}(r)\otimes_{\widehat{\mathbb{Z}}_p} \mathbb{B}_{\mathrm{crys}}  \bigr)^{G_K} \to \mathrm{H}^{i-1}_{\mathrm{gpkt}}\bigl(\widehat{\mathcal{L}}(r)\otimes_{\widehat{\mathbb{Z}}_p} \mathbb{B}_{\mathrm{crys}} \bigr)^{G_K} \oplus \mathrm{H}^{i-1}_{\mathrm{gpkt}}\bigl( \widehat{\mathcal{L}}(r)\otimes_{\widehat{\mathbb{Z}}_p}\mathrm{Gr}^r \mathbb{B}_{\mathrm{dR}} \bigr)^{G_K}.$$Thanks to Proposition \ref{prop:etcoh} this complex is identified with the complex
$(1-\frac{\varphi}{p^r}) \oplus   \iota$: $$D_{\mathrm{crys}}\bigl(\mathrm{H}^{i-1}_{\mathrm{et}}\bigl(U_{\overline{K}},  \mathcal{L}(r)\bigr)\bigr)  \longrightarrow D_{\mathrm{crys}}\bigl(\mathrm{H}^{i-1}_{\mathrm{et}}\bigl(U_{\overline{K}}, \mathcal{L}(r)\bigr)\bigr)\oplus \mathrm{Gr}^r D_{\mathrm{dR}}\bigl(\mathrm{H}^{i-1}_{\mathrm{et}}\bigl(U_{\overline{K}}, \mathcal{L}(r)\bigr)\bigr).$$Its cokernel $\mathbb{K}$ maps to $\mathrm{H}^{1}\left( G_{K},\mathrm{H}_{\mathrm{et
}}^{i-1}\left( U_{\overline{K}},\mathcal{L}(r)\right) \right) \left[ p^{-1}
\right]$ using the fundamental exact sequence and, in particular,  the restriction to $\mathrm{Gr}^r D_{\mathrm{dR}}\bigl(\mathrm{H}^{i-1}_{\mathrm{et}}\bigl(U_{\overline{K}}, \mathcal{L}(r)\bigr)\bigr)$ is the exponential map $\exp_r$. Using the spectral sequence $\mathrm{H}^j\bigl(G_K, \mathrm{H}^i_{\mathrm{gpkt}}(\_)\bigr) \Longrightarrow \mathrm{H}^{i-1}_{\mathrm{pket}}\bigl(\_\bigr)$ for the short exact sequence that gives $ \mathrm{H}^{i-1}_{\mathrm{gpkt}}\bigl( \widehat{\mathcal{L}}(r)\bigr)$ as the kernel of the surjective map (\ref{eq:gpkt}), we conclude that the composite map $$F^1 \mathrm{H}^i\bigl(\mathbf{Syn}(\mathcal{L},r)\bigr)\longrightarrow \mathbb{K} \longrightarrow   \mathrm{H}^{1}\left( G_{K},\mathrm{H}_{\mathrm{et
}}^{i-1}\left( U_{\overline{K}},\mathcal{L}(r)\right) \right) \left[ p^{-1}
\right]$$is the composite of  the map $F^1 \mathrm{H}^i\bigl(\mathbf{Syn}(\mathcal{L},r)\bigr)\to F^1 \mathrm{H}_{\mathrm{et}}^{i}\left( U,\mathcal{L}\left(
r\right) \right)[p^{-1}] $, induced by $\rho_{\mathrm{Syn}}^i(r) $,  and the map $\mathrm{HS}^{i}$. 
Finally, we remark that the composite $$\mathrm{Gr}^r \mathrm{H}_{\mathrm{logdR}}^{i-1}\left( X,\mathcal{E}\left( r\right)
\right) \longrightarrow F^{1}\mathrm{H}
^{i}\left( \mathbf{Syn}\left( \mathcal{E},\mathfrak{E},r\right)
\right)\longrightarrow \mathbb{K}$$coincides with the composite of the natural map $\mathrm{Gr}^r D_{\mathrm{dR}}\bigl(\mathrm{H}^{i-1}_{\mathrm{et}}\bigl(U_{\overline{K}}, \mathcal{L}(r)\bigr)\bigr)\to \mathbb{K}$ and the map  $\mathrm{Gr}^r \mathrm{H}_{\mathrm{logdR}}^{i-1}\left( X,\mathcal{E}\left( r\right)
\right) \to  \mathrm{Gr}^r D_{\mathrm{dR}}\bigl(\mathrm{H}^{i-1}_{\mathrm{et}}\bigl(U_{\overline{K}}, \mathcal{L}(r)\bigr)\bigr)$, induced by the map of complexes underlying the definition of $\rho_{\mathrm{Syn}}^i(r) $. The latter coincides with the isomorphism of \cite[Thm. 3.2.7(3)]{DLLZb}; see the discussion in \S \ref{sec:pfCorDst}. This concludes the proof.

\end{proof}

\section{Gysin morphisms for divisors}\label{sec:Gysin}

We fix a smooth divisor $F\subset X$ which is the analytic generic fiber of a Cartier divisor $\mathcal{F}\subset \mathcal{X}$ such that $F\cup D$ and $\mathcal{F}\cup \mathcal{D}$ are strict normal crossing divisors (meaning that the closed immersion is defined by a non-zero divisor at the level of formal schemes). 

In the notation of \S \ref{sec:Chech} we also take local charts so that that for every $\beta\in B$ we have lift $\mathfrak{F}_\beta \subset \mathfrak{X}_\beta$ such that $\mathfrak{F}_\beta\cup \mathfrak{D}_\beta$ is also a strict normal crossin divisor.
We write $X'$, $\mathcal{X}'$ and $\mathfrak{X}'_\beta$ for $X$, $\mathcal{X}$ and $\mathfrak{X}_\beta$ with log-structures defined by $F\cup D$,  $\mathcal{F}\cup \mathcal{D}$, $\mathfrak{F}_\beta\cup \mathfrak{D}_\beta$. We write $\mathbf{Syn}'(\mathcal{E},\mathfrak{E},r)$ for the syntomic complex associated to these log-structures. We then have an inclusion $\mathbf{Syn}(\mathcal{E},\mathfrak{E},r)\subset \mathbf{Syn}'(\mathcal{E},\mathfrak{E},r)$. 

\begin{lemma}\label{lemma:exactsyn} The cokernel of $\mathbf{Syn}(\mathcal{E},\mathfrak{E},r)\subset \mathbf{Syn}'(\mathcal{E},\mathfrak{E},r)$ is identified with the complex $\mathbf{Syn}(\mathcal{E}\vert_F,\mathfrak{E}\vert_{\mathcal{\overline{F}}},r-1)[-1]$ where we consider on $F$, $\mathfrak{F}$, $\mathfrak{F}_\beta$ the pull-back of the Cartier divisors $D\cap F$, $\mathcal{D}\cap \mathcal{F}$, $\mathfrak{D}_\beta\cap \mathfrak{F}_\beta$. 
\end{lemma}
\begin{proof}  We use the explicit representatives of these complexes of \S \ref{sec:dR}  and \S \ref{sec:analsyn}. For each $\underline{\beta} \in B^n$such that $\beta_1<\ldots < \beta_n$ we have morphisms of complexes
$$\mathrm{dR}'(\mathcal{E})_{\underline{\beta}} \longrightarrow \mathrm{dR}(\mathcal{E}\vert_{\mathcal{F}})_{\underline{\beta}}[-1], \qquad \mathrm{dR}'(\mathfrak{E}^{\rm conv})_{\underline{\beta}}\longrightarrow \mathrm{dR}(\mathfrak{E}^{\rm conv}\vert_{\mathfrak{F}_{\underline{\beta}}})_{\underline{\beta}}[-1],$$where $\_'$ stands for the logarithmic de Rham complexes for the log-structures defined by  $\mathcal{F}\cup \mathcal{D}$, $\mathfrak{F}_{\underline{\beta}}\cup \mathfrak{D}_{\underline{\beta}}$ and the morphisms are given by taking residues along $\mathcal{F}_{\underline{\beta}}$ and $\mathfrak{F}_{\underline{\beta}}$ respectively (see the proof of Proposition \ref{prop:GysDual} below for details). Moreover, the map on de Rham complexes sends $\mathrm{Fil}^r \mathrm{dR}'(\mathcal{E})_{\underline{\beta}}$ to $\mathrm{Fil}^{r-1}\mathrm{dR}(\mathcal{E}\vert_{\mathcal{F}})_{\underline{\beta}}[-1]$; compare with the proof of Proposition \ref{prop:samecohogroups}. Their kernel are $\mathrm{dR}(\mathcal{E})_{\underline{\beta}}$, resp.~$\mathrm{dR}(\mathfrak{E}^{\rm conv})_{\underline{\beta}}$. 
These morphisms are compatible for varying $\underline{\beta}$ and  provide a map $$\mathbf{Syn}'(\mathcal{E},\mathfrak{E},r)\longrightarrow \mathbf{Syn}(\mathcal{E}\vert_F,\mathfrak{E}\vert_{\mathcal{\overline{F}}},r-1)[-1]$$whose kernel is $\mathbf{Syn}(\mathcal{E},\mathfrak{E},r)$
\end{proof}

Taking cohomology groups we get a map:
$$\mathrm{Gys}_{\mathrm{syn},F}^i\colon \mathrm{H}^i\bigl(\mathbf{Syn}(\mathcal{E}\vert_F,\mathfrak{E}\vert_{\mathcal{\overline{F}}},r-1) \bigr) \longrightarrow \mathrm{H}^{i+2}\bigl(\mathbf{Syn}(\mathcal{E},\mathfrak{E},r)\bigr).$$By construction it is compatible with Gysin maps in de Rham cohomology 
$$ \mathrm{Gys}_{\mathrm{dR},F}^i\colon \mathrm{H}^{i}_{\mathrm{dR}}\bigl(F,\mathcal{E}\vert_F\bigr)\longrightarrow  \mathrm{H}^{i+2}_{\mathrm{dR}}(X,\mathcal{E})$$and analytic cohomology 
$$ \mathrm{Gys}_{\mathrm{an},F}^{i,r}\colon \mathrm{H}^i_{\mathrm{an}}(\mathfrak{E}\vert_{\mathcal{\overline{F}}},r-1) \longrightarrow \mathrm{H}^{i+2}_{\mathrm{an}}(\mathfrak{E},r).$$Let $V:=U\backslash F$ and let $J_V\colon V\to U$ be the inclusion. By \cite[Thm. 7.2]{Illusie} we have an exact  sequence $$ 0 \to \mathcal{L}(r) \longrightarrow j_{V,\mathrm{et},\ast} j_{V,\mathrm{et}}^\ast(\mathcal{L}(r)) \longrightarrow j_{F\cap U,\mathrm{et},\ast} j_{F\cap U,\mathrm{et}}^\ast (\mathcal{L})(r-1)[-1] \to 0.$$Taking long exact sequence in cohomology gives  Gysin maps in \'etale cohomology: 
$$ \mathrm{Gys}_{\mathrm{et},F}^i\colon \mathrm{H}^{i}\bigl(U\cap F_{\mathrm{et}}, j_{F\cap U,\mathrm{et}}^\ast (\mathcal{L})(r-1)\bigr) \longrightarrow \mathrm{H}^{i+2}(U_{\mathrm{et}},\mathcal{L}(r)).$$

\begin{theorem}\label{thm:Gysinandsyntcomp} The syntomic and \'etale Gysin maps are compatible with the syntomic-\'etale comparison maps, namely we have $$\rho_{\mathrm{Syn}}^{i+2}(r) \circ \mathrm{Gys}_{\mathrm{syn},F}^i= \mathrm{Gys}_{\mathrm{et},F}^i \circ \rho_{\mathrm{Syn},F}^{i}(r-1). $$

\end{theorem}
\begin{proof} One proceeds as in the proof of \cite[Prop. 4.3.17]{LLZ}. Let $\epsilon\colon X'\to X$ and $\iota\colon F\to X$ be the natural morphisms of log-adic spaces.  First of all one upgrades the exact sequence on \'etale sites to an exact sequence on $X_{\mathrm{pket}}$, using \cite[Lemma 4.5.7 \& Lemma 4.5.8]
{DLLZa}:

$$0\longrightarrow \widehat{\mathcal{L}} \longrightarrow R\varepsilon_{\mathrm{pket},\ast} \varepsilon_{\mathrm{pket}}^\ast\bigl(\widehat{\mathcal{L}}(r) \bigr) \longrightarrow  \iota_{\mathrm{pket},\ast} \iota_{\mathrm{pket}}^\ast\bigl(\widehat{\mathcal{L}}(r-1) \bigr) [-1] \longrightarrow 0.$$
Tensoring this with $\mathbb{B}_{\mathrm{crys}}$ and with $\mathbb{B}_{\mathrm{dR}}/\mathrm{Fil}^r \mathbb{B}_{\mathrm{dR}}$ one gets exact sequences that can be identified, see loc.~cit., with exact sequence:
 
$$0\longrightarrow \widehat{\mathcal{L}}\otimes_{\widehat{\mathbb{Z}}_p}  \mathbb{B}_{\mathrm{crys},X} \longrightarrow R\varepsilon_{\mathrm{pket},\ast}\left( \varepsilon_{\mathrm{pket}}^\ast\bigl(\widehat{\mathcal{L}}\bigr)\otimes_{\widehat{\mathbb{Z}}_p}  \mathbb{B}_{\mathrm{crys},X'}\right) \longrightarrow $$ $$\longrightarrow \iota_{\mathrm{pket},\ast}\left( \iota_{\mathrm{pket}}^\ast\bigl(\widehat{\mathcal{L}}\bigr)\otimes_{\widehat{\mathbb{Z}}_p}  \mathbb{B}_{\mathrm{crys},F}\right)[-1] \longrightarrow 0$$and
$$0\longrightarrow \widehat{\mathcal{L}}\otimes_{\widehat{\mathbb{Z}}_p}\frac{\mathbb{B}_{\mathrm{dR},X}}{\mathrm{Fil}^r\mathbb{B}_{\mathrm{dR},X} } \longrightarrow R\varepsilon_{\mathrm{pket},\ast}\left( \varepsilon_{\mathrm{pket}}^\ast\bigl(\widehat{\mathcal{L}}\bigr)\otimes_{\widehat{\mathbb{Z}}_p}  \frac{\mathbb{B}_{\mathrm{dR},X'}}{\mathrm{Fil}^r\mathbb{B}_{\mathrm{dR},X'} }\right) \longrightarrow $$ $$\longrightarrow \iota_{\mathrm{pket},\ast}\left( \iota_{\mathrm{pket}}^\ast\bigl(\widehat{\mathcal{L}}\bigr)\otimes_{\widehat{\mathbb{Z}}_p} \frac{\mathbb{B}_{\mathrm{dR},F}}{\mathrm{Fil}^r\mathbb{B}_{\mathrm{dR},F} }\right)[-1] \longrightarrow 0.$$These are compatible with the fundamental  exact sequences  of Lemma \ref{lemma:fundexactseq}:
$$
0 \to \widehat{\mathcal{L}}(r)  \longrightarrow  \widehat{\mathcal{L}}\otimes_{\widehat{\mathbb{Z}}_p}  \mathbb{B}_{\mathrm{crys},X}  {\longrightarrow} \widehat{\mathcal{L}}\otimes_{\widehat{\mathbb{Z}}_p}  \mathbb{B}_{\mathrm{crys},X}  \oplus \widehat{\mathcal{L}}\otimes_{\widehat{\mathbb{Z}}_p} \frac{\mathbb{B}_{\mathrm{dR},X}}{\mathrm{Fil}^r\mathbb{B}_{\mathrm{dR},X} } \longrightarrow 0,$$

$$0 \to R\varepsilon_{\mathrm{pket},\ast} \varepsilon_{\mathrm{pket}}^\ast\bigl(\widehat{\mathcal{L}}(r) \bigr)  \longrightarrow R\varepsilon_{\mathrm{pket},\ast}\left( \varepsilon_{\mathrm{pket}}^\ast\bigl(\widehat{\mathcal{L}}\bigr)\otimes_{\widehat{\mathbb{Z}}_p}  \mathbb{B}_{\mathrm{crys},X'}\right) \longrightarrow $$ $$ R\varepsilon_{\mathrm{pket},\ast}\left( \varepsilon_{\mathrm{pket}}^\ast\bigl(\widehat{\mathcal{L}}\bigr)\otimes_{\widehat{\mathbb{Z}}_p}  \mathbb{B}_{\mathrm{crys},X'}\right) \oplus R\varepsilon_{\mathrm{pket},\ast}\left(\varepsilon_{\mathrm{pket}}^\ast( \widehat{\mathcal{L}})\otimes_{\widehat{\mathbb{Z}}_p} \frac{\mathbb{B}_{\mathrm{dR},X'}}{\mathrm{Fil}^r\mathbb{B}_{\mathrm{dR},X'} }\right) \longrightarrow 0$$and 
$$ 0 \to \iota_{\mathrm{pket},\ast} \iota_{\mathrm{pket}}^\ast\bigl(\widehat{\mathcal{L}}(r) \bigr)  \longrightarrow  \iota_{\mathrm{pket},\ast}\left( \iota_{\mathrm{pket}}^\ast\bigl(\widehat{\mathcal{L}}\bigr)\otimes_{\widehat{\mathbb{Z}}_p}  \mathbb{B}_{\mathrm{crys},F}\right)\longrightarrow   $$ $$ {\longrightarrow} \iota_{\mathrm{pket},\ast}\left( \iota_{\mathrm{pket}}^\ast\bigl(\widehat{\mathcal{L}}\bigr)\otimes_{\widehat{\mathbb{Z}}_p}  \mathbb{B}_{\mathrm{crys},F}\right)\oplus \iota_{\mathrm{pket},\ast}\left(\iota_{\mathrm{pket}}^\ast( \widehat{\mathcal{L}})\otimes_{\widehat{\mathbb{Z}}_p}\frac{\mathbb{B}_{\mathrm{dR},F}}{\mathrm{Fil}^r\mathbb{B}_{\mathrm{dR},F} }\right)  \longrightarrow 0.$$It follows from the proof of \cite[Prop. 4.3.17]{LLZ} that this  produces a distinguished triangle:
$$0\longrightarrow \mathbf{Syn}(\mathcal{L},r) \longrightarrow \mathbf{Syn}\bigl(\varepsilon^\ast(\mathcal{L}),r\bigr) \longrightarrow \iota_\ast  \mathbf{Syn}\bigl(\iota^\ast(\mathcal{L}),r-1\bigr)[-1]  \longrightarrow 0  $$
upon taking $R\Gamma(X_{\mathrm{pket}},\_)$ as in \S \ref{sec:synL}. In loc.~cit.~a local description in terms of the Galois cohomology of a generator of inertia at $F$ for a basis of log affinoid perfectoid opens of $X_{\mathrm{pket}}$ is given, from which the right exactness of the sequence above follows. We are left to show that this triangle is compatible with  the exact sequence $$0\longrightarrow \mathbf{Syn}(\mathcal{E},\mathfrak{E},r)\longrightarrow  \mathbf{Syn}'(\mathcal{E},\mathfrak{E},r)\longrightarrow \mathbf{Syn}(\mathcal{E}\vert_F,\mathfrak{E}\vert_{\mathcal{\overline{F}}},r-1)[-1]\longrightarrow 0$$via the morphisms  $\mathbf{Syn}(\mathcal{E},\mathfrak{E},r)\to   \mathbf{Syn}(\mathcal{L},r)$ etc. of Proposition \ref{prop:comparecomplexes}. The compatibility of the first arrow is clear by functoriality as it is induced by a morphism of log adic spaces. We are left to show that the map induced on the quotients coincides with the one obtained from Proposition \ref{prop:comparecomplexes} applied on $F$. It can be checked locally using the charts of \S \ref{sec:Chech}. Using Theorem \ref{thm:OBcrys}, for each $\underline{\beta}\in B^n$ such that $\beta_1<\ldots < \beta_n$,  we can also replace  $\widehat{\mathcal{L}}\otimes_{\widehat{\mathbb{Z}}_p}  \mathbb{B}_{\mathrm{crys}}$ with $\mathfrak{E}\otimes_{\mathcal{O}} \mathcal{O}\mathbb{B}_{\mathrm{crys}}\otimes_{\mathcal{O}} \Omega^{\log,\bullet}$ on $\mathfrak{X}_{\underline{\beta}}$, $\mathfrak{X}'_{\underline{\beta}}$ and $\mathfrak{F}_{\underline{\beta}}$ and  $\widehat{\mathcal{L}}\otimes_{\widehat{\mathbb{Z}}_p}  \mathbb{B}_{\mathrm{dR}}$ with $\mathcal{E}\otimes_{\mathcal{O}} \mathcal{O}\mathbb{B}_{\mathrm{crys}}\otimes_{\mathcal{O}} \Omega^{\log,\bullet}$ on $X_{\underline{\beta}}$, $X'_{\underline{\beta}}$ and $F_{\underline{\beta}}$.
It suffices to show separately compatibilities of the associated short exact sequences: \smallskip

(1) for the analytic cohomology of  $\mathfrak{E}^{\rm conv}$ and the cohomology of $\mathcal{E}\otimes_{\mathcal{O}} \mathcal{O}\mathbb{B}_{\mathrm{crys}}\otimes_{\mathcal{O}} \Omega^{\log,\bullet} $ on $\mathfrak{X}_{\underline{\beta}}$, $\mathfrak{X}'_{\underline{\beta}}$ and $\mathfrak{F}_{\underline{\beta}}$ (the analytic part);

\smallskip
(2) for the log de Rham cohomology of 
$\mathcal{E}$ and the cohomology of $\mathcal{E}\otimes_{\mathcal{O}} \mathcal{O}\mathbb{B}_{\mathrm{dR}}\otimes_{\mathcal{O}} \Omega^{\log,\bullet}$ (the de Rham part) on $X_{\underline{\beta}}$, $X'_{\underline{\beta}}$ and $F_{\underline{\beta}}$. 

\smallskip
We can also assume that $\bigl(\mathcal{E},\mathfrak{E}\bigr)$ are free modules possibly after shrinking the charts  and then the statement reduces to the case that $\bigl(\mathcal{E},\mathfrak{E}\bigr)=(\mathcal{O}_{\mathcal{X}},\mathcal{O}_{\mathfrak{X}_{\underline{\beta}}})$ using that $\bigl(\mathcal{E},\mathfrak{E}\bigr)$ is crystalline associated to $\mathcal{L}$ (see Definition \ref{def:crysass}).  In this case the compatibility in the de Rham case, and with a similar argument in the crystalline case, is shown in the proof of \cite[Prop. 4.3.17]{LLZ}. 
\end{proof}

We also have the following variants for compact support, that will provide an adjunction property with respect to Poincar\'e pairing. 

We write $\mathbf{Syn}^{c-\infty,'}(\mathcal{E},\mathfrak{E},r)$ for the complex of $\mathbf{Syn}'(\mathcal{E}(-D),\mathfrak{E}(-\mathfrak{D}),r)$ as in \S \ref{sec:compactsupport}. As in Lemma \ref{lemma:exactsyn} one proves that it sits in an exact sequence of complexes 

$$0\to \mathbf{Syn}^{c-\infty}(\mathcal{E},\mathfrak{E},r) \to\mathbf{Syn}^{c-\infty,'}(\mathcal{E},\mathfrak{E},r) \to\mathbf{Syn}^{c-\infty}(\mathcal{E}\vert_{\mathcal{F}},\mathfrak{E}\vert_{\mathfrak{F}},r-1)[-1]  \to 0$$providing a map:
$$\mathrm{Gys}_{\mathrm{syn}^{c-\infty},F}^i\colon \mathrm{H}^i\bigl(\mathbf{Syn}^{c-\infty}(\mathcal{E}\vert_F,\mathfrak{E}\vert_{\mathcal{\overline{F}}},r-1) \bigr) \longrightarrow \mathrm{H}^{i+2}\bigl(\mathbf{Syn}^{c-\infty}(\mathcal{E},\mathfrak{E},r)\bigr),$$compatible with Gysin maps in de Rham and analytic cohomology.

\section{Syntomic polynomial complexes}\label{S PolyCmpl}

Assume that $\mathbb{E}:=\left( \mathcal{E},\mathrm{Fil}^{\bullet }\mathcal{E
},\nabla ,\mathfrak{E},\Phi _{\mathfrak{E}}\right) $ is a non-degenerate
Frobenius log-crystal relative to $\left( X,\alpha _{X},\overline{\mathcal{X}
},\alpha _{\overline{\mathcal{X}}}\right) $, as in \S \ref{sec:FrobIso}.

\begin{definition}\label{def:synpol}
Let $\mathbb{Q}_p\subset L$ be a field extension and let $P(X)\in 1+XL[X] $ be a polynomial. Define the syntomic complex  $\mathbf{Syn}(\mathcal{E},\mathfrak{E},P,r)$ to be the total complex associated to the double complexes $\bigl(P(\Phi),\gamma_{\rm an}\bigr)$:

$$\mathrm{R}\Gamma_{\mathrm{an}}\bigl((\mathcal{X}_k,\alpha_k)/\mathcal{S}_0,\mathfrak{E}^{\rm conv}\bigr)_L \longrightarrow \mathrm{R}\Gamma_{\mathrm{an}}\bigl((\mathcal{X}_k,\alpha_k)/\mathcal{S}_0,\mathfrak{E}^{\rm conv}\bigr)_L \oplus \frac{ \mathrm{R}\Gamma_{\mathrm{logdR}}\bigl(X, \mathcal{E}\bigr)_L }{\mathrm{Fil}^r \mathrm{R}\Gamma_{\mathrm{logdR}}\bigl(X, \mathcal{E}\bigr)_L}$$
\end{definition}

Here, the subscript $L$ stands for the extension of scalars $\otimes_{\mathbb{Q}_p} L$ and Frobenius $\Phi$ is extended $L$-linearly. Notice that the complex $\mathbf{Syn}(\mathcal{E},\mathfrak{E},r)$ of  Definition \ref{def:syntomiccomplex} corresponds to $L=\mathbb{Q}_p$ and $P(X)=1- \frac{X}{p^r}$.

Proceeding similarly we have variants, namely the one with support conditions  $\mathbf{Syn}^{c-\infty}(\mathcal{E},\mathfrak{E},P,r)$ as in Definition \ref{def:syntomiccompactcomplex} and, as in \S  \ref{sec:rigsyn} and in Definition \ref{def:syntomiccompactonU}  for $U\subset \mathcal{X}_k^{\rm sm}$ open and $Y\subset \mathcal{X}_k$ closed and reduced and containing $U$, the rigid syntomic complexes with or without support conditions, namely $\mathbf{Syn}_{{\rm rig},U\subset Y, \alpha_k}\bigl(\mathcal{E},\mathfrak{E},P,r\bigr)$  and $\mathbf{Syn}_{{\rm rig},U^o\subset Y, \alpha_k}^c(\mathcal{E},\mathfrak{E},P,r)$  with support on $U^o=U\backslash \mathcal{D}_k$, as well as those of \S \ref{sec:Egoodred}. We simply write 
$$ 
K_{P,r}=K_{P,r}\left( \mathbb{E}\right) :=\left[ P\left( \Phi
_{L}\right) \oplus \gamma \colon \mathrm{R}\Gamma _{L}\longrightarrow \Gamma
_{L}\oplus \frac{\mathrm{R}\Gamma _{\mathrm{dR},L}}{\mathrm{Fir}^{r}\left(
\mathrm{R}\Gamma _{\mathrm{dR},L}\right) }\right],
$$
to denote one of these complexes and the corresponding double complexes.  An analogue of Proposition \ref{prop:spectralanal} follows formally for the cohomology groups $\mathrm{H}^i(K_{P,r})$. Namely, we have an exact sequence $$0\longrightarrow F^1 \mathrm{H}^i(K_{P,r}) \longrightarrow  \mathrm{H}^i(K_{P,r}) \longrightarrow F^0\mathrm{H}^i(K_{P,r}) \longrightarrow 0$$with $ F^1 \mathrm{H}^i(K_{P,r})$ the cokernel of the map $P\left( \Phi
_{L}\right) \oplus \gamma$  
$$\mathrm{H}^{i-1}\bigl(\mathrm{R}\Gamma _{L}\bigr)\longrightarrow \mathrm{H}^{i-1}\bigl(\mathrm{R}\Gamma _{L}\bigr) \oplus \mathrm{H}^{i-1}\left(\mathrm{R}\Gamma _{\mathrm{dR},L}/\mathrm{Fir}^{r}\left(
\mathrm{R}\Gamma _{\mathrm{dR},L}\right)\right)$$and  $ F^0 \mathrm{H}^i(K_{P,r})$ the kernel of the map $P\left( \Phi
_{L}\right) \oplus \gamma$  $$\mathrm{H}^i\bigl(\mathrm{R}\Gamma _{L}\bigr)\longrightarrow \mathrm{H}^i\bigl(\mathrm{R}\Gamma _{L}\bigr) \oplus \mathrm{H}^i\left(\mathrm{R}\Gamma _{\mathrm{dR},L}/\mathrm{Fir}^{r}\left(
\mathrm{R}\Gamma _{\mathrm{dR},L}\right)\right).$$We have the following compatibilities for different choices of $P(X)$ and $r$.

\begin{remark}
\label{rmk:Poly R2} Suppose that $P\mid Q$, say $Q=PP^{\prime }$, and that $r\geq
s$. Define the morphism of double complexes $K_{P,r}\longrightarrow K_{Q,s}$, sending $u\in \mathrm{R}\Gamma _{L}$ to itself and $\left( x,y\right) \in \mathrm{R} \Gamma
_{L}\oplus \frac{\mathrm{R}\Gamma _{\mathrm{dR},L}}{\mathrm{Fir}^{r}\left(
\mathrm{R}\Gamma _{\mathrm{dR},L}\right) } $ to $\left( P^{\prime }\left(
\varphi \right) x,y\right) \in \mathrm{R} \Gamma
_{L}\oplus \frac{\mathrm{R}\Gamma _{\mathrm{dR},L}}{\mathrm{Fir}^{r}\left(
\mathrm{R}\Gamma _{\mathrm{dR},L}\right) }$. 
\end{remark}

The advantage of working with  general polynomials is the following:

\begin{lemma}\label{lemma:Poly R1} Assume that $P(\Phi) $ is an isomorphism on $\mathrm{H}^{i-1}\bigl(\mathrm{R}\Gamma _{L}\bigr)$. Then, we have $F^0\mathrm{H}^{i-1}(K_{P,r}) =0$ and the natural map  $$\mathrm{H}^{i-1}\left(\mathrm{R}\Gamma _{\mathrm{dR},L}/\mathrm{Fir}^{r}\left(
\mathrm{R}\Gamma _{\mathrm{dR},L}\right)\right)\longrightarrow  F^1 \mathrm{H}^i(K_{P,r})$$is an isomorphism with inverse 
$$F^1 \mathrm{H}^i(K_{P,r})\longrightarrow \mathrm{H}^{i-1}\left(\mathrm{R}\Gamma _{\mathrm{dR},L}/\mathrm{Fir}^{r}\left(
\mathrm{R}\Gamma _{\mathrm{dR},L}\right)\right), \quad [(x,y)] \mapsto y-\gamma \left( P\left(\Phi\right)^{-1}x\right).$$
\end{lemma}

\begin{proposition}\label{prop:Poly P finiteness} Assume that $\mathcal{X}$ is proper of relative dimension $d$. Then for $K_{P,r}=\mathbf{Syn}(\mathcal{E},\mathfrak{E},P, r)$ or $K_{P,r}=\mathbf{Syn}^{c-\infty}(\mathcal{E},\mathfrak{E},P,r)$ the group $\mathrm{ H}^{i}\left( K_{P,r}\right) $ is finite dimensional as an $L$-vector space  for every   $i$ and it is $0$ for $i>2d+2$.
\end{proposition}

\begin{proof}
The finite dimensionality follows from  the exact sequences in Proposition \ref{prop:spectralanal} and their analogue in the case of support, Proposition \ref{prop:anabs}, Proposition \ref{prop:anabscompact} and the finite dimensionality of de Rham cohomology and Hyodo-Kato cohomology. Since both these groups vanish in degree $> 2d$, also the vanishing follows. 
\end{proof}

%RIGID CASE?????? Thanks to the exact sequences generalizing Proposition \ref {prop:spectralanal}, the finiteness in the rigid case\ is a consequence of the finiteness of the de Rham cohomology and the cohomology groups \textrm{H}% $^{i-1}\left( \Gamma _{L}^{?}\right) $, which follow from our propernessassumption, taking into account Lemma \ref{lemma:rigBer} and Lemma \ref%{lemma:crigBer}. Because, in this case, e also have \textrm{H}$^{i}\left(\Gamma _{L}^{?}\right) =0$ for $i\geq 2d+1$ (and \textrm{H}$^{i}\left(\Gamma _{L}^{?}\right) =0$ for $i\geq 2d$ in the Berthelot rigid case), thevanishing of the polynomial cohomology follows. (TODO: GIVE REFERENCE FORFINITENESS OF BERTHELOT RIGID COHOMOLOGY; IF PROBLEMS WITH COEFFICIENTS, IN THE RIGID CASE, STATE THE RESULT ONLY FOR TRIVIAL COEFFICIENTS).\end{proof}

\section{Trace maps}\label{sec:trace}

Assume $\mathcal{X}$ is proper and geometrically connected of relative dimension $d$.
%\color{magenta}(see Remark \ref{CupProd RmkDef3} below
%for a slight generalization)\color{black}. 
Justified by this properness assumption, we will write $\mathbf{Syn}^{c}$ rather than $\mathbf{Syn}^{c\text{-}\infty }$ when considering the analytic syntomic cohomology. Let $f=[K_0:\mathbb{Q}_p]$ be the degree of the maximal unramified extension of $\mathbb{Q}_p$ in $K$. Assume that the $f$-roots of unity $\mu_f$ are in $L$. 

Let $\mathbb{I}$ be the unit object $\left( \mathcal{O},\mathfrak{O}\right)$. The aim of this section is to define a $K_{L}$-valued trace map on $\mathrm{H}^{2d+1}\left( \mathbf{Syn}^{c}\left( 
\mathbb{I},P,r\right) \right)$, assuming that the pair $(P,r)$ is admissible for $\mathcal{X}_{k}$ according to the  following  

\begin{definition}\label{def:CupProd DefAdm}
Suppose that $P\left( X\right) =\prod\nolimits_{i\in I}\left( 1-\alpha
_{i}^{-1}X\right) \in 1+XL\left[ X\right] $. We say that $\left( P,r\right) $
is admissible for $\mathcal{X}_{k}$ if $r\ges{} d+1$, and $\alpha$ is not equal to 
$p^{d}\zeta$ or to  $p^{d+1}\zeta$ for some root of unity $\zeta$ in $\mu_f$. 
\end{definition}

Denote by 
\[
\mathrm{H}_{\mathrm{logdR},c}^{i}:=\mathrm{H}_{\mathrm{logdR}}^{i}\left( X,%
\mathcal{O}_{X}\left( -D\right) \right) _{L}
\]%
and 
\[
\mathrm{H}_{\mathrm{HK},c}^{i}:=\mathrm{H}_{\mathrm{HK}}^{i}\left( \left( 
\mathcal{X}_{k},\alpha _{\mathcal{X}_{k}}\right) /\mathcal{S}_{0}^{0},%
\mathfrak{O}_{0}\left( -\mathfrak{D}_{0}\right) \right) _{L}
\]
the compactly supported logarithmic de Rham cohomology and  Hyodo--Kato cohomology of the structure sheaf respectively. 
Similarly set 
\[
    \mathrm{H}_{%
\mathrm{an},c}^{2d}:=\mathrm{H}_{\mathrm{an}}^{2d}\left( \left( 
\mathcal{X}_{k},\alpha _{\mathcal{X}_{k}}\right) /K_{0},\mathfrak{O}^{%
\mathrm{conv}}\left( -\mathfrak{D}\right) \right),
\]
and denote by $$\gamma^{\mathrm{HK}}_{c} : \mathrm{H}_{\mathrm{an},c}^{i}\lfre{}\mathrm{H}^{i}_{\mathrm{HK},c}$$
the Frobenius equivariant morphism defined by $Z=0$. According to Corollary $\ref{cor:DstHKDdR}$
the map  $\gamma_{c}^{\mathrm{an}} : \mathrm{H}^{i}_{\mathrm{an},c}\lfre{}\mathrm{H}^{i}_{\mathrm{logdR},c}$
factors as the composition $\mathrm{comp}_{c}\circ{}\gamma^{\mathrm{HK}}_{c}$, where 
\[
       \mathrm{comp}_{c} : \mathrm{H}^{i}_{\mathrm{HK},c}\hlfre{}\mathrm{H}_{\mathrm{logdR},c}^{i}
\] 
is the inclusion arising from the (compactly supported analogue of the) commutative diagram displayed in the statement of Corollary $\ref{cor:DstHKDdR}$. In particular the image of $\gamma_{c}^{\mathrm{an}}$ is contained in 
the image of  the comparison map $\mathrm{comp}_{c}$.
%, hence $\gamma _{c}^{\mathrm{an}}$ yields a Frobenious equivariant map from $%
%\mathrm{H}_{\mathrm{an},c}^{i}:=\mathrm{H}_{\mathrm{an}}^{i}\left( \left( 
%\mathcal{X}_{k},\alpha _{\mathcal{X}_{k}}\right) /K_{0},\mathfrak{E}^{%
%\mathrm{conv}}\left( -\mathfrak{D}\right) \right) $ to \textrm{I}$\mathrm{m}%
%\left( \gamma _{c}^{\mathrm{an}}\right) $.

Assume that $$r\ges d+1\ \ \text{ and}\ \  P\left( p^{d}\zeta \right) \neq 0$$
for any $\zeta \in \mu _{f}$. The Hyodo-Kato trace map identifies $\mathrm{%
H}_{\mathrm{HK},c}^{2d}$ with $K_{0,L}\left( -d\right)$ (cf.\ Proposition 4.6 of \cite{EY}), hence the assumption 
$P\left( p^{d}\zeta \right) \neq 0$, for $\zeta\in{}\mu_{f}$, implies that $P\left( \Phi _{L}\right) $ is invertible on $\mathrm{H}_{\mathrm{HK}%
,c}^{2d}$. For each $x$ in $\mathrm{H}^{2d}_{\mathrm{an},c}$, we can then define  
\[
     P\left( \Phi_{L}\right) ^{-1}\left( \gamma _{c}^{\mathrm{an}}\left( x\right) \right)=\mathrm{comp}_{c}\big(P(\Phi_{L})^{-1}\big(\gamma^{\mathrm{HK}}_{c}(x)\big)\big)
\]
in $\mathrm{H}_{\mathrm{logR},c}^{2d}$.  Because $r\ges d+1$,
one has $$\mathrm{H}^{2d}\big( \mathrm{Gr}\big( \Omega _{X/K}^{\mathrm{%
log},\bullet }\left( -D\right) \big) \big) _{L}=\mathrm{H}_{\mathrm{logdR%
},c}^{2d},$$ so that 
\begin{equation}\label{eq:ridondante }
    \mathrm{F}^{1}\mathrm{H}^{2d+1}\left( \mathbf{Syn}^{c}\left( \mathbb{I},P,r\right) \right)=\mathrm{coker}\Big(\mathrm{H}^{2d}_{\mathrm{an},c}\lfre{P(\Phi_{L})\oplus{}\gamma^{\mathrm{an}}_{c}}
    \mathrm{H}^{2d}_{\mathrm{an,c}}\oplus{}\mathrm{H}^{2d}_{\mathrm{logdR},c}\Big). 
%    \footnote{One has $\ker(\gamma^{\mathrm{an}}_{c})=\ker(\gamma^{\mathrm{HK}}_{c})$ and $$\mathrm{H}^{2d}_{\mathrm{an},c}/P(\Phi_L)\cdot{}\ker(\gamma^{\mathrm{HK}}_c)\simeq{}\mathrm{F}^{1}\mathrm{H}^{2d+1}\left( \mathbf{Syn}^{c}\left( \mathbb{I},P,r\right) \right)$$
%via the map sending the class of $x$ to that of $(x,0)$. Identifying $\mathrm{H}^{2d}_{\mathrm{HK},c}$
%    with a submodule of $\mathrm{H}^{2d}_{\mathrm{logdR},c}$ via $\mathrm{comp}_{c}$, the map  $\partial$ corresponds to 
%    \[
%            P(\Phi_{L})^{-1}\circ{}\gamma^{\mathrm{HK}}_{c} : \mathrm{H}^{2d}_{\mathrm{an},c}/P(\Phi_{L})\cdot\ker(\gamma^{\mathrm{HK}}_{c})\lfre{}
%            \mathrm{H}^{2d}_{\mathrm{HK},c}
%    \] }
\end{equation}
Since $\gamma_{c}^{\mathrm{HK}}$ is Frobenius equivariant, we can then define 
\[
     \partial : \mathrm{F}^{1}\mathrm{H}^{2d+1}\left( \mathbf{Syn}^{c}\left( \mathbb{I},P,r\right) \right)\lfre{}\mathrm{H}^{2d}_{\mathrm{logdR},c}
\]
to be the map which on the class $[x,y]$ represented by the pair $(x,y)$ takes the value  
\[
       \partial\big([x,y]\big)=y-P(\Phi_{L})^{-1}(\gamma^{\mathrm{an}}_{c}(x)).
\]
%It follows that, for every $\left( x,y\right)$ in, it makes sense to consider the expression%
%\[
%\partial \left( x,y\right) :=y-P\left( \Phi _{L}\right) ^{-1}\left( \gamma
%_{c}^{\mathrm{an}}\left( x\right) \right) \in \mathrm{H}_{\mathrm{logR}%
%,c}^{2d}\text{,}
%\]%
%where we remark that, indeed, the image $\left( P\left( \Phi _{L}\right)
%\left( z\right) ,\gamma _{c}^{\mathrm{an}}\left( z\right) \right) $ of any $%
%z\in \mathrm{H}_{\mathrm{an},c}^{2d}$ is killed by $\partial $.

\begin{remark}
Assume that $r\ges d+1$ and $P\left( p^{d}\zeta \right) \neq 0$ for any $%
\zeta \in \mu _{f}$. The above map $\partial $ is a section of the canonical
map from $\mathrm{H}_{\mathrm{logR},c}^{2d}$ to $\mathrm{F}^{1}\mathrm{H}%
^{2d+1}\left( \mathbf{Syn}^{c}\left( \mathbb{I},P,r\right) \right) $ sending 
$y$ to the class of $\left(0,y\right) $.
\end{remark}

The following lemma explains the relevance of the last assumption in Definition $\ref{def:CupProd DefAdm}$, namely $P\left( p^{d+1}\zeta \right) \neq 0$ for any $\zeta \in \mu _{f}$.

\begin{lemma}
\label{lemma:CupProd L1 Trace}If $P\left( p^{d+1}\zeta \right) \neq 0$ for
any $\zeta \in \mu _{f}$, then%
\[
\mathrm{H}^{2d+1}\left( \mathbf{Syn}^{c}\left( \mathbb{I},P,r\right) \right)
=\mathrm{F}^{1}\mathrm{H}^{2d+1}\left( \mathbf{Syn}^{c}\left( \mathbb{I}%
,P,r\right) \right) \text{.}
\]
\end{lemma}

\begin{proof}
We have that $\mathrm{F}^{0}\mathrm{H}^{2d+1}\left( \mathbf{Syn}^{c}\left( 
\mathbb{I},P,r\right) \right) \subset \mathrm{H}_{\mathrm{an}%
,c}^{2d+1,P\left( \Phi _{L}\right) =0}$, hence it suffices to show that 
$\mathrm{H}_{\mathrm{an},c}^{2d+1,P\left( \Phi _{L}\right) =0}=0$. It
follows from Proposition \ref{prop:anabscompact} $\left( 3\right) $ and $%
\left( 4\right) $ and the vanishing of $\mathrm{H}_{\mathrm{HK},c}^{2d+1}$
that we have $$\mathrm{H}_{\mathrm{an},c}^{2d+1,P\left( \Phi _{L}\right)
=0}\simeq \mathrm{H}_{\mathrm{HK},c}^{2d}\left( -1\right) ^{P\left( \Phi
_{L}\right) =0}\simeq K_{0,L}\left(
-d-1\right) ^{P\left( \Phi _{L}\right) =0},$$
which is zero since by assumption 
$P\left( \Phi _{L}\right) $ is an isomorphism on $K_{0,L}\left( -d-1\right) $.
\end{proof}

Let $\mathrm{Tr}%
_{U}\colon \mathrm{H}_{\mathrm{dR},c}^{2d}(U,\mathcal{O}_{U})\cong K(-d)$ be the
coherent trace map defined in Equation $(4.2.8)$, Theorem 4.2.7 of \cite{LLZ}, and define the \emph{de Rham trace map} $$\mathrm{Tr}_{\mathrm{dR%
}}\colon \mathrm{H}_{\mathrm{logdR},c}^{2d}\cong K_{L}\left( -d\right)$$
to be its composition with the isomorphism $\mathrm{H}_{\mathrm{logdR},c}^{2d}\cong{}\mathrm{H}_{\mathrm{dR},c}^{2d}(U,\mathcal{%
O}_{U})_{L}$.
We can finally define the sought for finite polynomial trace map. 

\begin{definition}
\label{def:CupProd F Trace} Assume that $\left( P,r\right) $ is admissible
for $\mathcal{X}_{k}$. Define 
\[
\mathrm{Tr}_{P,r}\colon \mathrm{H}^{2d+1}\big( \mathbf{Syn}^{c}\big( 
\mathbb{I},P,r\big) \big) \longrightarrow K_{L}
\]%
to be the composition of the natural isomorphism $$\mathrm{H}^{2d+1}\big( \mathbf{Syn}%
^{c}\big( \mathbb{I},P,r\big) \big) =\mathrm{F}^{1}\mathrm{H}%
^{2d+1}\big( \mathbf{Syn}^{c}\big( \mathbb{I},P,r\big) \big)$$ of
Lemma \ref{lemma:CupProd L1 Trace}, the map $\partial$, and the de Rham trace map $\mathrm{Tr}_{\mathrm{dR}}$. Explicitly, after identifying 
$\mathrm{H}^{2d+1}\big( \mathbf{Syn}^{c}\big( \mathbb{I},P,r\big)\big)$ with $\mathrm{F}^{1}\mathrm{H}^{2d+1}\big( \mathbf{Syn}^{c}\big( \mathbb{I},P,r\big) \big)$, one has 
\[
     \mathrm{Tr}_{P,r}([x,y])=\mathrm{Tr}_{\mathrm{dR}}\left( y-P\left( \Phi _{L}\right)
^{-1}\gamma _{\mathrm{an},c}\left( x\right) \right)
\]
for each $[x,y]$ in $F^{1}\mathrm{H}^{2d+1}\big(\mathbf{Syn}^{c}\left( \mathbb{I},P,r\right)\big)$ (cf.\ Equation $(\ref{eq:ridondante })$).
\end{definition}

\begin{remark}
If $P$ divides $Q$ in $1+XL[X]$, $r\ges s$, and $\left( Q,s\right) $ is admissible, then $(P,r)$ is admissible, and $\mathrm{Tr}_{P,r}$
equals the composition of $\mathrm{Tr}_{Q,s}$ with the morphism $\mathrm{H}^{2d+1}\left( 
 \mathbf{Syn
}^{c}\left( \mathbb{I},P,r\right) \right) \longrightarrow \mathrm{H}^{2d+1}\left(  \mathbf{Syn
}^{c}\left( \mathbb{I},Q,s\right)\right) $ of Remark \ref{rmk:Poly R2}.

%It then follows from Remark \ref{Poly R4} that, if $% P_{i}\mid Q_{i}$, $r_{i}\geq s_{i}$ and $\left( P,r\right) :=\left(P_{1}\ast P_{2},r_{1}+r_{2}\right) $ is admissible, then the pairing $\left\langle -,-\right\rangle _{P_{1},r_{1};P_{2},r_{2}}$ equals the composition of the tensor product of the morphisms of Remark \ref{rmk:Poly R2}followed by $\left\langle -,-\right\rangle _{Q_{1},s_{1};Q_{2},s_{2}}$.
\end{remark}

%\bigskip 
%
%We remark that the morphism from $\mathrm{H}_{\mathrm{an},c}^{2d}$ to $%
%\mathrm{H}_{\mathrm{HK},c}^{2d}$ (appearing in Corollary \ref{cor:DstHKDdR})
%is surjective and, indeed, it follows from Proposition \ref%
%{prop:anabscompact} $\left( 3\right) $ and $\left( 4\right) $ and the finite
%dimensionality of $\mathrm{H}_{\mathrm{HK},c}^{i}$ and $\mathrm{H}_{\mathrm{%
%abs}}^{i}\left( \left( \mathcal{X}_{k},\alpha _{\mathcal{X}_{k}}\right)
%/K_{0},\mathfrak{E}_{0}\left( -\mathfrak{D}_{0}\right) \right) $ that there
%is an exact sequence%
%\[
%0\longrightarrow \left( \frac{\mathrm{H}_{\mathrm{HK},c}^{2d-1}\left(
%-1\right) }{N\left( \mathrm{H}_{\mathrm{HK},c}^{2d-1}\left( -1\right)
%\right) }\right) ^{\Phi _{L}^{f}=p^{df}}\longrightarrow \mathrm{H}_{\mathrm{%
%an},c}^{2d,\Phi _{L}^{f}=p^{df}}\longrightarrow \mathrm{H}_{\mathrm{HK}%
%,c}^{2d}\longrightarrow 0\text{.}
%\]
%(Recall that $f$ is the degree of $K_{0}$ over $\Q_{p}$.)

We conclude this section with the following compatibility of trace maps. 
Let $I^{\mathrm{an}}_{c}$ be the kernel of  $\gamma^{\mathrm{an}}_{c} : \mathrm{H}^{2d}_{\mathrm{an},c}\lfre{}\mathrm{H}^{2d}_{\mathrm{logdR},c}$. It is also the 
kernel of  $\gamma_{c}^{\mathrm{HK}} : \mathrm{H}^{2d}_{\mathrm{an},c}\lfre{}\mathrm{H}^{2d}_{\mathrm{HK},c}$
(defined by $Z=0$), as $\gamma^{\mathrm{an}}_{c}=\mathrm{comp}_{c}\circ{}\gamma^{\mathrm{HK}}_{c}$ and $\mathrm{comp}_{c}$ is injective.

\begin{lemma}\label{lemma:traceHKlogdRcomp} The following diagram commutes.
$$\xymatrix{\mathrm{H}_{\mathrm{an}}^{2d}\left( \left( \mathcal{X}_k,\alpha _{\mathcal{X}
_k}\right) /\left( \mathfrak{S},\mathfrak{M}\right) ,\mathfrak{O}^{\mathrm{
conv}}(-\mathfrak{D})\right)\big/I_{c}^{\mathrm{an}}  \ar[d]_{\gamma^{\mathrm{HK}}_{c}}^{\wr} \ar[r]^{\ \ \ \ \ \ \ \ \ \ \ \ \gamma^{\mathrm{an}}_{c}} & \mathrm{H}_{\mathrm{logdR
}}^{2d}\left( X,\mathcal{O}_X(-D) \right) \ar[r]^-{\sim} & K \\ \mathrm{H}_{\mathrm{HK}}^{2d}\left( \left( \mathcal{X}_k,\alpha _{\mathcal{X
}_k}\right) /\mathcal{S}_{0}^{0},\mathfrak{O}_{0}(-\mathfrak{D}_0)\right)  \ar[r]^-{\sim} & K_0(-d) \ar@/_{5pt}/[ur] & }
$$
Here the horizontal isomorphisms are the trace maps $\mathrm{Tr}_{\rm dR}$ introduced above and $\mathrm{Tr}_{\rm HK}$ of \cite[Prop.~4.6]{EY} respectively, 
%the left vertical map is defined by $Z=0$ 
and the map $K_0(-d)\longrightarrow{} K$ is the inclusion. 
\end{lemma}
\begin{proof} First of all we reduce to the case that $\mathfrak{D}=D=\emptyset$ and $\alpha _{\mathcal{X}}$ is the log-structure $\alpha'$ given by the special fiber $\mathcal{X}_k$. Indeed, if the commutativity is proven in this case it follows also in the general case as the corresponding diagram maps to the one for general $\mathfrak{D}$   and log-structure $\alpha _{\mathcal{X}
_k}$. Second, if $n$ is the ramification index of $K_0\subset K$, in the diagram above we can replace the log-crystalline cohomology and trace map for $(\mathcal{X}_k,\alpha' _k)$ with the log-crystalline cohomology and trace map for the reduction $(\overline{\mathcal{X}}, \alpha_{\overline{\mathcal{X}}}')$ modulo $p$ of $(\mathcal{X},\alpha')$ with the left vertical map in the diagram given by $z^n=0$. In fact, taking powers of Frobenius we get an isomorphism of the corresponding log-crystalline cohomologies relative to $\mathcal{S}_{0}^{0}$, inverting $p$. The base change of $\mathrm{H}_{\mathrm{logcrys}}^{2d}\bigl((\overline{\mathcal{X}}, \alpha_{\overline{\mathcal{X}}}') /\mathcal{S}_{0}^{0}\bigr)$ and the trace map, where we omit the structure sheaf for simplicity, via the finite and flat map $\mathcal{O}_{K_0}\subset \mathcal{O}_K$ computes $\mathrm{H}_{\mathrm{logcrys}}^{2d}\bigl((\overline{\mathcal{X}}, \alpha_{\overline{\mathcal{X}}}')/\mathcal{O}_K^{\rm log}\bigr) $ and the corresponding trace map. Here, $\mathcal{O}_K^{\rm log}$ stands for $\mathcal{O}_K$, with log-structure given by the divisor $p=0$. As $\mathcal{X}$ provides a log-smooth lift of $\overline{\mathcal{X}}$  to $\mathcal{O}_K^{\rm log}$, this cohomology group is identified with $\mathrm{H}_{\mathrm{logdR
}}^{2d}( \mathcal{X},\mathcal{O}_{\mathcal{X}}) $ and the log-crystalline trace map is identified with the trace map in log de Rham cohomology. If we invert $p$, these coincide with  $\mathrm{H}_{\mathrm{logdR
}}^{2d}\left( X,\mathcal{O}_X \right)$ and $\mathrm{Tr}_{\rm dR}$. The  claim follows.
\end{proof}

%
%\begin{remark} \footnote{Forse lo rimuoverei.}
%\label{CupProd RmkDef3}One may also allow $\mathcal{X}$ to be proper and of
%relative dimension $d$ and the same results hold. (By applying the above
%discussion to each $\mathcal{X}_{i}$'s in order to deduce that $P\left(
%\Phi_L \right) $ acting on $\mathrm{H}_{\mathrm{an},c}^{2d}$ is invertible
%when $\left(P,r\right) $ is admissible, that $F^{0}%
%\mathrm{H}^{2d+1}\left( \mathbf{Syn}_{P,r}^{c}\right) =0$ and noticing that
%the existence of traces on each component imply that $\mathrm{H}^{2d}\left(
%\Gamma _{\mathrm{dR},L}^{c}\left( \mathbb{I}\right) \right) $ is a sum of
%copies of $K_{L}\left( d\right) $. The only difference is that $\mathrm{Tr}_{%
%\mathrm{dR}}$ is a morphism \color{magenta} which is not in general an isomorphism \color{black}\footnote{If $P(\Phi_{L})$ is not invertible on $\mathrm{H}^{2d}_{\mathrm{an},c}$, then 
%$\mathrm{Tr}_{P,r}$ is not necessarily injective even when $\mathcal{X}$ is geometrically connected. It is an isomorphism if and only if 
%$P(\Phi_{L})\cdot{}\ker(\gamma^{\mathrm{an}}_{c})=\ker(\gamma^{\mathrm{an}}_{c})$.}).
%\end{remark}

\section{Syntomic polynomial Kunneth morphisms and cup products}\label{SS PolyCup}

Consider $\mathbb{E}=\left( \mathcal{E},\mathrm{Fil}^{\bullet }
\mathcal{E},\nabla,\mathfrak{E},\Phi _{\mathfrak{E}}\right) 
$, a non-degenerate Frobenius log-crystal. Write $K_{P,r}(\mathbb{E})$ or simply $K_{P,r}$ for one of the complexes introduced in Definition \ref{def:synpol} and the following variants. We think about $K_{P,r}$ as a double complex $K_{P,r}^{p,q}$ as follows. We have:
\begin{itemize}
\item[i.] $K_{P,r}^{p,q}=0$ for $p\neq 0$, $1$ and $q\leq 0$;
\item[ii.] horizontal morphisms $d_1^{0,q}=(P(\Phi_L),\gamma)\colon K_{P,r}^{0,q}\rightarrow K_{P,r}^{1,q}$  and vertical morphisms $d_2^{p,q}\colon K_{P,r}^{p,q}\rightarrow K_{P,r}^{p,q+1}$ arising
from those of the Chech complexes $\mathrm{R}\Gamma _{L}$ and $\mathrm{R}\Gamma _{\mathrm{dR},L}$ such that $d_2^{1,q} \circ d_1^{0,q}  =d_1^{0,q+1}\circ d_2^{p,q}$
\end{itemize}
Then, the vertical filtration yields a spectral sequence
$$E_{P,r}^{p,q}\left( \mathbb{E}\right)_a \Longrightarrow
E_{P,r}^{p+q}\left( \mathbb{E}\right) \label{Poly F0}$$
whose first page $E_{P,r,1}^{p,q}$ is  $\mathrm{H}^{q}\left(\mathrm{R}\Gamma
_{L}\right) $ for $p=0$, $\mathrm{H}^{q}\left(\mathrm{R}\Gamma _{L}\right)
\oplus \mathrm{H}^{q}\left( \frac{\mathrm{R}\Gamma _{\mathrm{dR},L}}{\mathrm{
Fil}^{r}\left(\mathrm{R} \Gamma _{\mathrm{dR},L}\right) }\right) $ for $p=1$ and
zero otherwise, and the abutment $E_{P,r}^{p+q}$ is $\mathrm{H}^{p+q}(K_{P,r})$. It is easily checked that the filtration induced on $\mathrm{H}^{p+q}(K_{P,r})$  by the spectral sequence is described by $F^1 \mathrm{H}^{p+q}(K_{P,r})$ and $F^0 \mathrm{H}^{p+q}(K_{P,r})$  defined in the discussion after Definition \ref{def:synpol}.

\subsection{The setting}\label{sec:SS PolyCup Setting}

Assume that $\mathbb{E}_{i}=\left( \mathcal{E}_{i},\mathrm{Fil}^{\bullet }
\mathcal{E}_{i},\nabla _{i},\mathfrak{E}_{i},\Phi _{\mathfrak{E}_{i}}\right) 
$ are non-degenerate Frobenius log-crystal relative to $\left( X_{i},\alpha
_{X_{i}},\overline{\mathcal{X}_{i}},\alpha _{\overline{\mathcal{X}_{i}}
}\right) $, where $\left(X_i,D_i, \mathcal{X}_{i},\mathcal{D}_{i}\right) $ satisfy
Assumption \ref{ass:formalcase} for $i=1$, $2$, $12$, $3$. Set $U_{i}:=X_{i}-D_{i}$. We suppose that we have morphisms of formal
schemes
\begin{equation*}
\mathcal{X}_{3}\longrightarrow \mathcal{X}_{12}\rightrightarrows \mathcal{X}
_{1},\mathcal{X}_{2}
\end{equation*}
such that $\mathcal{D}_{3}$ contains the inverse image of $\mathcal{D}_{12}$
and $\mathcal{D}_{12}$ contains the inverse images of $\mathcal{D}_{i}$ and
the resulting morphism $\mathcal{X}_{12}\rightarrow \mathcal{X}_{1}\times 
\mathcal{X}_{2}$ induces an isomorphism $U_{12}\overset{\sim }{\rightarrow }
U_{1}\times U_{2}$. We can pull-back $\mathbb{E}_{i}$ to $\mathcal{X}_{3}$
and consider their tensor product $\mathbb{E}_{12}$. Finally, we assume that
there is a morphism $\mathbb{E}_{12}\rightarrow \mathbb{E}_{3}$. 

Write $K_i:=K_{P_i,r_i}(\mathbb{E}_i)$ for the syntomic double complexes, with or without  support.
Given the morphism $\mathbb{E}_{12}\rightarrow \mathbb{E}_{3}$ we wish to define a cup product of the cohomologies of $K_1$ and $K_2$ with values in $K_3$, respecting filtrations and compatible with cup products on analytic and de Rham cohomology. In order to make this precise we need some generalities.

\subsection{Cup products of double complexes}

Assume we have double complexes $K_1$, $K_2$ and $K_3$ (in our applications, mainly such that $K_i^{p,q}=0$ for $p\neq 0$, $1$ and $q\leq 0$).
We take the convention that the complexes are {\em commutative}, i.e., the horizontal $d_1^{p,q}:K_i^{p,q} \rightarrow K_i^{p+1,q}$ and vertical $d_2^{p,q}:K_i^{p,q} \rightarrow K_i^{p,q+1}$ differentials commute $d_1\circ d_2=d_2\circ d_1$. We give the following

\begin{definition}\label{def:CupProd SS F DblCmpx3}
A cup product $K_{1}\otimes K_{2}\rightarrow K_{3}$ of double complexes 
satisfying the Leibniz rule is a family of morphisms
$$
\cup ^{p_{1},q_{1};p_{2},q_{2}}:K_{1}^{p_{1},q_{1}}\otimes
K_{2}^{p_{2},q_{2}}\rightarrow K_{3}^{p_{1}+p_{2},q_{1}+q_{2}}
$$
such that, if $x_{i}\in K_{i}^{p_{i},q_{i}}$, we have
\begin{eqnarray}
d_{1}\left( x_{1}\cup x_{2}\right) &=&d_{1}\left( x_{1}\right) \cup
x_{2}+\left( -1\right) ^{p_{1}}x_{1}\cup d_{1}\left( x_{2}\right),
\notag \\
d_{2}\left( x_{1}\cup x_{2}\right) &=&d_{2}\left( x_{1}\right) \cup
x_{2}+\left( -1\right) ^{q_{1}}x_{1}\cup d_{2}\left( x_{2}\right).
\end{eqnarray}
We will think about complexes as being double complexes concentrated in the
column $p=0$: taking $p_{1}=p_{2}=0$ yields the notion of cup product of
complexes satisfying the Leibniz rule.
\end{definition}

Let us define $\widetilde{d}_{1}^{p,q}:=d_{1}^{p,q}$ and $\widetilde{d}%
_{2}^{p,q}:=\left( -1\right) ^{p}d_{2}^{p,q}$ for any one of the above
complexes, so that $\widetilde{d}:=\widetilde{d}_{1}+\widetilde{d}_{2}$ is
the differential of the associated total complex, and set $\widetilde{\cup }%
^{p_{1},q_{1};p_{2},q_{2}}:=\left( -1\right) ^{q_{1}p_{2}}\cup
^{p_{1},q_{1};p_{2},q_{2}}$. Then, if $x_{i}\in K_{i}^{p_{i},q_{i}}$, we have%
\begin{equation*}
\widetilde{d}\left( x_{1}\widetilde{\cup }x_{2}\right) =\widetilde{d}\left(
x_{1}\right) \widetilde{\cup }x_{2}+\left( -1\right) ^{n_{1}}x_{1}\widetilde{%
\cup }\widetilde{d}\left( x_{2}\right) 
\end{equation*}%
for the total degree $n_{1}:=p_{1}+q_{1}$ of $x_{1}$, as it follows from the
fact that the same formula holds for $\widetilde{d}_{1}$ and $\widetilde{d}%
_{2}$. In particular, $\widetilde{\cup }^{p_{1},q_{1};p_{2},q_{2}}$ yields a
cup product of complexes satisfying the Leibniz rule between the associated
total complexes and, in cohomology, for every $n_{1}$ and $n_{2}$, a cup
product (induced by $\widetilde{\cup }$)%
\begin{equation*}
\cup ^{n_{1},n_{2}}\colon \mathrm{H}^{n_{1}}\left( K_{1}\right) \otimes 
\mathrm{H}^{n_{2}}\left( K_{2}\right) \longrightarrow \mathrm{H}%
^{n_{1}+n_{2}}\left( K_{3}\right) \text{.}
\end{equation*}%
In the following proposition, whose verification we leave to the reader, we
consider the spectral sequences associated to the vertical filtration.
Writing $B_{i,\infty }^{p,q}$ (resp. $Z_{i,\infty }^{p,q}$) for the union
(resp. the intersection) of the coboundaries (resp. cocycles) induced by the
differentials of the pages and setting $E_{i,\infty }^{p,q}:=\frac{%
Z_{i,\infty }^{p,q}}{B_{i,\infty }^{p,q}}$, there is always a canonical
inclusion of $E_{i,\infty \infty }^{p,q}:=\mathrm{gr}^{p}\left(
E_{i}^{p+q}\right) $ into $E_{i,\infty }^{p,q}$ (which is an isomorphism
when $K_{i}^{p,q}=0$ for $p>0$ or $q<0$).

\begin{proposition}
\label{prop:cupproductspectral} The cup product $\cup ^{n_{1},n_{2}}$ on
cohomology preserves the filtrations, namely it induces cup products 
\begin{equation*}
\cup ^{n_{1},n_{2}}\colon \mathrm{Fil}^{p_{1}}\mathrm{H}^{n_{1}}\left(
K_{1}\right) \otimes \mathrm{Fil}^{p_{2}}\mathrm{H}^{n_{2}}\left(
K_{2}\right) \longrightarrow \mathrm{Fil}^{p_{1}+p_{2}}\mathrm{H}%
^{n_{1}+n_{2}}\left( K_{3}\right) 
\end{equation*}%
Moreover, the cup product of double complexes $\cup
^{p_{1},q_{1};p_{2},q_{2}}$ induces, via the associated morphisms $%
\widetilde{\cup }^{p_{1},q_{1};p_{2},q_{2}}$, a family of morphisms 
\begin{equation*}
\cup _{r}=\cup _{r}^{p_{1},q_{1};p_{2},q_{2}}\colon
E_{1,r}^{p_{1},q_{1}}\otimes E_{2,r}^{p_{2},q_{2}}\rightarrow
E_{3,r}^{p_{1}+p_{2},q_{1}+q_{2}},r\in \mathbb{N\cup }\left\{ \infty ,\infty
\infty \right\} 
\end{equation*}%
such that, for $r\neq \infty $ and $x_{i}\in E_{i,r}^{p_{i},q_{i}}$, we have 
\begin{equation*}
d_{r}\left( x_{1}\cup _{r}x_{2}\right) =d_{r}\left( x_{1}\right) \cup
_{r}x_{2}+(-1)^{p_{1}+q_{1}}x_{1}\cup _{r}d_{r}\left( x_{2}\right) \text{,}
\end{equation*}%
$\cup _{r+1}$ is induced by $\cup _{r}$, i.e., $\left[ x_{1}\right] \cup
_{r+1}\left[ x_{2}\right] =\left[ x_{1}\cup _{r}x_{2}\right] $ if $\left[
x_{i}\right] \in E_{i,r+1}^{p_{i},q_{i}}$ is the class of $x_{i}\in
E_{i,r}^{p_{i},q_{i}}$, $\cup _{\infty }$ is induced by any one of the
pairings $\cup _{r}$ with $r\in \mathbb{N}$, i.e. $\left[ x_{1}\right] \cup
_{\infty }\left[ x_{2}\right] =\left[ x_{1}\cup _{r}x_{2}\right] $ if $%
x_{i}\in Z_{i,\infty }^{p_{i},q_{i}}$, and $\cup _{\infty }$ restricts to $%
\cup _{\infty \infty }$, which is induced by $\cup ^{n_{1},n_{2}}$ on the
associated graded objects.
\end{proposition}

We finally remark that a cup product of double complexes can be interpreted
as a morphism of double complexes from $\mathrm{Tot}\left( K_{1}\otimes
K_{2}\right) $ to $K_{3}$ as follows. Here $\mathrm{Tot}\left( K_{1}\otimes
K_{2}\right) ^{p,q}$ is the sum of the $K_{1}^{p_{1},q_{1}}\otimes
K_{2}^{p_{2},q_{2}}$'s as long as $p_{1}+p_{2}=p$ and $q_{1}+q_{2}=q$ and
the differentials $d_{\mathrm{Tot},i}^{p,q}$ are obtained setting $d_{%
\mathrm{Tot},1}^{p,q}:=\widetilde{d}_{\mathrm{Tot},1}^{p,q}$ and $d_{\mathrm{%
Tot},2}^{p,q}:=\left( -1\right) ^{p}\widetilde{d}_{\mathrm{Tot},2}^{p,q}$,
where $\widetilde{d}_{\mathrm{Tot},i}^{p,q}$ is characterized by the
property that its restriction to $K_{1}^{p_{1},q_{1}}\otimes
K_{2}^{p_{2},q_{2}}$ is given by $\widetilde{d}_{i}^{p_{1},q_{1}}\otimes
1_{K_{2}^{p_{2},q_{2}}}+\left( -1\right)
^{n_{1}}1_{K_{1}^{p_{1},q_{1}}}\otimes \widetilde{d}_{i}^{p_{2},q_{2}}$.
Then the equalities appearing in Definition \ref{def:CupProd SS F DblCmpx3}
simply says that the family of the morphisms $\cup ^{p_{1},q_{1};p_{2},q_{2}}
$ defines a morphism of double complexes.

\subsection{Cup products in syntomic cohomology}
We go back to the notation of \S \ref{sec:SS PolyCup Setting}. 
Recall that, if $P_{1}\left( X\right) ,P_{2}\left( X\right) \in 1+XL\left[ X
\right] $ and write $P_{i}\left( X\right) =\prod\nolimits_{\alpha \in R_{i}}\left( 1-\alpha
X\right)$ where $\alpha \in \overline{L}$, then
$$
\left( P_{1}\ast P_{2}\right) \left( X\right) :=\prod\nolimits_{\left(
\alpha _{1},\alpha _{2}\right) \in R_{1}\times R_{2}}\left( 1-\alpha
_{1}\alpha _{2}X\right) \in 1+X L\left[ X\right] .
$$We choose $\lambda \in K$ and  $
p_{1}\left( X_{1},X_{2}\right)$, $p_{2}\left( X_{1},X_{2}\right) \in L\left[
X_{1},X_{2}\right] $ such that
\begin{equation}
\left( P_{1}\ast P_{2}\right) \left( X_{1}X_{2}\right) =p_{1}\left(
X_{1},X_{2}\right) P_{1}\left( X_{1}\right) +p_{2}\left( X_{1},X_{2}\right)
P_{2}\left( X_{2}\right).  \label{CupProd F1}
\end{equation}
Notice that (i) can be achieved thanks to \cite[Lemma 4.2]{Bes00}. We define a  cup product of double complexes
\begin{equation*}
\Cup _{P_{1},r_{1};P_{2},r_{2}}\colon K_{P_{1},r_{1}}^{p_{1},q_{1}}\left( 
\mathbb{E}_{1}\right) \otimes
_{K_{0,L}}K_{P_{2},r_{2}}^{c,p_{2},q_{2}}\left( \mathbb{E}_{2}\right)
\rightarrow K_{P_{1}\ast
P_{2},r_{1}+r_{2}}^{c,p_{1}+p_{2},q_{1}+q_{2}}\left( \mathbb{E}_{3}\right)
\end{equation*}
satisfying the condition of Definition \ref{def:CupProd SS F DblCmpx3}  as
follows. We first remark that we always have a commutative diagram
\begin{equation*}
\begin{array}{ccc}
\mathrm{R}\Gamma _{L}\left( \mathbb{E}_{1}\right) \otimes _{K_{0,L}}\mathrm{R}\Gamma
_{L}\left( \mathbb{E}_{2}\right) & \overset{\cup }{\longrightarrow } & 
\mathrm{R}\Gamma _{L}\left( \mathbb{E}_{3}\right) \\ 
\gamma _{1}\otimes _{K_{0,L}}\gamma _{2}\downarrow \text{ \ \ \ \ \ \ \ \ \
\ } &  & \text{ \ \ }\downarrow \gamma _{3} \\ 
\mathrm{R}\Gamma _{\mathrm{dR},L}\left( \mathbb{E}_{1}\right) \otimes
_{K_{L}}\mathrm{R}\Gamma_{\mathrm{dR},L}\left( \mathbb{E}_{2}\right) & \overset{
\cup }{\longrightarrow } & \mathrm{R}\Gamma_{\mathrm{dR},L}\left( \mathbb{E}_{3}\right)
\end{array}
\end{equation*}
with the property that vertical arrows are morphisms of complexes of $
K_{0,L} $-modules and the horizontal arrows are cup products of complexes
satisfying Leibniz rule; see the proof of Lemma \ref{lemma:CupProdCup}. We may assume that $\mathbb{E}_{12}=
\mathbb{E}_{3}$ in our definition of the pairings, so that $\mathrm{R}\Gamma
_{L}^c\left( \mathbb{E}_{3}\right) =\mathrm{R}\Gamma _{L}\left( \mathbb{E}_{1}\right)
\otimes _{K_{0,L}}\mathrm{R}\Gamma _{L}\left( \mathbb{E}_{2}\right) $ and $\mathrm{R}\Gamma _{
\mathrm{dR},L}\left( \mathbb{E}_{3}\right) =\mathrm{R}\Gamma _{\mathrm{dR},L}\left( 
\mathbb{E}_{1}\right) \otimes _{K_{L}}\mathrm{R}\Gamma _{\mathrm{dR},L}\left( \mathbb{E%
}_{2}\right) $. The product in the various degree
is described in the following table, where $u_i$, $x_i\in \mathrm{R}\Gamma
_{L}\left( \mathbb{E}_i\right) $ and $y_i\in \mathrm{R}\Gamma _{\mathrm{dR}
,L}\left( \mathbb{E}_{i}\right) $ for $i=1$, $2$.
\begin{equation*}
\resizebox{\textwidth}{!}{ \begin{tabular}{|l|l|l|} \hline $\Cup
_{P_{1},r_{1};P_{2},r_{2}}$ & $u_{2}$ & $\left( x_{2},y_{2}\right) $ \\
\hline $u_{1}$ & $u_{1}\otimes _{K_{0,L}}u_{2}$ & $\left( p_{2}\left(
\Phi _{1},\Phi _{2}\right) \left( u_{1}\otimes _{K_{0,L}}x_{2}\right)
,\left( 1-\lambda \right) \gamma \left( u_{1}\right) \otimes
_{K}y_{2}\right) $ \\ \hline $\left( x_{1},y_{1}\right) $ & $\left(
p_{1}\left( \Phi _{1},\Phi _{2}\right) \left( x_{1}\otimes
_{K_{0,L}}u_{2}\right) ,\lambda \cdot y_{1}\otimes _{K}\gamma \left(
u_{2}\right) \right) $ & $0$ \\ \hline \end{tabular}}
\end{equation*}

\begin{lemma}
\label{lemma:CupProdCup} We have that $\Cup _{P_{1},r_{1};P_{2},r_{2}}$ is a
cup product of double complexes satisfying Leibniz rule; see Definition \ref%
{def:CupProd SS F DblCmpx3}. Moreover, setting $P_{3}:=P_{1}\ast P_{2}$ and $%
r_{3}:=r_{1}+r_{2}$, we have the following commutative diagram of complexes
in which the horizontal arrows are the cup products of complexes satisfying
the Leibniz induced by the respective $\Cup _{P_{1},r_{1};P_{2},r_{2}}$'s
(via the associated $\widetilde{\Cup }%
_{P_{1},r_{1};P_{2},r_{2}}^{p_{1},q_{1};p_{2},q_{2}}:=\left( -1\right)
^{q_{1}p_{2}}\Cup _{P_{1},r_{1};P_{2},r_{2}}^{p_{1},q_{1};p_{2},q_{2}}$, see
the discusion before Proposition \ref{prop:cupproductspectral}) and where
the morphism%
\begin{equation*}
\mathbf{Syn}_{\mathrm{rig}}^{c}(\mathbb{E}_{3},P_{3},r_{3})\longrightarrow 
\mathbf{Syn}_{\mathrm{rig}}^{c}(\mathbb{E}_{3}^{\dagger },P_{3},r_{3})
\end{equation*}%
is a quasi-isomorphism:%
\begin{equation*}
\begin{array}{ccc}
\mathbf{Syn}(\mathbb{E}_{1},P_{1},r_{1})\otimes _{K_{0},L}\mathbf{Syn}%
^{c-\infty }(\mathbb{E}_{2},P_{2},r_{2}) & \longrightarrow  & \mathbf{Syn}%
^{c-\infty }(\mathbb{E}_{3},P_{3},r_{3}) \\ 
1\otimes \tau _{\mathrm{rig,c}}\uparrow \text{ \ \ \ \ \ \ \ \ \ \ \ \ \ \ \
\ \ \ \ \ \ \ \ } &  & \text{ \ \ \ \ \ \ }\uparrow \tau _{\mathrm{rig,c}}
\\ 
\mathbf{Syn}(\mathbb{E}_{1},P_{1},r_{1})\otimes _{K_{0},L}\mathbf{Syn}_{%
\mathrm{rig}}^{c}(\mathbb{E}_{2},P_{2},r_{2}) & \longrightarrow  & \mathbf{%
Syn}_{\mathrm{rig}}^{c}(\mathbb{E}_{3},P_{3},r_{3}) \\ 
\tau _{\mathrm{rig}}\otimes 1\downarrow \text{ \ \ \ \ \ \ \ \ \ \ \ \ \ \ \
\ \ \ \ \ \ } &  & \downarrow  \\ 
\mathbf{Syn}_{\mathrm{rig}}(\mathbb{E}_{1},P_{1},r_{1})\otimes _{K_{0},L}%
\mathbf{Syn}_{\mathrm{rig}}^{c}(\mathbb{E}_{2},P_{2},r_{2}) & 
\longrightarrow  & \mathbf{Syn}_{\mathrm{rig}}^{c}(\mathbb{E}_{3}^{\dagger
},P_{3},r_{3})\text{.}%
\end{array}%
\end{equation*}%
Here $\mathbf{Syn}_{\mathrm{rig}}^{c}(\mathbb{E}_{3}^{\dagger
},P_{3},r_{3}):=\mathbf{Syn}_{\mathrm{rig},U^{o}\subset Y,\alpha _{k}}^{c}(%
\mathbb{E}_{i}^{\dagger },P_{i},r_{i})$ is an overconvergent version of $%
\mathbf{Syn}_{\mathrm{rig}}^{c}(\mathbb{E}_{3},P_{3},r_{3})$ (obtained
considering the $\mathfrak{E}_{\underline{\beta }}^{\dagger }$'s of \S \ref%
{sec:rigsyn} and replacing the $\mathfrak{E}_{\underline{\beta }}^{\mathrm{%
conv}}$'s with them in Definition \ref{def:syntomiccompactonU}) and we have $%
\mathbf{Syn}_{\mathrm{rig}}:=\mathbf{Syn}_{\mathrm{rig},U\subset Y,\alpha
_{k}}$, $\mathbf{Syn}_{\mathrm{rig}}^{c}:=\mathbf{Syn}_{\mathrm{rig}%
,U^{o}\subset Y,\alpha _{k}}^{c}$ with $U=\mathcal{X}_{k}^{\mathrm{sm}}$, $%
U^{o}=U\backslash \mathcal{D}_{k}$ and $Y=\mathcal{X}_{k}$. Furthermore, $%
\tau _{\mathrm{rig}}=\tau _{\mathrm{rig},U\subset Y,\alpha _{k}}$ and $\tau
_{\mathrm{rig,c}}$ are the morphisms of double complexes defined in \S \ref%
{sec:rigsyn} and in \S \ref{sec:rigcmpct}. We can also replace everywhere in
the diagram $\mathbf{Syn}$ with $\mathbf{Syn}^{c\text{-}\infty }$.
\end{lemma}

\begin{proof}
We may assume $\mathbb{E}_{3}=\mathbb{E}_{12}$. All the syntomic
cohomologies are obtained by considering the total complex associated to a
double complex which is obtained from various log de Rham complexes $\Omega $%
. Thanks to the explicit description of \S \ref{sec:Chech} , \S \ref{sec:dR}%
, \S \ref{sec:analsyn}, \S \ref{sec:rigsyn} and \S \ref{sec:rigcmpct}, we
also know that these are double complexes \'{ }la Cech $\Omega =\left(
\Omega _{\underline{\beta }}\right) $ obtained from log de Rham complexes $%
\Omega _{\underline{\beta }}$ with differential $d_{\underline{\beta }}$.
Using the abbreviation $\underline{\beta }\left( j\right) :=\left( \beta
_{0},\ldots ,\widehat{\beta }_{j},\ldots ,\beta _{p}\right) $ , the total
differential is then given by 
\begin{equation*}
d\left( s_{\underline{\beta }}\right) =\left( \sum\nolimits_{j=0}^{p}s_{%
\underline{\beta }\left( j\right) }+\left( -1\right) ^{p}d_{\underline{\beta 
}}\left( s_{\underline{\beta }}\right) \right) \text{.}
\end{equation*}%
These log de Rham complexes $\Omega _{\underline{\beta }}=\Omega _{%
\underline{\beta }}\left( \mathcal{X}_{i}\right) $ are endowed with natural
cup products $\sqcup _{\underline{\beta }}$, arising from the morphisms 
\begin{equation*}
\sqcup _{\underline{\beta }}^{m,n}\colon \Omega _{\underline{\beta }%
}^{m}\left( \mathcal{X}_{1}\right) \otimes _{K_{0}}\Omega _{\underline{\beta 
}}^{n}\left( \mathcal{X}_{2}\right) \longrightarrow \Omega _{\underline{%
\beta }}^{m,n}\left( \mathcal{X}_{12}\right) 
\end{equation*}%
given by $\omega _{1}\sqcup _{\underline{\beta }}^{m,n}\omega
_{2}:=p_{1}^{\ast }\left( \omega _{1}\right) \wedge p_{2}^{\ast }\left(
\omega _{2}\right) $, where $p_{i}\colon \mathcal{X}_{12}\rightarrow 
\mathcal{X}_{i}$ for $i=1,2$ denote our given morphisms: a small correction,
explained below, is needed in order to handle the second and third
horizontal line of the diagram. The formula 
\begin{equation}
\left( s\sqcup t\right) _{\beta _{0},\ldots ,\beta
_{p_{1}+p_{2}}}:=\sum\nolimits_{r=0}^{p_{1}+p_{2}}\left( -1\right) ^{\left(
p_{1}+p_{2}+r\right) \left( p_{1}+q_{1}\right) +r\left( p_{1}+p_{2}\right)
+r}s_{\beta _{0},\ldots ,\beta _{r}}\sqcup _{\underline{\beta }}t_{\beta
_{r},\ldots ,\beta _{p_{1}+p_{2}}}  \label{eq:CupProd L1 FCechCup}
\end{equation}%
defines a cup product of double complexes satisfying Leibniz rule; see
[Stacks Project, (20.25.3.2)], where $p_{1}+p_{2}$ is denoted $p$, $%
p_{1}+q_{1}$ is denoted $n$ and $p_{2}+q_{2}$ is denoted $m$). In
particular, we get cup products $\cup $ of complexes satisfying the Leibniz
rule for these de Rham complexes $\Omega $. This means that, in order to
check Definition \ref{def:CupProd SS F DblCmpx3} holds for the double
complexes giving rise to $\mathbf{Syn}(\mathbb{E}_{1},P_{1},r_{1})$ and $%
\mathbf{Syn}^{c-\infty }(\mathbb{E}_{2},P_{2},r_{2})$, it suffices to check
the first relation, whose only non-trivial case is when $p_{1}=p_{2}=0$ in
the table above and the required relation follows thanks to (\ref{CupProd F1}%
).

For the second and the third line we argue as follows. We recall that, by
definition, for $?\in \left\{ \mathrm{conv},\dagger \right\} $, $R\Gamma _{%
\mathrm{rig},c}\left( (U^{o}\subset Y,\alpha _{k})/K_{0},\mathfrak{E}%
_{i}^{?}\right) $ is the Cech double complex $\Omega $\ obtained from the
complexes%
\begin{equation*}
\mathrm{Cone}_{\underline{\beta }}\left( \mathfrak{E}_{i}^{\dag }\right) :=%
\mathrm{Cone}\left( -1:\mathrm{dR}\left( \mathfrak{E}_{i}^{?}\left( -%
\mathfrak{D}\right) _{\underline{\beta }}\right) \longrightarrow \mathrm{dR}%
\left( j_{\underline{\beta }\ast }\circ j_{\underline{\beta }}^{\ast }\left( 
\mathfrak{E}_{i}^{?}\left( -\mathfrak{D}\right) \right) _{\underline{\beta }%
}\right) \right) \left[ -1\right] \text{.}
\end{equation*}%
For every $\underline{\beta }$, we get a cup product%
\begin{equation}
\sqcup _{\underline{\beta }}:\mathrm{dR}\left( \mathfrak{E}_{1}^{?}\right) _{%
\underline{\beta }}\otimes _{K_{0}}\mathrm{Cone}_{\underline{\beta }}\left( 
\mathfrak{E}_{2}^{\mathrm{conv}}\right) \longrightarrow \mathrm{Cone}_{%
\underline{\beta }}\left( \mathfrak{E}_{3}^{?}\right) 
\label{eq:CupProd L1 FCechCup2}
\end{equation}%
by means of the formula $z\sqcup _{\underline{\beta }}\left( x,y\right)
:=\left( p_{1}^{\ast }\left( z\right) \wedge p_{2}^{\ast }\left( x\right)
,p_{1}^{\ast }\left( z\right) \wedge p_{2}^{\ast }\left( x\right) \right) $.
For $R\Gamma _{\mathrm{rig}}\left( (U^{o}\subset Y,\alpha _{k})/K,\mathcal{E}%
_{i}\right) $ and $R\Gamma _{\mathrm{rig},c}\left( (U^{o}\subset Y,\alpha
_{k})/K,\mathcal{E}_{i}\right) $ one argues similarly. As above, we can use (%
\ref{eq:CupProd L1 FCechCup}) in order to get cup products of complexes
satisfying the Leibniz rule and then get the required cup products of double
complexes satisfying Leibniz rule from them. We remark that (\ref{eq:CupProd
L1 FCechCup2}) for $?=\mathrm{conv}$ and $\dagger $ are compatible with
respect to the restriction from the convergent to the overconvergent stuff.
As the map $\tau _{\mathrm{rig}}$ is restriction to an open and $\tau _{%
\mathrm{rig,c}}$ is the projection of the cone onto the first component, the
commutativity of the diagram follows. Finally, Corollary \ref{cor:BerCoh1}
proves that we have the claimed quasi-isomorphism.
\end{proof}

Applying Lemma \ref{lemma:CupProdCup}  we can now make the following:

\begin{definition}\label{def:CupProd Def1} 
When $\mathcal{X}_{3}\rightarrow \mathcal{X
}_{12}$ is the identity and $\mathbb{E}_3=\mathbb{E}_{12}$, we get a cup product called the {\em syntomic (polynomial) Kunneth morphism} that we denote them by $\sqcup _{P_{1},r_{1};P_{2},r_{2}}$ and satisfies the properties of Proposition \ref{prop:cupproductspectral}.

When $\mathcal{X}=\mathcal{X}_1=\mathcal{X}_2=\mathcal{X}_3$ and $\mathcal{X}_{3}\rightarrow \mathcal{X}_{1}\times \mathcal{X}_{2}$ is the
diagonal and $\mathbb{E}_3=\mathbb{E}_{12}$, we get   a cup product called the {\em syntomic
(polynomial) cup product} that we by $\cup
_{P_{1},r_{1};P_{2},r_{2}} $ and satisfies the properties of Proposition \ref{prop:cupproductspectral}.
\end{definition}

\begin{remark}\label{rmkPoly R4} (1) One can prove that the syntomic cup products of spectral sequence 
induced by the cup product of Lemma \ref{lemma:CupProdCup} do
not depend on the choice (\ref{CupProd F1}) starting
from the $2$-page on (and they are homotopic from the $1$-page).

(2) Suppose that $Q_{i}=P_{i}P_{i}^{\prime }$ and that $r_{i}\geq s_{i}$ for $i=1,2$ . Then $Q_{1}\ast Q_{2}=P_{1}\ast P_{2}\cdot P_{1}^{\prime }\ast P_{2}\cdot P_{1}\ast P_{2}^{\prime }\cdot P_{1}^{\prime }\ast P_{2}^{\prime }$, and $r_{1}+r_{2}\geq s_{1}+s_{2}$. If $p_{1}$ and $p_{2}$ are choosen as to satisfy (\ref{CupProd F1}) and we make analogous choices for the couples $\left( P_{1}^{\prime },P_{2}\right) $, $\left( P_{1},P_{2}^{\prime }\right) $ and $\left( P_{1}^{\prime },P_{2}^{\prime }\right) $, expanding the product we find an expression analogous to (\ref{CupProd F1})  for the couple $\left( Q_{1},Q_{2}\right) $\ involving two variables polynomials $q_{1}$ and $q_{2}$. With these choices, the morphisms arising from Remark \ref{rmk:Poly R2} give rise to a morphism of cup products of double complexes $\Cup _{P_{1},r_{1};P_{2},r_{2}}^{\prime }\rightarrow \Cup _{Q_{1},s_{1};Q_{2},s_{2}}^{\prime }$ and, hence, between the corresponding spectral sequences.

(3) As stated in Lemma \ref{lemma:CupProdCup}, the horizontal arrows appearing
in the diagram of loc.cit. come from cup products of double complexes
satisfying the Leibnitz rule. Indeed, the proof of the lemma shows that this
diagram can be promoted to a commutative diagram of double complexes and
that the morphism from the double complex giving rise to $\mathbf{Syn}_{%
\mathrm{rig}}^{c}(\mathbb{E}_{3},P_{3},r_{3})$ to its overconvergent version
gives rise to an isomorphism between the associated spectral sequences from
the $1$-page on.
\end{remark}

\subsection{Duality in syntomic cohomology}\label{sec:DualSyn}

Assume that $\mathcal{X}$ is proper of relative dimension $d$. As in \S \ref{sec:SS PolyCup Setting}, we will write $\mathbf{Syn}^{c}$ rather
than $\mathbf{Syn}^{c\text{-}\infty }$ when considering the analytic
syntomic cohomology. Take $\mathcal{X}_{3}\rightarrow 
\mathcal{X}_{1}\times \mathcal{X}_{2}$ to be the diagonal of $\mathcal{X}=
\mathcal{X}_{i}$ and  $\mathbb{E}_{3}=\left( \mathcal{O},\mathrm{Fil}
^{\bullet }\mathcal{O},d,\mathcal{O}_{\mathfrak{X}},\sigma \right) =:\mathbb{
I}$ the unit object. 
%Then we get a pairing (cfr. Remark \ref{CupProd RmkDef3}):

\begin{definition}
\label{DEF:CupProd Def2} If $\left( P_{1}\ast P_{2},r_{1}+r_{2}\right) $ is
admissible, then the (polynomial) syntomic cup products of Definition \ref{def:CupProd Def1} yields, by composition with the trace map $\mathrm{Tr}
_{P_{1}\ast P_{2},r_{1}+r_{2}}$ of Definition \ref{def:CupProd F Trace}, the (polynomial) syntomic pairing
$$
\left\langle -,-\right\rangle _{P_{1},r_{1};P_{2},r_{2}}\colon \mathrm{H}
^{i}\left( \mathbf{Syn}\left( \mathbb{E},P_1,r_1\right) \right)
\otimes _{K_{0,L}}\mathrm{H}^{2d+1-i}\left( \mathbf{Syn}^{c}
\left( \mathbb{E}^{\vee },P_2,r_2\right) \right) \longrightarrow K_{L}.
$$
\end{definition}

\begin{cor}
\label{cor:Poly R3bis} Assume that $\left( P_{1}\ast P_{2},r_{1}+r_{2}\right) $ is
admissible, then the following diagram is commutative. Here the pairings in the first row are the de Rham cup products and in the others are those provided by Lemma \ref{lemma:CupProdCup} and Proposition  \ref{prop:cupproductspectral}:
$$
\begin{array}{ccccc}
\frac{\mathrm{H}^{i}\left( \mathrm{R}\Gamma _{\mathrm{dR},L}\left( \mathbb{E}
\right) \right) }{\mathrm{Fir}^{r_{1}}\mathrm{H}^{i}\left(  \mathrm{R}\Gamma _{
\mathrm{dR},L}\left( \mathbb{E}\right) \right) } & \otimes _{K_{L}}
& \mathrm{Fir}^{r_{2}}\mathrm{H}^{2d-i}\left( \mathrm{R} \Gamma _{\mathrm{dR}
,L}^{c}\left( \mathbb{E}^\vee\right) \right) & \longrightarrow & K_{L} \\ 
\downarrow &  & \uparrow \gamma &  & \parallel \\ 
F^{1}\mathrm{H}^{i}\left( \mathbf{Syn}\left( \mathbb{E},P_1,r_1
\right) \right) & \otimes _{K_{0,L}} & F^{0}\mathrm{H}^{2d+1-i}\left( 
\mathbf{Syn}^{c}\left( \mathbb{E}^\vee,P_2,r_2\right) \right) & 
\longrightarrow & K_{L} \\ 
\downarrow &  & \uparrow &  & \parallel \\ 
\mathrm{H}^{i}\left( \mathbf{Syn}\left( \mathbb{E},P_1,r_1
\right) \right) & \otimes _{K_{0,L}} & \mathrm{H}^{2d+1-i}\left( \mathbf{
Syn}^{c}\left( \mathbb{E}^\vee,P_2,r_2\right) \right) & \longrightarrow
& K_{L}.
\end{array}
$$
\end{cor}
\begin{proof} The map $\gamma$ is defined using that an element $x\in F^{0}\mathrm{H}^{2d+1-i}\left( 
\mathbf{Syn}^{c}\left( \mathbb{E}^\vee,P_2,r_2\right) \right)$ is sent to $0$ via $\gamma$ in  $\mathrm{H}^{2d-i}\left( \mathrm{R} \Gamma _{\mathrm{dR}
,L}^{c}\left( \mathbb{E}^\vee\right) \right)/\mathrm{Fir}^{r_{2}}\mathrm{H}^{2d-i}\left( \mathrm{R} \Gamma _{\mathrm{dR} ,L}^{c}\left( \mathbb{E}^\vee\right) \right)$. 
The commutativity of the first square is a direct check. The commutativity of the second square follows
from Proposition  \ref{prop:cupproductspectral}.
\end{proof}

Write $W_1,\ldots,W_s$ for the irreducible components of $\mathcal{X}_k^{\rm sm}$ and fix an open non-empty subscheme $U\subset W_\ell$ in one of these components.  Definition \ref{def:CupProd Def1} provides a pairing, also denoted
$\left\langle -,-\right\rangle _{P_{1},r_{1};P_{2},r_{2}}$  from $$\mathrm{H}^i\bigl(\mathbf{Syn}_{{\rm rig},U\subset W_\ell, \alpha_k}(\mathbb{E},P_1,r_1)\bigr)
\otimes\mathrm{H}^{2d+1-i}\left( \mathbf{Syn}_{\mathrm{rig}
,U^{o}\subset W_\ell, \alpha_k}^{c}\left( \mathbb{E}^\vee,P_2,r_2\right)
\right) $$to $\mathrm{H}^{2d+1}\left( \mathbf{Syn}_{\mathrm{rig}
,U^{o}\subset W_\ell, \alpha_k}^{c}\left( \mathbb{I},P_{1}\ast P_{2},r_1+r_2\right)
\right)$. As in (\ref{eq:tauRigc}) and (\ref{eq:tauRig}) in \S \ref{sec:rigcmpct}, for every $i\in  \mathbb{N}$,   we have morphisms
$$
\tau_{{\rm rig},U,\ell}^i\colon \mathrm{H}^i\bigl(\mathbf{Syn}(\mathbb{E},P_1,r_1)\bigr) \longrightarrow \mathrm{H}^i\bigl(\mathbf{Syn}_{{\rm rig},U\subset W_\ell, \alpha_k}(\mathbb{E},P_1,r_1)\bigr),$$ as well as its variant in which $\mathbf{Syn}$ is
replaced by $\mathbf{Syn}^{c\text{-}\infty }$, and
$$\tau_{\mathrm{rig},U, \ell,c}^i\colon\mathrm{H}^{i}\left( \mathbf{Syn}_{\mathrm{rig}
,U^{o}\subset W_\ell, \alpha_k}^{c}\left( \mathbb{E}^\vee,P_2,r_2\right)
\right) \longrightarrow \mathrm{H}^{i}\left( \mathbf{Syn}^{c-\infty}\left( \mathbb{
E}^\vee,P_2,r_2\right) \right).$$
\begin{cor}\label{cor:compdual} Assume $\left( P_{1}\ast P_{2},r_{1}+r_{2}\right) $ is admissible. Let $?\in \left\{ \emptyset ,c-\infty
\right\} $. Then, for  $f\in \mathrm{H}^{i}\left( \mathbf{Syn}^{?}\left( \mathbb{E}
,P_{1},r_{1}\right) \right) $ and  $g\in \mathrm{H}^{2d+1-i}\left( \mathbf{Syn}_{\mathrm{rig}
,U\subset W_\ell, \alpha_k}^c\left( \mathbb{E}^\vee,P_2,r_2\right)
\right)$, we have 
$$
\tau_{\mathrm{rig},U,\ell,c}^{2d+1}\left(\left\langle \tau _{\mathrm{rig},U,\ell}^i\left( f\right) ,g\right\rangle
_{P_{1},r_{1};P_{2},r_{2}}\right)=\left\langle f,\tau _{\mathrm{rig},U,\ell,c}^{2d+1-i}\left(
g\right) \right\rangle _{P_{1},r_{1};P_{2},r_{2}}\in K_L.
$$
\end{cor} 
\begin{proof} 
For $U=W_\ell\cap \mathcal{X}_k^{\rm sm}$ the claim follows from Lemma \ref{lemma:CupProdCup} and Lemma \ref{lemma:xiiso}. For general $U$ we remark that the morphism of complexes 
$$\mathbf{Syn}_{{\rm rig},W_\ell\subset W_\ell, \alpha_k}(\mathbb{E},P_1,r_1) \longrightarrow\mathbf{Syn}_{{\rm rig},U\subset W_\ell, \alpha_k}(\mathbb{E},P_1,r_1)$$and $$\mathbf{Syn}_{{\rm rig},U^o\subset W_\ell, \alpha_k}^c(\mathbb{E}^\vee,P_2,r_2) {\longrightarrow} \mathbf{Syn}_{{\rm rig},W_\ell^o\subset W_\ell, \alpha_k}^c(\mathbb{E}^\vee,P_2,r_2),$$of \S \ref{sec:rigsyn} and \S \ref{sec:rigcmpct}, are morphisms of double complexes and by construction they are compatible with the pairings of Definition \ref{def:CupProd Def1} and Lemma \ref{lemma:CupProdCup}.
\end{proof}

\subsection{Dualities for coefficients with good reduction}\label{sec:DualSynGood} Assume that $W_\ell$ is smooth and let $U\subset W_\ell$ be an open dense subscheme. Assume that $\mathbb{E}=(\mathcal{E},\mathfrak{E})$ satisfies the Assumption of \S \ref{sec:Egoodred}. It follows from \cite[Prop. 3.15]{EY} that $\mathrm{H}^i_{\rm rig,c}\bigl((U^o\subset W_\ell,\alpha_k^{\rm hor})/K_0, \mathfrak{E}^{\rm conv}\bigr)$ and $\mathrm{H}^i_{\rm rig,c}\bigl((U^o\subset W_\ell,\alpha_k^{\rm hor})/K, \mathfrak{E}^{\rm conv}\bigr)$ coincides with Berthelot's rigid cohomology $\mathrm{H}^i_{{\rm rig},c}\bigl(U^o/K_0\bigr)$ and $\mathrm{H}^i_{{\rm rig},c}\bigl(U^o/K\bigr)$ respectively. In particular, these groups vanish for $i\geq 2d+1$ and we have trace maps $$\mathrm{Tr}\colon \mathrm{H}^{2d}_{{\rm rig},c}\bigl(U^o/K_0\bigr)\cong K_0(-d), \quad \mathrm{Tr}\colon \mathrm{H}^i_{{\rm rig},c}\bigl(U^o/K\bigr)\cong K. $$
For $\left( P,r\right) $ admissible, $P(\Phi_L)$ is an isomorphism on $\mathrm{H}^{2d}_{{\rm rig},c}\bigl(U^o/K_0) $ and we get from Lemma \ref{lemma:Poly R1} for $\mathbb{E}=\mathbb{I}=(\mathcal{O},\mathfrak{O})$  the unit object that 
$$\mathrm{H}^{2d}_{{\rm rig},c}\bigl(U^o/K\bigr)_L\cong F^1\mathrm{H}^{2d+1}\bigl(\mathbf{Syn}^c_{{\rm rig},U^o\subset W_\ell,\alpha_k^{\rm hor}}(\mathbb{I},P,r)\bigr)$$and $$F^1\mathrm{H}^{2d+1}\bigl(\mathbf{Syn}^c_{{\rm rig},U^o\subset W_\ell,\alpha_k^{\rm hor}}(\mathbb{I},P,r)\bigr)=\mathrm{H}^{2d+1}\bigl(\mathbf{Syn}^c_{{\rm rig},U^o\subset W_\ell,\alpha_k^{\rm hor}}(\mathbb{I},P,r)\bigr) .$$We use this to define a trace map $$\mathrm{Tr} \colon \mathrm{H}^{2d+1}\bigl(\mathbf{Syn}^c_{{\rm rig},U^o\subset W_\ell,\alpha_k^{\rm hor}}(\mathbb{I}^\vee,P,r)\bigr)\cong K_L$$and proceeding as in Lemma \ref{lemma:CupProdCup} and  in Definition \ref{DEF:CupProd Def2} pairings $\left\langle -,-\right\rangle_{P_{1},r_{1};P_{2},r_{2}}$

\begin{equation}\label{CupProd F Def2Rig}
\mathrm{H}^{i}\left(
\mathbf{Syn}_{{\rm rig},U\subset W_\ell,\alpha_k^{\rm hor}}\left(\mathbb{E},P_1,r_1\right)
\right) \otimes \mathrm{H}^{2d+1-i}\left(
\mathbf{Syn}^c_{{\rm rig},U^o\subset W_\ell,\alpha_k^{\rm hor}}\left(\mathbb{E}^{\vee},P_2,r_2
\right)\right) \rightarrow K_{L} . 
\end{equation}
We have morphisms 
$$
\nu_{{\rm rig},U,\ell,\ast}^i\colon \mathrm{H}^i\bigl(\mathbf{Syn}_{{\rm rig},U\subset W_\ell, \alpha_k^{\rm hor}}(\mathbb{E},P_1,r_1)\bigr) \longrightarrow \mathrm{H}^i\bigl(\mathbf{Syn}_{{\rm rig},U\subset W_\ell, \alpha_k}(\mathbb{E},P_1,r_1)\bigr),
$$with a left inverse $\nu_{{\rm rig},U,\ell}^{\ast,i}$, see  (\ref{eq:rhoRig}) in \S \ref{sec:Egoodred},
and
$$\nu_{\mathrm{rig},U,\ell,c,\ast}^i\colon \mathrm{H}^{i}\left( \mathbf{Syn}_{\mathrm{rig}
,U^{o}\subset W_\ell, \alpha_k^{\rm hor}}^{c}\left( \mathbb{E}^\vee,P_2,r_2\right)
\right) \longrightarrow \mathrm{H}^{i}\left( \mathbf{Syn}_{\mathrm{rig}
,U^{o}\subset W_\ell, \alpha_k}^{c}\left( \mathbb{E}^\vee,P_2,r_2\right)
\right),$$with a left inverse $\nu_{\mathrm{rig},U,\ell,c}^{\ast,i}$ see (\ref{eq:rhoRigc}) in \S \ref{sec:Egoodred}. Since all these morphisms arise from morphisms of double complexes, they are readily checked to be compatible with cup products we get the following:

\begin{lemma}\label{lemma:pairnu} For every element $f\in \mathrm{H}^{i}\left( \mathbf{Syn}_{\mathrm{rig},U\subset
W_{l},\alpha _{k}^{hor}}^{?}\left( \mathbb{E},P_{1},r_{1}\right) \right) $
with $?\in \left\{ \emptyset ,c-\infty \right\} $ and for every element $g\in \mathrm{H}^{2d+1-i}\left( \mathbf{Syn}_{\mathrm{rig}
,U^{o}\subset W_\ell, \alpha_k^{\rm hor}}^{c}\left( \mathbb{E}^\vee,P_2,r_2\right)
\right)$, we have 
$$
\nu _{\mathrm{rig},U, \ell,c}^{\ast,2d+1}  \left( \left\langle  f  ,\nu _{\mathrm{rig},U, \ell,c,\ast}^{2d+1-i}(g)\right\rangle
_{P_{1},r_{1};P_{2},r_{2}}\right)=\left\langle \nu _{\mathrm{rig},U,\ell}^{\ast,i}( f), g \right\rangle _{P_{1},r_{1};P_{2},r_{2}}\in  K_L.
$$
\end{lemma}
Consider $\mathbb{E}=\mathbb{I}$. Berthelot's trace map is compatible via $\tau_{\mathrm{rig},\ell,c}^{2d+1} \circ \nu_{\mathrm{rig},U,\ell,c,\ast}^{2d+1}$ with the trace map of Definition \ref{def:CupProd F Trace} and $\tau_{\mathrm{rig},\ell,c}^{2d+1}=\nu _{\mathrm{rig},U, \ell,c}^{\ast,2d+1}$. Then, for a general $\mathbb{E}$ satisfying the Assumption of \S \ref{sec:Egoodred} we get:  

\begin{cor}\label{cor:adjointdualityrig} For $f\in \mathrm{H}^{i}\left( \mathbf{Syn}^{?}\left( \mathbb{E}
,P_{1},r_{1}\right) \right) $ with $?\in \left\{\emptyset ,c-\infty
\right\} $ and $g$ as in Lemma \ref{lemma:pairnu}, we have 
$$
\left\langle \nu_{{\rm rig},U,\ell}^{\ast,i}\left( \tau _{\mathrm{rig},U,\ell}^i(f)\right) ,g\right\rangle
_{P_{1},r_{1};P_{2},r_{2}}=\left\langle f, \tau _{\mathrm{rig},U,\ell,c}^{2d+1-i}\circ \nu_{\mathrm{rig},U,\ell,c,\ast}^{2d+1-i}  \left(
g\right) \right\rangle _{P_{1},r_{1};P_{2},r_{2}}\in K_L.
$$
\end{cor} 
\begin{proof} It follows from  Corollary \ref{cor:compdual} and Lemma \ref{lemma:pairnu}.
\end{proof}

\subsection{Adjointness of Gysin morphisms and pull-back morphisms}\label{SS PolyAdj}

We suppose to be in the setting of \S \ref{sec:Gysin}, except we replace the complexes $\mathbf{Syn}\left( \mathcal{E},\mathfrak{E},r\right) $ with the  syntomic polynomial complexes $\mathbf{Syn}\left(\mathbb{E},P,r\right) $ of \S \ref{S PolyCmpl}, with and without support conditions. We use
the abbreviation $\mathbf{Syn}^c$ for $\mathbf{Syn}^{c-\infty} $. In particular, there is a closed
immersion $\iota $ of $\mathcal{F}$ in $\mathcal{X}$. 

We note that Lemma \ref{lemma:exactsyn} (and its compact version mentioned
at the end of \S \ref{sec:Gysin}) generalizes to give rise to exact sequences
\begin{equation}\label{CupProd F GysEx}
0\longrightarrow \mathbf{Syn}\left( \mathbb{E},P,r\right)
\longrightarrow \mathbf{Syn}\left( \mathbb{E},P,r\right)
^{\prime }\stackrel{\mathrm{Res}}{\longrightarrow }\mathbf{Syn}\left( \mathbb{E}_{\mid F}\left( -1\right),P,r \right) \left[
-1\right] \longrightarrow 0. 
\end{equation}
Taking cohomology we get morphsims for $i\geq 2$
$$
\mathrm{Gys}_{\mathrm{syn},F}^{P,r}\colon \mathrm{H}^{i-2}\left( 
\mathbf{Syn}\left( \mathbb{E}_{\mid F}\left( -1\right),P,r
\right) \right) \longrightarrow \mathrm{H}^{i}\left( \mathbf{Syn}
\left( \mathbb{E},P,r\right) \right) ,
$$and
$$
\mathrm{Gys}_{\mathrm{syn}^c,F}^{P,r}\colon \mathrm{H}^{i-2}\left( 
\mathbf{Syn}^c\left( \mathbb{E}_{\mid F}\left( -1\right),P,r
\right) \right) \longrightarrow \mathrm{H}^{i}\left( \mathbf{Syn}^c
\left( \mathbb{E},P,r\right) \right) ,
$$
In the opposite direction, there are pull-back morphisms
$$
\iota ^{\ast }\colon \mathrm{H}^{i}\left( \mathbf{Syn}\left( 
\mathbb{E},P,r\right) \right) \longrightarrow \mathrm{H}^{i}\left( 
\mathbf{Syn}\left( \mathbb{E}_{\mid F}^{\vee },P,r\right)\right)$$and
$$
\iota ^{\ast }\colon \mathrm{H}^{i}\left( \mathbf{Syn}^c\left( 
\mathbb{E},P,r\right) \right) \longrightarrow \mathrm{H}^{i}\left( 
\mathbf{Syn}^c\left( \mathbb{E}_{\mid F}^{\vee },P,r\right)\right).$$We write
$$
\mathrm{Gys}_{\mathrm{syn},F}^{\mathrm{top},P,r}:\mathrm{H}^{2d-1}\left( 
\mathbf{Syn}^c\left( \mathbb{I}_{F}\left( -1\right),P,r \right) \right)
\longrightarrow \mathrm{H}^{2d+1}\left( \mathbf{Syn}^c\left( \mathbb{
I},P,r\right) \right)
$$
for the map obtained in the case $i=2d+1$ and $\mathbb{E}=\mathbb{I}$ 
the unit object. Fix pairs $(P_1,r_1)$ and $P_2,r_2)$ so that $( P,r)=(P_1\ast P_2, r_1+r_2) $ 
is admissible for $\mathcal{X}_k$ and that $\left( P\left( pX\right) ,r-1\right) $ is
admissible for $\mathcal{F}_k$ in the sense of Definition \ref{def:CupProd DefAdm}. Then, we can consider the following syntomic pairings $\left\langle -,-\right\rangle_{P_{1},r_{1};P_{2},r_{2}}$ of
Definition \ref{DEF:CupProd Def2}:
$$ \mathrm{H}^{i}\left(
\mathbf{Syn}\left( \mathbb{E},P_1,r_1\right) \right) \otimes
_{K_{0,L}}\mathrm{H}^{2d-i+1}\left( \mathbf{Syn}^{c}\left(
\mathbb{E}^{\vee },P_2,r_2\right) \right) \longrightarrow K_{L}$$and pairings
$\left\langle -,-\right\rangle_{P_{1}\left( pX\right) ,r_{1}-1;P_{2},r_{2}}$
$$
\mathrm{H}^{i-2}\left( \mathbf{Syn}\left(
\mathbb{E}_{\mid F}\left( -1\right),P_1,r_1 \right) \right) \otimes
_{K_{0,L}}\mathrm{H}^{2d-i+1}\left( \mathbf{Syn}^{c}\left(
\mathbb{E}_{\mid F}^{\vee },P_2,r_2\right) \right) \longrightarrow
K_{L},
$$where we identify $
\mathbf{Syn}\left( \mathbb{E}_{\mid F}\left( -1\right),P_1,r_1
\right) =\mathbf{Syn}\left( \mathbb{E}
_{\mid F},P_{1}( pX) ,r_{1}-1\right) $ and we notice that $P_{1}\left( pX\right) \ast
P_{2}=P\left( pX\right) $. Then,

\begin{proposition}
\label{prop:GysDual}The map $\mathrm{Gys}_{\mathrm{syn},F}^{\mathrm{top}
,P,r}$ is compatible with the trace maps on $X$ and $F$ respectively and
$$
\left\langle \mathrm{Gys}_{\mathrm{syn},F}^{P_{1},r_{1}}\left( f\right)
,g\right\rangle _{P_{1},r_{1};P_{2},r_{2}}=\left\langle f,\iota ^{\ast
}\left( g\right) \right\rangle _{P_{1}\left( pX\right) ,r_{1}-1;P_{2},r_{2}}\in K_L
$$
for $f\in \mathrm{H}^{i-2}\left( \mathbf{Syn}\left( 
\mathbb{E}_{\mid F}\left( -1\right),P_1,r_1 \right) \right) $ and $g\in \mathrm{H}
^{2d+1-i}\left( \mathbf{Syn}^c\left( \mathbb{E}^{\vee
},P_2,r_2\right) \right) $. Analogously,
$$
\left\langle f, \mathrm{Gys}_{\mathrm{syn}^c,F}^{P_2,r_2}\left( g\right)
\right\rangle _{P_{1},r_{1};P_{2},r_{2}}=\left\langle \iota ^{\ast
}\left( f\right),g \right\rangle _{P_{1}\left( pX\right) ,r_{1}-1;P_{2},r_{2}}\in K_L
$$
for every $f\in \mathrm{H}
^i\left( \mathbf{Syn}\left( \mathbb{E}^{\vee
},P_1,r_1\right) \right) $ and $g\in \mathrm{H}^{2d+i-1}\left( \mathbf{Syn}^c\left( 
\mathbb{E}_{\mid F}\left( -1\right),P_2,r_2 \right) \right) $.

\end{proposition}

\begin{proof}
The compatibility of $\mathrm{Gys}_{\mathrm{syn}^{c},F}^{\mathrm{top},P,r}$
with the trace maps reduces to an analogous statement on de Rham cohomology
(in view of Definition \ref{def:CupProd F Trace}) and this is
proven in \cite[Prop. 4.3.7]{LLZ}. We prove the first equality. We have to show the compatibility of
the pairings with values in $\mathrm{H}^{2d+1}\left( \mathbf{Syn}^c\left( \mathbb{I},P,r\right) \right) $ and $\mathrm{H}^{2d-1}\left( 
\mathbf{Syn}^c\left( \mathbb{I}_{F}\left( -1\right),P,r \right) \right) $
respectively via $\mathrm{Gys}_{\mathrm{syn}^{c},F}^{\mathrm{top},P,r}$,
namely that we have
\begin{equation}\label{eq:GysAdjointiota}
\mathrm{Gys}_{\mathrm{syn},F}^{P_{1},r_{1}}\left( f\right) \cup _{P_{1},r_{1};P_{2},r_{2}}g=
\mathrm{Gys}^{\mathrm{top},P,r}_{\mathrm{syn}^c,F}\left( f\cup _{P_{1}
,r_{1};P_{2},r_{2}}\iota ^{\ast }\left( g\right) \right).
\end{equation} Write $\tilde{g}$ for a cocycle in the total  complex $ \mathbf{Syn}^{c}\left(
\mathbb{E}^{\vee },P_2,r_2\right)$ representing $g$. Let $\overline{g}$ be the image of $g$ in $ \mathbf{Syn}^{c}\left(\mathbb{E}^{\vee }_{\mid F},P_2,r_2\right)$; it represents $\iota^\ast(g)$. Denote by $g'$ the cocycle $g$ viewd as a cocycle in the total complex $ \mathbf{Syn}^{c}\left(
\mathbb{E}^{\vee },P_2,r_2\right)^\prime$ via (\ref{CupProd F GysEx}). 

Similarly, let $\tilde{f}$ be a cocycle in the total complex $\mathbf{Syn}\left( 
\mathbb{E}_{\mid F}\left( -1\right),P_1,r_1 \right)[-1]$  representing $f$
Let $f'$ be a lift of $\tilde{f}$  in the total complex $\mathbf{Syn}\left( 
\mathbb{E},P_1,r_1 \right)^\prime$ again using the exactness of (\ref{CupProd F GysEx}). Since $d(f)=0$, then $d(f')$ lies in $\mathbf{Syn}\left( \mathbb{E},P_1,r_1 \right)$. It is a cocycle as $d \circ d=0$ and the class of $d(f')$ in cohomology is $\mathrm{Gys}_{\mathrm{syn},F}^{P_{1},r_{1}}\left( f\right)$.

Write $z:=f' \cup _{P_{1},r_{1};P_{2},r_{2}} g'$ for the element in $\mathbf{Syn}^{c}\left(
\mathbb{I},P,r\right)^\prime$ induced by the cup product of total complexes associated to the cup product of double complexes satisfying Leibniz rule of Lemma \ref{lemma:CupProdCup}. As the morphisms in the exact sequence of total complexes (\ref{CupProd F GysEx}) are compatible with cup products, we conclude that $z$ lifts the cocycle $\tilde{f} \cup _{P_{1},r_{1};P_{2},r_{2}} \overline{g}$ in  $\mathbf{Syn}\left( 
\mathbb{I}_{\mid F}\left( -1\right),P,r\right)[-1]$. Hence $d(z)$ represents the cohomology class
$\mathrm{Gys}^{\mathrm{top},P,r}_{\mathrm{syn}^c,F}\left( f\cup _{P_{1}
,r_{1};P_{2},r_{2}}\iota ^{\ast }\left( g\right) \right)$. Since $d(g')=0$, then $d(z)=d(f')\cup _{P_ {1},r_{1};P_{2},r_{2}} g'$ which represents $\mathrm{Gys}_{\mathrm{syn},F}^{P_{1},r_{1}}\left( f\right)\cup _{P_{1},r_{1};P_{2},r_{2}} g$. This proves (\ref{eq:GysAdjointiota}) as claimed.
\end{proof}

\newcommand{\noopsort}[1]{\relax}

\providecommand{\bysame}{\leavevmode\hbox to3em{\hrulefill}\thinspace}
\providecommand{\MR}[1]{\relax}
\renewcommand{\MR}[1]{%
 MR \href{http://www.ams.org/mathscinet-getitem?mr=#1}{#1}.
}
\providecommand{\href}[2]{#2}
\newcommand{\articlehref}[2]{\href{#1}{#2}}

\bibliography{RefFab}{}
\bibliographystyle{alpha}

\end{document}